%% file: DCov2.tex
\documentclass[12pt]{amsart}
\usepackage{etex}
\usepackage{amssymb}
\usepackage[noadjust]{cite}
\usepackage{booktabs}
\usepackage{url}
\usepackage{hyphenat}
\usepackage{mathtools}
\usepackage{cancel} 
\usepackage{ifpdf}
\usepackage[T1]{fontenc}
\usepackage[utf8]{inputenc}
\usepackage{mdwtab}
\usepackage{enumitem}
\usepackage{xr-hyper}
\ifpdf
  \usepackage[pdftex]{graphicx}
  \usepackage[pdftex,lmargin=1in,rmargin=1in,tmargin=1in,bmargin=1in]{geometry}
  \usepackage[bookmarks=true, bookmarksopen=true,%
    bookmarksdepth=3,bookmarksopenlevel=2,%
    colorlinks=true,%
    linkcolor=blue,%
    citecolor=blue,%
    filecolor=blue,%
    menucolor=blue,%
    urlcolor=blue]{hyperref}
  \hypersetup{pdfauthor={Robert Lipshitz, Peter S. Ozsváth, and Dylan
      P. Thurston}}
\else
  \usepackage[dvips]{graphicx}
  \usepackage[dvips,lmargin=1in,rmargin=1in,tmargin=1in,bmargin=1in]{geometry}
  \usepackage[draft]{hyperref}
\fi
\usepackage{color}
\usepackage[all]{xypic}

\usepackage{tikz}
\usetikzlibrary{matrix,arrows}
\tikzset{modarrow/.style={->, dashed}}
\tikzset{tensoralgarrow/.style={double, double equal sign distance, -implies}}
\tikzset{algarrow/.style={->, thick}}
\tikzset{taa/.style={double, double equal sign distance, -implies}}

\usepackage[percent]{overpic}

\input{defs}
\input{macros}

\begin{document}
\title[Bordered Floer homology and the  branched double cover II]
{Bordered Floer homology and the spectral sequence of a branched
  double cover II: the spectral sequences agree}

\author[Lipshitz]{Robert Lipshitz}
\thanks{RL was supported by NSF Grants DMS-0905796 and DMS-1149800, and a Sloan Research
  Fellowship.}
\address{Department of Mathematics, University of North Carolina\\
  Chapel Hill, NC }
\email{lipshitz@math.columbia.edu}

\author[Ozsv\'ath]{Peter~S.~Ozsv\'ath}
\thanks{PSO was supported by NSF grant number DMS-1247323.}
\address {Department of Mathematics, Princeton University\\ 
Princeton, NJ 08544}
\email {petero@math.princeton.edu}

\author[Thurston]{Dylan~P.~Thurston}
\thanks {DPT was supported by NSF grants number DMS-1358638 and DMS-1507244.}
\address{Department of Mathematics,
         Indiana University,
         Bloomington, IN 47405}
\email{dpthurst@indiana.edu}

\begin{abstract}
  Given a link in the three-sphere, Ozsv{\'a}th and Szab{\'o} 
  showed that there is a spectral sequence starting at the Khovanov
  homology of the link and converging to the Heegaard Floer homology
  of its branched double cover.  The aim of this paper is to
  explicitly calculate this spectral sequence in terms of bordered
  Floer homology.  There are two primary ingredients in this
  computation: an explicit calculation of bimodules associated to Dehn
  twists, and a general pairing theorem for polygons. The previous
  part~\cite{LOT:DCov1} focuses on computing the bimodules; this part
  focuses on the pairing theorem for polygons, in order to prove that
  the spectral sequence constructed in the previous part agrees with
  the one constructed by Ozsv\'ath and Szab\'o.
\end{abstract} 

\maketitle

\tableofcontents
\input{intro}
\input{algebra}
\input{complexes}
\input{polygons}
\input{pairing}
\input{diagram}

\bibliographystyle{hamsalpha}\bibliography{heegaardfloer}
\end{document}

%% file: defs.tex
\newcommand{\RR}{\mathbb R}

\newcommand{\ZZ}{\mathbb Z}

\newcommand{\FF}{\mathbb F}

\newcommand{\Ker}{\mathrm{Ker}}

\newcommand{\bD}{\mathbb{D}}

\newcommand{\conn}{\mathbin \#}

\newcommand{\co}{\colon}

\newcommand{\bdy}{\partial}

\newcommand{\lbracket}{[}
\newcommand{\rbracket}{]}

\newcommand{\Hyph}{\text{-}}

\DeclareMathOperator{\Sym}{Sym}

\DeclareMathOperator{\spin}{spin}
\newcommand{\SpinC}{\spin^c}

\DeclareMathOperator{\ind}{ind}

\DeclareMathOperator{\ev}{ev}
\newcommand{\evbig}[1]{(\ev\mathop{\tilde{\times}}\kappa)_{#1}}

\DeclareMathOperator{\br}{br} 

\newcommand{\emb}{{\mathrm{emb}}} 

\newcommand{\Barop}{{\mathrm{Bar}}}

\newcommand\dr{\mathrm{dr}}

\theoremstyle{plain}
\newtheorem{theorem}{Theorem}
\numberwithin{equation}{section}

\newtheorem{proposition}[equation]{Proposition}
\newtheorem{lemma}[equation]{Lemma}
\newtheorem{corollary}[equation]{Corollary}

\newtheorem{convention}[equation]{Convention}
\newtheorem{definition}[equation]{Definition}

\theoremstyle{definition}

\theoremstyle{remark}
\newtheorem{example}[equation]{Example}
\newtheorem{remark}[equation]{Remark}

\newcommand{\HF}{\mathit{HF}}
\newcommand{\HFa}{\widehat {\HF}}

\newcommand{\CFa}{\widehat {\mathit{CF}}}

\newcommand{\x}{\mathbf x}
\newcommand{\y}{\mathbf y}

\newcommand\HH{\mathit{HH}}

\newcommand\Hochschild\HH

\newcommand{\Ainf}{A_\infty}

\newcommand{\Alg}{\mathcal{A}}

\newcommand\Blg{\mathcal{B}}
\newcommand\Clg{\mathcal{C}}

\newcommand{\alphas}{{\boldsymbol{\alpha}}}
\newcommand{\betas}{{\boldsymbol{\beta}}}
\newcommand{\gammas}{{\boldsymbol{\gamma}}}
\newcommand{\rhos}{{\boldsymbol{\rho}}}

\newcommand{\cM}{\mathcal{M}}
\newcommand{\Mod}{\cM}
\newcommand{\tMod}{\widetilde{\cM}}
\newcommand{\Cone}{\mathrm{Cone}}

\newcommand{\ocM}{\widebar{\cM}{}} 

\newcommand{\cMM}{\mathcal{MM}}

\newcommand{\DD}{\textit{DD}}
\newcommand{\DA}{\textit{DA}}

\newcommand{\AAm}{\textit{AA}} 
\newcommand{\CFD}{\mathit{CFD}}
\newcommand{\CFDD}{\mathit{CFDD}}
\newcommand{\CFA}{\mathit{CFA}}

\newcommand{\CFDA}{\mathit{CFDA}}
\newcommand{\CFDAa}{\widehat{\CFDA}}
\newcommand{\CFAA}{\mathit{CFAA}}
\newcommand{\CFAAa}{\widehat{\CFAA}}
\newcommand{\CFDDD}{\mathit{CFDDD}}
\newcommand{\CFDDA}{\mathit{CFDDA}}
\newcommand{\CFDAA}{\mathit{CFDAA}}
\newcommand{\CFAAA}{\mathit{CFAAA}}
\newcommand{\CFDDDa}{\widehat{\CFDDD}}
\newcommand{\CFDDAa}{\widehat{\CFDDA}}
\newcommand{\CFDAAa}{\widehat{\CFDAA}}
\newcommand{\CFAAAa}{\widehat{\CFAAA}}
\newcommand{\CFDa}{\widehat{\CFD}}

\newcommand{\CFDDa}{\widehat{\CFDD}}

\newcommand{\CFAa}{\widehat{\CFA}}

\newcommand{\cZ}{\mathcal{Z}}
\newcommand{\PtdMatchCirc}{\cZ}
\newcommand{\PMC}{\PtdMatchCirc}

\newcommand{\CircPts}{{\mathbf{a}}}
\newcommand{\glue}{\mathbin{\natural}}

\newcommand{\dg}{\textit{dg} }

\newcommand\Id{\mathbb{I}}
\newcommand\Ground{\mathbf k}
\newcommand\Groundl{\mathbf l}

\newcommand\DT{\boxtimes}
\newcommand\Gen{\mathfrak{S}}

\newcommand\Tensor{\mathcal T}
\newcommand\Zmod[1]{\mathbb{Z}/{#1}\mathbb{Z}}
\newcommand{\Field}{{\FF_2}}
\DeclareMathOperator{\nbd}{nbd}

\newcommand{\Heegaard}{\mathcal{H}}
\newcommand{\HD}{\Heegaard}

\renewcommand{\th}{^\text{th}}
\newcommand{\st}{^\text{st}}

\DeclareMathOperator{\Mor}{Mor}

\makeatletter
\newcommand\honestalg[3]{\bigl\lbracket
\begin{smallmatrix} #1\@ifempty{#3}{}{&#3} \\ #2 \end{smallmatrix}
\bigr\rbracket}

\makeatother
\newcommand{\lab}[1]{$\scriptstyle #1$}

\newcommand{\lsub}[2]{{}_{#1}#2}
\newcommand{\lsup}[2]{{}^{#1}\mskip-.6\thinmuskip#2}

\newcommand\Conf{\mathrm{Conf}}
\newcommand\oConf{\overline{\Conf}}
\newcommand\CDisk{\mathbb{D}}
\newcommand{\deltas}{{\boldsymbol{\delta}}}

\newcommand\rKh{\widetilde{Kh}}

\newcommand\CCFa{\mathbf{\widehat{CF}}}
\newcommand\CCFAa{\mathbf{\widehat{CFA}}}
\newcommand\CCFDa{\mathbf{\widehat{CFD}}}
\newcommand\CCFDAa{\mathbf{\widehat{CFDA}}}
\newcommand\CCFDDa{\mathbf{\widehat{CFDD}}}
\newcommand\CCFAAa{\mathbf{\widehat{CFAA}}}

\newcommand{\Filt}{\mathcal{F}}

\newcommand{\HFuk}{\mathsf{TFuk}}

\newcommand{\arcz}{\mathbf{z}}

\newcommand\IndI{\mathbb I}
\newcommand\IndJ{\mathbb J}

\renewcommand{\ij}[1]{i\!j_{#1}}

\newcommand\dimSXM{\dim_{SM}}
\newcommand\dimXM{\dim_{XM}}
\newcommand\CrossMatched{\mathcal{XM}}
\newcommand\SimpCrossMatched{\mathcal{SM}}
\newcommand\ChordMatched{\mathcal{N}}
\newcommand\MatchedComplex{{\mathcal C}_{[0]}}
\newcommand\tMatchedComplex{{\mathcal C}_{[t]}}
\newcommand\ChordMatchedComplex{{\mathcal C}_{\mathcal N}}
\newcommand\tsicComplex{{\mathcal C}_{tsic}}
\newcommand\ComplexCross{{\mathcal C}_{[\infty]}}

\newcommand\xis{\boldsymbol{\xi}}


%% file: macros.tex
\newread\testin

\graphicspath{{draws/}{mpdraws/}{}}
\makeatletter
\def\input@path{{}{draws/}}
\makeatother

\def\mathcenter#1{%
  \vcenter{\hbox{$#1$}}%
}

\DeclareRobustCommand{\widebar}[1]{\overline{#1}{}}

\makeatletter
\newcommand\mi@kern[1]{%
  \settowidth\@tempdima{$\mi@obj^{#1}$}
  \kern-\@tempdima
  #1
  \settowidth\@tempdima{$\mi@obj$}
  \kern\@tempdima
}

\newtoks\mi@toksp
\newtoks\mi@toksb
\DeclareRobustCommand{\manyindices}[5]{
  \def\mi@obj{#5}
  \mi@toksp\expandafter{\mi@kern{#2}}
  \mi@toksb\expandafter{\mi@kern{#1}}
  \@mathmeasure4\textstyle{#5_{#1}^{#2}}
  \@mathmeasure6\textstyle{#5_{#3}^{#4}}
  \dimen0-\wd6 \advance\dimen0\wd4
  \@mathmeasure8\textstyle{\hphantom{{}_{#1}^{#2}}#5^{\the\mi@toksp#4}_{\the\mi@toksb#3}}
  \hbox to \dimen0{}{\kern-\dimen0\box8}
}
\makeatother 

\usepackage{ifpdf}
\ifpdf
  
\else
  
\fi


%% file: intro.tex
\section{Introduction}

This paper concerns the relationship between Khovanov homology for
links $L$ in the three-sphere and the Heegaard Floer homology groups
of their branched double covers. These two link invariants are related,
as was shown in joint work of Zolt{\'a}n Szab{\'o} and the second
author:

\begin{theorem}\cite[Theorem~1.1]{BrDCov}
  \label{thm:MultiDiagramSpectralSequence}
  For any link $L\subset S^3$ there is a spectral sequence
  with $E_2$-page the reduced Khovanov homology of the mirror of
  $L$ and $E_\infty$-page $\HFa(\Sigma(L))$.
\end{theorem}

The differentials in the spectral sequence come from counts of
pseudo-holomorphic polygons. As such, they depend on a number of
choices.  Baldwin showed~\cite{Baldwin11:ss} that the filtered
chain homotopy type inducing the spectral sequence (and hence each
page after $E_2$) is a link invariant, see also~\cite{RobertsMultiDiagrams}.

The present work is a sequel to~\cite{LOT:DCov1}. In that paper, we
gave a combinatorially-defined spectral sequence from the
reduced Khovanov homology $\rKh(m(L))$ of the mirror of a link $L$ to
the Heegaard Floer homology $\HFa(\Sigma(L))$ of the double cover of
$S^3$ branched over~$L$.  That spectral sequence was \emph{a priori}
different from the one in \cite{BrDCov}.
The spectral sequence from \cite{LOT:DCov1} arises from
decomposing the branched double cover of a link into a union of
bordered manifolds as in~\cite{LOT1,LOT2}.  Bordered Floer homology
associates a bimodule to (the branched double cover of) each link crossing. The
bimodule associated to a crossing can be interpreted as a mapping cone
of bimodules
corresponding to the (branched double covers of the) two resolutions:

\begin{proposition} 
  \label{prop:TriangleCounting}
  \cite[Theorem~\ref*{DCov1:thm:MappingCone}]{LOT:DCov1}
  Consider the tangle consisting of $k$ horizontal segments lying in a
  plane. We think of this as a tangle in $S^2\times I$.  Its branched
  double cover is a bordered three-manifold representing
  $[0,1]\times \Sigma$. We denote this bordered three-manifold by $\Id$ (as it represents the identity cobordism).
  Next, consider the tangle $\sigma_i$ (respectively
  $\sigma_i^{-1}$) consisting of horizontal segments 
  with exactly one positive (respectively negative)
  crossing, which occurs between the $i\th$ and $(i+1)\st$ strands.
  Let ${\check \sigma}_i$ denote the
  ``anti-braidlike'' resolution of $\sigma_i$ and let $\CFDAa(\sigma_i)$ and $\CFDAa({\check
    \sigma}_i)$ denote the bimodules associated to the branched double covers
  of $\sigma_i$ and ${\check \sigma}_i$, respectively (which are bordered $3$-manifolds).  Then there
  are distinguished bimodule morphisms
  \begin{align*}
    F^-\co \CFDAa(\Id) &\to  \CFDAa({\check \sigma}_i)\\
    F^+\co \CFDAa({\check \sigma}_i) &\to \CFDAa(\Id)
  \end{align*}
  with the property that $\CFDAa(\sigma_i^{-1})$ and $\CFDAa(\sigma_i)$ are
  isomorphic to the mapping cones of the
  bimodule morphisms
  $F^-$ and $F^+$ respectively. (See Figure~\ref{fig:braid-gens}.)
\end{proposition}

\begin{remark}
  If we orient the tangle locally so that the strands are all oriented
  from left to right then the sign of the crossing appearing in the above
  proposition coincides with the sign of the crossing from knot theory.
\end{remark}

\begin{figure}
  \centering
  \includegraphics[scale=.857]{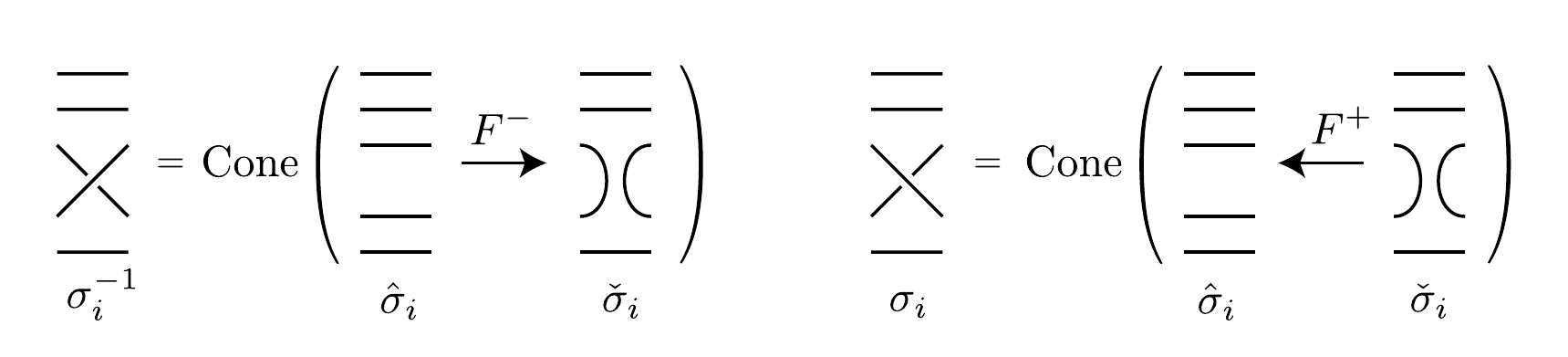}
  \caption{\textbf{Braid generators as mapping cones.}}
  \label{fig:braid-gens}
\end{figure}

Next, the pairing theorem for bordered Floer homology realizes the
Heegaard Floer complex of the branched double cover as a tensor
product of the constituent bimodules. Since each of these is a mapping
cone, the iterated tensor product inherits a filtration by
$\{0,1\}^n$.  The homology of the associated graded complex can be
identified with the chain groups computing reduced Khovanov
homology; see Corollary~\ref{cor:surg-tri-agree}, below.
Thus, one might expect that the
spectral sequence associated to the filtration is related to the 
spectral sequence from
Theorem~\ref{thm:MultiDiagramSpectralSequence}.

Indeed, the aim of the present paper is to identify these two spectral sequences; i.e. to prove:
\begin{theorem}
  \label{thm:IdentifySpectralSequences}
  The spectral sequence coming
  from~\cite[Theorem~\ref*{DCov1:thm:BorderedSpectralSequence}]{LOT:DCov1}
  agrees with the spectral sequence for multi-diagrams
  from~\cite[Theorem 1.1]{BrDCov}
  (Theorem~\ref{thm:MultiDiagramSpectralSequence} above). In fact, the
  filtered chain complexes inducing these two spectral sequences
  are filtered chain homotopy equivalent.
\end{theorem}

Our strategy of proof involves the following four key steps:
\begin{enumerate}
  \item 
    \label{step:DefinePolygons}
    Generalize the polygon counts which appear in Heegaard Floer homology (or more generally, in Fukaya categories)
    to the bordered context.
  \item 
    \label{step:PairingForPolygons}
    Give a pairing theorem which identifies polygon counts in a closed diagram in terms of pairings of 
    polygon counts appearing in the bordered context.
  \item 
    \label{step:ComputeTriangleMap}
    Construct a triple diagram corresponding to each crossing, and
    identify the corresponding triangle-counting morphism with the
    $\DA$ bimodule morphism appearing in
    Proposition~\ref{prop:TriangleCounting}.
  \item 
    \label{step:GlobalDiagram}
    Draw a Heegaard multi-diagram for the branched double cover with
    the property that the polygon counts in the diagram decompose as
    pairings, in the sense of Step~\ref{step:PairingForPolygons}, of the triangles appearing in
    Step~\ref{step:ComputeTriangleMap}. 
\end{enumerate}

In the above outline, Step~\ref{step:DefinePolygons} is a
routine. Step~\ref{step:ComputeTriangleMap} will follow quickly from a
uniqueness property of the maps appearing in
Proposition~\ref{prop:TriangleCounting}, which was
established in~\cite{LOT:DCov1}. Step~\ref{step:GlobalDiagram} is in fact
the diagram used in~\cite{LOT:DCov1}, which in turn is a stabilization of a
diagram studied by Greene~\cite{Greene:spanning-tree}. 

Thus, the crux of the present paper is the pairing theorem in
Step~\ref{step:PairingForPolygons}.  To this end, it is important to
formulate a well-defined object: holomorphic polygon counts depend on
the precise conformal parameters of the multi-diagram, while pairing
theorems (such as the pairing theorem for closed
manifolds,~\cite[Theorem~\ref*{LOT1:thm:TensorPairing}]{LOT1}) depend
on degenerating these parameters. Thus, we need an object depending on
holomorphic polygon counts which is sufficiently robust to be
invariant under variation of parameters.  This object is provided by
the notion of a {\em $\IndI$-filtered chain complex of attaching circles},
which consists
of a collection of sets of attaching circles $\{\betas^i\}$ together
with a suitable collection of chains
$\eta^{i_1<i_2}\in\CFa(\betas^{i_1},\betas^{i_2},z)$
(Definition~\ref{def:ChainComplex}). 
(This is a special case of a notion of twisted complexes
which play a prominent role in the theory of Fukaya categories~\cite{SeidelBook}; see also Remark~\ref{rem:ChainComplexIsTypeD}.)
The key point is that this data
makes $\bigoplus_{i\in\IndI} \CFa(\alphas,\betas^i,z)$ into a filtered
chain 
complex (Definition~\ref{def:AssociatedComplex}):

\begin{proposition}
  \label{intro:ChainComplexesOfAttachingCircles}
  Given a finite partially ordered set $\IndI$, an $\IndI$-filtered chain complex of attaching circles $\{\betas^i\}_{i\in\IndI}$, and one more set of attaching circles
  $\alphas$, there is a naturally associated $\IndI$-filtered complex denoted
  $\CFa(\alphas,\{\betas^i\}_{i\in\IndI},z)$ whose associated graded complex is
  \[\bigoplus_{i\in\IndI}\CFa(\alphas,\betas^i,z). \]
\end{proposition}
(This is reformulated and proved as
Proposition~\ref{prop:ChainComplexesOfAttachingCircles}, below.)

Now, the spectral sequence from
Theorem~\ref{thm:MultiDiagramSpectralSequence} can be thought of as
coming from a gluing statement for complexes of attaching circles, as follows:

\begin{proposition}
  \label{intro:GlueChainComplexes}
  Fix surfaces-with-boundary $\Sigma_1$ and $\Sigma_2$ with common boundary~$C$, and let
  $\{\betas^i\}_{i\in\IndI}$ be a chain complex of attaching circles in
  $\Sigma_1$, filtered by a partially ordered set $\IndI$, and let
  $\{\gammas^j\}_{j\in\IndJ}$ be a chain complex of attaching circles in
  $\Sigma_2$ filtered by a partially ordered set $\IndJ$.  Then these 
  chain complexes of attaching circles can be glued to form a natural
  $\IndI\times\IndJ$-filtered chain complex of attaching circles 
  $\betas\conn\gammas$ in $\Sigma_1\cup_C \Sigma_2$.
\end{proposition}

The construction of the $\IndI\times\IndJ$-filtered chain complex is given in Definition~\ref{def:ConnectedSum} below; 
a more precise version of Proposition~\ref{intro:GlueChainComplexes}
is given as Proposition~\ref{prop:GlueChainComplexes}.

Simplified versions of Step~\ref{step:DefinePolygons} are provided by 
the following two bordered analogues of Proposition~\ref{intro:ChainComplexesOfAttachingCircles}. 

\begin{theorem}
  \label{intro:ChainComplexToCFA-CFD}
  Let $\Sigma$ be a surface-with-boundary and $\{\betas^i\}_{i\in\IndI}$ an
  $\IndI$-filtered chain complex of attaching circles in $\Sigma$.  Let $\alphas$ be a collection of
  $\alpha$-arcs and $\alpha$-circles so that each $(\Sigma,\alphas,\betas^i)$
  is a bordered Heegaard diagram. Then there is a
  naturally associated $\IndI$-filtered $\Ainf$-module
  $\CCFAa(\alphas,\allowbreak\{\betas^i\}_{i\in\IndI},
  \allowbreak z)$ (respectively
  $\IndI$-filtered type $D$ structure
  $\CCFDa(\alphas,\{\betas^i\}_{i\in\IndI},z)$)
whose associated graded object is
  $\bigoplus_{i\in\IndI}\CFAa(\Sigma,\alphas,\betas^i,z)$
  (respectively $\bigoplus_{i\in\IndI}\CFDa(\Sigma,\alphas,\betas^i,z)$).
\end{theorem}
(This is proved as Propositions~\ref{prop:DefCCFAa} and~\ref{prop:DefCCFDa}.)

The tensor product over the bordered algebra $\Alg(\PMC)$ of an $\IndI$-filtered $\Ainf$-module and a
$\IndJ$-filtered type $D$ structure is naturally an $\IndI\times\IndJ$-filtered
chain complex; see for instance~\cite[Section~\ref*{DCov1:sec:FilteredComplexes}]{LOT:DCov1} and Lemma~\ref{lem:FilteredTensorProduct} below.
A simplified version of the pairing theorem for polygons needed in
Step~\ref{step:PairingForPolygons} can be stated as follows.

\begin{theorem}
  \label{intro:PairingForPolygons}
  Let $\Sigma_1$ and $\Sigma_2$ be surfaces, and
  $\alphas^i\subset \Sigma_i$ be $\alpha$-curves which intersect the
  boundary of $\Sigma_i$ in a pointed matched circle $\PMC_i=\partial
  \Sigma_i\cap\alphas_i$, with
  $\PMC_1=-\PMC_2$. Let $\alphas=\alphas_1\cup \alphas_2\subset \Sigma=\Sigma_1\cup_\bdy\Sigma_2$.
  Fix an $\IndI$-filtered chain complex of attaching circles $\{\betas^i\}_{i\in\IndI}$
  in $\Sigma_1$ and a $\IndJ$-filtered chain complex of attaching circles  $\{\gammas^j\}_{j\in\IndJ}$
  in $\Sigma_2$. We can form the 
  $\IndI\times\IndJ$-filtered chain complex
  of attaching circles $\betas\conn\gammas$, as in Proposition~\ref{intro:GlueChainComplexes}. 
  Then there is a homotopy equivalence of $\IndI\times\IndJ$-filtered chain complexes
  \[\CFa(\alphas_1\cup
  \alphas_2,\{\betas^i\conn\gammas^j\}_{(i,j)\in\IndI \times\IndJ},z)
  \simeq
  \CCFAa(\alphas_1,\{\betas^i\}_{i\in\IndI})
  \DT 
  \CCFDa(\alphas_2,\{\gammas^j\}_{j\in\IndJ}).\]
\end{theorem}
(This is proved as Theorem~\ref{thm:PolygonPairing}, in Section~\ref{sec:PairingTheorem}.)

We prove Theorem~\ref{intro:PairingForPolygons} by applying ``time dilation'', the deformation
used in~\cite{LOT1} to prove the pairing theorem for three-manifolds.
As a preliminary step, though, we will need to translate time,
as in the proof of the self-pairing theorem~\cite[Theorem~\ref*{LOT2:thm:Hochschild}]{LOT2}.

The generalizations of Theorems~\ref{intro:ChainComplexToCFA-CFD}
and~\ref{intro:PairingForPolygons} to the bimodule case
(Theorem~\ref{thm:PolygonPairingDA} below), which we need in
Step~\ref{step:PairingForPolygons}, is a fairly routine extension.

\subsection{Organization}

In
Section~\ref{sec:filtered-alg}, we remind the reader of the definition of
the $\IndI$-filtered algebraic objects which will appear throughout:
complexes, $\Ainf$-modules, type $D$ structures, and bimodules. We
also explain how the tensor products of filtered objects give filtered
complexes.  Section~\ref{sec:Complexes} concerns
pseudo-holomorphic polygons in Heegaard multi-diagrams.  In that
section, we explain the definition of chain complexes of attaching
circles, showing how these can be used to construct filtered chain
complexes (Proposition~\ref{intro:ChainComplexesOfAttachingCircles}),
and how they can be glued
(Proposition~\ref{intro:GlueChainComplexes}).  In
Section~\ref{sec:Polygons}, the constructions are generalized to
bordered multi-diagrams. That section contains the proof of
Theorem~\ref{intro:ChainComplexToCFA-CFD}
(restated as
Propositions~\ref{prop:DefCCFAa} and~\ref{prop:DefCCFDa} in the type
$A$ and $D$ cases, respectively).  
Proposition~\ref{prop:BimoduleMorphism} is an analogous result for bimodules.

Having described all the ingredients in the statement of the pairing
theorem for polygons, in Section~\ref{sec:PairingTheorem} 
we turn to the precise statement and proof of
Theorem~\ref{intro:PairingForPolygons}. (Theorem~\ref{thm:PolygonPairing}
is the simpler, module version and Theorem~\ref{thm:PolygonPairingDA} is the more general,
bimodule version.)  As a consequence of the pairing theorem, we
deduce that the surgery exact sequence implied by bordered Floer
homology~\cite[Section 11.2]{LOT1} agrees with the original surgery sequence from~\cite{OS04:HolDiskProperties}. 
In Section~\ref{subsec:IdentifySS}, we describe
the Heegaard multi-diagram for double covers, and show how it can be
used, in conjunction with the pairing theorem, to prove
Theorem~\ref{thm:IdentifySpectralSequences}.

\subsection{Further remarks}

The present paper is devoted to computing a spectral sequence from
Khovanov homology to the Heegaard Floer homology of a branched double
cover. It is worth pointing out that this spectral sequence has a
number of generalizations to other contexts. For example, Grigsby and
Wehrli have established an analogous spectral sequence starting at
various colored versions of Khovanov homology and converging to knot
homology of the branch locus in various branched covers of $L$,
leading to a proof that these colored Khovanov homology groups detect
the unknot~\cite{GrigsbyWehrli:detects}. Bloom~\cite{Bloom:ss} proved an analogue
of Theorem~\ref{thm:MultiDiagramSpectralSequence} using Seiberg-Witten
theory in place of Heegaard Floer homology. More recently, Kronheimer
and Mrowka~\cite{KronheimerMrowka} have proved an analogue with $\ZZ$
coefficients, converging to a version of instanton link homology,
showing that Khovanov homology detects the unknot.

We have relied in this paper extensively on the machinery of bordered
Floer homology, which gives a technique for computing $\HFa$ for
three-manifolds. Another powerful technique for computing this
invariant is the technique of {\em nice diagrams};
see~\cite{SarkarWang07:ComputingHFhat}. Indeed, although nice diagrams
were originally conceived as a way to compute $\HFa$ for closed
three-manifolds, it has been extended to a tool for computing triangle
maps in~\cite{LMW06:CombinatorialCobordismMaps}, see
also~\cite{Sarkar11:IndexTriangles}. At present, a ``nice'' technique
for directly computing the polygons needed in
Theorem~\ref{thm:MultiDiagramSpectralSequence} has not been
developed. See~\cite[Theorem~11.10]{ManolescuOzsvath:surgery} for an
alternative approach studying these maps.

Ideas from~\cite{AurouxBordered} suggest
another route to the pairing theorem for polygons.

\subsection{Acknowledgements}
The authors thank the Simons Center for Geometry and Physics
for its hospitality while this work was written up. We thank Mohammed
Abouzaid and Fr{\'e}d{\'e}ric Bourgeois for helpful
conversations. Finally, we thank the referees for thorough readings
and carefully considered comments and corrections.


%% file: algebra.tex
\section{Algebraic preliminaries}\label{sec:filtered-alg}

In this section we recall some notions about filtered objects, partly
to fix notation. Much of the material overlaps with~\cite[Section~\ref*{DCov1:sec:FilteredComplexes}]{LOT:DCov1}; we have also included here some of the discussion
from~\cite[Section~\ref*{LOT2:sec:algebra-modules}]{LOT2}.

We use $\Mor$ to denote the chain complex of morphisms in a \dg
category. For example, in the \dg category of chain complexes (over a ground ring
$\Ground$ of characteristic $2$), $\Mor(C_*,D_*)$ denotes the $\Ground$-module maps from
$C_*$ to $D_*$, with differential given by $d(f) =
f\circ\bdy_C+\bdy_D\circ f$. So, for instance, the cycles in
$\Mor(C_*,D_*)$ are the chain maps. For more details on the categories
we work with, see~\cite[Section 2]{LOT2}.

\begin{definition}
  \label{def:IFilteredCx}
  Let $\IndI$ be a finite, partially ordered set. An
  \emph{$\IndI$-filtered chain complex} is a collection $\{C^i\}_{i\in
    \IndI}$ of chain complexes over $\Field$ together with a
  collection of degree-zero morphisms $D^{i<j}\in \Mor(C^i,C^j[1])$ for each $i,j\in
  I$ with $i<j$, satisfying the compatibility condition
  \[ 
  d(D^{i<k}) = \sum_{\{j\mid i<j<k\}} D^{j<k}\circ D^{i<j}, 
  \]
  where $d$ denotes the differential on the morphism space $\Mor(C^i,C^k[1])$.
\end{definition}

(Here and elsewhere, $[1]$ denotes a grading shift by $1$. In the rest of this paper, the modules and chain complexes considered will be ungraded, but we will keep track of gradings in this background section.)

In particular, the compatibility condition implies that if $i$ and $j$
are consecutive,
then $D^{i<j}$ is a chain map.

This can be reformulated in the following more familiar terms.  Given
$\{C^i, D^{i<j}\}$, we can form the graded vector space $C= \bigoplus_{i\in \IndI} C^i$,
equipped with the degree $-1$ endomorphism $D\co C \to C$ defined by
$D = \sum_i \partial^i +\sum_{i<j} D^{i<j},$ where $\partial^i$ is the
differential on~$C^i$. The
compatibility condition is simply the statement that $D$ is a
differential. The pair
$(C,D)$ is naturally a filtered complex.
The associated graded complex to $C$ is simply $\bigoplus_{i\in \IndI}
C^i$, equipped with the differential $\partial$ which is the sum of
the differentials on the~$C^i$.

This notion has several natural generalizations to the case of
$\Ainf$-modules. The one we will use is the following. (See also
Remark~\ref{rem:filt-mod}.) First, fix an $\Ainf$-algebra $\Alg$ over
a ground ring $\Ground$ of characteristic $2$, and assume that $\Alg$
is \emph{bounded} in the sense that the operations $\mu_i$ on $\Alg$
vanish identically for all sufficiently large $i$.  Recall that the
category of $\Ainf$-modules over $\Alg$ is a \dg
category~\cite[Section 2.2.2]{LOT2}.
\begin{definition}\label{def:filt-Ainf-module}
  Let $\IndI$ be a finite, partially ordered set.
  An \emph{$\IndI$-filtered $\Ainf$-module over $\Alg$} is a collection $\{M^i\}_{i\in \IndI}$ of
  $\Ainf$-modules over $\Alg$, equipped with a preferred morphism
  $F^{i<j}\in \Mor(M^i,M^j[1])$ for each pair $i,j\in I$ with $i<j$, satisfying
  the compatibility condition
  \begin{equation}
    \label{eq:CompatibilityCondition}
  d(F^{i<k})= \sum_{j\mid i<j<k} F^{j<k}\circ F^{i<j},
  \end{equation}
  where $d$ denotes the differential on the morphism space $\Mor(M^i,M^k[1])$.
  Moreover, we say that the filtered module is {\em bounded} 
  if for all sufficiently large $n$, the maps
  \[ m_n\co M^i\otimes
  \overbrace{\Alg\otimes\dots\otimes\Alg}^{n-1}\to
  M^i[2-n]\qquad\text{and}\qquad F^{i<k}_n\co M^i\otimes \overbrace{\Alg\otimes\dots\otimes\Alg}^{n-1}\to M^k[2-n]\] 
  vanish.
\end{definition}

We can form the graded $\Ground$-module $M=\bigoplus_{i\in \IndI} M^i$. The higher products $m_n$ on the $M^i$ and the maps $F^{i<j}$
can be assembled to give a map
$M\otimes \Tensor^*(\Alg[1])\to M[1]$ defined,
for $x_i\in M^i$ and $a_1,\dots,a_n\in \Alg$, by
\begin{equation}\label{eq:filt-amalg}
x_i\otimes a_1\otimes\dots\otimes a_n \mapsto m^i_{n+1}(x_i,a_1,\dots,a_n) +
\sum_{i<j} F^{i<j}_{n+1}(x_i,a_1,\dots,a_n),
\end{equation}
where $m^i$
denotes the $\Ainf$ action on $M^i$.  The
compatibility condition on the $F^{i<j}$ is equivalent to
Formula~(\ref{eq:filt-amalg}) defining an $\Ainf$-module structure.
The boundedness condition ensures that the above constructions define
a map $M \otimes \overline{\Tensor}^*(\Alg[1]) \to M[1]$,
replacing the tensor algebra $\Tensor^*(\Alg)=\bigoplus_{i\geq 0}\Alg^{\otimes i}$ by its completion
$\overline{\Tensor}^*(\Alg)=\prod_{i\geq 0}\Alg^{\otimes i}$.

Recall that a type $D$ structure (a variant of twisted complexes) over
the $\Ainf$-algebra $\Alg$ over $\Ground=\bigoplus_{i=1}^N\Field$ is a
$\Ground$-module $X$ together with a $\Ground$-linear map $\delta^1\co
X\to\Alg[1]\otimes_\Ground X$ so that
\[
\mathcenter{
\begin{tikzpicture}
  \node at (0,0) (top) {};
  \node at (0,-1) (delta1) {$\delta^1$};
  \node at (-1,-4) (mu) {$\mu_1$};
  \node at (-1,-5) (bl) {};
  \node at (0,-5) (bc) {};
  \draw[->] (top) to (delta1);
  \draw[->] (delta1) to (mu);
  \draw[->] (mu) to (bl);
  \draw[->] (delta1) to (bc);
\end{tikzpicture}}
\mathcenter{+}
\mathcenter{
\begin{tikzpicture}
  \node at (0,0) (top) {};
  \node at (0,-1) (delta1) {$\delta^1$};
  \node at (0,-2) (delta2) {$\delta^1$};
  \node at (-1,-4) (mu) {$\mu_2$};
  \node at (-1,-5) (bl) {};
  \node at (0,-5) (bc) {};
  \draw[->] (top) to (delta1);
  \draw[->] (delta1) to (mu);
  \draw[->] (delta2) to (mu);
  \draw[->] (mu) to (bl);
  \draw[->] (delta1) to (delta2);
  \draw[->] (delta2) to (bc);
\end{tikzpicture}}
\mathcenter{+}
\mathcenter{
\begin{tikzpicture}
  \node at (0,0) (top) {};
  \node at (0,-1) (delta1) {$\delta^1$};
  \node at (0,-2) (delta2) {$\delta^1$};
  \node at (0,-3) (delta3) {$\delta^1$};
  \node at (-1,-4) (mu) {$\mu_3$};
  \node at (-1,-5) (bl) {};
  \node at (0,-5) (bc) {};
  \draw[->] (top) to (delta1);
  \draw[->] (delta1) to (mu);
  \draw[->] (delta2) to (mu);
  \draw[->] (delta3) to (mu);
  \draw[->] (mu) to (bl);
  \draw[->] (delta1) to (delta2);
  \draw[->] (delta2) to (delta3);
  \draw[->] (delta3) to (bc);
\end{tikzpicture}}
\mathcenter{+\cdots=0.}
\]
(See~\cite[Definition~\ref*{LOT2:def:TypeD}]{LOT2}.)
Over a \dg algebra only the first two terms occur.
 If $X$
and $Y$ are type $D$ structures, the morphism space $\Mor^D(X,Y)$ is
the space of $\Ground$-linear maps $h^1\co X \to \Alg\otimes Y$, equipped with a
differential consisting of terms each of which applies $\delta^1_X$
some number of times, then $h^1$, and then finally $\delta^1_Y$ a
number of times, and then feeds all the $n$ resulting algebra elements
into the operation $\mu_n$ on $\Alg$. When $\Alg$ is a $\dg$ algebra, this differential can be written
\[
\mathcenter{d(h^1)=}
\mathcenter{
\begin{tikzpicture}
  \node at (0,0) (top) {};
  \node at (0,-2) (delta1) {$h^1$};
  \node at (-1,-4) (mu) {$\mu_1$};
  \node at (-1,-5) (bl) {};
  \node at (0,-5) (bc) {};
  \draw[->] (top) to (delta1);
  \draw[->] (delta1) to (mu);
  \draw[->] (mu) to (bl);
  \draw[->] (delta1) to (bc);
\end{tikzpicture}}
\mathcenter{+}
\mathcenter{
\begin{tikzpicture}
  \node at (0,0) (top) {};
  \node at (0,-1) (delta1) {$\delta^1_X$};
  \node at (0,-2) (delta2) {$h^1$};
  \node at (-1,-4) (mu) {$\mu_2$};
  \node at (-1,-5) (bl) {};
  \node at (0,-5) (bc) {};
  \draw[->] (top) to (delta1);
  \draw[->] (delta1) to (mu);
  \draw[->] (delta2) to (mu);
  \draw[->] (mu) to (bl);
  \draw[->] (delta1) to (delta2);
  \draw[->] (delta2) to (bc);
\end{tikzpicture}}
\mathcenter{+}
\mathcenter{
\begin{tikzpicture}
  \node at (0,0) (top) {};
  \node at (0,-2) (delta1) {$h^1$};
  \node at (0,-3) (delta2) {$\delta^1_Y$};
  \node at (-1,-4) (mu) {$\mu_2$};
  \node at (-1,-5) (bl) {};
  \node at (0,-5) (bc) {};
  \draw[->] (top) to (delta1);
  \draw[->] (delta1) to (mu);
  \draw[->] (delta2) to (mu);
  \draw[->] (mu) to (bl);
  \draw[->] (delta1) to (delta2);
  \draw[->] (delta2) to (bc);
\end{tikzpicture}}.
\]
Suppose for simplicity that $\Alg$ is a $\dg$ algebra. Given
$g^1\in\Mor^{\Alg}(X,Y)$ and $h^1\in\Mor^{\Alg}(Y,Z)$, we can define
their composite $(h^1\circ g^1)$ by the expression
\[ 
\mathcenter{
\begin{tikzpicture}
  \node at (0,0) (top) {};
  \node at (0,-1) (delta1) {$g^1$};
  \node at (0,-2) (delta2) {$h^1$};
  \node at (-1,-3) (mu) {$\mu_2$};
  \node at (-1,-4) (bl) {};
  \node at (0,-4) (bc) {};
  \draw[->] (top) to (delta1);
  \draw[->] (delta1) to (mu);
  \draw[->] (delta2) to (mu);
  \draw[->] (mu) to (bl);
  \draw[->] (delta1) to (delta2);
  \draw[->] (delta2) to (bc);
\end{tikzpicture}}\,\,.
\]

We define a filtered type $D$ structure similarly to a filtered
$\Ainf$-module:
\begin{definition}\label{def:filt-D}
  Fix a \dg algebra $\Alg$, 
  and let $\IndI$ be a finite, partially ordered
  set.  An \emph{$\IndI$-filtered type $D$ structure over $\Alg$} is a
  collection $\{P^i\}_{i\in \IndI}$ of type $D$ structures over $\Alg$,
  equipped with preferred degree-zero morphisms $h^1_{i<j}\co P^i \to \Alg[1]
  \otimes P^j$ (for
  each pair $i,j\in I$ with $i<j$) satisfying the compatibility
  condition~\eqref{eq:CompatibilityCondition}.
\end{definition}

\begin{remark}\label{rem:type-d-ainf}
  Over an $\Ainf$-algebra, type $D$ structures form an
  $\Ainf$-category~\cite[Lemma \ref*{LOT2:lem:D-higher-comp}]{LOT2},
  making this definition more complicated.
  We will not need that level of generality in our present
  considerations.
\end{remark}

Let $M=(\{M_i\}_{i\in\IndI},\{F^{i<i'}\}_{i<i'\in\IndI})$ be an $\IndI$-filtered $\Ainf$-module over $\Alg$
and let $P=(\{P^j\}_{j\in\IndJ},\allowbreak\{h^1_{j<j'}\}_{j<j'\in\IndJ})$ be a $\IndJ$-filtered type $D$ structure over~$\Alg$.
Suppose moreover that $M$ is bounded.
We can form the tensor product
\[ M\DT P = \bigoplus_{i\times j\in\IndI\times\IndJ} M^i\otimes P^j, \]
endowed with a differential
\[ \partial \co M^i\otimes P^j \to M^i\otimes P^j[1] \]
given by
\[ 
\mathcenter{\partial=}
\mathcenter{
\begin{tikzpicture}
  \node at (0,0) (tl) {};
  \node at (1,0) (tr) {};
  \node at (1,-1) (delta) {$\delta_{j}$};
  \node at (0,-2) (m) {$m^i$};
  \node at (0,-3) (bl) {};
  \node at (1,-3) (br) {};
  \draw[modarrow] (tr) to (delta);
  \draw[modarrow] (delta) to (br);
  \draw[modarrow] (tl) to (m);
  \draw[modarrow] (m) to (bl);
  \draw[tensoralgarrow] (delta) to (m);
\end{tikzpicture}
},
\]
where $\delta_j\co P^j \to {\overline{\Tensor}^*}(\Alg[1]) \otimes P^j$
is the map obtained by iterating $\delta^1$ on $P^j$
and
\[m^i\co M_i\otimes\overline{\Tensor}^*(\Alg[1]) \to M_i[1]\]
is the map induced by the $\Ainf$-action.
Define maps $D^{i\times j<i'\times j'}$ by the following expression:
\begin{equation}
  \label{eq:DefineDifferentialOnDT}
  \mathcenter{\displaystyle D^{i\times j<i'\times j'}=
    \sum_{n=0}^\infty
    \sum_{j=j_0<\dots<j_n=j'}
  \mathcenter{
  \begin{tikzpicture}
    \node at (-2,0) (tlblank) {};
    \node at (1,0) (trblank) {};
    \node at (1,-1) (delta1) {$\delta_{j_0}$};
    \node at (1,-2) (f1) {$h^1_{j_0<j_1}$};
    \node at (1,-3) (delta2) {$\delta_{j_{1}}$};
    \node at (1,-4) (rdots) {$\vdots$};
    \node at (1,-5) (delta3) {$\delta_{j_{n-1}}$};
    \node at (1,-6) (f2) {$h^1_{j_{n-1}<j_n}$};
    \node at (1,-7) (delta4) {$\delta_{j_n}$};
    \node at (-2,-8) (m) {$F^{i\leq i'}$};
    \node at (1,-9) (brblank) {};
    \node at (-2,-9) (blblank) {};
    \draw[modarrow] (tlblank) to (m);
    \draw[modarrow] (m) to (blblank);
    \draw[modarrow] (trblank) to (delta1);
    \draw[modarrow] (delta1) to (f1);
    \draw[modarrow] (f1) to (delta2);
    \draw[modarrow] (delta2) to (rdots);
    \draw[modarrow] (rdots) to (delta3);
    \draw[modarrow] (delta3) to (f2);
    \draw[modarrow] (f2) to (delta4);
    \draw[modarrow] (delta4) to (brblank);
    \draw[tensoralgarrow, bend right=10] (delta1) to (m);
    \draw[tensoralgarrow, bend right=8] (delta2) to (m);
    \draw[tensoralgarrow, bend right=5] (delta3) to (m);
    \draw[tensoralgarrow] (delta4) to (m);
    \draw[algarrow, bend right=9] (f1) to (m);
    \draw[algarrow, bend right=2] (f2) to (m);
  \end{tikzpicture}},
}
\end{equation}
with the understanding that 
\[ F^{i\leq i'}
= \left\{\begin{array}{ll}
F^{i<i'} & {\text{if $i<i'$}} \\
m^i & {\text{if $i=i'$}}
\end{array}\right.
\]

We can formulate a boundedness condition on filtered type~$D$ structures
analogous to the one for type~$A$ structures:

\begin{definition}
  \label{def:BoundedTypeD}
  The
  filtered type $D$ structure on~$P$ is called {\em bounded} if for each
  $j\in\IndJ$, sufficiently large iterates of $\delta^1$ on $P_j$
  vanish, i.e., the map $\delta_j$ maps to
  $\Tensor^*(\Alg[1])\otimes P^j\subset
  \overline{\Tensor^*}(\Alg[1])\otimes P^j$. 
\end{definition}

If $M$ is not bounded, but $P$ is, the tensor product $M\DT P$ still makes sense.
In fact, the following is part of~\cite[Lemma~\ref*{DCov1:lem:DT-descends}]{LOT:DCov1}.
\begin{lemma}
  \label{lem:FilteredTensorProduct}
  Let $M$ be an $\IndI$-filtered $\Ainf$-module over $\Alg$
  and let $P$ be a $\IndJ$-filtered type $D$ structure over $\Alg$. Assume
  that one of $M$ or $P$ is bounded, so $M\DT P$ is defined.
  Then, $M\DT P$ is naturally an $\IndI\times \IndJ$-filtered chain complex.
\end{lemma}

\subsection{Bimodules}

There are analogous constructions of filtered bimodules, subsuming
both filtered $\Ainf$-modules and filtered type $D$ structures.  Fix
$\Ainf$-algebras $\Alg$ and $\Blg$ over ground rings
$\Ground=\bigoplus_{i=1}^N\Field$ and
$\Groundl=\bigoplus_{i=1}^{N'}\Field$, respectively. Recall that a
\emph{type \DA\ bimodule} over $\Alg$ and $\Blg$ consists of a
$(\Ground,\Groundl)$-bimodule $X$ together with a map $\delta^1\co
X\otimes_\Groundl T^*(\Blg[1])\to (\Alg\otimes_\Ground X)[1]$
satisfying:
\[
\mathcenter{
\begin{tikzpicture}
  \node at (0,0) (top) {};
  \node at (0,-1) (delta1) {$\delta^1$};
  \node at (-1,-4) (mu) {$\mu_1$};
  \node at (-1,-5) (bl) {};
  \node at (0,-5) (bc) {};
  \node at (1,0) (tr) {};
  \draw[taa] (tr) to (delta1);
  \draw[->] (top) to (delta1);
  \draw[->] (delta1) to (mu);
  \draw[->] (mu) to (bl);
  \draw[->] (delta1) to (bc);
\end{tikzpicture}}
\mathcenter{+}
\mathcenter{
\begin{tikzpicture}
  \node at (0,0) (top) {};
  \node at (0,-1) (delta1) {$\delta^1$};
  \node at (0,-2) (delta2) {$\delta^1$};
  \node at (-1,-4) (mu) {$\mu_2$};
  \node at (-1,-5) (bl) {};
  \node at (0,-5) (bc) {};
  \node at (1,0) (tr) {};
  \node at (1.1,0) (trr) {};
  \draw[taa] (tr) to (delta1);
  \draw[taa] (trr) to (delta2);
  \draw[->] (top) to (delta1);
  \draw[->] (delta1) to (mu);
  \draw[->] (delta2) to (mu);
  \draw[->] (mu) to (bl);
  \draw[->] (delta1) to (delta2);
  \draw[->] (delta2) to (bc);
\end{tikzpicture}}
\mathcenter{+}
\mathcenter{
\begin{tikzpicture}
  \node at (0,0) (top) {};
  \node at (0,-1) (delta1) {$\delta^1$};
  \node at (0,-2) (delta2) {$\delta^1$};
  \node at (0,-3) (delta3) {$\delta^1$};
  \node at (-1,-4) (mu) {$\mu_3$};
  \node at (-1,-5) (bl) {};
  \node at (0,-5) (bc) {};
  \node at (1,0) (tr) {};
  \node at (1.1,0) (trr) {};
  \node at (1.2,0) (trrr) {};
  \draw[taa] (tr) to (delta1);
  \draw[taa] (trr) to (delta2);
  \draw[taa] (trrr) to (delta3);
  \draw[->] (top) to (delta1);
  \draw[->] (delta1) to (mu);
  \draw[->] (delta2) to (mu);
  \draw[->] (delta3) to (mu);
  \draw[->] (mu) to (bl);
  \draw[->] (delta1) to (delta2);
  \draw[->] (delta2) to (delta3);
  \draw[->] (delta3) to (bc);
\end{tikzpicture}}
\mathcenter{+\cdots=0.}
\]
See~\cite[Definition~\ref*{LOT2:def:DA-structure}]{LOT2}.  The space
of morphisms between type~\DA\ bimodules $X$ and $Y$ is the graded space of
maps $h^1_k\co X\otimes \Tensor^*(\Blg[1])\to \Alg\otimes Y$.
Restricting to the case that
$\Alg$ is a \dg algebra, this morphism space has differential
\[
\mathcenter{d(h^1)=}
\mathcenter{
\begin{tikzpicture}
  \node at (0,0) (top) {};
  \node at (0,-2) (delta1) {$h^1$};
  \node at (-1,-4) (mu) {$\mu_1$};
  \node at (-1,-5) (bl) {};
  \node at (0,-5) (bc) {};
  \node at (1,0) (tr) {};
  \draw[taa] (tr) to (delta1);
  \draw[->] (top) to (delta1);
  \draw[->] (delta1) to (mu);
  \draw[->] (mu) to (bl);
  \draw[->] (delta1) to (bc);
\end{tikzpicture}}
\mathcenter{+}
\mathcenter{
\begin{tikzpicture}
  \node at (0,0) (top) {};
  \node at (0,-1) (delta1) {$\delta^1$};
  \node at (0,-2) (delta2) {$h^1$};
  \node at (-1,-4) (mu) {$\mu_2$};
  \node at (-1,-5) (bl) {};
  \node at (0,-5) (bc) {};
  \node at (1,0) (tr) {};
  \node at (1.1,0) (trr) {};
  \draw[taa] (tr) to (delta1);
  \draw[taa] (trr) to (delta2);
  \draw[->] (top) to (delta1);
  \draw[->] (delta1) to (mu);
  \draw[->] (delta2) to (mu);
  \draw[->] (mu) to (bl);
  \draw[->] (delta1) to (delta2);
  \draw[->] (delta2) to (bc);
\end{tikzpicture}}
\mathcenter{+}
\mathcenter{
\begin{tikzpicture}
  \node at (0,0) (top) {};
  \node at (0,-2) (delta1) {$h^1$};
  \node at (0,-3) (delta2) {$\delta^1$};
  \node at (-1,-4) (mu) {$\mu_2$};
  \node at (-1,-5) (bl) {};
  \node at (0,-5) (bc) {};
  \node at (1,0) (tr) {};
  \node at (1.1,0) (trr) {};
  \draw[taa] (tr) to (delta1);
  \draw[taa] (trr) to (delta2);
  \draw[->] (top) to (delta1);
  \draw[->] (delta1) to (mu);
  \draw[->] (delta2) to (mu);
  \draw[->] (mu) to (bl);
  \draw[->] (delta1) to (delta2);
  \draw[->] (delta2) to (bc);
\end{tikzpicture}}
.
\]
If $h^1\in\Mor(M,N)$ and $g^1\in\Mor(N, P)$ are morphisms of~\DA~bimodules over $\Alg$ and $\Blg$, and $\Alg$ is a $\dg$ algebra,
we can define their composite by the expression:
\[ 
\mathcenter{(g^1\circ h^1)=}
\mathcenter{
\begin{tikzpicture}
  \node at (0,0) (top) {};
  \node at (0,-1) (delta1) {$h^1$};
  \node at (0,-2) (delta2) {$g^1$};
  \node at (-1,-3) (mu) {$\mu_2$};
  \node at (-1,-4) (bl) {};
  \node at (0,-4) (bc) {};
  \node at (1,0) (tr) {};
  \node at (1.1,0) (trr) {};
  \draw[taa] (tr) to (delta1);
  \draw[taa] (trr) to (delta2);
  \draw[->] (top) to (delta1);
  \draw[->] (delta1) to (mu);
  \draw[->] (delta2) to (mu);
  \draw[->] (mu) to (bl);
  \draw[->] (delta1) to (delta2);
  \draw[->] (delta2) to (bc);
\end{tikzpicture}}
\]

Boundedness for type~\DA~ bimodules is slightly subtle to formulate. Before doing so, we need some more notation. First,
we will assume that our $\DA$ bimodule is defined over $\Alg$ and $\Blg$, where
$\Alg$ and~$\Blg$ are augmented over their ground rings, with
augmentation $\epsilon\co \Alg\to \Ground$ and $\epsilon \co
\Blg \to \Groundl$ and augmentation ideals $\Alg_+$ and~$\Blg_+$.

For each partition of $(1,\dots, i)$ into $m$ subsequences and each
length $m-j$ subsequence of $(1,\dots,m)$, we have a corresponding map
$M \otimes \Blg_+[1]^{\otimes i} \to \Alg[1]^{\otimes j}\otimes M$, defined by applying
$\delta^1$ $m$ times with input $x\in M$ and further algebra inputs
specified by the partition, followed by the map $\Alg[1]^{\otimes m}\to
\Alg[1]^{\otimes j}$ obtained by applying the augmentation map $\epsilon$ to the tensor
factors in the length $m-j$ subsequence. Such a map is called a {\em
  spinal $\DA$ bimodule operation} with $i$ algebra inputs and $j$
outputs.

\begin{definition}\cite[Definition~\ref{LOT2:def:DA-bounded}]{LOT2}
  The filtered type~\DA~ bimodule $M$ is called {\em operationally
    bounded} if for each $x$ there is an $n$ so that all spinal $\DA$
  bimodule operations with module input~$x$, $i$~algebra inputs, and
  $j$~algebra outputs, with $i+j>n$, vanish. It is called {\em left
    bounded} if for each $x\in M$ and each $i$, there is an $n$ so
  that all spinal $\DA$ bimodule operations with module input~$x$,
  $i$~algebra inputs, and $j>n$ algebra outputs vanish.  It
  is
  called {\em right bounded} if for each $x\in M$ and each $j$,
  there is an $n$ so that all spinal $\DA$ bimodule operations with
  module input~$x$, $i > n$ algebra inputs, and $j$~algebra outputs
  vanish.
\end{definition}

We define filtered type \DA\ structures similarly to filtered
$\Ainf$-modules and type $D$ structures:

\begin{definition}\label{def:filt-DA}
  Fix \dg algebras $\Alg$ and $\Blg$, and let $\IndI$ be a finite,
  partially ordered set.  An \emph{$\IndI$-filtered type $\DA$ structure
    over $\Alg$ and $\Blg$} is a collection $\{Q^i\}_{i\in \IndI}$ of type
  $\DA$ structures over $\Alg$ and $\Blg$, equipped with a preferred degree-zero
  morphism $F^{i<j}\in \Mor(P^i,P^j[1])$ for each pair $i,j\in \IndI$ with
  $i<j$, satisfying the compatibility
  condition~\eqref{eq:CompatibilityCondition}.
\end{definition}

The following is the bimodule analogue of
Lemma~\ref{lem:FilteredTensorProduct} (which is
also~\cite[Lemma~\ref*{DCov1:lem:FilteredProduct}]{LOT:DCov1}):

\begin{lemma}
  \label{lem:FilteredTensorProductDA}
  Let $\Alg$, $\Blg$, and $\Clg$ be $\dg$ algebras.
  Let $\lsup{\Alg}M_{\Blg}$ be an $\IndI$-filtered $\DA$-bimodule
  and $\lsup{\Blg}N_{\Clg}$ be a $\IndJ$-filtered type
  $\DA$-bimodule. Assume that either $M$ is right-bounded or $N$ is
  left-bounded, so $\lsup{\Alg}M_{\Blg}\DT_{\Blg} \lsup{\Blg}N_{\Clg}$ is defined.
  Then, $\lsup{\Alg}M_{\Blg}\DT_{\Blg} \lsup{\Blg}N_{\Clg}$ is naturally an $\IndI\times \IndJ$-filtered 
  type $DA$ bimodule.
\end{lemma}
\begin{proof}
  This is immediate from the definitions. (See also~\cite[Proof of
  Lemma~\ref*{DCov1:lem:DT-descends}]{LOT:DCov1}.)
\end{proof}

\begin{remark}\label{rem:filt-mod}
  The notion of a filtered $\Ainf$-module in
  Definition~\ref{def:filt-Ainf-module} is somewhat restrictive. For
  example, consider the algebra $\Alg=\Field[x]$ and let $M=\Alg$ as a
  right $\Alg$-module. The two-step filtration $\Filt_0M=M\supset
  \Filt_1M=xM$ is not a filtered module in the sense of
  Definition~\ref{def:filt-Ainf-module}.
\end{remark}

\begin{remark}\label{rem:filt-is-tw}
  Definitions~\ref{def:filt-Ainf-module},~\ref{def:filt-D}
  and~\ref{def:filt-DA} look quite similar. Indeed, all are roughly
  the same as a twisted complex
  \cite{BondalKapranov,Kontsevich,SeidelBook} in the relevant \dg
  category. There are some
  minor differences: unlike \cite[Definition~1]{BondalKapranov}
  and \cite[p.~15]{Kontsevich}, we allow indexing by arbitrary finite, partially
  ordered sets, rather than just by $\ZZ$; unlike~\cite[Section
  (3l)]{SeidelBook} we do view the partial ordering as part of the
  data.
\end{remark}


%% file: complexes.tex
\section{Chain complexes of  attaching circles in closed surfaces}
\label{sec:Complexes}

Counting pseudo-holomorphic bigons in a symmetric product of a
Heegaard surface gives rise to the differential appearing in Heegaard
Floer homology. Given a collection of Heegaard tori in the surface,
one can generalize these to counts of holomorphic polygons in a natural way.
We review these constructions in Subsection~\ref{subsec:HolCurves},
with a special emphasis on the ``cylindrical reformulation''
from~\cite{Lipshitz06:CylindricalHF}, as that generalizes most readily
to the bordered setting (cf.\ Section~\ref{sec:Polygons}).

Polygon counts can be organized using the notion of ``chain
complexes of attaching circles'' mentioned in the introduction. In
Subsection~\ref{subsec:ChainComplexes} we describe this construction
(Definition~\ref{def:ChainComplex}), and show how it can be used to
construct chain complexes (in the usual sense), as was promised in
Proposition~\ref{intro:ChainComplexesOfAttachingCircles}.

In Subsection~\ref{subsec:MorComplexes}, we explain some of the functorial
properties of these chain complexes of attaching circles. 

In Subsection~\ref{subsec:ConnSum} (see especially
Proposition~\ref{prop:GlueChainComplexes}), we show how to glue an
$\IndI$-filtered chain complex of attaching circles in a Heegaard
surface $\Sigma_1$ to an $\IndJ$-filtered chain complex in $\Sigma_2$,
to give an $\IndI\times \IndJ$-filtered chain complex in the connected
sum $\Sigma_1\conn \Sigma_2$. This construction relies on a preliminary
construction, {\em approximation}, introduced in
Subsection~\ref{subsec:Approximations} (see especially
Definition~\ref{def:Approximation}).  Approximation gives a way of
enhancing a chain complex of attaching circles indexed by a set
$\IndI$ to a chain complex of attaching circles indexed by a set
$\IndI\times\IndJ$. Approximations are constructed in
Proposition~\ref{prop:CloseApproximations}.

In Subsection~\ref{sec:Links}, it is shown that the constructions
given here (specifically, the filtered complex associated to a chain
complex of attaching circles gotten by gluing) generalize the
construction of the filtered complex associated to the branched
double cover of a link from~\cite{BrDCov}.

\subsection{Holomorphic curves in Heegaard multi-diagrams}
\label{subsec:HolCurves}

\begin{definition}
  \label{def:AttachingCircles}
  Let $\Sigma$ be a compact,
  oriented surface without boundary of some genus $g$, equipped with a
  basepoint $z\in \Sigma$.
  A {\em complete set of attaching circles} is a collection 
  $\betas=\{\beta_1,\dots,\beta_g\}$ of homologically independent, pairwise 
  disjoint embedded circles in $\Sigma\setminus z$.
  A {\em pointed Heegaard multi-diagram} is a surface $\Sigma$
  equipped with some number $n$ of complete sets of attaching circles.
\end{definition}

\begin{definition}
  \label{def:AdmissibleAttachingCircles}
  Let $\IndI$ be a finite, partially ordered set.  An {\em $\IndI$-filtered
    admissible collection of attaching circles} is a collection
  $\{\betas^i\}_{i\in \IndI}$ of $g$-tuples of attaching circles
  $\betas^i$ with the property that for each sequence
  $i_1<\dots<i_n$ in $\IndI$, the Heegaard multi-diagram
  $(\Sigma,\betas^{i_1},\dots,\betas^{i_n},z)$ is weakly admissible for all
  $\SpinC$-structures in
  the sense of~\cite[Definition 4.10]{OS04:HolomorphicDisks}.
\end{definition}

We will be primarily interested in the case where $\IndI=\{0,1\}^n$.

Given two pairs of complete sets of attaching circles
$\alphas=\{\alpha_1,\dots,\alpha_g\}$ and
$\betas=\{\beta_1,\dots,\beta_g\}$ so that
$(\Sigma,\alphas,\betas,z)$ is weakly admissible for all
$\SpinC$-structures, let $\CFa(\alphas,\betas,z)$ denote the
Lagrangian intersection Floer complex of the tori
$T_\alpha=\alpha_1\times\cdots\times\alpha_g$ and
$T_\beta=\beta_1\times\cdots\times\beta_g$ in
$\Sym^g(\Sigma\setminus\{z\})$. The complex $\CFa(\alphas,\betas,z)$
is generated by $\Gen(\alphas,\betas)\coloneqq T_\alpha\cap T_\beta$.

For an $\IndI$-filtered admissible collection of attaching circles and any
sequence $i_0<i_1<\dots<i_n$ in $\IndI$, there is a map
\begin{equation}\label{eq:mn-order}
m_{n}\co
\CFa(\betas^{i_{n-1}},\betas^{i_n},z)
\otimes\dots\otimes
\CFa(\betas^{i_0},\betas^{i_1},z)
\to \CFa(\betas^{i_0},\betas^{i_n},z)
\end{equation}
defined by counting pseudo-holomorphic polygons; see, for
instance,~\cite[Section 4.2]{BrDCov}.
In the case where $n=1$, this is the
usual differential on $\CFa(\betas^{i_0},\betas^{i_1},z)$. The maps
$m_n$ are well-known to satisfy $\Ainf$ relations; see, for
instance,~\cite{SeidelBook} or~\cite{FOOO09:1,FOOO09:2}. 

\begin{convention}\label{conv:order}
  We have reversed the order of the arguments to $m_n$ from what is
  standard in the Heegaard Floer literature, so that $m_3$ agrees with
  the standard order for function composition in a category. See also
  Definition~\ref{def:n-m-k} and Remark~\ref{rem:FukABimodule} for
  further justification of this choice. We will use the order from
  Equation~\ref{eq:mn-order} throughout this paper.
\end{convention}

Since we are working in the cylindrical formulation of Heegaard Floer
homology, we spell out these polygon counts a little more. (The cases
of triangles and rectangles were discussed in~\cite[Section 10]{Lipshitz06:CylindricalHF}.)

\begin{figure}
  \centering
    \begin{overpic}[tics=10]{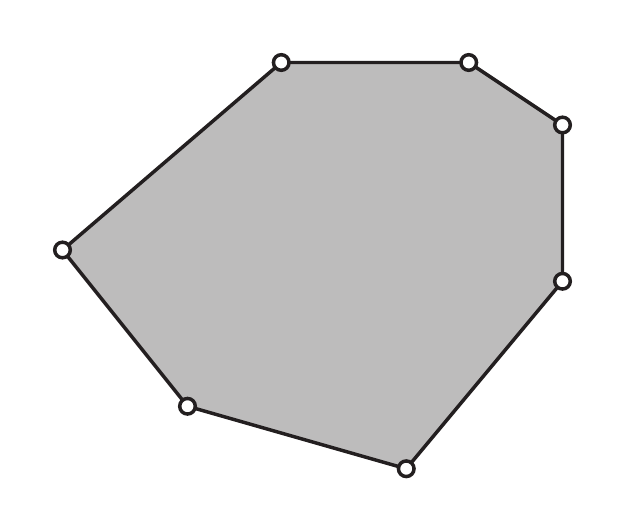}
    \put(92,39){$p_{01}$}
    \put(92,64){$p_{60}$}
    \put(64,4){$p_{12}$}
    \put(25,15){$p_{23}$}
    \put(0,44){$p_{34}$}
    \put(40,78){$p_{45}$}
    \put(70,78){$p_{56}$}
    \put(92,50){$e_{0}$}
    \put(80,25){$e_{1}$}
    \put(41,11){$e_{2}$}
    \put(15,30){$e_{3}$}
    \put(23,62){$e_{4}$}
    \put(58,77){$e_{5}$}
    \put(82,72){$e_{6}$}
  \end{overpic}
  \caption{\textbf{The disk $D_7$.} The labeling of the arcs and
    punctures in the boundary is shown.}
  \label{fig:Dn}
\end{figure}
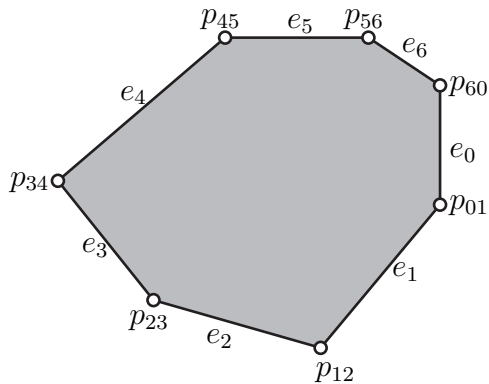

Let $D_n$ denote a disk with $n$ labeled punctures on its
boundary (a polygon). Label the arcs in $\bdy D_n$ as $e_0,\dots,e_{n-1}$, in
clockwise order, and let $p_{i,i+1}$ denote the puncture between $e_i$
and $e_{i+1}$. (See Figure~\ref{fig:Dn}.) 
Let $\Conf(D_n)$ denote the moduli space of (positively oriented) complex structures
on~$D_n$, up to biholomorphism respecting the labeling of the edges.  (For $n 
\ge 3$,
  $\Conf(D_n)$ is an $(n-3)$-dimensional ball.) The space $\Conf(D_n)$ has a 
  natural Deligne-Mumford
compactification $\oConf(D_n)$, which is diffeomorphic to the
associahedron (see, e.g.,~\cite{Devadoss98:tessellations}). The boundary
$\oConf(D_n)\setminus\Conf(D_n)$ of $\oConf(D_n)$ consists of trees of
polygons.

Points of
$\Conf(D_n)$ are
equivalence classes $[j]$ of complex structures $j$ on
$D_n$. We would like to view elements of $\Conf(D_n)$ as honest
complex structures, instead of equivalence classes, so we explain how
to do so.  There is an infinite-dimensional bundle $\mathcal{E}_n\to
\Conf(D_n)$ of honest complex structures $j$. Since $\Conf(D_n)$ is contractible, this bundle is
necessarily trivial and, in particular, admits a section. Less
trivially, if we replace $\mathcal{E}_n$ with the bundle of 
complex structures together with choices of strip-like ends (in the
sense of~\cite[Section 9(a)]{SeidelBook}), which we still denote
$\mathcal{E}_n$, then the projection map $\mathcal{E}_n\to \Conf(D_n)$ extends to a continuous map
$\overline{\mathcal{E}}_n\to\oConf(D_n)$, where
$\overline{\mathcal{E}}_n$ is a
partial compactification of
$\mathcal{E}_n$ given as follows: a point $p\in\bdy\oConf(D_n)$
corresponds to a tree of disks together with an equivalence class of
complex structures on each of those disks, and the fiber of
$\overline{\mathcal{E}}_n$ over $p$ is the space of complex
structures on those disks inducing the specified equivalence class,
together with a choice of strip-like ends for each of the
disks. Seidel proves that one can choose sections of
$\overline{\mathcal{E}}_n\to\oConf(D_n)$, for each $n$, which agree
at the boundary strata~\cite[Lemma
9.3]{SeidelBook}. For the rest of the paper, we will identify
$\oConf(D_n)$ with its image in $\overline{\mathcal{E}}_n$ under the
chosen section, and view elements of $\Conf(D_n)$ as honest 
complex structures, not equivalence classes of them.

Similarly, fix a family $\omega_j$, $j\in\Conf(D_n)$ of symplectic
forms on $D_n$ with cylindrical ends, so that $\omega_j$ is adjusted
to $j$~\cite[Section 3.3]{BEHWZ03:CompactnessInSFT}, and so that as
$j$ approaches the boundary of $\oConf(D_n)$, the $\omega_j$ split as
in (a very simple case of) symplectic field theory~\cite[Section
3.3]{BEHWZ03:CompactnessInSFT} to the corresponding pairs of
symplectic forms on $D_m\amalg D_{n+2-m}$.

\begin{definition}
  \label{def:AdmissibleJs}
  Fix a symplectic form $\omega_\Sigma$ on $\Sigma$.
  An {\em admissible collection of almost-complex
    structures} consists of
  \begin{itemize}
  \item a choice of $\RR$-invariant almost-complex structure $J$ on
    $\Sigma\times[0,1]\times\RR$ of the kind used to compute Heegaard
    Floer homology (i.e., satisfying~\cite[Conditions
    (J1)--(J5)]{Lipshitz06:CylindricalHF}), and
  \item a choice of a smooth family $\{J_j\}_{j\in\Conf(D_n)}$ of
    almost-complex structures on $\Sigma\times D_n$, for each $n\geq
    3$,
  \end{itemize}
  satisfying the following conditions:
  \begin{enumerate}[label=(J-\arabic*),ref=(J-\arabic*)]
  \item\label{item:J-proj} For each $j\in\Conf(D_n)$, the projection map
    $$\pi_{\CDisk}\co \Sigma\times D_n \to D_n$$
    is $(J_j,j)$-holomorphic.
  \item\label{item:J-fiber} For each $j\in\Conf(D_n)$, the fibers of
    $\pi_{\CDisk}$ are $J_j$-holomorphic.
  \item\label{item:J-tame} Each almost-complex structure $J_j$ is
    adjusted to the split symplectic form $\omega_{\Sigma}\oplus
    \omega_{j}$ on $\Sigma\times D_n$.
  \item\label{item:J-cylindrical} Each almost-complex structure $J_j$
    agrees with $J$ near the punctures of $D_n$, in the sense that each puncture
    has a strip-like neighborhood $U$ in $D_n$ so that $(\Sigma \times
    U,J_j|_{\Sigma \times U})$ is bi-holomorphic to
    $(\Sigma\times[0,1]\times(0,\infty), J)$.
  \item\label{item:J-compatible} Suppose that $\{j_\alpha\}\subset
    \Conf(D_n)$ is a sequence converging to a point $j_\infty\in
    \bdy\oConf(D_n)$. For notational simplicity, suppose $j_\infty$
    lies in the codimension-one boundary, and so corresponds to a
    point
    $(j_{\infty,1},j_{\infty,2})\in \Conf(D_{m+1})\times
    \Conf(D_{n-m+1})$. Then the complex structures $J_{j_\alpha}$
    are required to converge to the complex structure $J_{j_{\infty,1}}\amalg
    J_{j_{\infty,2}}$ on $(\Sigma\times D_{m+1})\amalg (\Sigma\times
    D_{n-m+1})$.

    (Convergence of the $J_{j_\alpha}$ means the following. As
    $\alpha\to\infty$, certain arcs in $D_{m+1}$ collapse. Over
    neighborhoods of these arcs, the complex structures $J_{j_\alpha}$
    should be obtained by inserting longer and longer necks, as
    in~\cite[Section 3.4]{BEHWZ03:CompactnessInSFT}. Outside these
    neighborhoods, we require convergence in the $C^\infty$-topology.)

    We require the analogous compatibility condition for the
    higher-codimension boundary of $\oConf(D_n)$, as well.
  \end{enumerate}
\end{definition}


\begin{definition}\label{def:holo-polygon}
  Let $\{J_j\}_{j\in\Conf(D_n)}$ be an admissible collection of
  almost-complex structures. Given $g$-tuples of attaching circles
  $\betas^0,\dots,\betas^n$ and generators $\x^{i,i+1}\in
  \CFa(\betas^i,\betas^{i+1},z)$ (with the understanding that
  $\betas^{n+1}=\betas^0$), consider surfaces~$S$ with boundary and boundary punctures, and
  proper maps
  \begin{equation}\label{eq:polygon-map}
  u\co (S,\bdy S)\to \bigl(\Sigma\times D_{n+1},(\betas^0\times
  e_0)\cup\dots\cup(\betas^n\times e_n) \bigr)
  \end{equation}
  asymptotic to $\x^{i,i+1}$ at $p_{i,i+1}$. This space decomposes
  into homology classes (compare~\cite[Sections 2 and
  10.1.1]{Lipshitz06:CylindricalHF}); let
  $\pi_2(\x^{n,0},\dots,\x^{0,1})$ denote the set of homology classes
  of such maps. 

  Let $\cM(\x^{n,0},\x^{n-1,n},\dots,\x^{0,1})$ denote the moduli space of pairs $(j,u)$ where $j\in \Conf(D_{n+1})$ and $u$ is a 
  map as in Formula~(\ref{eq:polygon-map}) such that:
  \begin{enumerate}[label=(M-\arabic*),ref=(M-\arabic*),start=0]
  \item The image of $u$ is disjoint from $\{z\}\times D_{n+1}$ (i.e., has
    multiplicity $0$ at $z$).
  \item $u$ is $(i,J_j)$-holomorphic for some complex structure $i$ on
    $S$,
  \item $\pi_\bD\circ u$ is a $g$-fold branched cover, and
  \item $u$ is an embedding.
  \end{enumerate}
  We will often abuse notation and write elements of
  $\cM(\x^{n,0},\x^{n-1,n},\dots,\x^{0,1})$ as maps $u$, but the
  complex structure $j$ is part of the data.

  The space $\cM(\x^{n,0},\dots,\x^{0,1})$ decomposes as a
  disjoint union
  \[
  \cM(\x^{n,0},\x^{n-1,n},\dots,\x^{0,1})=\coprod_{B\in\pi_2(\x^{n,0},\dots,\x^{0,1})}\cM^B(\x^{n,0},\x^{n-1,n},\dots,\x^{0,1}).
  \]
  We will often abbreviate $\cM^B(\x^{n,0},\x^{n-1,n},\dots,\x^{0,1})$
  to $\cM^B$.
 
  For each $B\in\pi_2(\x^{n,0},\dots,\x^{0,1})$ the space
  $\cM^B(\x^{n,0},\dots,\x^{0,1})$ has a well-defined expected
  dimension $\ind(B)+n-2$.
\end{definition}

\begin{proposition}\label{prop:closed-trans}
  Admissible collections of almost-complex structures exist. Moreover,
  with respect to a generic admissible collection of almost-complex structures,
  each of the moduli spaces $\cM^B(\x^{n,0},\dots,\x^{0,1})$ is transversely
  cut out by the $\overline{\bdy}$-operator, and hence is a smooth
  manifold of dimension $\ind(B)+n-2$.
\end{proposition}

\begin{proof}
  The first part of the statement (existence) follows from the
  observation that the space of almost-complex structures satisfying
  Conditions~\ref{item:J-proj}--\ref{item:J-cylindrical} is
  contractible, so the extension problem specified by
  Condition~\ref{item:J-compatible} has a solution. 
  The second part (transversality) follows by a similar argument to
  \cite[Section~3]{Lipshitz06:CylindricalHF}.
\end{proof}

\begin{remark}\label{rem:non-compact}
In general, of course, the moduli spaces $\cM^B(\x^{n,0},\dots,\x^{0,1})$ are 
non-compact, though they admit compactifications 
$\ocM^B(\x^{n,0},\dots,\x^{0,1})$ in terms of trees of holomorphic curves. For 
generic admissible collections of almost-complex structures, if $\ind(B)=-n+2$ 
then $\cM^B(\x^{n,0},\dots,\x^{0,1})$ is a compact $0$-manifold: all broken 
curves in $\ocM^B(\x^{n,0},\dots,\x^{0,1})$ belong to negative expected 
dimension families, and hence by transversality do not occur.
\end{remark}

\begin{definition}
  \label{def:closed-polygon-map}
  Let $(\Sigma,\betas^0,\dots,\betas^n,z)$ be a weakly admissible Heegaard 
  multi-diagram. Define a map
  \[
  m_{n}\co
  \CFa(\betas^{n-1},\betas^{n},z)\otimes
  \CFa(\betas^{n-2},\betas^{n-1},z)
  \otimes\dots\otimes
  \CFa(\betas^{0},\betas^{1},z)
  \to \CFa(\betas^{0},\betas^{n},z)
  \]
  as follows. Choose a generic admissible collection of almost-complex
  structures (as guaranteed by Proposition~\ref{prop:closed-trans}),
  and define
  \[
  m_n(\x^{n-1,n},\dots,\x^{0,1})
  =\!\!\!\!\!\sum_{\x^{0,n}\in\Gen(\betas^0,\betas^n)}\,\,\sum_{\substack{B\in\pi_2(\x^{n,0},\x^{n-1,n},\dots,\x^{0,1})\\\ind(B)=3-n}}\,
  \bigl(\#\cM^B(\x^{n,0},\x^{n-1,n},\dots,\x^{0,1})\bigr)\x^{0,n}.
  \]
  (Here, $\x^{n,0}$ is the generator of $\CFa(\betas^n,\betas^0,z)$
  corresponding to $\x^{0,n}\in\CFa(\betas^0,\betas^n,z)$.) As explained 
  Remark~\ref{rem:non-compact}, each moduli space being counted is finite, and 
  it follows from weak admissibility of the Heegaard multi-diagram that the sum 
  itself is finite.
\end{definition}

Note that when $n=1$ the map $m_n$ is just the differential on
$\CFa(\betas^0,\betas^1,z)$.

\begin{proposition}\label{prop:closed-Ainf-rel}
  Let $\betas^0,\dots,\betas^n$ be $g$-tuples of attaching circles in $\Sigma$. Then the maps $m_n$ satisfy the following $\Ainf$ relation:
  \begin{equation}\label{eq:closed-polys-Ainf}
  \sum_{i=1}^n\sum_{j=1}^{n-i-1} 
  m_{n-j+1}(
  \x^{n-1,n},\dots,\x^{i+j,i+j+1},m_j(\x^{i-1,i},\dots,\x^{i+j-1,i+j}),\x^{i-2,i-1},\dots,
  \x^{0,1})=0.
  \end{equation}
\end{proposition}
\begin{proof}
  This follows in the usual way, by considering the ends of the
  1-dimensional moduli space
  \[
  \bigcup_{\substack{B\in\pi_2(\x^{n,0},\dots,\x^{0,1})\\\ind(B)=4-n}}
  \#\cM^B(\x^{n,0},\x^{n-1,n},\dots,\x^{0,1}).
  \]
  Conditions~\ref{item:J-proj},~\ref{item:J-fiber},~\ref{item:J-tame}
  and~\ref{item:J-cylindrical} guarantee that this moduli space has a
  compactification in terms of broken holomorphic polygons. 
  (See~\cite{BEHWZ03:CompactnessInSFT}
  or~\cite{Abbas14:compactness}. In particular, by Condition~\ref{item:J-compatible}, approaching the boundary of
  $\Conf(D_n)$ has the effect of splitting along a hypersurface, as
  in~\cite[Theorem 10.3]{BEHWZ03:CompactnessInSFT}.)
  Condition~\ref{item:J-compatible} allows us to identify the counts
  of these broken polygons with the counts used to define $m_{n-j+1}$
  and $m_j$.
\end{proof}

We conclude this section by noting that, via the \emph{tautological
correspondence}, the maps $m_n$ can also be defined by counting polygons in the
symmetric product. That is, given a family $J_j$, $j\in\Conf(D_n)$, of
admissible almost-complex structures there is a corresponding family
$\Sym^g(J_j)$, $j\in\Conf(D_n)$ of maps from $D_n$ to the space of
almost-complex structures on the symmetric product
$\Sym^g(\Sigma)$. Given a $J_j$-holomorphic map $u$ as in
Formula~\eqref{eq:polygon-map} there is a corresponding
$\Sym^g(J_j)$-holomorphic map $\phi_u\co D_n\to \Sym^g(\Sigma)$
defined by $\phi_u(x)=(\pi_\Sigma\circ u)\bigl((\pi_\bD\circ
u)^{-1}(x)\bigr)$. The map $\phi_u$ sends the edges of $D_n$ to
$T_{\betas^1},T_{\betas^2},\dots,T_{\betas^n}$, where $T_{\betas^i}$
is the image of $\beta^i_1\times\cdots\times\beta^i_g\subset
\Sigma^{\times g}$ in $\Sym^g(\Sigma)$.
\begin{lemma}\label{lem:tautological}
  The assignment $u\mapsto \phi_u$ gives a bijection between the
  moduli space of polygons as in Definition~\ref{def:holo-polygon} and
  the moduli space of polygons in
  $(\Sym^g(\Sigma),T_{\betas^1},\dots,T_{\betas^n})$ which are
  holomorphic with respect to one of the complex structures
  $\Sym^g(J_j)$, $j\in\Conf(D_n)$. In particular, the maps $m_n$ from
  Definition~\ref{prop:closed-Ainf-rel} agree with the maps defined by
  counting holomorphic polygons in the symmetric product as in,
  e.g.,~\cite[Section 8]{OS04:HolomorphicDisks} or~\cite[Section
  4.2]{BrDCov}.
\end{lemma}
\begin{proof}
  The proof is the same as the proof for bigons~\cite[Proposition
  13.6]{Lipshitz06:CylindricalHF}, with only notational changes.
\end{proof}

\subsection{Chain complexes of attaching circles: definition}
\label{subsec:ChainComplexes}

\begin{definition}
  \label{def:ChainComplex}
  Let $\IndI$ be a finite, partially ordered set and let $(\Sigma,z)$
  be a pointed closed, oriented surface.  An {\em $\IndI$-filtered
    chain complex of attaching circles} (or simply \emph{chain complex
    of attaching circles}) consists of the following data:
  \begin{itemize}
  \item an admissible collection of attaching circles
    $\{\betas^i\}_{i\in \IndI}$ and
  \item for each pair of elements $i_1<i_2$ in $\IndI$,
    a chain $\eta^{i_1<i_2}\in \CFa(\betas^{i_1},\betas^{i_2},z)$.
  \end{itemize}
  The chains $\eta^{i<j}$ are required to satisfy the following
  compatibility conditions, indexed by pairs $i, j\in \IndI$ with $i<j$:
  \begin{equation}
    \label{eq:Compatibility}
    \sum_{n=1}^\infty \,\,\sum_{i=i_0<i_1<\dots<i_{n-1}<i_n=j} m_{n}(\eta^{i_{n-1}<i_n},\dots,\eta^{i_0<i_1})=0,
  \end{equation}
  where the sum is taken over the sequences $i_1,\dots, i_{n-1}$ and
  $m_n$ denotes the counts of holomorphic $n+1$-gons
  (Definition~\ref{def:closed-polygon-map}).  We will usually suppress
  $\IndI$ and $\Sigma$ from the notation and write a chain complex of
  attaching circles as a triple $(\{\betas^i\}_{i\in
    \IndI},\{\eta^{i_1<i_2}\}_{i_1,i_2\in \IndI},z)$.
\end{definition}

For example, Equation~\eqref{eq:Compatibility} implies that if $i<j$ are
consecutive (i.e.  there is no other $k\in \IndI$ between $i$ and $j$),
then $\eta^{i<j}$ is a cycle.

\begin{remark}
  \label{rem:ChainComplexIsTypeD}
  The attentive reader might notice a similarity between
  Definition~\ref{def:ChainComplex} and
  Definition~\ref{def:IFilteredCx} (say). Indeed, let $\HFuk$ denote
  the full subcategory of the Fukaya category of $\Sym^g(\Sigma\setminus\{z\})$
  generated by Heegaard tori. Then a chain complex of attaching
  circles is just a twisted complex (type $D$ structure) in $\HFuk$
  (modulo the caveats in Remarks~\ref{rem:type-d-ainf}
  and~\ref{rem:filt-is-tw}).
\end{remark}

The following example plays a pivotal role in the proof of the surgery
exact triangle for Heegaard Floer homology~\cite[Theorem~9.1]{OS04:HolDiskProperties}.

\begin{example}
  \label{ex:AttachingCircles}
  Let $\IndI=\{0,1,\infty\}$ with the obvious ordering, and let $\betas^0$, $\betas^1$, and
  $\betas^\infty$ be three sets of attaching circles in a genus one
  surface, with the property that
  $(\Sigma,\betas^0,\betas^1,\betas^\infty,z)$ is a Heegaard triple
  representing ${\overline {\mathbb CP}}^2$ (i.e., with $\beta^0$,
  $\beta^1$, $\beta^\infty$ at slopes $0$, $1$, $\infty$,
  respectively). Choose $\eta^{0,1}$ and
  $\eta^{1,\infty}$ to be cycles representing the non-trivial
  homology class in 
  \[
  \HFa(\betas^0,\betas^1,z)=\HFa(S^3)\cong \Zmod{2}
  \qquad\text{and}\qquad
  \HFa(\betas^1,\betas^\infty,z)=\HFa(S^3)\cong \Zmod{2},
  \]
  respectively.
  We can find a chain $\eta^{0,\infty}$ so that
  $d\eta^{0,\infty}=m_2(\eta^{1,\infty},\eta^{0,1})$.
  The data
  $(\{\betas^0,\betas^1,\betas^\infty\},\allowbreak\{\eta^{0,1},\eta^{1,\infty},\eta^{0,\infty}\})$
  forms a chain complex of attaching circles. (See
  Figure~\ref{fig:ExampleComplex}.)
\end{example}

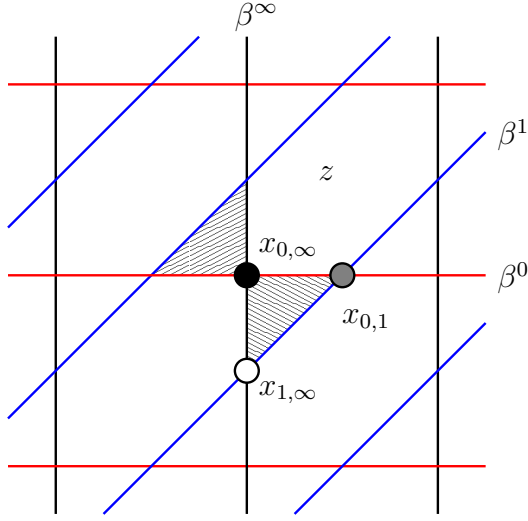
\begin{figure}
  \centering
  \input{ExampleComplex}
  \caption{\textbf{Example of a chain complex of attaching circles.}
    The picture takes place in the torus, with three curves
  $\beta^0$, $\beta^1$, and $\beta^\infty$, which pairwise intersect in 
  three points
  $x_{0,\infty}$, $x_{0,1}$, and $x_{1,\infty}$.
  The fact that $m_2(x_{1,\infty},x_{0,1})=0$ is illustrated
  by the two hatched (canceling) triangles. Thus, in this model,
  the chains $\eta^{0,1}=x_{0,1}$, $\eta^{1,\infty}=x_{1,\infty}$
  and $\eta^{0,\infty}=0$ give a chain complex of attaching circles.}
  \label{fig:ExampleComplex}
\end{figure}

Chain complexes of attaching circles can be used to turn sets of attaching
circles into chain complexes, using the following Yoneda embedding.

\begin{definition}
  \label{def:AssociatedComplex}
  Suppose that $(\{\betas^i\}_{i\in \IndI},\{\eta^{i_1<i_2}\}_{i_1,i_2\in \IndI},z)$ is a chain
  complex of attaching circles, and $\alphas$ is an additional set of
  attaching circles with the property that for all 
  sequences $i_1<\dots <i_n$ in $\IndI$, the multi-diagram
  $(\Sigma,\alphas,\betas^{i_1},\dots,\betas^{i_n},z)$ is weakly
  admissible.  We call the collection of chain complexes
  \[
  \{\CFa(\alphas,\betas^i,z)\}_{i\in \IndI}
  \]
  equipped with the morphisms
  $$
  D^{i<j}(\x)=\sum_{i=i_1<\dots<i_n=j}
  m_n(\eta^{i_{n-1}<i_n},\dots,\eta^{i_1<i_2},\x)
  $$
  where the sum is over all subsequences of $\IndI$ starting at $i$ and ending at $j$
  the 
 \emph{Heegaard Floer complex
    associated to $\alphas$ and $(\{\betas^i\}_{i\in \IndI},\{\eta^{i_1<i_2}\}_{i_1,i_2\in \IndI})$}. We
  denote this filtered chain complex by $\CCFa(\alphas,\{\betas^i\}_{i\in\IndI},\{\eta^{i_1<i_2}\}_{i_1,i_2\in \IndI},z)$.
\end{definition} 
The terminology is justified by the following more precise version of Proposition~\ref{intro:ChainComplexesOfAttachingCircles}:
\begin{proposition}
  \label{prop:ChainComplexesOfAttachingCircles} 
  Let $\alphas$ be a set of attaching circles and $(\{\betas^i\}_{i\in
    \IndI},\{\eta^{i_1<i_2}\}_{i_1,i_2\in\IndI},z)$ be a chain complex
  of attaching circles. Suppose that for all
  sequences $i_1<\dots <i_n$ in~$\IndI$, the
  multi-diagram $(\Sigma,\alphas,\betas^{i_1},\dots,\betas^{i_n},z)$
  is weakly admissible. Then the Heegaard Floer complex associated to
  $\alphas$ and $(\{\betas^i\}_{i\in
    \IndI},\{\eta^{i_1<i_2}\}_{i_1,i_2\in\IndI},z)$ is an $\IndI$-filtered chain
  complex in the sense of Definition~\ref{def:IFilteredCx}.
\end{proposition}

\begin{proof}
  As we will see, the structure equation is an easy consequence of the
  associativity formula for counts of holomorphic polygons,
  Proposition~\ref{prop:closed-Ainf-rel}. (Compare
  Figure~\ref{fig:attach-circ-cx}.)
  For any $i<k$ we have
  \begin{multline}\label{eq:terms-1}
  \sum_{i<j<k} D^{j<k}\bigl(D^{i<j}(\x)\bigr) =\\
   \sum_{i=i_0<\cdots<i_m=j<\cdots<i_n=k}m\bigl(\eta^{i_{n-1}<i_n},\dots,\eta^{i_{m+1}<i_m},m(\eta^{i_{m-1}<i_m},\dots,\eta^{i_0<i_1},\x)\bigr)
  \end{multline}
  (omitting the indices on the $m$'s), while
  \begin{multline}\label{eq:terms-2}
    d (D^{i<k})(\x)=\\
   \sum_{i=i_0<\dots<i_n=k}
    m_1(m_n(\eta^{i_{n-1}<i_n},\dots,\eta^{i_0<i_1},\x))+
    m_n(\eta^{i_{n-1}<i_n},\dots,\eta^{i_0<i_1},m_1(\x)).
  \end{multline}
  Equation~\eqref{eq:Compatibility} gives
  \begin{equation}
    \label{eq:terms-3}
    0=\sum_{i=i_0<\dots<i_n=k}\,\,\sum_{1< j\leq \ell< n} m(\eta^{i_{n-1}<i_n},\dots, m(\eta^{i_{\ell-1}<i_\ell},\dots,\eta^{i_{j}<i_{j+1}}),\dots,\eta^{i_0<i_1},\x).
  \end{equation}
  The sum of the right hand sides of Equations~\eqref{eq:terms-1},~\eqref{eq:terms-2} and~\eqref{eq:terms-3} is the left hand side of Equation~\eqref{eq:closed-polys-Ainf}, and hence is equal to zero.
  \begin{figure}
    \centering
    \begin{overpic}[tics=10,height=1.75in]{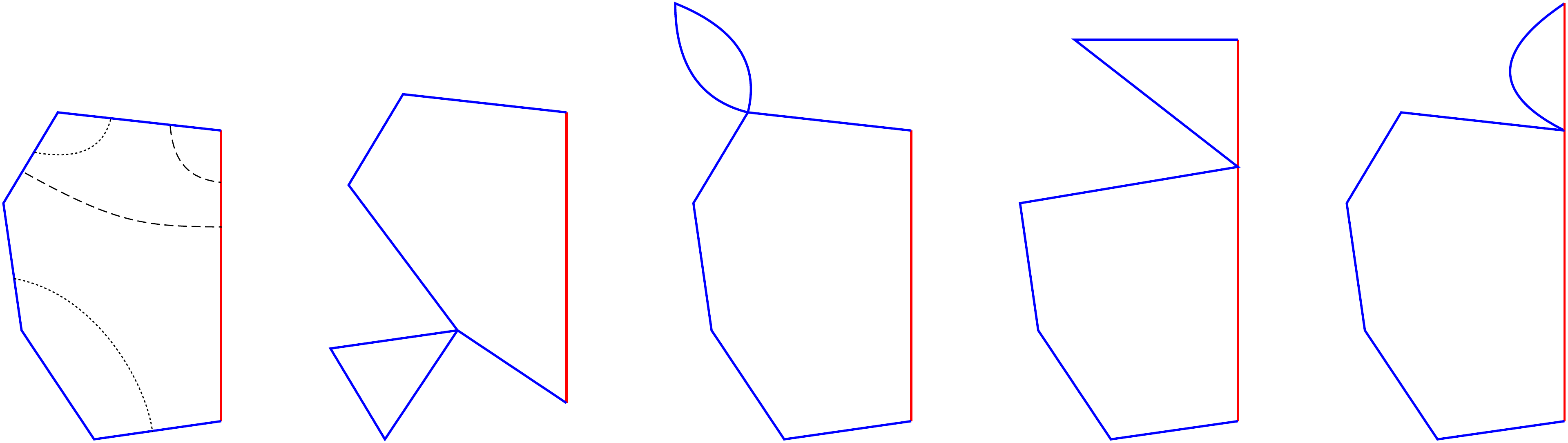}
      \put(15,10){\textcolor{red}{$\alphas$}}
      \put(10,21){\textcolor{blue}{$\betas^5$}}
      \put(10,3){\textcolor{blue}{$\betas^1$}}
      \put(0,2){\textcolor{blue}{$\betas^2$}}
      \put(-2,9){\textcolor{blue}{$\betas^3$}}
      \put(-1,18){\textcolor{blue}{$\betas^4$}}
      \put(8,-3){(a)}
      \put(28,-3){(b)}
      \put(50,-3){(c)}
      \put(72,-3){(d)}
      \put(94,-3){(e)}
    \end{overpic}
    \vspace{.1in}
    \caption{\textbf{Proof of
        Proposition~\ref{prop:ChainComplexesOfAttachingCircles}.} 
      The hexagon illustrated in (a) can be degenerated along one of the dotted lines
      to give (b) or (c). Degenerating along the dashed lines gives one of (d) or (e).
      Degenerations of type (b) and (c) cancel in Formula~\eqref{eq:Compatibility}; 
      those of type (d) give terms of the form $D^{j<k}\circ D^{i<j}$;
      and those of type (e) give terms in $d D^{i<k}$.}
    \label{fig:attach-circ-cx}
  \end{figure}
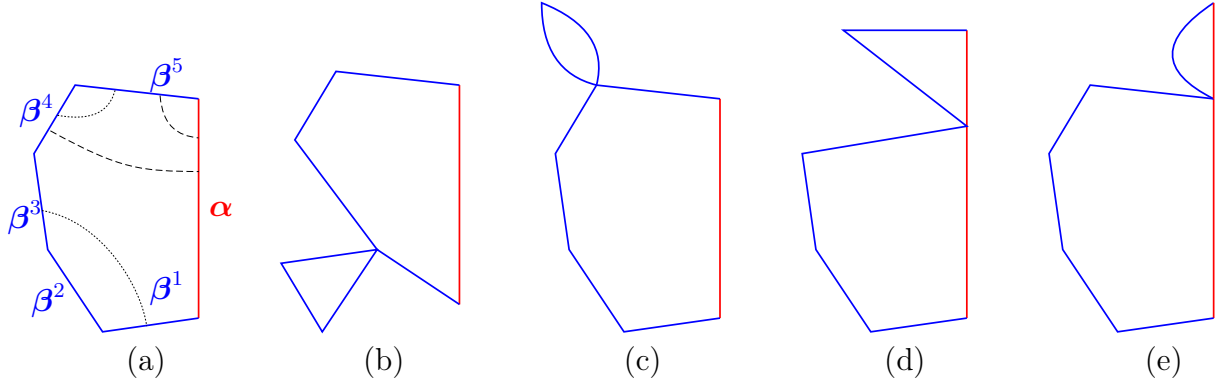
  Thus,
  \[ 
  d (D^{i<k}) + \sum_{\{j\mid i<j<k\}} D^{j<k}\circ D^{i<j}=0.\qedhere
  \]
\end{proof}

\subsection{Morphisms between chain complexes of attaching circles}
\label{subsec:MorComplexes}
The next step in developing the theory of chain complexes of attaching circles is to verify that the chain complex considered
in Proposition~\ref{prop:ChainComplexesOfAttachingCircles} is
invariant under change of admissible collection of almost-complex
structures and isotopies of the $\alpha$- and $\beta$-circles. The
argument is based on the usual ``continuation maps'' in Floer
homology, and is similar to the proofs in~\cite{RobertsMultiDiagrams}
and~\cite{Baldwin11:ss}; see also~\cite[Section 10(a)]{SeidelBook} in
the more general setting of Fukaya categories. In this section we will
prove invariance under isotopies, leaving invariance under changes of
admissible collection of almost-complex structures as an exercise to
the reader.

So, fix an $\IndI$-filtered admissible collection of attaching circles
$\{\betas^i\}_{i\in \IndI}$ in a pointed surface $(\Sigma,z)$, an element
$k\in \IndI$, and a collection of attaching circles $\gammas^k$
Hamiltonian isotopic to $\betas^k$. Let $\gammas^i=\betas^i$ for
$i\neq k$.  We will assume that the $\gammas^k$ are close enough to
the $\betas^k$ in a sense that will be made precise in two places
below; in practice, by breaking a Hamiltonian isotopy up into a
sequence of smaller isotopies, this closeness assumption can always be
achieved. Fix also an admissible collection of almost-complex
structures, chosen generically in the sense of
Proposition~\ref{prop:closed-trans}, and so that the moduli spaces
described below are transversally cut out. The argument that such a
family of almost-complex structures exists is similar to the (largely
omitted) proof of Proposition~\ref{prop:closed-trans}.

Our first goal is to define maps
\[
f_n\co 
\CFa(\betas^{i_{n-1}},\betas^{i_n},z)
\otimes\dots\otimes
\CFa(\betas^{i_0},\betas^{i_1},z)\to
\CFa(\gammas^{i_0},\gammas^{i_n},z)
\]
satisfying the $\Ainf$-homomorphism relation, i.e., so that
\begin{multline}\label{eq:Ainf-hom}
0=\sum_{1\leq a\leq b\leq m}
f_{m-b+a}\bigl(x_m,x_{m-1},\dots,m_{b-a+1}(x_b,\dots,x_a),\dots,x_1\bigr)\\
+\sum_c\sum_{n_1+\dots+n_c=m}
m_c\bigl(f_{n_1}(x_m,x_{m-1},\dots,x_{m-n_1+1}),\\
f_{n_2}(x_{m-n_1},\dots,x_{m-n_1-n_2+1}),\dots,
f_{n_c}(x_{n_c},\dots,x_1)\bigr)
\end{multline}
for any sequence of sets of attaching circles
$\betas^{i_0},\dots,\betas^{i_m}$ and elements $x_j\in
\CFa(\betas^{i_{j-1}},\betas^{i_{j}},z)$.

\begin{figure}
  \centering
  \begin{overpic}[tics=10]{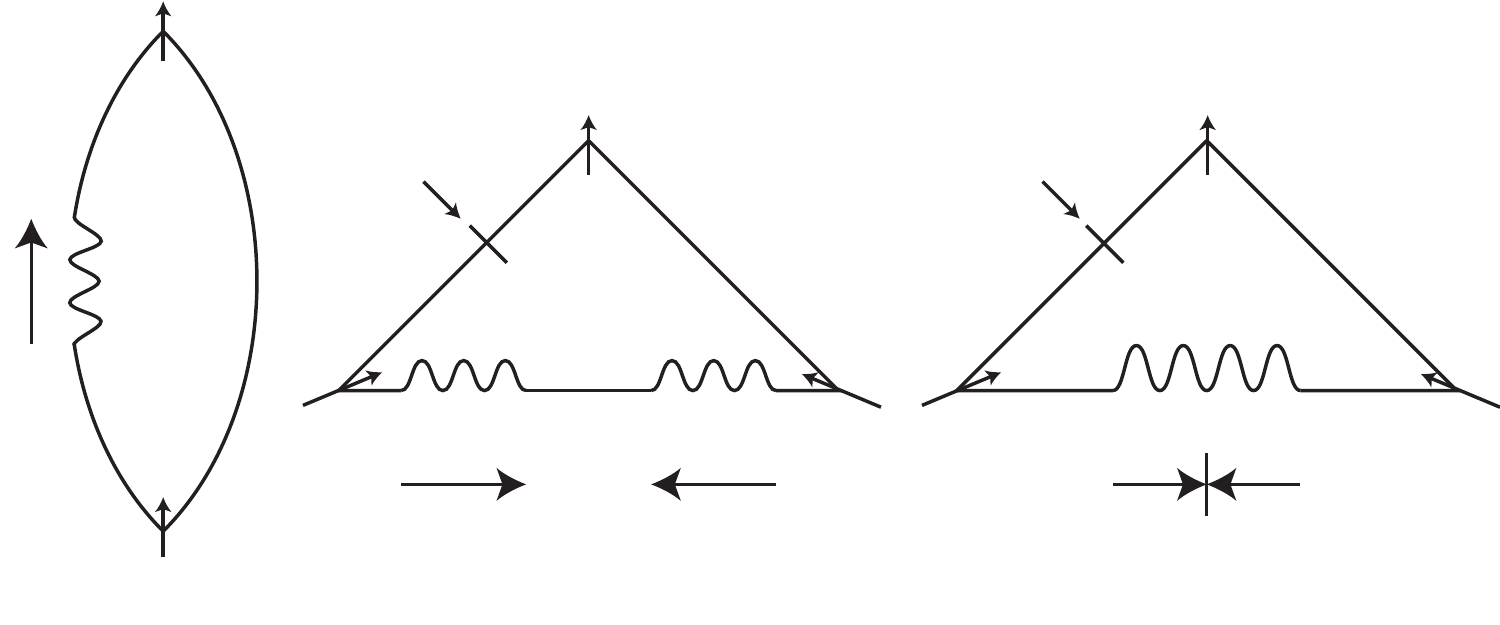}
    \put(9.5,0){(a)}
    \put(38,0){(b)}
    \put(79,0){(c)}
    \put(.5,22){$t$}
    \put(4,8){$\betas^k$}
    \put(18,22){$\betas^i$}
    \put(4,34){$\gammas^k$}
    \put(7,22){$\phi_t(\betas^k)$}
    \put(30,7){$t$}
    \put(48,7){$t$}
    \put(26.5,12){$\phi_t(\betas^k)$}
    \put(44,12){$\phi_t(\betas^k)$}
    \put(38,18){$\gammas^{i_1=k}$}
    \put(48,25){$\betas^{i_0=i}$}
    \put(31,30){$\betas^{i_3=j}$}
    \put(25,22){$\betas^{i_2}$}
    \put(76,7){$t$}
    \put(84,7){$t$}
    \put(72,12){$\phi_t(\betas^k)$}
    \put(81,12){$\phi_t(\betas^k)$}
    \put(89,25){$\betas^{i_0=i}$}
    \put(72,30){$\betas^{i_3=j}$}
    \put(66,22){$\betas^{i_2}$}
  \end{overpic}
  \caption{\textbf{Continuation maps associated to a
      Hamiltonian isotopy.} (a) A polygon defining Case~\ref{item:f3}
    of the map $f_1$. (b) and (c) Examples of polygons defining
    Case~\ref{item:f5} of the map $f_3$ for $k=i_1$. (b) represents to the
    boundary conditions in Formula~(\ref{eq:Ls-1}) while (c)
    represents the boundary conditions in Formula~(\ref{eq:Ls-2}).}
  \label{fig:CxHam}
\end{figure}
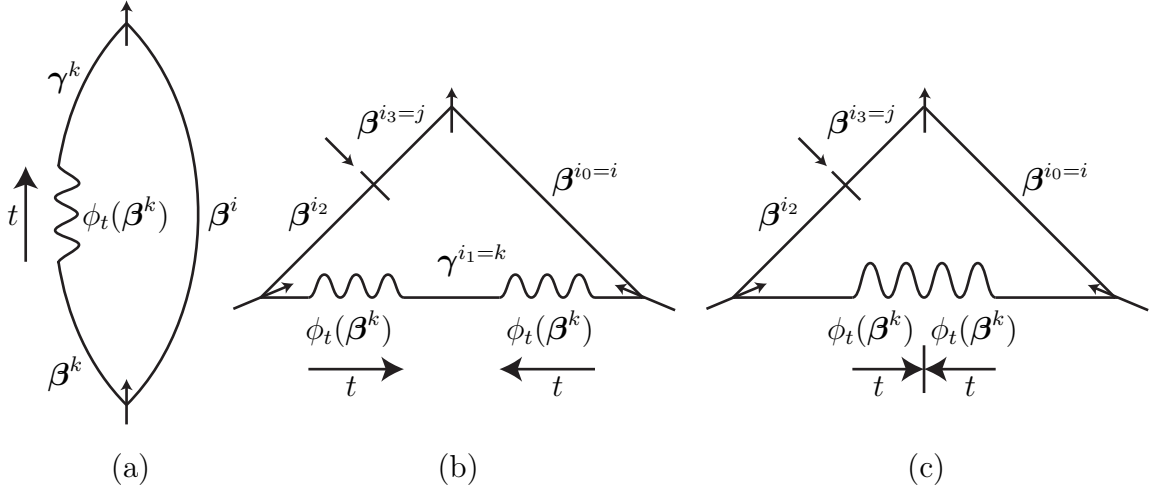

The maps $f_n$ are defined as follows:
\begin{enumerate}[label=(f-\arabic*),ref=(f-\arabic*)]
\item\label{item:f1} If $n=1$ and $k\not\in \{i_0,i_1\}$ then $f_1\co
  \CFa(\betas^{i_0},\betas^{i_1},z)\to
  \CFa(\gammas^{i_0},\gammas^{i_1},z)=\CFa(\betas^{i_0},\betas^{i_1},z)$
  is the identity map.
\item\label{item:f2} If $n>1$ and $k\not\in\{i_0,\dots,i_n\}$ then $f_n=0$.
\item\label{item:f3} For any $n\geq 1$, if $k=i_0$ then $f_n$ is defined by counting
  holomorphic $(n+1)$-gons with dynamic boundary conditions along one
  edge. That is, we count holomorphic maps
  \[
  u\co (S,\bdy S)\to ((\Sigma\setminus\{z\})\times P, L_C\cup (\betas^{i_1}\times e_2)\cup\dots\cup
  \betas^{i_n}\times e_{n+1})
  \]
  where $C\in\RR$ is allowed to vary and $L_C$ is given as follows. For each $j\in\Conf(D_n)$, fix
  an identification of a neighborhood of the edge
  $e_1$ in $P$ with $\RR\times [0,\epsilon)$ so that the symplectic
  form $\omega_j$ on $P$ is the pullback of the usual symplectic form
  on $\RR^2$. 
  We require these identifications to be continuous, consistent with
  the cylindrical ends,
  and consistent across strata of
  $\Conf(n+1)$;
  compare~\cite[Sections (9e)--(9i) and (10e)]{SeidelBook}.
  For each $j$, fix a Hamiltonian isotopy $\phi_t$, $t\in[0,1]$, from
  $\betas^k$ to $\gammas^k$, induced by a time-dependent Hamiltonian
  $H_t$. Then
  \begin{multline}\label{eq:LC}
  L_C=\left[\betas^k\times(-\infty,C]\times\{0\}\right]\cup \{(\phi_t(\betas^k),C+t,0)\mid
  t\in[0,1]\}\\
  \cup \left[\gammas^k\times[C+1,\infty)\times\{0\}\right].
  \end{multline}

  The manifold $L_C$ is not Lagrangian with respect to $\omega=\omega_\Sigma\times\omega_j$, but
  rather with respect to the deformed form
  \begin{equation}\label{eq:deformed-form}
    \omega - \left(d (\psi(s)H_t)\right)\wedge dt,
  \end{equation}
  where $\psi\co [0,\epsilon]\to \RR$ is a smooth cut-off function
  taking the value $1$ on a neighborhood of $0$ and $0$ on a
  neighborhood of $\epsilon$.  Assuming that the Hamiltonian isotopy
  $H_t$ is small enough, this deformed form is still symplectic and
  still tames the almost-complex structures $J_j$ under
  consideration. This is one of the two ``close'' requirements we place
  on the $\betas^i$ and $\gammas^i$.
\item\label{item:f4} For any $n\geq 1$, if $k=i_n$ then $f_n$ is defined similarly to
  the previous case, but with dynamic boundary conditions along the
  edge $e_{n+1}$.
\item\label{item:f5} For any $n\geq 2$, if $k=i_j$ for some $0<j<n$ then $f_n$ is
  defined by counting holomorphic $(n+1)$-gons
  \[
  u\co (S,\bdy S)\to ((\Sigma\setminus\{z\})\times P, (\betas^{i_0}\times
  e_1)\cup\dots\cup L'_s \cup\dots\cup \betas^{i_n}\times e_{n+1})
  \]
  for some $s\in(-\infty,1]$
  where the boundary condition $L'_s$ along the edge $e_{j+1}$ is
  given as follows.
  For each $j\in\Conf(D_n)$, fix an identification of a neighborhood of the
  edge $e_{j+1}$ with $\RR\times[0,\epsilon)$, as in
  Case~\ref{item:f3}.
  Then for $s<0$,
  \begin{equation}\label{eq:Ls-1}
    \begin{split}
      L'_s&=\left[\betas^k\times(-\infty,s-1]\times\{0\}\right]\cup \{(\phi_t(\betas^k),s+t-1,0)\mid
      t\in[0,1]\}\\ 
      &\qquad\qquad\cup\left[\gammas^k\times[s,-s]\times\{0\}\right]\cup\{(\phi_{1-t}(\betas^k),-s+t,0)\mid t\in[0,1]\}\\
      &\qquad\qquad\cup \left[\betas^k\times[-s+1,\infty)\times\{0\}\right].
    \end{split}
  \end{equation}
  For $s\in[0,1]$, 
  \begin{multline}\label{eq:Ls-2}
  L'_s=\left[\betas^k\times(-\infty,s-1]\times\{0\}\right]\cup \{(\phi_t(\betas^k),s+t-1,0)\mid
  t\in[0,1-s]\}\\ 
  \cup \{(\phi_{1-t}(\betas^k),-s+t,0)\mid t\in[s,1]\}
  \cup \left[\betas^k\times[1-s,\infty)\times\{0\}\right].
  \end{multline}
  See Figure~\ref{fig:CxHam}. 
  (Again, the submanifolds $L'_s$ are not Lagrangian with respect to
  $\omega$, but rather with respect to a deformed symplectic form as
  in Equation~\eqref{eq:deformed-form}. This is the second of the two
  ``close'' requirement we place on $\betas^i$ and $\gammas^i$.)
\end{enumerate}

\begin{lemma}\label{lem:f-ainf-rel}
  The maps $f_n$ satisfy the $\Ainf$-homomorphism relation~\eqref{eq:Ainf-hom}.
\end{lemma}
\begin{proof}
  There are several cases. The case that $k\not\in\{i_0,\dots,i_m\}$
  is trivial. Next, consider the case that $k=i_0$. One-dimensional
  moduli spaces of polygons as in Case~\ref{item:f3} of the definition
  of $f_m$ have three kinds of ends:
  \begin{itemize}
  \item Breakings of polygons where the edge $e_1$ does not
    break. These correspond to the terms in the first sum in
    Equation~\eqref{eq:Ainf-hom} with $i>1$.
  \item Breakings of polygons where the edge $e_1$ breaks below
    $[C,C+1]$. These correspond to the terms in the first sum in
    Equation~\eqref{eq:Ainf-hom} with $i=1$.
  \item Breakings of polygons where the edge $e_1$ breaks above
    $[C,C+1]$. These correspond to the terms in the second sum in
    Equation~\eqref{eq:Ainf-hom}.
  \end{itemize}
  
  The case that $k=i_m$ is similar to the case that $k=i_0$.

  \begin{figure}
    \centering
    \begin{overpic}[width=\textwidth, tics=10]{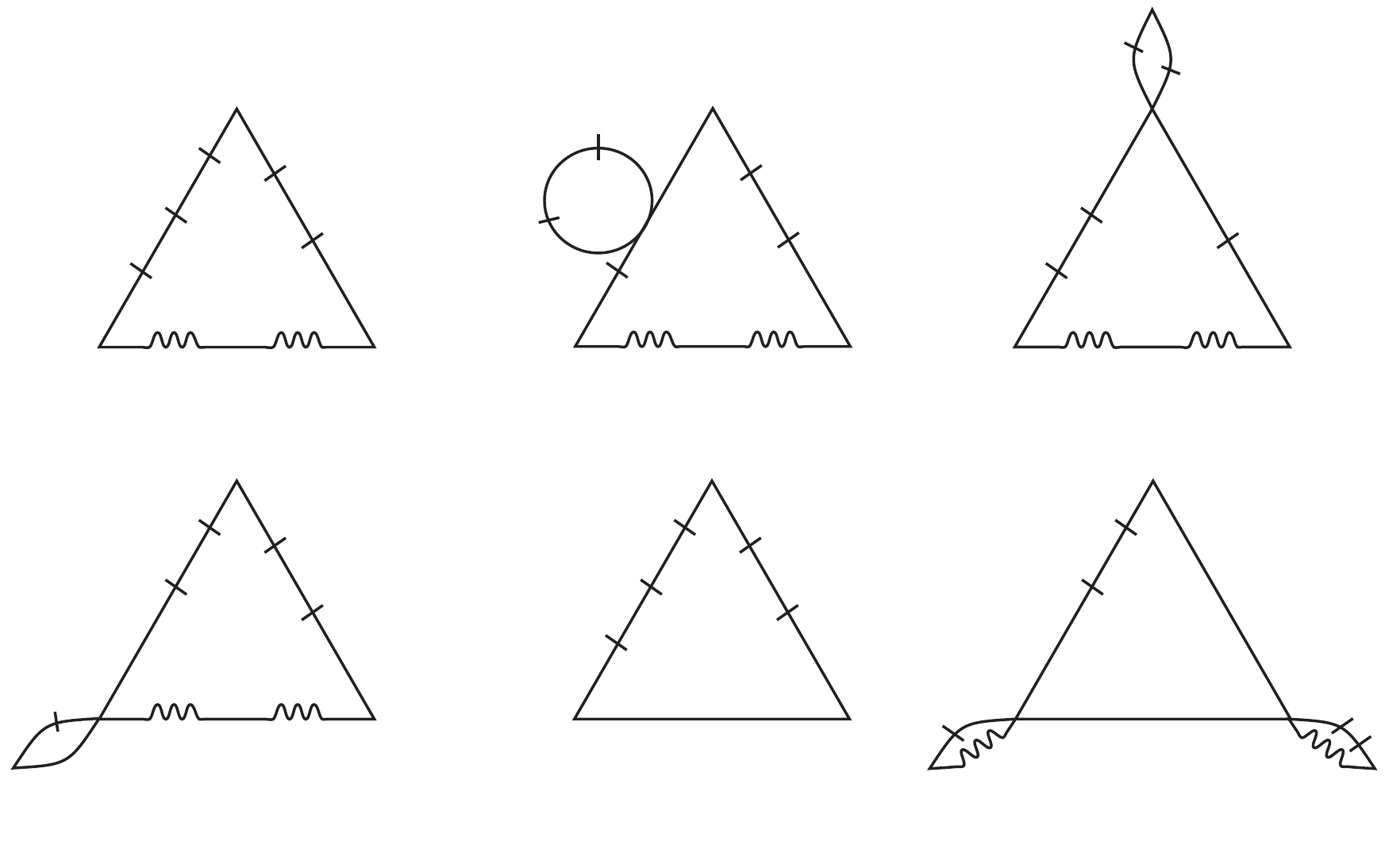}
      \put(3,36){$x_{k+1}$}
      \put(27.5,36.5){$x_{k}$}
      \put(23,34.5){$\betas^k$}
      \put(16,34.5){$\gammas^k$}
      \put(8,34.5){$\betas^k$}
      \put(51,8){$\betas^k$}
      \put(81.5,8){$\gammas^k$}
      \put(-1,6){$x_{k+1}$}
      \put(27.5,10){$x_{k}$}
      \put(37,45){$x_a$}
      \put(42.5,53.5){$x_b$}
      \put(38,36){$x_{k+1}$}
      \put(61.5,36.5){$x_{k}$}
      \put(38,9.5){$x_{k+1}$}
      \put(61.5,10){$x_{k}$}
      \put(70,36){$x_{k+1}$}
      \put(93.5,36.5){$x_{k}$}
      \put(65,5.5){$x_{k+1}$}
      \put(99.6,6){$x_{k}$}
      \put(15.5,30){(a)}
      \put(50,30){(b)}
      \put(81.5,30){(c)}
      \put(15.5,2){(d)}
      \put(50,2){(e)}
      \put(81.5,2){(f)}
    \end{overpic}
    \caption{\textbf{Proof that $f$ satisfies the $\Ainf$ relation.}
      (a) shows an $8$-gon: the output corner and the corners
      corresponding to $x_k$ and $x_{k+1}$ are drawn as corners, and
      the rest as tick marks. (b)--(f) show the ways this moduli space
      of $8$-gons can break. For cases (b) and (d) we have only shown
      one of two cases: the other case is given by reflecting the
      picture horizontally.}
    \label{fig:maps-breaking}
  \end{figure}
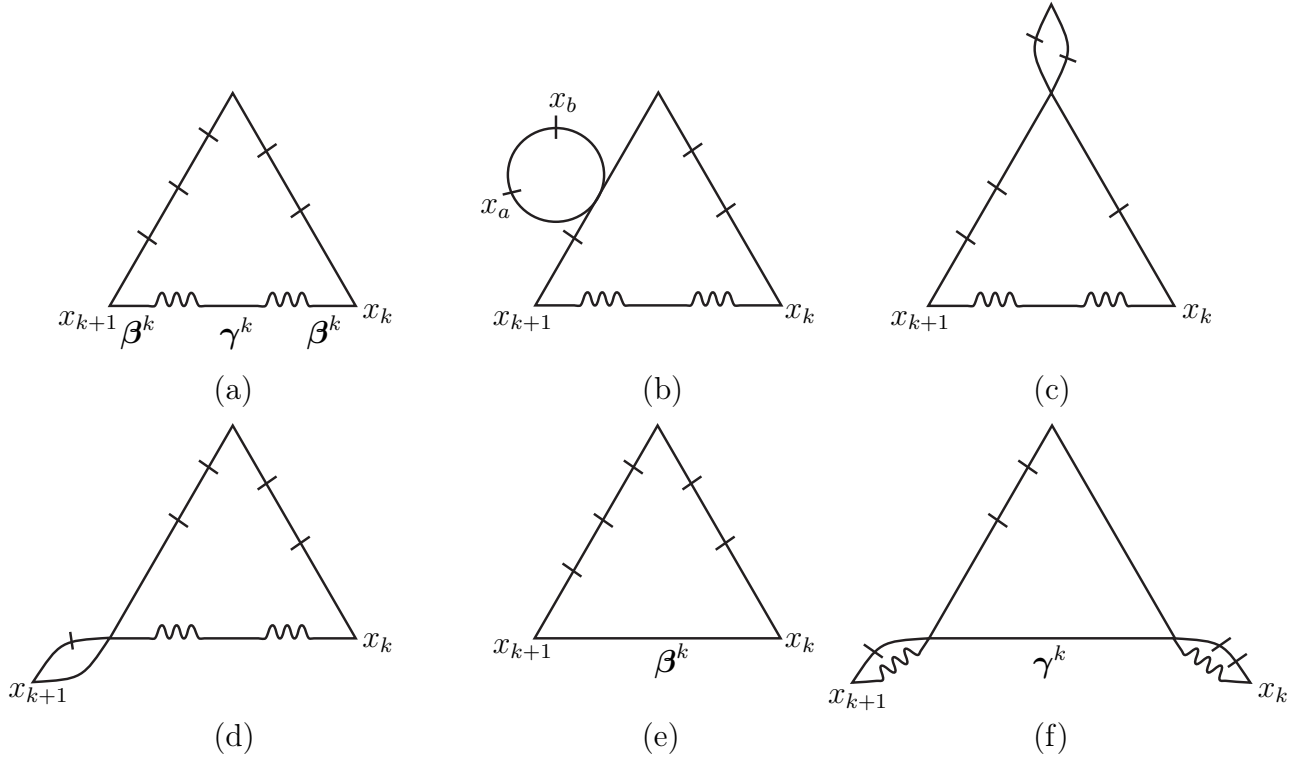

  Finally, consider the case that $k=i_j$ for some $0<j<m$. One-dimensional
  moduli spaces of polygons as in Case~\ref{item:f5} of the definition
  of $f_m$ have five kinds of ends, as shown in
  Figure~\ref{fig:maps-breaking}. The ends correspond to terms in
  Equation~\eqref{eq:Ainf-hom}, as follows:
  \begin{enumerate}
  \item Ends of type (b) correspond to terms in the first sum where
    neither $x_k$ nor $x_{k+1}$ is an input to the multiplication $m$.
  \item Ends of type (c) correspond to terms in the second sum where
    $x_k$ and $x_{k+1}$ are inputs to the same map $f$.
  \item Ends of type (d) correspond to terms in the first sum where
    exactly one of $x_k$ and $x_{k+1}$ is an input to the
    multiplication $m$.
  \item Ends of type (e) (which correspond to $s=1$) correspond to
    terms in the first sum where $x_k$ and $x_{k+1}$ are both input to
    the multiplication $m$. (Note that in this case, we must have $m-j+i=1$.)
  \item Ends of type (f) (which correspond to $s=-\infty$) correspond
    to terms in the second sum where $x_k$ and $x_{k+1}$ are inputs to
    different $f$'s. (Note that the symmetry in the definition of
    $L'_s$ forces degenerations to occur at two corners at
    once.)\qedhere
  \end{enumerate}
\end{proof}

Next, we turn to the continuation maps for chain complexes of
attaching circles. With notation as above, suppose that we are given chains
$\eta^{i_1<i_2}$ making
$(\{\betas^i\}_{i\in\IndI},\allowbreak\{\eta^{i<j}\}_{i,j\in\IndI},z)$
into a chain complex of attaching circles. Let
\[
\zeta^{i<j}=\sum_{i=i_0<i_1<\cdots<i_n=j}f_n(\eta^{i_{n-1}<i_n},\eta^{i_{n-2}<i_{n-1}}\dots,\eta^{i_0<i_1}).
\]

\begin{lemma}
  \label{lem:ChangeByIsotopy} The data
  $(\{\gammas^i\}_{i\in\IndI},\{\zeta^{i<j}\}_{i<j\in\IndI},z)$
  forms a chain complex of attaching circles.
\end{lemma}
\begin{proof}
  This is straightforward from the definitions and
  Lemma~\ref{lem:f-ainf-rel}.
\end{proof}

\begin{proposition}
  \label{prop:IsotopiesInduceEquivalences}
  Let
  $(\{\betas^i\}_{i\in\IndI},\{\eta^{i_1<i_2}\}_{i_1,i_2\in\IndI},z)$
  be a chain complex of attaching circles, $\phi_t$ an
  exact Hamiltonian isotopy, and $\{\gammas^i\}_{i\in\IndI}$
  the new collection of attaching circles gotten by letting $\phi_1$
  act on the $k\th$ tuple of attaching circles. Let
  $(\{\gammas^i\}_{i\in\IndI},\{\zeta^{i_1<i_2}\}_{i_1,i_2\in\IndI},z)$
  denote the new chain complex of attaching circles as in
  Lemma~\ref{lem:ChangeByIsotopy}. Then given 
  another $g$-tuple of attaching circles $\alphas$,
  there is a filtered quasi-isomorphism between the associated filtered complexes 
  $\CCFa(\alphas,\allowbreak\{\betas^i\}_{i\in\IndI},\allowbreak\{\eta^{i_1<i_2}\}_{i_1,i_2\in \IndI},\allowbreak z)$
  and 
  $\CCFa(\alphas,\{\gammas^i\}_{i\in\IndI},\allowbreak\{\zeta^{i_1<i_2}\}_{i_1,i_2\in \IndI},\allowbreak z)$
  (as defined in Definition~\ref{def:AssociatedComplex}).
\end{proposition}
\begin{proof}
  The quasi-isomorphism 
  \[
  F\co
  \CCFa(\alphas,\{\betas^i\}_{i\in\IndI},\{\eta^{i<j}\}_{i,j\in \IndI},z)\to
  \CCFa(\alphas,\{\gammas^i\}_{i\in\IndI},\{\zeta^{i<j}\}_{i,j\in \IndI},z)
  \]
  is defined similarly to the map $f$ above. Specifically, for
  $i\leq j\in\IndI$, define
  \[F^{i\leq j}\co\CFa(\alphas,\betas^{i},z)\to
  \CFa(\alphas,\gammas^{j},z)\] by the following cases:
  \begin{enumerate}
  \item If $i=j\neq k$ then $F^{i\leq j}$ is the identity
    map. (Compare~\ref{item:f1}.)
  \item If $i<j<k$ or $k<i<j$ then define $F^{i\leq
      j}=0$. (Compare~\ref{item:f2}.)
  \item If $k=i$ then define $F^{i\leq j}$ by counting holomorphic
    polygons with boundary
    \[
    (\alphas\times e_1)\cup L_C\cup (\gammas^{i_1}\times
    e_3)\cup\cdots \cup (\gammas^{i_n}\times e_{n+2}),
    \]
    where $i<i_1<\cdots<i_n=j$ is a sequence in $\IndI$, asymptotic to
    $\eta^{i_m<i_{m+1}}$ (or equivalently $\zeta^{i_m<i_{m+1}}$) at the corner between $\gammas^{i_m}$ and 
    $\gammas^{i_{m+1}}$. Here, $L_C$ is as in Formula~\eqref{eq:LC}.
    (Compare~\ref{item:f3}.)
  \item If $k=j$ then define $F^{i\leq j}$ similarly to the previous
    case, but with $\betas$ in place of $\gammas$: count polygons with
    boundary
    \[
    (\alphas\times e_1)\cup  (\betas^{i_0}\times
    e_2)\cup\cdots \cup (\betas^{i_n-1}\times e_{n+1}) \cup L_C,
    \]
    asymptotic to $\eta^{i_m<i_{m+1}}$ (or equivalently
    $\zeta^{i_m<i_{m+1}}$) at the corner between $\betas^{i_m}$ and
    $\betas^{i_{m+1}}$. (Compare~\ref{item:f4}.)

    Note that both this and the previous item cover the case $i=j=k$,
    and define the same map in this case: it is the usual Floer
    continuation map associated to the isotopy from $\betas^k$ to
    $\gammas^k$.
  \item If $i<k<j$ then define $F^{i\leq j}$ by counting holomorphic
    polygons with boundary
    \[
    (\alphas\times e_1)\cup (\betas^{i_0}\times e_2)\cup\dots\cup L'_s
    \cup\dots\cup (\betas^{i_n}\times e_{n+2}),
    \]
    where $i=i_0<i_1<\cdots<i_n=j$, asymptotic to $\eta^{i_m<i_{m+1}}$
    at the corner between $\betas^{i_m}$ and $\betas^{i_{m+1}}$. Here,
    $L'_s$ is as defined in Formulas~\eqref{eq:Ls-1}
    and~\eqref{eq:Ls-2}. (Compare~\ref{item:f5}.)
  \end{enumerate}
  It is straightforward to verify that $F$ is a
  chain map. The map of associated graded complexes is the usual
  Floer continuation map, and hence is a quasi-isomorphism. It follows
  that $F$ is a quasi-isomorphism, as well.
\end{proof}

\subsection{Close approximations of attaching circles}
\label{subsec:Approximations}

We would like to  describe the gluing of chain complexes of attaching circles 
which appears in Proposition~\ref{intro:GlueChainComplexes}. Before doing this,
we discuss a preliminary construction which goes into the definition:
approximations to  attaching circles.

\begin{definition}
  \label{def:Approximation}
  Let $\{\betas^i\}_{i\in \IndI}$ be an admissible collection of attaching
  circles, and $\IndJ$ another partially ordered set. 
  We say that an $\IndI\times \IndJ$-filtered admissible collection of
  attaching circles $\{\betas^{i\times j}\}_{i\times j\in \IndI\times \IndJ}$ is an {\em
    approximation} to $\{\betas^i\}_{i\in \IndI}$ if for each $i,j$,
  $\betas^{i\times j}$ is Hamiltonian isotopic to 
  $\betas^i$ where the Hamiltonian is supported in a tubular neighborhood of $\betas^i$.

  An approximation is called \emph{efficient} if for each $j_0<j_1$ in
  $\IndJ$ and each $i\in \IndI$, the differential on
  $\CFa(\betas^{i\times j_0},\betas^{i\times j_1},z)$ vanishes.  In
  particular, it has a unique generator of top degree
    \[\Theta^{i\times j_0<i\times j_1}\in 
    \CFa(\betas^{i\times j_0},\betas^{i\times j_1},z).\]
\end{definition}
It is easy to construct efficient approximations; see, for instance,
Figure~\ref{fig:close-approx}.

Let $\{\betas^i\}_{i\in \IndI}$ be an admissible set of attaching circles
and $\IndJ$ another partially ordered set. Let $\{\betas^{i\times
  j}_\epsilon\}_{i\times j\in \IndI\times \IndJ}$ be a one-parameter family of
efficient approximations, where $\lim_{\epsilon\to 0} \betas^{i\times
  j}_\epsilon=\betas^i$ (in the $C^\infty$ topology). Fix a
sequence $i_0\times j_0<\dots<i_n\times j_n$.
For $a_1\leq\dots\leq a_n$ a non-decreasing sequence, let
$R(a_1, \dots, a_n)$ denote the number of repeated entries, counted
with multiplicity (so that there are $n+1-R(a_1,\dots,a_n)$ distinct entries).
Let $k = R(i_0,\dots,i_n)$, and let $i_{s_0}<\dots<i_{s_{n-k}}$ be the
subsequence with all but the last of
each repeated entry removed.

\begin{convention}\label{conv:ij}
  To shorten notation, we will often write $\ij{k}$ for $i_k\times j_k$.
\end{convention}

Suppose that $\epsilon$ is sufficiently small. Then
given a sequence of $n+1-k$ generators
$\x_0^{i_{s_0}<i_{s_1}},\dots,\x_0^{i_{s_{n-k-1}}<i_{s_{n-k}}}$ 
and $\x_0^{i_{s_0}<i_{s_{n-k}}}$ with
the understanding that 
$\x_0^{i<i'}\in\Gen(\betas^{i},\betas^{i'})$, there
is a canonically associated sequence
of $n+1$ generators
\[\x_\epsilon^{\ij{0}<\ij{1}},\dots,\x_\epsilon^{\ij{n-1}<\ij{n}}~\text{and}~\x_{\epsilon}^{\ij{0}<\ij{n}},\]
with
$\x_\epsilon^{\ij{m}<\ij{m+1}}\in
\Gen(\betas^{i\times j}_{\epsilon},\betas^{i'\times j'}_{\epsilon})$
defined by
\begin{align*}
\x_\epsilon^{\ij{m}<\ij{m+1}}&=
\begin{cases}
  \text{nearest-point to }\x_0^{i_{s_\ell}<i_{s_{\ell+1}}} & \text{if }m=s_\ell\text{ for some }\ell\\
  \Theta^{\ij{m}<\ij{m+1}} & \text{otherwise (i.e., if $i_m=i_{m+1}$)}
\end{cases}\\
\x_\epsilon^{\ij{0}<\ij{n}} &= \text{nearest-point to }\x_0^{i_0<i_n}.
\end{align*}

\begin{definition}
  \label{def:NearEquivalent}
  Two elements $B_1,B_2$ of 
  $\pi_2(\x_\epsilon^{\ij{0}<\ij{n}},\x_\epsilon^{\ij{n-1}<\ij{n}},\dots,\x_\epsilon^{\ij{0}<\ij{1}})$
  are said to be {\em nearly equivalent} if there is a collection of doubly-periodic domains
  for the Heegaard diagrams $(\betas^{\ij{m}},\betas^{\ij{m+1}})$ for various $m$ with $i_m=i_{m+1}$
  which can be added to $B_2$ to get $B_1$. We denote the set of 
  ``near equivalence'' classes of domains by
  \[
  \pi'_2(\x_\epsilon^{\ij{0}<\ij{n}},\x_\epsilon^{\ij{n-1}<\ij{n}},\dots,\x_\epsilon^{\ij{0}<\ij{1}}).
  \]
\end{definition}

\begin{definition}
  \label{def:ApproximatingHomotopyClass}
  There is an obvious map 
  \[
  \phi_\epsilon\co \pi_2(\x_\epsilon^{\ij{0}<\ij{n}},\x_{\epsilon}^{\ij{n-1}<\ij{n}},
  \dots,\x_\epsilon^{\ij{0}<\ij{1}})\to \pi_2(\x_0^{i_{s_0}<i_{s_{n-k}}},\x_0^{i_{s_{n-k}}<i_{s_{n+1-k}}},\dots,\x_0^{i_{s_0}<i_{s_1}})
  \]
  gotten by taking the multiplicities away from the isotopy
  region. More precisely, for each attaching circle $\betas^i_k$ in
  $\betas^i$, choose two basepoints, one on either side of
  $\betas^i_{k}$, but far enough away so as to be disjoint from all
  translates $\betas^i_{\epsilon}$ for all $\epsilon>0$
  sufficiently small. Then $\phi_\epsilon(B_\epsilon)$ is the unique
  domain with the same multiplicities as $B_\epsilon$ at all of these
  points, and at the basepoint $z$. We call $B_\epsilon\in
  \phi_\epsilon^{-1}(B')$ an \emph{approximation to $B'$}.
\end{definition}

\begin{lemma}
  \label{lem:ApproximatingHomotopyClass}
  The map $\phi_\epsilon$ descends to a bijection 
  \[
  \pi_2'(\x_\epsilon^{\ij{0}<\ij{n}},\x_{\epsilon}^{\ij{n-1}<\ij{n}},
  \dots,\x_\epsilon^{\ij{0}<\ij{1}})
  \stackrel{\cong}{\longrightarrow} \pi_2(\x_0^{i_{s_0}<i_{s_{n-k}}},\x_0^{i_{s_{n-k}}<i_{s_{n+1-k}}},\dots,\x_0^{i_{s_0}<i_{s_1}}).
  \]
  Also, two nearly equivalent homotopy classes have the same index
  and, for an appropriate choice of symplectic form, the same energy.
\end{lemma}
\begin{proof}
  The first part of the statement follows readily from the identification
  \[\pi_2(\x_0^{i_0<i_n},\x_0^{i_{s_{n-k}}<i_{s_{n+1-k}}},\dots,\x_0^{i_{s_0}<i_{s_1}})
  \cong \ZZ\oplus \Ker\Big(\bigoplus_{s=0}^{n-1}{\mathrm{Span}}([\betas^{i_s<i_{s+1}}])\to H_1(\Sigma;\ZZ)\Big)\]
  (see for example~\cite[Proposition~8.2]{OS04:HolomorphicDisks}).  
  The fact that nearly equivalent homotopy classes have the same index
  and energy follows from the fact that
  periodic domains appearing in the equivalence relation have Maslov
  index zero and, with respect to a symplectic form constructed  by Perutz~\cite[Section 7]{Perutz07:handleslides} which agrees with the area
  form~\cite[Lemma 4.12]{OS04:HolomorphicDisks} away from the diagonal, zero
  energy. 
\end{proof}

\begin{figure}
  \centering
  \includegraphics{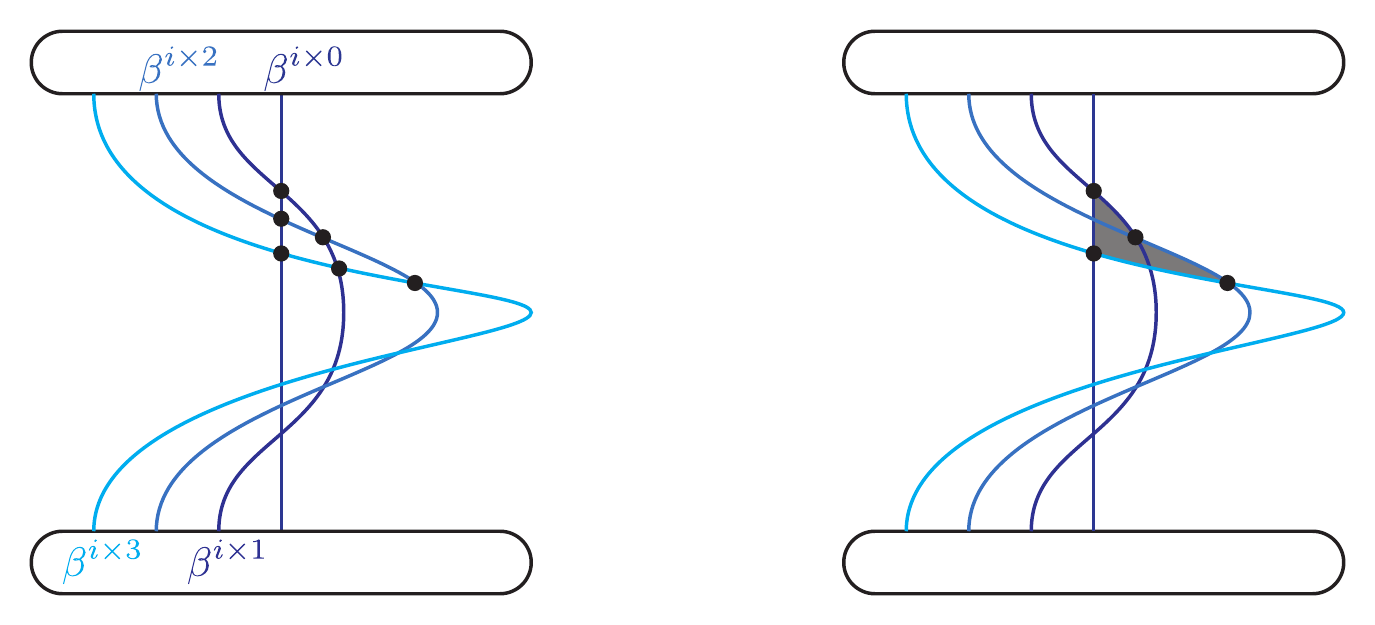}
  \caption{\textbf{A close approximation.} Left: Hamiltonian
    perturbations of $\beta^{i,0}$ giving an efficient
    approximation. If the distance $\epsilon$ between the
    $\beta^{i,j}$'s is sufficiently small then this approximation is
    close (Definition~\ref{def:CloseApproximation}). (The generators $\Theta^{i\times j<i\times j'}$ are
    marked.) Right: the domain $B_1$ from
    Lemma~\ref{lem:FiberedProductDegreeOne}.}
  \label{fig:close-approx}
\end{figure}

The dimensions of the moduli spaces $\cM^{B_\epsilon}$ and $\cM^{B'}$ are related by the following:

\begin{lemma}
  \label{lem:ApproximateDimension}
  Fix a subsequence $\ij{0}<\dots<\ij{n}$ of $\IndI\times \IndJ$, and
  let $k=R(i_0,\dots,i_n)$ be the number of repeated entries in the sequence $i_0\leq \cdots\leq i_n$.  Fix
  a homology class
  \[B'\in\pi_2(\x_0^{i_{s_0}<i_{s_{n-k}}},\x_0^{i_{s_{n-k-1}}<i_{s_{n-k}}},\dots,\x_0^{i_{s_0}<i_{s_1}})\]
  and let
  $B_\epsilon\in\pi_2(\x_\epsilon^{\ij{0}<\ij{n}},\x_\epsilon^{\ij{n-1}<\ij{n}}, \dots, \x_\epsilon^{\ij{0}<\ij{1}})$ be an 
  approximation to it.
  Then,
  \[
  \ind(B_\epsilon)=\ind(B').
  \]
  So, for a generic choice of admissible almost-complex structure,
  \[
  \dim(\cM^{B_\epsilon})-\dim(\cM^{B'})=k
  \]
  except if either:
  \begin{enumerate}
  \item\label{case:AD-2gons} $B'$ is the trivial homology class of bigons, i.e., the domain
    corresponding to $B'$ is $0$ and $k=n-1$. In this case,
    $\dim(\cM^{B'})=0$ and $\dim(\cM^{B_\epsilon})=k-1=n-2$.
  \item\label{case:AD-1gons} $B'$ is the trivial homology class of $1$-gons, i.e., the
    domain corresponding to $B'$ is $0$ and $k=n$. In this case,
    $\dim(\cM^{B'})=0$ and $\dim\cM^{B_\epsilon}=\max\{n-2,0\}$.
  \end{enumerate}
\end{lemma}

\begin{proof}
  It suffices to consider the case $k=1$: the general case follows by
  induction. Relabeling, let $\{\betas^i\}_{i=1}^n$ be the curves in
  the boundary of $B_\epsilon$ and $\{\gammas^i\}_{i=1}^{n+1}$ be the curves
  in the boundary of $B'$.
  Invariance properties of the index imply that
  $\ind(B_\epsilon)-\ind(B')$ is independent of the choice of
  Hamiltonian isotopy used to define the efficient approximation. In
  particular, we may assume that there is a $j$ so that:
  \begin{itemize}
  \item $\gammas^i=\betas^i$ for $i<j$
  \item $\gammas^i=\betas^{i-1}$ for $i>j$.
  \item $\gammas^j$ intersects $\betas^{j-1}=\gammas^{j-1}$ in two
    points, as in
  Figure~\ref{fig:partic-close-approx}.
  \end{itemize}

  \begin{figure}
    \centering
    \includegraphics{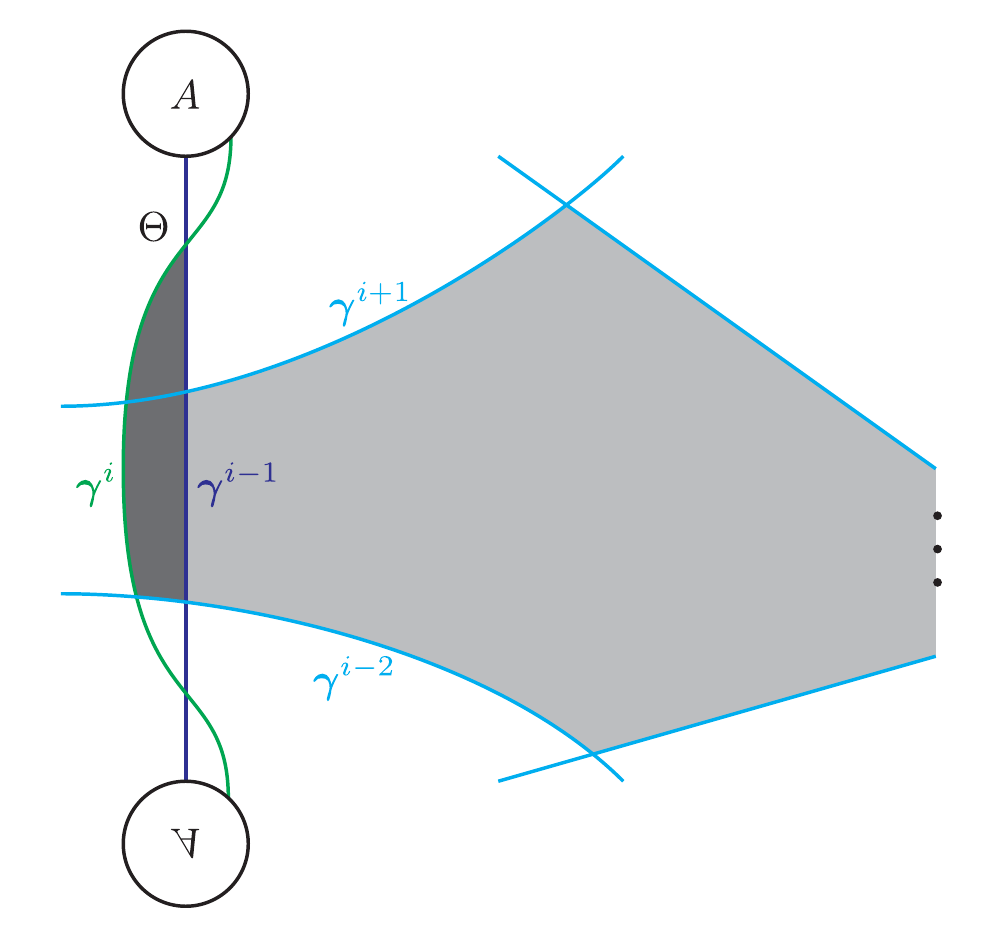}
    \caption{\textbf{A particular close approximation.} The domain
      $B_\epsilon$ is shaded; the domain
      $B'\subset B_\epsilon$ is lightly shaded and $T\subset
      B_\epsilon$ is darkly shaded.}
    \label{fig:partic-close-approx}
  \end{figure}

  In this case, the homotopy class $B_\epsilon$ decomposes as a
  juxtaposition of $B'$ and a triangle $T$. (Again, see
  Figure~\ref{fig:partic-close-approx}.) Now, $\ind(T)=0$, so by
  additivity of the index,
  \[
  \ind(B_\epsilon)=\ind(B')+\ind(T)=\ind(B').
  \]
  The expected dimension of $\cM^{B'}$ is $\ind(B')+n-3$,
  and the expected dimension of $\cM^{B_\epsilon}$ is
  $\ind(B_\epsilon)+n+1-3=\ind(B')+n+1-3$. Thus:
  \begin{itemize}
  \item If $\cM^{B'}$ and $\cM^{B_\epsilon}$ are transversally
    cut out then $\dim\cM^{B_\epsilon}-\dim\cM^{B'}=1=k$. This
    happens for generic $J$ except if $B'$ is the homology
    class of a constant map (which must be a bigon or $1$-gon).
  \item If $B'$ is the homology class of a constant bigon then
    $\dim\cM^{B'}=0$ and $\dim\cM^{B_\epsilon}=0=k-1$.
  \item If $B'$ is the homology class of a constant $1$-gon then
    $\dim\cM^{B'}=0$ and $\dim\cM^{B_\epsilon}=0$.
  \end{itemize}
  This completes the proof.
\end{proof}

\begin{definition}
  \label{def:CloseApproximation}
  An efficient approximation $\{\betas^{i\times j}\}_{i\times j\in \IndI\times \IndJ}$ is
  said to be {\em close} if the following conditions hold:
  \begin{enumerate}[label=(cl-\arabic*),ref=(cl-\arabic*)]
  \item 
    \label{cl:ApproximatePoints}
    Each generator $\x^{i_0<i_1}\in \Gen(\betas^{i_0},\betas^{i_1})$ has a canonical corresponding
    ``nearest-point'' generator $\x^{i_0\times j_0<i_1\times j_1}\in \Gen(\betas^{i_0\times j_0},\betas^{i_1\times j_1})$,
    and the corresponding ``nearest-point map'' induces an isomorphism
    of chain complexes.
    \[\CFa(\betas^{i_0},\betas^{i_1},z)
    \to
    \CFa(\betas^{i_0\times j_0},\betas^{i_1\times j_1},z).
    \]
  \item 
    \label{cl:Triangles}
    For each $i_0<i_1$ and $j_0<j_1$, the map
    \[
    m_2(\Theta^{i_1\times j_0<i_1\times j_1},\cdot)\co
    \CFa(\betas^{i_0\times j_0},\betas^{i_1\times j_0},z)\to
    \CFa(\betas^{i_0\times j_0},\betas^{i_1\times j_1},z) 
    \]
    gotten by counting
    holomorphic triangles using $\Theta^{i_1\times j_0<i_1\times j_1}$
    coincides with the nearest-point map
    \[ \CFa(\betas^{i_0\times j_0},\betas^{i_1\times j_0},z)
    \to \CFa(\betas^{i_0\times j_0},\betas^{i_1\times j_1},z).\]
    Similarly, the map
    \[
    m_2(\cdot,\Theta^{i_0\times j_0<i_0\times j_1})\co
    \CFa(\betas^{i_0\times j_1},\betas^{i_1\times j_1},z)\to
    \CFa(\betas^{i_0\times j_0},\betas^{i_1\times j_1},z) 
    \]
    gotten by counting holomorphic triangles using $\Theta^{i_0\times
      j_0<i_0\times j_1}$ coincides with the corresponding
    nearest-point map.
  \item
    \label{cl:Polygons}
    For each
    $i_0\times j_0<\dots <i_n\times j_n$ with $R(i_0,\dots,i_n)=k$, if
    $B_\epsilon\in\pi_2(\x_\epsilon^{\ij{0}<\ij{n}},
    \x_\epsilon^{\ij{n-1}<\ij{n}},\allowbreak\dots,\allowbreak
    \x_\epsilon^{\ij{0}<\ij{1}})$ approximates some
    $B'\in \pi_2(\x_0^{i_{s_0}<i_{s_{n-k}}},\x_0^{i_{s_{n-k-1}}<i_{s_{n-k}}},\dots,
    \x_0^{i_{s_0}<i_{s_1}})$ with $\ind(B')<3-n$,
    then $\cM^{B_\epsilon}$ is empty.
  \end{enumerate}
\end{definition}
Note that, in light of Lemma~\ref{lem:ApproximateDimension},
Condition~\ref{cl:Polygons} is only interesting when $\ind(B') = 2-n$.

The goal of the rest of this subsection is to prove the existence of close approximations.

\begin{lemma}
  \label{lem:PseudoHolomorphicRepresentative}
  For $\{\betas^{i\times j}_\epsilon\}_{i\times j\in \IndI \times \IndJ}$ a family of
  efficient approximations as above,
  fix a subsequence $\ij{0}<\dots<\ij{n}$ of $\IndI\times
  \IndJ$ with $R(i_0,\dots,i_n)=k$ and a homology class
  \[B'\in\pi_2(\x_0^{i_{s_0}<i_{s_{n-k}}},\x_0^{i_{s_{n-k-1}}<i_{s_{n-k}}},\dots,\x_0^{i_{s_0}<i_{s_1}}).\]
  Let
  $B_\epsilon\in\pi_2(\x_\epsilon^{\ij{0}<\ij{n}},
  \x_\epsilon^{\ij{n-1}<\ij{n}}, \dots,
  \x_\epsilon^{\ij{0}<\ij{1}})$ be approximations to
  $B'$ (in the sense of Definition~\ref{def:ApproximatingHomotopyClass}). If $\cM^{B_\epsilon}$ is non-empty for $\epsilon>0$ arbitrarily
  small then $\cM^{B'}$ is non-empty, as well.
\end{lemma}

\begin{proof}
  For both conceptual and notational clarity, we give this proof in
  the symmetric product formulation. The result in the cylindrical
  formulation then follows from the tautological correspondence, Lemma~\ref{lem:tautological}.

  Suppose we have curves
  $u_\epsilon\in \cM^{B_\epsilon}$ for some infinite subsequence of $\epsilon>0$
  converging to zero. Since we are fixing the homotopy class $B'$,
  Lemma~\ref{lem:ApproximatingHomotopyClass} gives an {\em a priori}
  energy estimate throughout the sequence.  Gromov compactness implies
  that away from finitely many points and arcs in the source, there is
  a subsequence of the $u_{\epsilon}$ which converges in
  $C^{\infty}_{\mathrm{loc}}$ to a possibly degenerate pseudo-holomorphic
  $(n+1)$-gon. (See Theorem~4.1.1 of~\cite{MS04:HolomorphicCurvesSymplecticTopology} for a nice
  treatment of Gromov compactness. Note that the current application
  appears to be slightly more general than the version stated there
  for the following three reasons: we have not one but several
  Lagrangians, our almost-complex structures are parameterized by
  points in the source, and the Lagrangians move in the
  sequence. Because the proof of Theorem~4.1.1 is local in the source,
  the fact that we have several Lagrangians causes no additional
  complications.  Since pseudo-holomorphic curves with variable almost-complex structures in the target give rise to pseudo-holomorphic
  curves in a product space, the second point also causes no
  additional difficulties. Finally, by taking a product with ${\mathbb
    C}$, a similar reduction takes care of the third point.)

  The limiting object is a possibly degenerate pseudo-holomorphic
  $(n+1)$-gon; i.e. the source of this curve is equipped with $n+1$
  punctures, the {\em obligatory punctures}, arising as the limits of
  the punctures in the sources, and possibly additional punctures
  corresponding to where the conformal structure on the $(n+1)$-gon
  degenerates, which we call {\em optional punctures}. By the
  removable singularities theorem or convergence of holomorphic curves
  to Lagrangian intersection points (depending on the puncture), we
  can extend the map continuously across all the
  punctures.\footnote{As explained in the next paragraph, punctures can be either ephemeral, if the Lagrangians on the two sides of the puncture are the same, or persistent if the Lagrangians on the two sides are different. See, for instance,~\cite[Theorem
    4.1.2]{MS04:HolomorphicCurvesSymplecticTopology} for the removable
    singularities theorem, which is relevant at the ephemeral
    punctures. Convergence of holomorphic curves to intersection
    points, which is relevant at the persistent punctures, is a
    special case of~\cite[Theorem 2]{Floer88:unregularized} or see,
    for instance,~\cite[Lemma 4.2.1]{WehrheimWoodward10:quilted}. In
    both cases, we use Perutz's result~\cite{Perutz07:handleslides}
    that the symmetric product is symplectic and the Heegaard tori are
    Lagrangian, so that the hypotheses of the theorems are satisfied;
    see also the discussion around~\cite[Theorem
    4.1.2]{MS04:HolomorphicCurvesSymplecticTopology}.}
  At the optional punctures, the limit can be taken carefully
  to get a sequence of bigons which connect pairs of optional
  punctures, so that the resulting nodal surface is connected.

  We consider now the obligatory punctures.  The $s\th$ obligatory
  puncture marks two consecutive edges, which are mapped to the
  Heegaard tori associated to $\betas^{i_{s}}$ and $\betas^{i_{s+1}}$.
  There are two cases, according to whether or not
  $i_s=i_{s+1}$. Punctures with $i_s=i_{s+1}$ we call {\em ephemeral
    punctures}, and those with $i_s\neq i_{s+1}$ we call {\em
    persistent punctures}.  Suppose that a puncture $p$ is
  persistent. Then, if the limiting curve $u'$ does not take $p$ to
  $\x_0^{i_{s}<i_{s+1}}$, but rather to some other
  $\y^{i_s<i_{s+1}}\in\Gen(\betas^{i_s},\betas^{i_{s+1}})$,
  we can reparameterize the sources to extract a chain of bigons
  connecting $\x_0^{i_{s}<i_{s+1}}$ and $\y^{i_s<i_{s+1}}$. We think
  of these bigons as forming part of the Gromov limiting curve $u'$.

  At each ephemeral puncture, we do not apply any further reparameterization;
  simply apply the removable singularities theorem
  to extend the map smoothly across that puncture.

  We now have a possibly degenerate pseudo-holomorphic $(n+1)$-gon
  representing some class 
  \[
  B''\in\pi_2( \x_\epsilon^{\ij{0}<\ij{n}}, \x_\epsilon^{\ij{n-1}<\ij{n}}, \dots, \x_\epsilon^{\ij{0}<\ij{1}}).
  \]
  Next, we argue that, perhaps after including some more bigons in the limit $u'$, we have $B''=B'$.
  
  To this end let $\{p_i\}_{i=1}^m$ be the points appearing as in Definition~\ref{def:ApproximatingHomotopyClass}.
  We can assume that all the curves in our sequence
  intersect the divisor
  $\{p_i\}\times\Sym^{g-1}(\Sigma)$ transversally; in fact, there is some integer
  $m_i$ with the property that
  $u_\epsilon^{-1}(p_i\times\Sym^{g-1}(\Sigma))$ is a degree $m_i$
  divisor in the source of $u_{\epsilon}$. We must argue that
  the preimage under $u'$ of 
  $p_i\times \Sym^{g-1}(\Sigma))$ is also a degree $m_i$
  divisor in the (possibly nodal) source of $u'$.

  Consider the sequence of divisors
  $u_\epsilon^{-1}(p_i\times\Sym^{g-1}(\Sigma))$.  Let $S$ denote the
  source of $u'$ and $\overline{S}$ the result of filling in the
  punctures of $S$, so $\overline{S}$ is a compact surface with
  boundary. We can assume, by passing to a subsequence, that the
  sequence $u_\epsilon^{-1}(p_i\times\Sym^{g-1}(\Sigma))$ converges to
  a divisor in $\overline{S}$.

  Clearly, if a subsequence of points in the divisors converge to some
  point $p$ in the interior of $\overline{S}$, then $u'$ maps $p$ into
  $p_i\times \Sym^{g-1}(\Sigma)$. 

  Taking the Gromov limit carefully, we can rule out the case where
  some subsequence in the divisors converges to a persistent
  puncture. If a subsequence runs off a puncture, reparameterize the
  surface to extract a sequence of bigons. In the limit, then, the
  sequence of divisors limit to points in the interior of a chain of
  disks attached at the persistent puncture.  
  We can similarly eliminate a subsequence of the divisors 
  converging to an optional puncture: by attaching bigons to the
  source, the subsequence will limit to the interior of one of the
  bigons attached at the optional puncture.

  It is impossible for some subsequence in the divisors
  to converge to a point in the boundary of the source of $u'$: for in
  that case, we could extract a holomorphic disk which meets
  $p_i\times \Sym^{g-1}(\Sigma)$ with boundary in one of the Heegaard
  tori (since the divisor $p_i\times \Sym^{g-1}(\Sigma)$ is disjoint
  from the Heegaard tori). But any holomorphic disk with boundary
  contained in some Heegaard torus has the same local multiplicity at
  $p_i$ as it does at $z$, which we assumed to be zero.

  In a similar vein, it is impossible for some subsequence in the
  divisors to converge to a point which is an ephemeral puncture. For
  if there were such a subsequence, we would be able to translate back from the puncture
  to obtain a pseudo-holomorphic strip $[0,1]\times {\mathbb R}$ with
  finite energy, satisfying the boundary conditions that $\{0\} \times
  {\mathbb R}$ and $\{1\}\times {\mathbb R}$ are mapped into some
  fixed Heegaard torus, and the normalization condition that point $s
  \times 0$ is obtained as a limit of points mapped into $p_i\times
  \Sym^{g-1}(\Sigma)$. In the case where the limit point lies in the
  interior (i.e. $0<s<1$), by the local $C^{\infty}$ convergence we
  conclude that $u(s,0)\in p_i\times \Sym^{g-1}(\Sigma)$. (We are
  using the fact that there are no sphere components in the
  Gromov limit: non-trivial spheres have non-zero local multiplicity
  at $z$.) Thus, after removing the singularities at the two
  punctures on the boundary, we can view the result as a holomorphic
  disk with boundary in a Heegaard torus with non-zero multiplicity at
  $p_i$, a contradiction. The possibility that $s=0$ or $1$ is ruled
  out similarly.
\end{proof}

\begin{proposition}
  \label{prop:CloseApproximations}
  Let $\{\betas^i\}_{i\in \IndI}$ be an admissible collection of attaching
  circles, and $\IndJ$ another partially ordered set. Then, given a
  generic, admissible collection of almost-complex structures, there is an
  $\IndI\times \IndJ$-filtered close approximation
  $\{\betas^{i\times j}\}_{i\times j \in \IndI\times \IndJ}$ to
  $\{\betas^{i}\}_{i\in \IndI}$  (in the sense of
  Definitions~\ref{def:Approximation}
  and~\ref{def:CloseApproximation}).
\end{proposition}

\begin{proof}
  For each circle in $\betas^i$ consider the Hamiltonian perturbations
  shown in Figure~\ref{fig:close-approx} (for some total order
  extending the partial order on $\IndI$). For any $\epsilon$, these
  perturbations are efficient. Given a
  generic, admissible almost-complex structure, the
  moduli spaces of holomorphic curves in
  $(\Sigma,\betas^{i_0},\betas^{i_1},z)$ are transversally cut out. In
  particular, the moduli spaces do not change if we perturb
  $\betas^{i_0}$ and $\betas^{i_1}$ slightly. So, for $\epsilon$
  small enough, Condition~\ref{cl:ApproximatePoints} holds.

  Inspecting the diagram, there is a small triangle giving the nearest-point
  map as a term in $m_2(\Theta^{i_1\times j_0<i_1\times
    j_1},\cdot)$. It follows from
  Lemma~\ref{lem:PseudoHolomorphicRepresentative} that these are the
  only terms. Specifically, let $B_\epsilon$ be the domain of an index $0$
  triangle and suppose that $B_\epsilon$ admits a holomorphic
  representative for arbitrarily small $\epsilon$. The domain
  $B_\epsilon$ approximates some domain of bigons $B$, and by
  Lemma~\ref{lem:PseudoHolomorphicRepresentative} the domain $B$ has a
  holomorphic representative. By Lemma~\ref{lem:ApproximateDimension},
  the index of $B$ is also $0$. But this implies $B$ must be the
  trivial homology class. Hence, $B_\epsilon$ is supported in the
  isotopy region. Inspecting the diagram, this in turn implies that
  $B_\epsilon$ is the small triangle class already considered. So,
  Condition~\ref{cl:Triangles} holds.

  The argument for Condition~\ref{cl:Polygons} is similar. Again, if
  $B_\epsilon$ admits a holomorphic representative for arbitrarily
  small $\epsilon$ then by
  Lemma~\ref{lem:PseudoHolomorphicRepresentative} $B'$ admits a
  holomorphic representative. But since $\ind(B')<3-n$ this
  contradicts the assumption that we were working with a generic
  family of almost-complex structures.
\end{proof}

\subsection{Connected sums of chain complexes of attaching circles}
\label{subsec:ConnSum}

To define the gluing construction from
Proposition~\ref{intro:GlueChainComplexes}, we will form connected
sums of Heegaard surfaces.

\begin{definition}
  \label{def:ConnectedSum}
  Let $(\Sigma_1,\IndI,\{\betas^i\}_{i\in \IndI}, \{\eta^{i_1<i_2}\}_{i_1,i_2\in
    \IndI},z)$ and $(\Sigma_2,\IndJ,\{\gammas^j\}_{j\in \IndJ},
  \{\zeta^{j_1<j_2}\}_{j_1,j_2\in \IndJ},z)$ be chain complexes
  of attaching circles. We form a new $\IndI\times \IndJ$-filtered chain
  complex of attaching circles
$$(\Sigma_1\conn\Sigma_2,\{\deltas^{i\times j}\}_{i\times j\in \IndI\times \IndJ},
\{\omega^{\ij{1}<\ij{2}}\}_{\ij{1},\ij{2}\in \IndI\times \IndJ},z),$$
as follows. The surface $\Sigma_1\conn\Sigma_2$ is obtained by taking the
connected sum of $\Sigma_1$ and $\Sigma_2$ near the basepoint $z$.
Let $\{\betas^{i\times j}\}_{i\times j\in \IndI\times \IndJ}$ be a close approximation 
to $\{\betas^i\}_{i\in \IndI}$ (in the sense of Definitions~\ref{def:Approximation}
  and~\ref{def:CloseApproximation}) 
and let $\{\gammas^{i\times j}\}_{i\times j\in \IndI\times \IndJ}$ be a close approximation to
$\{\gammas^{j}\}_{j\in \IndJ}$.
We then define
\[
\deltas^{i\times j}=\betas^{i\times j}\cup \gammas^{i\times j}
\subset \Sigma_1\conn\Sigma_2.
\]
Define chains
$$\omega^{\ij{0}<\ij{1}}\in \CFa(\betas^{i_0\times j_0},
\gammas^{i_1\times j_1},z)$$
by
\begin{equation}
  \label{eq:DefineProductChain}
\omega^{\ij{0}<\ij{1}}=
\left\{\begin{array}{ll}
    \Theta^{i_0\times j_0<i_0\times j_1}\otimes \zeta^{i_0\times j_0<i_0\times j_1}
    & {\text{if $i_0=i_1$}} \\
    \eta^{i_0\times j_0<i_1\times j_0}\otimes \Theta^{i_0\times j_0<i_1\times j_0}
    & {\text{if $j_0=j_1$}} \\
    0 & {\text{otherwise.}}
  \end{array}
\right.
\end{equation}
Here, $\zeta^{i_0\times j_0<i_0\times j_1}$ and $\eta^{i_0\times j_0<i_1\times j_0}$
are generators in $\Sigma_2$ and $\Sigma_1$ respectively gotten by
applying the nearest-point map to 
$\zeta^{j_0<j_1}$ and $\eta^{i_0<i_1}$ respectively; 
and the right-hand-side of Equation~\eqref{eq:DefineProductChain}
uses the K{\"u}nneth isomorphism of chain groups (ignoring the differential)
\begin{equation}\label{eq:Kunneth}
\CFa(\betas^{\ij{0}},\betas^{\ij{1}},z)\otimes 
\CFa(\gammas^{\ij{0}},
\gammas^{\ij{1}},z) 
\stackrel{\cong}{\longrightarrow}
\CFa(\betas^{\ij{0}}\cup \gammas^{\ij{0}},
\betas^{\ij{1}}\cup \gammas^{\ij{1}},z).
\end{equation}
The
resulting $\IndI\times \IndJ$-filtered complex is called the {\em connected
  sum of chain complexes of attaching circles $(\Sigma_1,\IndI,\{\betas^i\}_{i\in \IndI},
\{\eta^{i_1<i_2}\}_{i_1,i_2\in \IndI},z)$ and $(\Sigma_2,\IndJ,\{\gammas^j\}_{j\in \IndJ}, 
\{\zeta^{j_1<j_2}\}_{j_1,j_2\in \IndJ},z)$} and
denoted
$$(\Sigma_1,\{\betas^i\}_{i\in \IndI}, \{\eta^{i_1<i_2}\}_{i_1,i_2\in \IndI},z)
\#(\Sigma_2,\{\gammas^j\}_{j\in \IndJ}, \{\zeta^{j_1<j_2}\}_{j_1,j_2\in \IndJ},z).$$
\end{definition}

We would like to show that the construction from
Definition~\ref{def:ConnectedSum} indeed gives a chain complex of
attaching circles. This will involve analyzing moduli spaces of
polygons in a connected sums of surfaces (Lemmas~\ref{lem:FiberedProduct}
and~\ref{lem:FiberedProductDegreeOne} below), and their limits as curves
are isotoped (Proposition~\ref{prop:CloseApproximations} above).

Before stating the lemmas we will need about polygons in connected
sums, we review some topological preliminaries.  Fix
$\IndI\times\IndJ$-filtered chain complexes of attaching circles
$\{\betas^{i\times j}\}_{i\times j\in\IndI\times\IndJ}$ and
$\{\gammas^{i\times j}\}_{i\times j\in\IndI\times\IndJ}$ in $\Sigma_1$
and $\Sigma_2$ respectively.  Every homotopy class 
\[
B\in
  \pi_2(\x^{\ij{0}<\ij{n}}\otimes \y^{\ij{0}<\ij{n}},
  \x^{\ij{n-1}<\ij{n}}\otimes\y^{\ij{n-1}<\ij{n}},\dots,
  \x^{\ij{0}<\ij{1}}\otimes\y^{\ij{0}<\ij{1}})
\]
with $n_z(B)=0$ has a corresponding splitting as
$B=B_1\conn B_2$, where 
\[
B_1\in\pi_2(\x^{\ij{0}<\ij{n}}, \x^{\ij{n-1}<\ij{n}},\dots, \x^{\ij{0}<\ij{1}})
\]
for the multi-diagram
$(\Sigma_1,\{\betas^{i\times j}\}_{i\times j\in\IndI\times \IndJ},z)$
and 
\[
B_2\in\pi_2(\y^{\ij{0}<\ij{n}},\y^{\ij{n-1}<\ij{n}},\dots,\y^{\ij{0}<\ij{1}})
\]
for the multi-diagram
$(\Sigma_2,\{\gammas^{i\times j}\}_{i\times j\in \IndI\times
  \IndJ},z)$.

Consider the forgetful map
\begin{equation}\label{eq:kappa}
  \kappa_{B_1}\co\cM^{B_1} \to \Conf(D_{n+1}),
\end{equation} 
defined by $\kappa_{B_1}(j,u)=j$, and the
corresponding map
$\kappa_{B_2}\co \cM^{B_2}\to \Conf(D_{n+1})$. 

\begin{definition}
  Let $\{J^1_j\}$ and $\{J^2_j\}$ be admissible collections of almost-complex structures for $\Sigma_1$ and $\Sigma_2$, respectively. Let
  $\{J_j\}$ be an admissible collection of almost-complex structures
  for $\Sigma=\Sigma_1\conn\Sigma_2$ which agrees with $J^1_j\amalg
  J^2_j$ away from the connected sum region. We call such an admissible collection
  of almost-complex structures on $\Sigma$ {\em compatible
    with the splitting} $\Sigma=\Sigma_1\conn\Sigma_2$.
\end{definition}

Admissible collections of almost-complex structures compatible with
the splitting $\Sigma=\Sigma_1\conn\Sigma_2$ clearly exist: they are easy
to construct explicitly from $\{J^1_j\}$ and $\{J^2_j\}$.

\begin{lemma}
  \label{lem:FiberedProduct}
  Fix admissible collections of almost-complex structures on $\Sigma$
  which are compatible with the splitting $\Sigma=\Sigma_1\conn\Sigma_2$.
  Then 
  the moduli space $\Mod^{B_1\conn B_2}$ is the fibered product
  of $\kappa_{B_1} \co \cM^{B_1}\to \Conf(D_{n+1})$ with $\kappa_{B_2}\co
  \cM^{B_2}\to \Conf(D_{n+1})$. 
  Further, if $J_j^1$ and $J_j^2$ are chosen generically then $\cM^{B_1}$ and $\cM^{B_2}$ are transversely cut out and the maps $\kappa_{B_1}$ and $\kappa_{B_2}$ are transverse to each other, so $\cM^B$ is a smooth manifold and
  \begin{equation}
    \label{eq:AdditiveIndex}
    \dim(\cM^B)=\dim(\cM^{B_1})+\dim(\cM^{B_2})-n+2.
  \end{equation}
\end{lemma}
\begin{proof}
  The description of $\Mod^{B_1\conn B_2}$ as a fibered product
  is immediate from the definitions. To wit, suppose 
  \[
  \bigl(u\co S\to \Sigma\times D_{n+1}\bigr)\in
  \cM^{B_1 \conn B_2}.
  \]
  So, $u$ is $J_j$-holomorphic for some $j\in\Conf(D_{n+1})$. 
  Since the fibers of $\pi_\bD$ are holomorphic
  (Condition~\ref{item:J-fiber}), the surface $S$ decomposes as a
  disjoint union $S=S_1\amalg S_2$, where $u(S_i)\subset
  \Sigma_i\times D_{n+1}$. The map $u|_{S_i}$ is a $J^i_j$-holomorphic
  representative of $B_i$.
  Conversely, given $J^i_j$-holomorphic representatives $u_i$ of
  $B_i$ ($i=1,2$), the map $u_1\amalg u_2$ is a $J_j$-holomorphic
  representative of $B$.

  The  transversality statement follows along the lines of~\cite[Chapter 3]{MS04:HolomorphicCurvesSymplecticTopology} but with three additional complications:
  \begin{enumerate}
  \item We are working in the Lagrangian boundary case, rather than the closed case, so we have to work with slightly different Sobolev spaces; see, for instance,~\cite[Section 3]{Lipshitz06:CylindricalHF} for a review of the relevant spaces.
  \item We are working with varying conformal structures on the source polygons. In the proof of transversality, one simply multiplies the source of the $\overline{\bdy}$-map by $\Conf(D_{n+1})$. Of course, this makes it easier to achieve transversality (in a sense that can be made precise). In particular, the arguments from~\cite{MS04:HolomorphicCurvesSymplecticTopology} go through essentially without change.
  \item We want to ensure that $\kappa_{B_1}$ is transverse to $\kappa_{B_2}$. It is not hard to see that for any choice of $J_j^2$ so that $\cM^{B_2}$ is transversely cut out we can choose a family $J_j^1$ so that $\cM^{B_1}$ is transversely cut out and $\ev\co \cM^{B_1}\to \Conf(D_{n+1})$ is transverse to $\cM^{B_2}$, using the fact that in the configuration space for $B_1$ we have multiplied by $\Conf(D_{n+1})$, and the (universal) $\overline{\bdy}$ operator was already transverse to the $0$-section before this multiplication.
  \end{enumerate}
  With these hints, we leave the details of the transversality argument to the reader.
\end{proof}

Lemma~\ref{lem:FiberedProduct} has the following simple special case:

\begin{lemma}
  \label{lem:FiberedProductTriangles}
  Let $(\Sigma_1,\betas^0,\betas^1,\betas^\infty,z)$ and
  $(\Sigma_2,\gammas^0,\gammas^1,\gammas^\infty,z)$ be admissible
  triples of attaching circles.  Fix admissible collections of almost-complex structures on $\Sigma$ which are compatible with the
  splitting $\Sigma=\Sigma_1\conn \Sigma_2$.  Fix
  $B_1\in\pi_2(\x^{0<2},\x^{1<2},\x^{0<1})$ and
  $B_2\in\pi_2(\y^{0<2},\y^{1<2},\y^{0<1})$ with $\ind(B_i)=0$,
  and consider the connected sum
  \[B=B_1\conn B_2\in\pi_2(\x^{0<2}\conn \y^{0<2},\x^{1<2}\conn \y^{1<2},\x^{0<1}\conn \y^{0<1}).\]
  Then, 
  \begin{equation}
    \label{eq:ProductDescription}
    \#\cM^{B_1\conn B_2}=\#\cM^{B_1}\cdot\#\cM^{B_2}.
  \end{equation}
  In particular,
  \begin{equation}
    \label{eq:TrianglesDecompose}
    m_2(\x^{1<2}\conn \y^{1<2},\x^{0<1}\conn \y^{0<1})
  = m_2(\x^{1<2},\x^{0<1})\otimes m_2(\y^{1<2},\y^{0<1}).
  \end{equation}
\end{lemma}
\begin{proof}
  Equation~\eqref{eq:ProductDescription} follows immediately from
  Lemma~\ref{lem:FiberedProduct}: the space of conformal structures
  $\Conf(D_3)$ consists of a single point, and hence the fibered
  product description reduces to a simple Cartesian product.
  Equation~\eqref{eq:TrianglesDecompose} follows immediately from
  Equation~\eqref{eq:ProductDescription}.
\end{proof}

\begin{lemma}
  \label{lem:FiberedProductDegreeOne}
  Fix an admissible collection of almost-complex structures on $\Sigma$
  which is compatible with the splitting $\Sigma=\Sigma_1\conn \Sigma_2$.
  Fix some $g$-tuple of attaching circles $\betas^{i_0}$ in $\Sigma_1$,
  let $\IndJ$ be a
  partially ordered set and let $\{\betas^{i_0\times j}\}_{j\in\IndJ}$ be close approximations to
  $\betas^{i_0}$.
  Then the map 
  \[
  \kappa=\coprod_{B_1}\kappa_{B_1}\co \coprod_{B_1\in
    \pi_2(\Theta^{i_0\times j_0<i_0\times j_n}, \Theta^{i_{0}\times
      j_{n-1}<i_0\times j_n},\dots, \Theta^{i_0\times j_0<i_0\times
      j_1})} \Mod^{B_1}\to \Conf(D_{n+1})
  \]
  has degree $1\pmod{2}$.

  This has the following consequence. Let $\{\gammas^j\}_{j\in\IndJ}$ be a
  $\IndJ$-filtered admissible
  collection of attaching circles in $\Sigma_2$ and fix a sequence
  $j_0<\dots<j_n$,
  generators $\y^{j_0<j_1},\dots,\y^{j_{n-1}<j_n}$ and $\y^{j_0<j_n}$
  and some $B_2\in\pi_2(\y^{j_0<j_n},\y^{j_{n-1}<j_n},\dots,
  \y^{j_0<j_1})$ for the multi-diagram
  $(\Sigma_2,\{\gammas^{j}\}_{j\in J},z)$.  
Let $\x^{i_0\times
    j_p<i_0\times j_q} = \Theta^{i_0\times j_p<i_0\times j_q}\conn 
  \y^{j_p<j_q}$. If $\dim(\cM^{B_2})=0$ then
  \begin{equation}\label{eq:FiberedProductDegreeOne}
  \#\cM^{B_2}\equiv\#\biggl(\bigcup_{B_1\in \pi_2(\Theta^{i_0\times j_0<i_0\times
      j_n},\dots,
    \Theta^{i_0\times j_0<i_0\times j_1})}\cM^{B_1\conn B_2}\biggr)\pmod{2}.
  \end{equation}
\end{lemma}
\begin{proof}
  First, consider the case of triangles ($n=2$). The space
  $\Conf(D_3)$ is a single point, so the statement is equivalent to
  the statement that the triangle map is given by
  \[
  F(\Theta^{i_0\times j_1<i_0\times j_2}\otimes \Theta^{j_0\times
    j_0<i_0\times j_1})=\Theta^{i_0\times j_0<i_0\times j_2}.
  \]
  This is well-known. To verify it, work first with a particular
  Heegaard triple---for instance, a sub-diagram of
  Figure~\ref{fig:close-approx}---and then observe that the triangle
  count is an isotopy invariant.

  Next we turn to the general case. The top-dimensional boundary of 
  \[
  \ocM\coloneqq \coprod_{B_1\in \pi_2(\Theta^{i_0\times j_0<i_0\times j_n},
    \Theta^{i_{0}\times j_{n-1}<i_0\times j_n},\dots,
    \Theta^{i_0\times j_0<i_0\times j_1})} \ocM^{B_1}\to
  \oConf(D_{n+1})
  \]
  (where $\ocM^{B_1}$ denotes the compactification of $\cM^{B_1}$)
  has two kinds of points:
  \begin{itemize}
  \item Boundary points where the $(n+1)$-gon $D_{n+1}$ degenerates as
    a union of two polygons with at least $3$ vertices each, and
  \item Boundary points where the $(n+1)$-gon $D_{n+1}$ does not
    degenerate but rather the curve splits off a bigon at one of the corners.
  \end{itemize}
  The map $\kappa$ sends the first kind of boundary points to points
  in the boundary of $\oConf(D_{n+1})$. The bigons occurring in points
  of the second kind are counted in the differential on the complex
  $\CFa(\betas^{i_0\times j},\betas^{i_0\times j'},z)$. So, since the
  differential on $\CFa(\betas^{i_0\times j},\betas^{i_0\times j'},z)$
  vanishes, boundary points of the second kind occur in pairs, both
  lying over the same point of $\Conf(D_{n+1})$. Thus, $\kappa$
  defines a relative cycle in
  $C_{n-2}(\oConf(D_{n+1}),\bdy\oConf(D_{n+1});\ZZ/2\ZZ)$ and hence
  $\kappa$ has a well-defined degree modulo $2$.
  

  Next, consider the preimage of the corner $q$ of
  $\oConf(D_{n+1})$ corresponding to the decomposition of $B_1$ into
  triangles, $B_1=B_1^n*B_1^{n-1}*\cdots*B_1^2$, where 
  \[
  B_1^{k}\in\pi_2(\Theta^{i_0\times j_0<i_0\times j_k},
  \Theta^{i_0\times j_{k-1}<i_0\times j_k}, \Theta^{i_0\times
    j_0<i_0\times j_{k-1}}),
  \]
  say. By the triangle case,
  \[
  \bigcup_{B_1^k\in \pi_2(\Theta^{i_0\times j_0<i_0\times j_k},
    \Theta^{i_0\times j_{k-1}<i_0\times j_k}, \Theta^{i_0\times
      j_0<i_0\times j_{k-1}})}\cM^{B_1^k}
  \]
  has (algebraically) one holomorphic representative. Moreover, by
  Oh's boundary perturbation technique~\cite{Oh96:boundary}, say, this
  representative is transversely cut out. Thus, standard gluing
  techniques (see, e.g.,~\cite[Proposition 5.39]{LOT1} for a more detailed discussion in
  a similar situation) imply that near this corner $q$, the moduli
  space $\cM$ is modeled on (an odd number of copies of)
  $[0,\epsilon)^{n-1}$, projecting to $\oConf(D_{n+1})$ by local
  homeomorphisms of topological manifolds with corners. In particular,
  this implies that
  \[
  \bigcup_{B_1^2,\dots,B_1^n}\cM^{B_1^n}\times\cM^{B_1^{n-1}}\times\dots\times\cM^{B_1^2}
  \]
  has (algebraically) one holomorphic representative, as well. Thus, $\kappa$ has degree $1\pmod{2}$.


  Finally, Equation~\eqref{eq:FiberedProductDegreeOne} now follows from
  Lemma~\ref{lem:FiberedProduct}.
\end{proof}

The following is a precise version of Proposition~\ref{intro:GlueChainComplexes};

\begin{proposition}
  \label{prop:GlueChainComplexes}
  If the curves $\betas^{i\times j}$ and $\gammas^{i\times j}$ in
  Definition~\ref{def:ConnectedSum} are close
  approximations to $\betas^i$ and $\gammas^j$ then the connected sum of
  attaching circles
  \[
  (\Sigma_1\cup \Sigma_2,\IndI\times \IndJ,\{\deltas^{i\times j}\}_{i\times j\in \IndI\times \IndJ},
  \{\omega^{\ij{1}<\ij{2}}\}_{\ij{1},\ij{2}\in \IndI\times \IndJ},z)
  \] 
  is an $\IndI\times \IndJ$-filtered
  chain complex of attaching circles
  (Definition~\ref{def:ChainComplex}).
\end{proposition}

\begin{proof}
  We must show the structure equation~\eqref{eq:Compatibility} holds
  for the classes $\omega$. In fact, we will show that most terms in
  the Equation~\eqref{eq:Compatibility} vanish identically.
  
  Consider a sequence $\ij{0}<\dots<\ij{n}$ of
  indices in $\IndI\times\IndJ$. Let $k=R(i_0,\dots,i_n)$ and
  $\ell=R(j_0,\dots j_n)$.

  \textbf{Claim 1.} If $k+\ell\neq n$ then
  \begin{equation}
    \label{eq:term-vanish-1}
    m_{n}(\omega^{\ij{n-1}<\ij{n}},\dots,\omega^{\ij{0}<\ij{1}})=0.
  \end{equation}
  Indeed, in this case, some
  $\omega^{i_\alpha\times j_\alpha<i_{\alpha+1}\times j_{\alpha+1}}$
  is itself zero.

  So, from now on, we restrict to sequences with $k+\ell=n$.

  \textbf{Claim 2.} We have
  \begin{equation}
    \label{eq:term-vanish-2}
    m_{n}(\omega^{\ij{n-1}<\ij{n}},\dots,\omega^{\ij{0}<\ij{1}})=0
  \end{equation}
  unless one of:
  \begin{enumerate}
    \item $n=2$ and $k=\ell=1$,
    \item $k=n$ and $\ell=0$, or
    \item $k=0$ and $\ell=n$.
  \end{enumerate}

  To see this, fix for each $m=0,\dots,n-1$ a generator
  $\x^m\otimes\y^m\in\CFa(\betas^{\ij{m}},\betas^{\ij{m+1}},z)$ which is a component (summand) of the chain
  $\omega^{\ij{m},\ij{m+1}}$. Fix also a homotopy
  class of $n+1$-gons
  $B\in\pi_2(\x^n\otimes\y^n,\dots,\x^0\otimes\y^0)$ with
  $n_z(B)=0$ and
  \begin{equation}
    \label{eq:ZeroDimensionalModuliSpace}
    \dim(\cM^B)=0.
  \end{equation}
  We can decompose $B=B_1\conn B_2$, where
  $B_1\in\pi_2(\x^n,\dots,\x^0)$ and
  $B_2\in\pi_2(\y^n,\dots,\y^0)$.
  The homotopy class $B_1$ is an approximation to some
  $B'_1\in\pi_2(\x_0^{i_{s_{n+1-k}}},\dots,\x_0^{i_{s_1}})$
  (in the Heegaard tuple
  $(\Sigma,\betas^{i_{s_0}},\dots,\betas^{i_{s_{n-k+1}}},z)$).
  By Lemma~\ref{lem:ApproximateDimension}, 
  \begin{align}
    \dim(\cM^{B_1})-\dim(\cM^{B'_1})
    &=k
    \label{eq:LeftLimit}
  \end{align}
  unless $k\in\{n-1,n\}$ and $B'_1$ is trivial.
  Similarly, $B_2$ is an approximation to a homotopy
  class $B'_2$ of $(n+1-\ell)$-gons with
  \begin{equation}
    \label{eq:RightLimit}
    \dim(\cM^{B_2})-\dim(\cM^{B'_2})= \ell
  \end{equation}
  unless $\ell\in\{n-1,n\}$ and $B'_2$ is trivial.
  Combining
  Equations~\eqref{eq:ZeroDimensionalModuliSpace},
  \eqref{eq:AdditiveIndex}, \eqref{eq:LeftLimit},
  \eqref{eq:RightLimit},  
  and the fact that $n=k+\ell$,
  we conclude that
  \begin{equation}
    \label{eq:DimensionInequality}
  \dim(\cM^{B'_1})+\dim(\cM^{B'_2})<0
  \end{equation}
  unless one of:
  \begin{enumerate}
  \item\label{case:bigons} Both $B'_1$ and $B'_2$ are constant bigons.
    In particular, this gives $k=n-1$, $\ell=n-1$. So, since
    $k+\ell=n$, $k=\ell=1$.
  \item\label{case:n-gon} $B'_1$ is a constant $1$-gon. Then, in
    particular, $k=n$ and $\ell=0$.
  \item\label{case:n-gon-2} $B'_2$ is a constant $1$-gon. Then, in
    particular, $\ell=n$ and $k=0$.
  \end{enumerate}
  If one of $\dim(\cM^{B'_1})$ or $\dim(\cM^{B'_2})$ is negative,
  Property~\ref{cl:Polygons} ensures that $\cM^{B_1}$ or $\cM^{B_2}$
  (respectively) is empty, and hence, by
  Lemma~\ref{lem:FiberedProduct} $\cM^B$ is empty. So, for $\cM^B$ to
  be non-empty, one of cases~(\ref{case:bigons}),~(\ref{case:n-gon})
  or~(\ref{case:n-gon-2}) must occur.

  \textbf{Claim 3.} When $n=2$ and $k=\ell=1$, terms in
  Formula~\eqref{eq:Compatibility} cancel in pairs.

  To see this, observe that
  \begin{align*}
    m_2(\omega^{i_1\times j_0<i_1\times j_1},&\omega^{i_0\times
      j_0<i_1\times j_0}) \\
    &=m_2(\Theta^{i_1\times j_0<i_1\times j_1}\conn \zeta^{i_1\times j_0<i_1\times j_1},\eta^{i_0\times j_0<i_1\times j_0}\conn \Theta^{i_0\times j_0<i_1\times j_1}) \\
    &= m_2(\Theta^{i_1\times j_0<i_1\times j_1},\eta^{i_0\times j_0<i_1\times j_0})
    \otimes m_2(\zeta^{i_1\times j_0<i_1\times j_1},\Theta^{i_0\times j_0<i_1\times j_1}) \\
    &= \eta^{i_0\times j_0<i_1\times j_1}\otimes\zeta^{i_0\times j_0<i_1\times j_1},
  \end{align*}
  by the definition of $\omega$,
  Lemma~\ref{lem:FiberedProductTriangles}
  (Equation~\eqref{eq:TrianglesDecompose}), and
  Property~\ref{cl:Triangles} of the close approximation, in turn. 
  The same
  argument shows that
  \begin{align*}
    m_2( \omega^{i_0\times j_1<i_1\times j_1},\omega^{i_0\times
      j_0<i_0\times j_1})
    &=\eta^{i_0\times j_0<i_1\times j_1}\otimes\zeta^{i_0\times j_0<i_1\times j_1}.
  \end{align*}
  This proves the claim.

  \textbf{Claim 4.} The terms in Formula~\eqref{eq:Compatibility} with 
  $k=n$ (and $\ell=0$) cancel with each other. 

  Indeed, if $k=n$ then
  Lemma~\ref{lem:FiberedProductDegreeOne} applies: $B=B_1$ from
  Lemma~\ref{lem:FiberedProductDegreeOne}, and that lemma ensures that
  $\#\cM^{B_1\conn B_2}= \#\cM^{B_2}$. Thus,
  Equation~\eqref{eq:Compatibility} for the $\omega$ is a consequence
  of the corresponding condition on the $\zeta$: for any fixed
  $i_0\in\IndI$ and $j<j'\in\IndJ$,
  \begin{align}
    \sum_{j=j_0<j_1<\dots<j_{n-1}<j_n=j'} 
    &m_{n}(\omega^{i_{0}\times j_{n-1}<
      i_0\times j_n},\dots,
    \omega^{i_0\times j_0<i_0\times j_1}) \label{eq:MoveFirstFactor} \\
    &=
    \sum_{j=j_0<j_1<\dots<j_{n-1}<j_n=j'} 
    m_{n}(\zeta^{j_{n-1}<j_n},\dots,
    \zeta^{j_0<j_1}) \nonumber \\
    &=
    0. \nonumber
  \end{align}

  \textbf{Claim 5.} The terms in Formula~\eqref{eq:Compatibility} with 
  $\ell=n$ (and $k=0$) cancel with each other. This follows from the
  same argument used to prove Claim 4, with the two sides of the
  diagram reversed.

  The five claims account for all of the terms in
  Formula~\eqref{eq:Compatibility}, so
  Formula~\eqref{eq:Compatibility} holds and the proposition is
  proved.
\end{proof}

\subsection{The chain complex for a link}
\label{sec:Links}
 
A particular chain complex of attaching circles is constructed
in~\cite{BrDCov}, though without using this terminology.

Let $L$ be a $c$-component link in a three-manifold. A Heegaard
diagram subordinate to $L$ in $Y$ is a pointed Heegaard diagram
\[
(\Sigma,\alphas=\{\alpha_1,\dots,\alpha_g\},\betas=\{\beta_1^\infty,\dots,\beta_c^\infty,\beta_{c+1},\dots,\beta_g\},z)
\]
with the property that the components of $L$ are boundary-parallel
circles in the $\beta$-handlebody,
the attaching disk of $\beta_k^\infty$ meets
the $k\th$ component of $L$ transversally in one point if
$k=1,\dots,c$, and these are the only intersection points of the
$\beta$-attaching disks and $L$.

A tuple of framings for $L$ can be specified by curves
$\{\beta_k^0\}_{k=1}^{c}$, so that $\beta^0_k\cap \beta_\ell=\emptyset$ for $k=c+1,\dots,g$
and $\beta_k^0\cap\beta^\infty_\ell=\emptyset$ unless $k=\ell$, in which case the two curves
meet transversally in one point. For $k=1,\dots,c$, let $\beta_k^1$ be a standard
resolution of $\beta_k^{\infty}\cup \beta_k^0$, as in
Figure~\ref{fig:ExampleComplex}. With these choices, we have a
collection of attaching circles indexed by
$\IndI=\{0,1,\infty\}^c$. Given a sequence $i_0<\dots<i_n$ in $\IndI$
we can consider the Heegaard multi-diagram
$(\Sigma,\alphas,\betas^{i_0},\dots,\betas^{i_n},z)$. 

We perturb the
$\betas^{i_j}$'s to make this multi-diagram admissible, in an
efficient way, i.e., so that the differential on
$\CFa(\betas^{i_k},\betas^{i_{k+1}},z)$ vanishes identically.
(Any periodic domain can be written as a sum of doubly- and triply-periodic domains. The doubly-periodic domains have zero area by the above construction.
The triply periodic domains can be arranged to have zero area as well, see
for example Figure~\ref{fig:ExampleComplex}.)
If
$i_k$ and $i_{k+1}$ are consecutive (i.e., $i_k$ and $i_{k+1}$ differ
in exactly one coordinate) then let
$\Theta^{i_k<i_{k+1}}\in\CFa(\betas^{i_k},\betas^{i_{k+1}},z)$ be the
unique top-dimensional generator. Otherwise, let
$\Theta^{i_k<i_{k+1}}=0$.

\begin{definition}
  \label{def:ChangeFramingsComplex}
  The chain complex of attaching
  circles
  \[(\IndI=\{0,1,\infty\}^c,\{\betas^i\}_{i\in\IndI},\{\Theta^{i<i'}\}_{i,i'
    \in \IndI},z)\]
  is called {\em the chain complex of framing changes}.
  If we further specify an additional $g$-tuple of attaching circles $\alphas$,
  and let $Y$ be specified by the Heegaard diagram $(\Sigma,\alphas,\betas^{\infty},z)$, and $L\subset Y$ be the corresponding framed link, then we say
  that 
  $(\IndI=\{0,1,\infty\}^c,\{\betas^i\}_{i\in\IndI},\{\Theta^{i<i'}\}_{i,i'
    \in \IndI},z)$
  is the {\em chain complex of framing changes on the link $L\subset Y$
  specified by $\alphas$}.
\end{definition}
Definitions~\ref{def:AssociatedComplex} and~\ref{def:ChangeFramingsComplex} make 
\[
\bigoplus_{i\in \{0,1,\infty\}^c}
\CFa(\alphas,\{\betas^i\}_{i\in\{0,1,\infty\}^c},z)
\]
into a chain complex. The paper~\cite{BrDCov} also makes this vector
space into a chain complex, denoted $X$, by counting
pseudo-holomorphic polygons~\cite[Section 4.2]{BrDCov}.

\begin{lemma}
  The chain complex $X$ is the filtered complex associated in
  Definition~\ref{def:AssociatedComplex} to $\alphas$ and the chain
  complex associated to framing changes on the link
  (Definition~\ref{def:ChangeFramingsComplex}).
\end{lemma}

\begin{proof}
  The complex $X$ from~\cite[Section 4.2]{BrDCov} is 
  $
  \bigoplus_{i\in\IndI}\CFa(\alphas,\betas^i,z)
  $
  with differential given by 
  \[
  D(\xi) =\!\!\!\!\! \sum_{i_1<\cdots<i_k\text{ consecutive}}\!\!\!\!\! m_k(\Theta^{i_{k-1}<i_k},\cdots,\Theta^{i_1<i_2},\xi).
  \]
  This is exactly the complex from
  Definition~\ref{def:AssociatedComplex}.
\end{proof}


%% file: draws/ExampleComplex.tex
\begin{picture}(0,0)%
\includegraphics{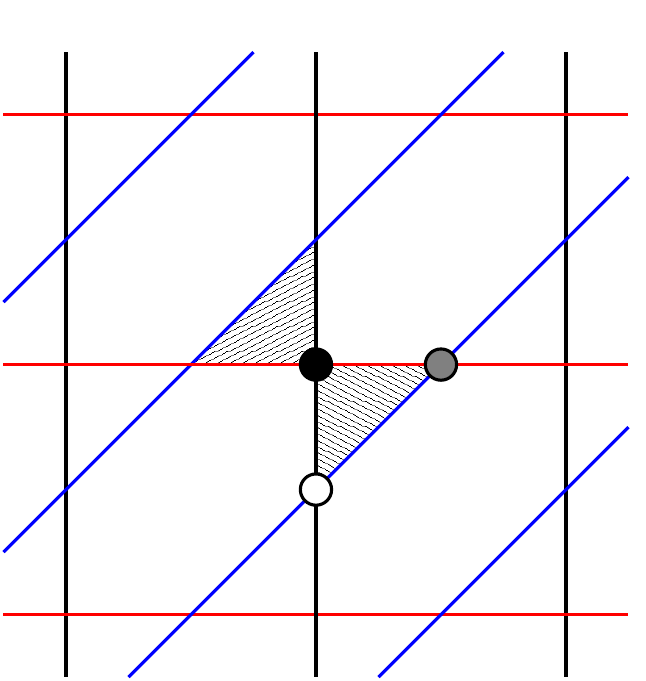}%
\end{picture}%
\setlength{\unitlength}{1973sp}%
\begingroup\makeatletter\ifx\SetFigFont\undefined%
\gdef\SetFigFont#1#2#3#4#5{%
  \reset@font\fontsize{#1}{#2pt}%
  \fontfamily{#3}\fontseries{#4}\fontshape{#5}%
  \selectfont}%
\fi\endgroup%
\begin{picture}(6198,6522)(1768,-6394)
\put(6001,-3961){\makebox(0,0)[lb]{\smash{{\SetFigFont{12}{14.4}{\familydefault}{\mddefault}{\updefault}{\color[rgb]{0,0,0}$x_{0,1}$}%
}}}}
\put(7951,-3511){\makebox(0,0)[lb]{\smash{{\SetFigFont{12}{14.4}{\familydefault}{\mddefault}{\updefault}{\color[rgb]{0,0,0}$\beta^0$}%
}}}}
\put(7951,-1711){\makebox(0,0)[lb]{\smash{{\SetFigFont{12}{14.4}{\familydefault}{\mddefault}{\updefault}{\color[rgb]{0,0,0}$\beta^1$}%
}}}}
\put(5701,-2161){\makebox(0,0)[lb]{\smash{{\SetFigFont{12}{14.4}{\familydefault}{\mddefault}{\updefault}{\color[rgb]{0,0,0}$z$}%
}}}}
\put(4951,-3061){\makebox(0,0)[lb]{\smash{{\SetFigFont{12}{14.4}{\familydefault}{\mddefault}{\updefault}{\color[rgb]{0,0,0}$x_{0,\infty}$}%
}}}}
\put(4951,-4861){\makebox(0,0)[lb]{\smash{{\SetFigFont{12}{14.4}{\familydefault}{\mddefault}{\updefault}{\color[rgb]{0,0,0}$x_{1,\infty}$}%
}}}}
\put(4651,-211){\makebox(0,0)[lb]{\smash{{\SetFigFont{12}{14.4}{\familydefault}{\mddefault}{\updefault}{\color[rgb]{0,0,0}$\beta^\infty$}%
}}}}
\end{picture}%

%% file: polygons.tex
\section{Polygon counting in bordered manifolds}
\label{sec:Polygons}

The present section contains a generalization of some of the material
from Section~\ref{sec:Complexes} to the bordered setting.  In
Subsection~\ref{subsec:BorderedMultiDiagrams}, we introduce bordered
multi-diagrams (generalizing the earlier Heegaard multi-diagrams).  In
Subsection~\ref{subsec:BorderedPolygons}, we consider moduli spaces of
pseudo-holomorphic polygons in bordered multi-diagrams.  Counting
points in these moduli spaces gives the maps generalizing the
pseudo-holomorphic polygon counts considered earlier.  The algebra of
these holomorphic curve counts is described in
Subsection~\ref{subsec:BorderedPolygonMaps}. In
Subsection~\ref{subsec:ComplexAttachCircles}, it is shown that an
$\IndI$-filtered chain complex of attaching circles, together with a set
of bordered attaching curves, gives rise to a filtered $\Ainf$-module
over the algebra associated to a pointed matched circle.  These
results, Proposition~\ref{prop:DefCCFAa} (the Type~$A$ version)
and Proposition~\ref{prop:DefCCFDa} (the Type~$D$ version), can be
viewed as bordered analogues of the
filtered chain complexes constructed in
Proposition~\ref{prop:ChainComplexesOfAttachingCircles}. 
The filtered type $A$ modules and type $D$ modules will appear in the statement 
of the pairing theorem for polygons in Section~\ref{sec:PairingTheorem}.
Subsection~\ref{subsec:PolygonBimodule}
describes the further generalization to bordered Heegaard diagrams
with two boundary components (in the spirit of~\cite{LOT2}).
These generalizations appear in the statement of a pairing theorem used to 
prove the main theorem of this paper, Theorem~\ref{thm:IdentifySpectralSequences}.

\subsection{Bordered multi-diagrams}
\label{subsec:BorderedMultiDiagrams}

If $\Sigma$ is a surface with genus $g$ and a single boundary component,
one can consider $g$-tuples of attaching circles as in
Definition~\ref{def:AttachingCircles}. There is another kind of tuple
of curves which is natural in the bordered case:

\begin{definition}
  \label{def:AttachingCurves}
  Let $\Sigma$ be a compact, oriented surface with one boundary component.
  Fix a pointed matched circle $\PMC$ consisting of $2k$ pairs
  of points ${\mathbf a}\subset S^1=\bdy\Sigma$. A {\em complete set of bordered attaching curves compatible with
    $\PMC$} is a collection $\alphas=\{\alpha_1,\dots,\alpha_{g+k}\}$ of
  curves in $\Sigma$ such that:
  \begin{itemize}
  \item The curves $\alpha_i\in\alphas$ are pairwise disjoint.
  \item  $\alphas\cap \partial \Sigma = \bdy\alphas=\PMC$.
    We sometimes abbreviate this condition as
    $\partial (\Sigma,\alphas) = \PMC$.
  \item  The relative cycles $\{[\alpha_i]\}_{i=1}^{g+2k}$,
    where $[\alpha_i]\in H_1(\Sigma,\partial \Sigma)$,
    are linearly independent.
  \end{itemize}
\end{definition}

When considering holomorphic curves, we will attach a
cylindrical end to $\bdy\Sigma$. We will still denote the result by
$\Sigma$, and hope that this will not cause confusion.

\begin{definition}
  \label{def:Admissible}
  Let $\alphas$ be a complete set of bordered attaching curves in $\Sigma$ in
  the sense of Definition~\ref{def:AttachingCurves} (compatible with some $\PMC$),
  and let $\{\betas^i\}_{i=1}^n$ be an $n$-tuple of complete sets of 
  attaching circles (in the sense of
  Definition~\ref{def:AttachingCircles}). 
  We call the data $(\Sigma,\alphas,\betas^1,\dots\betas^n,z)$ a 
  {\em bordered multi-diagram}.

  A {\em generalized
    multi-periodic domain} is a relative homology class
  $B\in H_2(\Sigma,\alphas\cup\betas^{1}\cup\dots\cup\betas^{n}\cup \partial
  \Sigma)$ whose boundary
  $\bdy B$, viewed as an element of 
  \[
  H_1(\alphas\cup\betas^{1}\cup\dots\cup\betas^{n}\cup\bdy\Sigma,\partial\Sigma)\cong
H_1(\alphas\cup\betas^{1}\cup\dots\cup\betas^{n},{\mathbf
    a})
  \]
  is contained in the image of the
  inclusion 
  \[H_1(\alphas,{\mathbf a})\oplus
  H_1(\betas^{1})\oplus\dots\oplus H_1(\betas^{n})\to
  H_1(\alphas\cup\betas^{1}\cup\dots\cup\betas^{n},{\mathbf a}).\]
  A generalized multi-periodic domain $P$ has a local multiplicity $n_x(P)$ at any point $x\in
  \Sigma\setminus (\alphas\cup\betas^{1}\cup\dots\cup\betas^{n})$.
  A {\em multi-periodic domain} is one whose local
  multiplicity at (the region adjacent to) the point $z$ vanishes.

  A multi-periodic domain is called {\em provincial} if all of its
  local multiplicities near $\partial \Sigma$ vanish; equivalently, if
  it has trivial boundary in $H_0({\mathbf a})$.

  We say that $(\Sigma,\alphas,\{\betas^{i}\}_{i\in\IndI},z)$ is {\em
    admissible} if any non-zero multi-periodic domain has both
  positive and negative local multiplicities. The diagram
  $(\Sigma,\alphas,\{\betas^{i}\}_{i\in\IndI},z)$ is called {\em
    provincially admissible} if any non-zero provincial 
  multi-periodic domain has
  both positive and negative local multiplicities.
\end{definition}

\begin{lemma}
\label{lem:AdmissibleMultiDiagram}
If $\{\betas^i\}_{i\in\IndI}$ is an $\IndI$-filtered, admissible collection of attaching circles
(in the sense of Definition~\ref{def:AdmissibleAttachingCircles}),
and $\alphas$ is a complete set of bordered attaching curves compatible with some pointed matched
circle $\PMC$ on $\partial\Sigma$, then we can always find another complete set of bordered attaching
curves $\alphas'$ compatible with $\PMC$ so that 
\begin{itemize}
\item $(\Sigma,\alphas',\{\betas^i\}_{i\in\IndI},z)$ is admissible.
\item $\alphas'$ is isotopic to $\alphas$.
\end{itemize}
\end{lemma}
\begin{proof}
  This follows by winding transversely to the $\alpha$ curves, as in
  the case of bigons~\cite[Lemma 5.4]{OS04:HolomorphicDisks}. (The
  corresponding result for bordered Heegaard diagrams (i.e.,
  $|\IndI|=1$), is~\cite[Proposition~\ref*{LOT1:prop:admis-achieve-maintain}]{LOT1}.)
\end{proof}
  
\subsection{Moduli spaces of polygons in bordered manifolds}  
\label{subsec:BorderedPolygons}
 
As discussed in Section~\ref{sec:Complexes}, there are polygon counts
defined in Heegaard Floer homology which satisfy the
$\Ainf$ relations.  The goal of this subsection and the next is to
generalize these polygon maps to the bordered context.  Suppose that
$(\Sigma,\alphas,\{\betas^i\}_{i\in\IndI},z)$ is an admissible
multi-diagram.  We will define maps
\[
\lsub{n}m_k\co \CFa(\betas^{i_{n-1}},\betas^{i_n},z)\otimes\dots \otimes 
\CFa(\betas^{i_1},\betas^{i_2},z)
\otimes \CFAa(\alphas,\betas^{i_1})\otimes
 \Alg^{\otimes k}
 \to
 \CFAa(\alphas,\betas^{i_n})
\]
by combining the definition of type $A$ modules
from~\cite[Chapter~\ref*{LOT1:chap:type-a-mod}]{LOT1} with
the usual polygon counts.

In this subsection, we set up the relevant moduli spaces. As is usual
in the cylindrical setting, this is a two-step process. First we
introduce moduli spaces of polygons with a fixed source.  The expected
dimension of these moduli spaces depends on the Euler characteristic
of the source.  We then show that for embedded polygons, the Euler
characteristic of the source is determined by the homology
class. (This is an extension of Sarkar's index formula for
polygons~\cite{Sarkar11:IndexTriangles}.) In the next section, we will
count moduli spaces of rigid, embedded polygons to define the polygon
maps.

To start, attach a cylindrical end $S^1\times[0,\infty)$ to
$\bdy\Sigma$. We will denote the result by~$\Sigma$, as well; this
should not cause confusion.

We generalize Definition~\ref{def:AdmissibleJs} to the bordered context:
\begin{definition}
  \label{def:AdmissibleJ}
  An {\em admissible collection of almost-complex structures} is a
  choice of smooth family $\{J_j\}_{j\in\Conf(D_n)}$ of almost-complex
  structures on $\Sigma\times D_n$ for each $n\geq 3$, satisfying all
  the conditions in Definition~\ref{def:AdmissibleJs},
  and the following further condition:
  \begin{itemize}
    \item over the cylindrical end $S^1\times [0,\infty)$ of $\Sigma$,
      the complex structure $J_j$ splits as a product $j_0\times j$,
      where $j_0$ is a standard cylindrical complex structure on $S^1\times [0,\infty)$.
    \end{itemize}
\end{definition}

\begin{definition}\label{def:bord-moduli}
  Let $\{J_j\}_{j\in\Conf(D_{n+1})}$ be an admissible collection of almost-complex structures. Fix a complete set of bordered attaching curves $\alphas$
  compatible with $\PMC$ (Definition~\ref{def:AttachingCurves}), and a
  further collection of $n$ complete sets of attaching circles
  $\betas^{1},\dots,\betas^n$ (Definition~\ref{def:AttachingCircles}).
  Fix generators $\x^k\in\Gen(\betas^{k},\betas^{{k+1}})$ for
  $k=1\dots,{n-1}$, as well as $\x^0\in\Gen(\alphas,\betas^{1})$,
  $\x^n\in\Gen(\alphas,\betas^{n})$,
  and consider maps
  \begin{equation}\label{eq:bordered-polygon-map}
    u\co (S,\bdy S)\to \bigl(\Sigma\times D_{n+1},(\alphas\times e_{0}) \cup
    (\betas^1\times
    e_1)\cup\dots\cup(\betas^n\times e_n) \bigr)
  \end{equation}
  where $S$ is a punctured Riemann surface and $D_{n+1}$ is equipped
  with a set of
  points $q_i\in \bdy D_{n+1}$ for $i=1,\dots,k$,
  with the following properties:
  \begin{enumerate}[label=(c-\arabic*),ref=c-\arabic*,start=0]
  \item The projection $\pi_\Sigma\circ u\co S\to \Sigma$ has degree $0$ at the region adjacent to the basepoint $z$.
  \item The punctures of $S$ are mapped to the punctures
    $\{p_{i,i+1}\}_{i=0}^\ell\cup \{q_i\}_{i=1}^k$ of
    $D_{n+1}\setminus \{q_i\}_{i=1}^k$.
  \item 
    The curve $u$ is asymptotic to  $\x^i \times \{p_{i,i+1}\}$ at the preimage of the 
    puncture $p_{i,i+1}$.
  \item The curve $u$ is asymptotic to $\rhos_i \times \{q_i\}$ at the
    punctures above $q_i$, for some set of Reeb chords $\rhos_i$ (in
    $Z\setminus\{z\}$ with endpoints in $\CircPts$).
  \item 
    \label{item:BoundaryMonotone} At each point $q\in (e_{0}\setminus
    \{q_i\}_{i=1}^\ell)$, the $g$ points
    $(\pi_\Sigma\circ u)\bigl((\pi_\CDisk\circ u)^{-1}(q)\bigr)$
    lie in $g$ distinct $\alpha$-curves.
  \end{enumerate}
  The set of such $u$ decomposes into homology classes, denoted
  $\pi_2(\x^{n},\x^{n-1},\dots,\x^0;\rhos_1,\dots,\rhos_m)$.
  For fixed $B\in\pi_2(\x^{n},\x^{n-1},\dots,\x^0;\rhos_1,\dots,\rhos_m)$, let
  \[
  \Mod^{B,S}=\Mod^B(\x^n,\dots,\x^0;\rhos_1,\dots,\rhos_m;S)
  \]
  denote the
  moduli space pairs $(j,u)$ where $j\in\Conf(D_{n+1})$ and $u$ is a $J_j$- holomorphic representative of
  $B\in\pi_2(\x^n,\dots,\x^0;\rhos_1,\dots,\rhos_m)$.
\end{definition}

Condition~(\ref{item:BoundaryMonotone}) can be formulated as a
combinatorial condition on the $(\x,\rhos_1,\dots,\rhos_m)$: it is the
{\em strong boundary monotonicity}
of~\cite[Section~\ref*{LOT1:sec:curves-in-sigma}]{LOT1}. It is
equivalent to the algebraic condition that $\x\otimes
a(\rhos_1)\otimes\dots\otimes a(\rhos_m)$ is a non-vanishing element in
$\CFAa(\alphas,\betas^1)\otimes\Alg(\PMC)\otimes\dots\otimes\Alg(\PMC)$,
where the tensor is taken over the ring of idempotents in
$\Alg(\PMC)$;
see~\cite[Lemma~\ref*{LOT1:lem:strong-monotone-tensor}]{LOT1}.

\begin{lemma}\label{lem:ind-fixed-source}
  The expected dimension of the moduli space
  $\Mod^{B}(\x^n,\dots,\x^0;\rhos_1,\dots,\rhos_m;S)$ is given by
  $\ind(B,S;\rhos_1,\dots,\rhos_m)+n-2$ where
  \begin{equation}\label{eq:ind-fixed-source}
  \ind(B,S;\rhos_1,\dots,\rhos_m)\coloneqq \left(\frac{3-n}{2}\right)g-\chi(S)+2e(B)+m.
  \end{equation}
  Here, $g$ is the genus of $\Sigma$ (which is also the number of
  elements in each $\x^i$) and $e(B)$ denotes the Euler measure of
  $B$.
\end{lemma}
\begin{proof}
  This is a simple adaptation of the proof in the bigon
  case~\cite[Proposition~\ref*{LOT1:Prop:Index}]{LOT1}; see
  also~\cite[Section 10.2]{Lipshitz06:CylindricalHF}
  and~\cite{Sarkar11:IndexTriangles} in the closed case.
\end{proof}

Next, we observe that, as with bigons, embeddedness is equivalent to a
condition on the Euler characteristic of $S$. To state the formula for
$\chi(S)$ we need a little more notation:
\begin{itemize}
\item Given a domain $B$, let $\bdy_i B$ denote the part of $B$ lying
  along $\betas^i$, and $\bdy_0B$ the part of $B$ lying along
  $\alphas$.
\item Given two curves $\gamma,\eta$ in $\Sigma$, with
  $\gamma\pitchfork\eta$ but possibly intersecting at the endpoints of
  $\gamma$ or $\eta$, we can define the \emph{jittered intersection
    number} of $\gamma$ and $\eta$, denoted $\gamma\cdot\eta$, by
  pushing $\eta$ slightly so that the endpoints of $\eta$
  (respectively $\gamma$) become disjoint from $\gamma$ (respectively
  $\eta$) in the four obvious ways, and averaging the
  results. See~\cite[Section~\ref*{LOT1:sec:index-additive}]{LOT1}.
\item Given a pair of Reeb chords $\rho_1,\rho_2$ in $Z$, let
  $L(\rho_1,\rho_2)$ be the linking number of $\bdy \rho_1$ and
  $\bdy\rho_2$, i.e., the multiplicity with which $\rho_2$ covers
  $\bdy \rho_1$. (This can be a half-integer, if $\rho_1$ and $\rho_2$
  share an endpoint.) Extend $L$ bilinearly to a function
  $L(\rhos_1,\rhos_2)$ on pairs of sets of Reeb chords.
  See~\cite[Section~\ref*{LOT1:sec:pregrading}
  and~\ref*{LOT1:sec:ind-at-emb}]{LOT1}.
\item Given a set of Reeb chords $\rhos$, let
  $\iota(\rhos)=-\frac{|\rhos|}{2}-\sum_{\{\rho_1,\rho_2\}\subset\rhos}|L(\rho_1,\rho_2)|.$
  See~\cite[Section~\ref*{LOT1:sec:ind-at-emb}]{LOT1}.
\end{itemize}

\begin{proposition}\label{prop:emb-ind}
  Suppose that $u\in \Mod^B(\x^n,\dots,\x^0;\rhos_1,\dots,\rhos_m;S)$
  is an embedded holomorphic $(n+1)$-gon. Then 
  \begin{align}
    \chi(S)&=g+e(B)-n_{\x^0}(B)-n_{\x^n}(B)\label{eq:emb-chi}\\
    &\qquad-\sum_{n\geq
      j>\ell\geq1}\bdy_j(B)\cdot\bdy_\ell(B)-\sum_i\iota(\rhos_i)-\sum_{i<j}L(\rhos_i,\rhos_j)\nonumber\\
    \ind(B;S;\rhos_1,\dots,\rhos_m)&=e(B)+n_{\x^0}(B)+n_{\x^n}(B)-\left(\frac{n-1}{2}\right)g\label{eq:emb-ind}
    \\
    &\qquad+\sum_{n\geq
      j>\ell\geq1}\bdy_j(B)\cdot\bdy_\ell(B)+m+\sum_i\iota(\rhos_i)+\sum_{i<j}L(\rhos_i,\rhos_j).\nonumber
  \end{align}
  Moreover, if $\Mod^B(\x^n,\dots,\x^0;\rhos_1,\dots,\rhos_m;S)$
  contains a non-embedded holomorphic polygon then
  \begin{multline*}
    \ind(B;S;\rhos_1,\dots,\rhos_m)\leq 
    \Biggl[e(B)+n_{\x^0}(B)+n_{\x^n}(B)-\left(\frac{n-1}{2}\right)g+\sum_{n\geq
      j>\ell\geq1}\bdy_j(B)\cdot\bdy_\ell(B)
    \\
    +m+\sum_i\iota(\rhos_i)+\sum_{i<j}L(\rhos_i,\rhos_j)\Biggr]-2.
  \end{multline*}
\end{proposition}
\begin{proof}
  The formula is a combination of the embedded index formula for
  bigons~\cite[Proposition~\ref*{LOT1:prop:asympt_gives_chi}]{LOT1}
  and Sarkar's index formula for
  $n$-gons~\cite{Sarkar11:IndexTriangles}. We prove it by
  imitating the proof
  of~\cite[Proposition~\ref*{LOT1:prop:asympt_gives_chi}]{LOT1}; see
  also the proof
  of~\cite[Proposition~\ref*{LOT1:prop:tri-emb-ind-general}]{LOT1}. We
  will be brief.

  \begin{figure}
    \centering
    \includegraphics{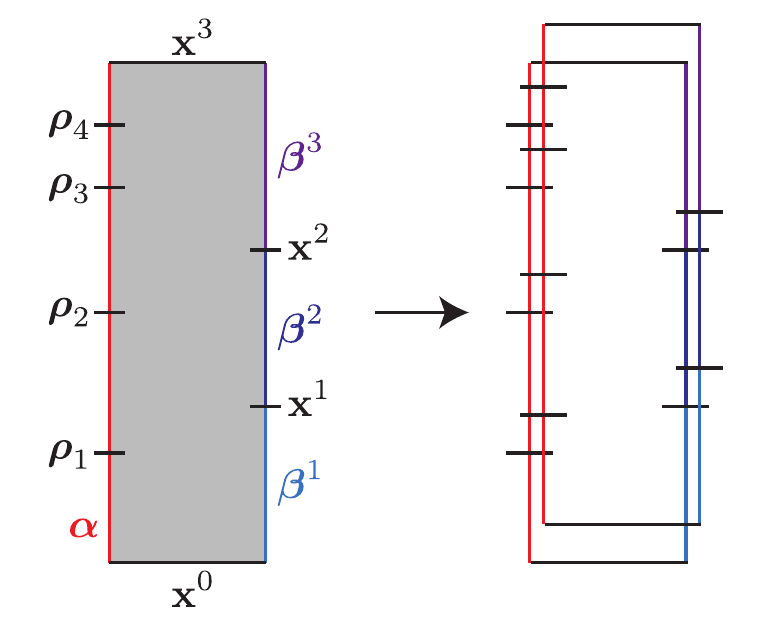}
    \caption{\textbf{Computing the embedded index for polygons.} The
      polygon shown is a $4$-gon (rectangle).}
    \label{fig:polygon-ind}
  \end{figure}

  We start with the formula for $\chi$. View the disk $D_{n+1}$ as the
  strip $[0,1]\times\RR$ with $(n-1)$ marked points on
  $\{1\}\times\RR$, so that the boundary conditions $\alphas$
  correspond to $\{0\}\times\RR$ and $\x^0$ corresponds to
  $[0,1]\times\{-\infty\}$. There is an induced $\RR$-action on
  $D_{n+1}$ which is holomorphic but does not preserve (all but two
  of) the corners; continuing to think of $D_{n+1}$ as a strip, we
  will call this $\RR$-action \emph{translation}. Let $\tau_r(u)$ be
  the result of translating $u$ by $r$ units.

  Given a map $v$ between Riemann surfaces, let $\br(v)$ denote the
  ramification degree of $v$. Recall that $\pi_\bD\co \Sigma\times
  D_n\to D_n$ and $\pi_\Sigma\co \Sigma\times D_n\to \Sigma$ are the
  two projections. Let $\frac{\partial}{\partial t}$ be the vector
  field generated by the $\RR$-action on $D_n$.

  Viewing the $\x^i$ punctures of $S$ as right-angled corners and each
  Reeb chord as having two right-angled corners, we have
  \begin{equation}\label{eq:ind-chi-e}
  \chi(S)=e(S)+\frac{(n+1)g}{4}+\sum_i\frac{|\rhos_i|}{2}.
  \end{equation}
  By the Riemann-Hurwitz formula, 
  \begin{equation}\label{eq:ind-R-H}
  e(S)=e(B)-\br(\pi_\Sigma\circ u).
  \end{equation}
  By definition $\br(\pi_\Sigma\circ u)$ is the number of tangencies
  of $\frac{\partial}{\partial t}$ to $u$. Taking into account that
  sliding Reeb chords in $\rhos_i$ past each other introduces boundary
  double points we have
  \begin{equation}\label{eq:br-to-tau}
  \br(\pi_\Sigma\circ u)=u\cdot
  \tau_\epsilon(u)-\sum_i\sum_{\{\rho_a,\rho_b\}\subset \rhos_i}L(\rho_a,\rho_b),
  \end{equation}
  where $u\cdot \tau_\epsilon(u)$ denotes the intersection
  number---algebraic or geometric does not matter, since holomorphic
  curves intersect positively.

  Translating farther, for $R$ sufficiently large we have
  \begin{equation}
    \label{eq:ind-R-large}
    u\cdot \tau_\epsilon(u)=u\cdot\tau_R(u)+\sum_{n\geq
      j>\ell\geq1}\bdy_j(B)\cdot\bdy_\ell(B)+\frac{g(n-1)}{4}+\sum_{i<j}L(\rhos_i,\rhos_j).
  \end{equation}
  Again, the contribution of $L(\rhos_i,\rhos_j)$ comes from Reeb
  chords sliding past each other; see, e.g., the proof
  of~\cite[Proposition~\ref*{LOT1:prop:asympt_gives_chi}]{LOT1}. The
  contribution of $\sum_{n\geq
    j>\ell\geq1}\bdy_j(B)\cdot\bdy_\ell(B)+\frac{g(n-1)}{4}$ comes
  from intersections appearing or disappearing along the boundary,
  where $\betas^i$ and $\betas^j$ intersect; see, the proof
  of~\cite[Proposition~\ref*{LOT1:prop:tri-emb-ind-general}]{LOT1}. 

  As in the bigon case, 
  \begin{equation}
    \label{eq:ind-R-v-large}
    u\cdot\tau_R(u)=n_{\x^0}(B)+n_{\x^n}(B)-\frac{g}{2}.
  \end{equation}
  Combining
  Equations~\eqref{eq:ind-chi-e},~\eqref{eq:ind-R-H},~\eqref{eq:br-to-tau},~\eqref{eq:ind-R-large}
  and~\eqref{eq:ind-R-v-large} gives
  \begin{align*}
    \chi(S)&=\frac{(n+1)g}{4}+\sum_i\frac{|\rhos_i|}{2}+e(B)-\Bigl[-\sum_i\sum_{\{\rho_a,\rho_b\}\subset
        \rhos_i}L(\rho_a,\rho_b)+\sum_{n\geq
      j>\ell\geq1}\bdy_j(B)\cdot\bdy_\ell(B)\\
    &\qquad\qquad+\frac{g(n-1)}{4}+\sum_{i<j}L(\rhos_i,\rhos_j)+n_{\x^0}(B)+n_{\x^n}(B)-\frac{g}{2}\Bigr]\\
    &=g+e(B)-n_{\x^0}(B)-n_{\x^n}(B)+\sum_i\frac{|\rhos_i|}{2}+\sum_i\sum_{\{\rho_a,\rho_b\}\subset
        \rhos_i}L(\rho_a,\rho_b)\\
      &\qquad\qquad-\sum_{n\geq
      j>\ell\geq1}\bdy_j(B)\cdot\bdy_\ell(B)-\sum_{i<j}L(\rhos_i,\rhos_j)\\
    &=g+e(B)-n_{\x^0}(B)-n_{\x^n}(B)-\sum_i\iota(\rhos_i)
      -\sum_{n\geq
      j>\ell\geq1}\bdy_j(B)\cdot\bdy_\ell(B)-\sum_{i<j}L(\rhos_i,\rhos_j),
  \end{align*}
  as claimed.

  Combining Formulas~\eqref{eq:ind-fixed-source}
  and~\eqref{eq:emb-chi} gives
  \begin{align*}
    \ind(u)&=\left(\frac{3-n}{2}\right)g+2e(B)+m\\
    &\qquad-\biggl[
      g+e(B)-n_{\x^0}(B)-n_{\x^n}(B)-\sum_i\iota(\rhos_i)\\
    &\qquad\qquad
      -\sum_{n\geq
        j>\ell\geq1}\bdy_j(B)\cdot\bdy_\ell(B)-\sum_{i<j}L(\rhos_i,\rhos_j)
    \biggr]\\
    &=\left(\frac{1-n}{2}\right)g+e(B)+m
      +n_{\x^0}(B)+n_{\x^n}(B)+\sum_i\iota(\rhos_i) \\
    &\qquad  +\sum_{n\geq
        j>\ell\geq1}\bdy_j(B)\cdot\bdy_\ell(B)+\sum_{i<j}L(\rhos_i,\rhos_j),
  \end{align*}
  as claimed.
  
  Finally, for non-embedded curves, each double point or equivalent
  singularity increases $\chi$ by $2$, and consequently drops $\ind$
  by $2$.
\end{proof}

Finally, we define the moduli spaces of embedded curves: 
\begin{definition}\label{def:emb-moduli-sp}
  Let $\cM^B=\cM^B(\x^n,\dots,\x^0;\rhos_1,\dots,\rhos_m)$ denote the
  set of embedded holomorphic maps in the homology class $B$ with
  asymptotics $\x^n,\dots,\x^0$ and $\rhos_1,\dots,\rhos_m$;
  equivalently, 
  \begin{equation*}
    \cM^B(\x^n,\dots,\x^0;\rhos_1,\dots,\rhos_m)\coloneqq
    \bigcup_{\chi(S)\text{ given by 
    }\eqref{eq:emb-chi}}\cM^B(\x^n,\dots,\x^0;\rhos_1,\dots,\rhos_m;S).
  \end{equation*}
\end{definition}

\begin{proposition}
  For a generic, admissible family $\{J_j\}$ of almost-complex
  structures, the moduli spaces
  $\Mod^{B,S}=\Mod^B(\x^n,\dots,\x^0;\rhos_1,\dots,\rhos_m;S)$, and
  hence also the embedded moduli spaces
  $\cM^B(\x^n,\dots,\x^0;\rhos_1,\dots,\rhos_m)$, are transversely cut
  out, and hence are smooth manifolds whose dimensions are given by
  $\ind(B,S;\rhos_1,\dots,\rhos_m)+n-2$ and
  $\ind(B;\rhos_1,\dots,\rhos_m)+n-2$, respectively.
\end{proposition}
\begin{proof}
  This follows from standard transversality results; see, for
  instance,~\cite[Proposition~\ref{LOT1:prop:transversality}]{LOT1}
  for the analogous result for bigons.
\end{proof}

The spaces $\Mod^{B,S}$ and $\cM^B(\x^n,\dots,\x^0;\rhos_1,\dots,\rhos_m)$ are, 
of course, typically non-compact, except when they are $0$-dimensional, in 
which case they are compact; see also Remark~\ref{rem:non-compact}.

\subsection{Maps induced by polygon counts}
\label{subsec:BorderedPolygonMaps}
\begin{definition}\label{def:n-m-k}
  Suppose that $(\Sigma,\alphas,\{\betas^i\}_{i=1}^n,z)$ is a
  provincially admissible multi-diagram.
  Define the map
  \[ \lsub{n}m_k\co \CFa(\betas^{{n-1}},\betas^{n},z)\otimes\dots \otimes 
  \CFa(\betas^{1},\betas^{2},z)\otimes \CFAa(\alphas,\betas^1)\otimes
  \overbrace{\Alg \otimes \dots \otimes \Alg}^k \to
  \CFAa(\alphas,\betas^n),\]
  by extending the following formulas linearly.
  For fixed $\eta^{i}\in\Gen(\betas^{i-1},\betas^{i})$
  and sequence of sets of Reeb chords $\vec\rhos=(\rhos_1,\dots,\rhos_k)$ 
  define:
  \begin{equation}
    \lsub{n}m_k(\eta^{n},\dots,\eta^{1},\x,a(\rhos_1),\dots,a(\rhos_k)) 
    = \!\!\!\!\!\!\!\!\!\sum_{\substack{\y\in\Gen(\alphas,\betas^{n})\\
    		    B\in\pi_2(\y,\eta^{n-1},\dots,\eta^1,\x;\rhos_1,\dots,\rhos_m)\\ 
    \ind(B;\rhos_1,\dots,\rhos_m)=2-n}}\!\!\!\!\!\!\!\!\!
    \#\Mod^B(\y,\eta^{n-1},\dots,\eta^{1},\x;\rhos_1,\dots,\rhos_k)
    \y    \label{eq:DefineAction}
  \end{equation}
\end{definition}
So, for example, $\lsub{0}m_0$ is the differential on
$\CFAa(\alphas,\betas^0)$. Since the moduli spaces being counted are
$0$-dimensional, they are compact and hence finite. Provincial
admissibility implies that the sum defining $\lsub{n}m_k$ is also
finite
(compare~\cite[Proposition~\ref*{LOT1:prop:provincial-admis-finiteness}]{LOT1}).

Admissibility guarantees
that $\lsub{n}m_k=0$ for a fixed multi-diagram and all sufficiently
large $k$
(compare~\cite[Proposition~\ref*{LOT1:prop:admis-finiteness}]{LOT1}).

The polygon counts satisfy an $\Ainf$ relation. Before stating this
relation, note that $\Alg$ is a $\dg$ algebra, so $\mu_{i-j+1}=0$
unless $j=i$ or $j=i-1$.

\begin{proposition}
  \label{prop:AinftyRelation}
  The polygon counts defined above satisfy the following $\Ainf$ relation:
  \begin{align*}
    0&=\sum_{0\leq i\leq k;\ 0\leq p \leq n}
    \lsub{n-p}m_{k-i}(\eta^{n},\dots \eta^{p+1}, \lsub{p}m_i(\eta^p,\dots,\eta^1,\x,a_1,\dots,a_i),a_{i+1},\dots a_k) \\
    &+ \sum_{1\leq p<q\leq n} 
    \lsub{n-q+p}m_k(\eta^{n},\dots,\eta^{q+1},m_{q-p+1}(\eta^{q},\dots,\eta^p),\eta^{p-1},\dots,\eta^1,\x,a_1,\dots,a_k) \\
    &+ \sum_{1\leq i\leq j\leq n}
    \lsub{n}m_{k-j+i}(\eta^n,\dots,\eta^1,\x,a_1,\dots,\mu_{i-j+1}(a_i,\dots,a_j),\dots,a_k),
  \end{align*}
  for any $\x\in \Gen(\alphas,\betas^1)$, 
  $\eta^i\in \CFa(\betas^i,\betas^{i+1},z)$,
  $a_i\in\Alg(\PMC)$.
\end{proposition}

\begin{proof}
  This is a straightforward synthesis of the proof of the $\Ainf$ relation for polygon counting
  (Proposition~\ref{prop:closed-Ainf-rel}) with the proof of the $\Ainf$ relation for $\CFAa$~\cite[Proposition~\ref*{LOT1:prop:A-module-defined}]{LOT1}.
\end{proof}

\begin{remark}
  \label{rem:FukABimodule}
  Proposition~\ref{prop:AinftyRelation} has an interpretation in terms
  of Fukaya categories. As in Remark~\ref{rem:ChainComplexIsTypeD},
  let $\HFuk$ denote the full subcategory of the Fukaya category of
  $\Sym^g(\Sigma)$ generated by Heegaard tori. Then
  Proposition~\ref{prop:AinftyRelation} can be interpreted as saying
  that $\CFAa(\alphas,\cdot)$ is an $\Ainf$-bimodule over $\HFuk$
  and $\Alg(\PMC)$. (Convention~\ref{conv:order} means that
  $\CFAa(\alphas,\cdot)$ is a left-right bimodule; with the usual
  composition conventions for Heegaard Floer homology
  $\CFAa(\alphas,\cdot)$ would be a right-right bimodule.)
\end{remark}

There is a type $D$ analogue of the above construction (compare~\cite[Chapter~\ref*{LOT1:chap:type-d-mod}]{LOT1}).
Recall that
for type $D$ structures, one considers a different orientation
convention: in that case, one considers a collection $\alphas$ of
bordered attaching curves in $\Sigma$ which are compatible with $-\PMC$.

Now, if $(\Sigma,\alphas,\{\betas^i\}_{i=1}^n,z)$ is an admissible
bordered multi-diagram, polygon counts give maps
\[
\delta_n^1\co \CFa(\betas^{{n-1}},\betas^{n},z)\otimes \dots \otimes
\CFa(\betas^{1},\betas^{2},z)\otimes \CFDa(\alphas,\betas^{1}) \to
\Alg(\PMC)\otimes \CFDa(\alphas,\betas^{n}),
\]
defined as follows:
\[ 
\delta_n^1(\eta^{n-1},\dots,\eta^1,\x) =
\sum_{\substack{\y\in\Gen(\alphas,\betas^{i_n})\\\vec{\rho}\\
    B\in\pi_2(\y,\eta^{n-1},\dots,\eta^1,\x;\vec{\rho})\\ \ind(B,\vec{\rho})=2-n}}
\#\Mod^B(\y,\eta^{n-1},\dots,\eta^1,\x;\vec\rho)
a(-\vec\rho) \otimes \y,
\] 
where $\vec\rho$ runs over all sequences
$\vec\rho=(\{\rho_1\},\dots,\{\rho_k\})$ of (singleton sets of) Reeb
chords for which $(\x,\{-\rho_1\},\dots,\{-\rho_k\})$ is strongly
boundary monotone; and if $\vec\rho=(\rho_1,\dots,\rho_k)$, then
\[a(-\vec\rho)=\prod_{i=1}^k a(-\rho_i).\]
So, for example, $\delta^1_1$ is the differential on
$\CFDa(\alphas,\betas^1)$. 

The map $\delta_{n}^1$ can be naturally extended to a map 
\[\delta_{n}^1\co \CFa(\betas^{n-1},\betas^n,z)\otimes\dots \otimes \CFa(\betas^1,\betas^2,z)\otimes \Alg(\PMC)\otimes \CFDa(\alphas,\betas^1) \to \Alg(\PMC)\otimes \CFDa(\alphas,\betas^n)\]
by the formula 
\[ 
{\widetilde \delta}^1_n(\eta^{n-1},\dots,\eta^1,a,\x)=
(\mu_2(a,\cdot)\otimes\Id_{\CFDa})\circ
\delta^1_n(\eta^{n-1},\dots,\eta^1,\x).
\]

\begin{proposition}
  \label{prop:AinftyRelationD}
  The polygon counts defined above satisfy the following $\Ainf$ relation:
  \begin{align*} \sum_{1\leq p\leq n} & {\widetilde\delta}^1_{n-p+1}
  (\eta^{i_{n-1}<i_n},\dots,\eta^{i_p<i_{p+1}},\delta^1_{p+1}(\eta^{i_{p-1}<i_p},\dots,\eta^{i_0<i_1},\x)) \\
  & + (\mu_1\otimes \Id) \circ \delta^1_{n+1}(\eta^{i_{n-1}<i_n},\dots,\eta^{i_0<i_1},\x) \\
  & + \sum_{1\leq p\leq q\leq n} \delta^1_{n-q+p+1}(\eta^{i_{n-1}<i_n},\dots,m_{q-p+1}(\eta^{i_{q-1}<i_q},\dots,\eta^{i_{p-1}<i_p}),\eta^{i_{p-2}<i_{p-1}},\dots,\eta^{i_0<i_1},\x) \\
  &=0.\end{align*}
\end{proposition}
\begin{proof}
  As in the proof of Proposition~\ref{prop:AinftyRelationD}, the proof
  is a combination of the proof of the usual polygon counting
  $\Ainf$ relation (Proposition~\ref{prop:closed-Ainf-rel}) with the
  proof of the corresponding relation
  in the bordered case~\cite[Proposition~\ref*{LOT1:prop:typeD-d2}]{LOT1}. 
\end{proof}

\begin{remark}
  \label{rem:FukDBimodule}
  Continuing with the notation from Remark~\ref{rem:FukABimodule}, we
  can think of Proposition~\ref{prop:AinftyRelationD} as giving
  $\CFDa(\alphas,\cdot)$ the structure of a type~\DA~bimodule over
  $\Alg(-\PMC)$ and $\HFuk$.
\end{remark}

\begin{remark}\label{remark:CFD-is-CFA-DT-Id}
  Proposition~\ref{prop:AinftyRelationD} can alternatively be thought of as a
  formal consequence of Proposition~\ref{prop:AinftyRelation}.  Before
  describing this, we give an analogous way of deducing the type $D$
  structure relation for $\CFDa(\alphas,\betas)$ from the type $A$
  structure relation for
  $\CFAa(\alphas,\betas)$.  Consider the type \DD\ bimodule
  $X=\CFDDa(\Id)$ associated to the identity cobordism from $F(\PMC)$
  to itself. This bimodule was computed
  in~\cite[Definition~\ref*{LOT4:def:DDstruct}]{LOT4}: it is generated
  by pairs of complementary idempotents, and the differential is given
  by
  \[
  \bdy(I\otimes I')=\sum_{\text{chords }\xi}Ia(\xi)\otimes a'(\xi)I'.
  \]
  It follows that $\CFDa(\alphas,\betas)=X\DT\CFAa(\alphas,\betas)$.
  (Here, equality denotes a canonical identification, not merely a homotopy equivalence.)
  Taking this as a definition of
  $\CFDa(\alphas,\betas)$, the type $D$ structure equation on
  $\CFDa(\alphas,\betas)$ is a formal consequence of the type $DD$
  structure equation on $X$ (as verified directly
  in~\cite[Proposition~\ref*{LOT4:prop:DDsquaredZero}]{LOT4}), the
  type $A$ structure relation on $\CFAa(\alphas,\betas)$, and the fact
  that a type $DD$ bimodule tensored with a type $A$ module gives a
  type $D$
  module~\cite[Section~\ref*{LOT2:sec:tensor-products}]{LOT2}.
  
  In an analogous manner, we could have defined the $\Alg(-\PMC)\Hyph
  \HFuk$ bimodule structure on $\CFDa(\alphas,\cdot)$ as
  $X\DT\CFAa(\alphas,\cdot)$. Now the desired~\DA~structure equations
  follow from the $\Ainf$-bimodule relations
  (Proposition~\ref{prop:AinftyRelation}), the~\DD\ structure
  relations on $X$, and the fact that a type~\DD\ bimodule tensored with an
  $\Ainf$-bimodule gives a type~\DA~structure (see
  again~\cite[Section~\ref*{LOT2:sec:tensor-products}]{LOT2}).
\end{remark}

\subsection{Complexes of attaching circles and filtered bordered modules}
\label{subsec:ComplexAttachCircles}

In Proposition~\ref{prop:ChainComplexesOfAttachingCircles}, we described how an $\IndI$-filtered
chain complex of attaching
circles along with another set of attaching circles gives rise to an $\IndI$-filtered chain complex.

Our aim here is to prove the analogue in the bordered
setting. Specifically, we will show how a chain complex of attaching
circles in a Heegaard surface with boundary, along with a further set
of bordered attaching curves compatible with a pointed matched circle~$\PMC$
gives rise to an $\IndI$-filtered $\Ainf$-module.

\begin{definition}\label{def:CCFA}
  Let $\Sigma$ be a surface with boundary, $z$ a basepoint in
  $\bdy\Sigma$, and $\alphas$ a complete set of
  bordered attaching curves compatible with some fixed pointed matched circle
  $\PMC$.  Let $(\{\betas^i\}_{i\in \IndI},
  \{\eta^{i<i'}\}_{i,i'\in \IndI})$ be a chain complex of attaching
  circles.  Suppose moreover that the multi-diagram
  $(\Sigma,\alphas,\{\betas^i\}_{i\in\IndI},z)$ is provincially
  admissible (see Definition~\ref{def:AttachingCurves}).

  Define the $\IndI$-filtered $\Ainf$-module
  $\CCFAa(\alphas,\{\betas^i\}_{i\in \IndI},\{\eta^{i<i'}\}_{i,i'\in\IndI},z)$ over $\Alg(\PMC)$
  to be
  \[\{\CFAa(\alphas,\betas^i)\}_{i\in \IndI}\] together with the
  $\Ainf$-homomorphisms
  \[
  F^{i<i'}\co \CFAa(\alphas,\betas^i)\to
  \CFAa(\alphas,\betas^{i'})
  \] 
  for $i,i'\in \IndI$ with $i<i'$ defined by
  \[
  F^{i<i'}(\x,a_1,\dots,a_k)=\sum_{i=i_0<\dots<i_n=i'}
  \lsub{n}m_k(\eta^{i_{n-1}<i_{n}},\dots,\eta^{i_{0}<i_{1}},\x,a_1,\dots,a_k).
  \]
\end{definition}

When they are clear from the context, we will drop the indexing set from the notation, writing
\[
\CCFAa(\alphas,\{\betas^i\},\{\eta^{i<i'}\})\coloneqq
\CCFAa(\alphas,\{\betas^i\}_{i\in \IndI},\{\eta^{i<i'}\}_{i,i'\in\IndI}).
\]

The following is a precise version of the type $A$ case of Theorem~\ref{intro:ChainComplexToCFA-CFD}:

\begin{proposition}
  \label{prop:DefCCFAa}
  If the diagram $(\Sigma,\alphas,\{\betas^i\}_{i\in\IndI},\{\eta^{i<i'}\}_{i,i'\in\IndI},z)$
  is provincially admissible then
  the object $\CCFAa(\alphas,\{\betas^i\},\{\eta^{i<i'}\})$ 
  of Definition~\ref{def:CCFA} is an $\IndI$-filtered $\Ainf$-module over $\Alg(\PMC)$.
  Its associated graded object is $\bigoplus_{i\in \IndI}\CFAa(\alphas,\betas^i)$.
  If $(\alphas,\{\betas^i\}_{i\in\IndI},\{\eta^{i<i'}\}_{i,i'\in\IndI},z)$
  is admissible (Definition~\ref{def:Admissible}) then
  $\CCFAa(\alphas,\{\betas^i\},\{\eta^{i<i'}\})$
  is bounded.
\end{proposition}

\begin{proof}
  The compatibility condition (from
  Equation~\eqref{eq:CompatibilityCondition}) is a
  consequence of the $\Ainf$ relation
  (Proposition~\ref{prop:AinftyRelation}), together with the
  compatibility conditions for a chain complex of attaching
  circles (Equation~\eqref{eq:Compatibility}). See also the proof of
  Proposition~\ref{prop:ChainComplexesOfAttachingCircles}.
\end{proof}

\begin{remark}
  Recall (Remark~\ref{rem:ChainComplexIsTypeD}) that a chain complex
  of attaching circles can be viewed as an $\IndI$-filtered type $D$
  structure (twisted complex) over $\HFuk$. Now, $\CCFAa$ can be
  thought of as the tensor product of this filtered type $D$ structure
  with $\CFAa(\alphas,\cdot)$, thought of as a bimodule over $\HFuk$
  and $\Alg(\PMC)$ as in Remark~\ref{rem:FukABimodule}. Accordingly,
  this tensor product is naturally an $\IndI$-filtered $\Ainf$-module over
  $\Alg(\PMC)$ (since a type $D$ structure tensored with an
  $\Ainf$-bimodule is a type $A$ module~\cite[Section~\ref*{LOT2:sec:tensor-products}]{LOT2}.
\end{remark}

Similarly, we can form 
\[
\CCFDa(\alphas,\{\betas^i\},\{\eta^{i<i'}\},z)=
\CCFDa(\alphas,\{\betas^i\}_{i\in \IndI},\{\eta^{i<i'}\}_{i,i'\in\IndI},z),
\]
which is an $\IndI$-filtered type~$D$ structure
over $\Alg(-\PMC)$.
In this case,
\[ \lsup{i<i'}\delta^1 \co \CFDa(\alphas,\betas^i)\to \Alg(\PMC)\otimes \CFDa(\alphas,\betas^{i'}) \]
is defined by
\[
\lsup{i<i'}\delta^1=\sum_{i=i_1<\dots<i_n=i'}
\delta^1_n(\eta^{i_{n-1}<i_{n}},\dots,\eta^{i_{1}<i_{2}},\x).
\]

Here is the more precise version of the type $D$ case of Theorem~\ref{intro:ChainComplexToCFA-CFD}:
\begin{proposition}
\label{prop:DefCCFDa}  
  If the diagram $(\Sigma,\alphas,\{\betas^i\}_{i\in\IndI},\{\eta^{i<i'}\}_{i,i'\in\IndI},z)$
  is provincially admissible, the object
  $\CCFDa(\alphas,\allowbreak\{\betas^i\},\allowbreak\{\eta^{i<i'}\},z)$
  defined above is an $\IndI$-filtered type $D$ structure over $\Alg(-\PMC)$. Its associated
  graded object is
  $\bigoplus_{i\in\IndI}\CFDa(\alphas,\betas^i,z)$.
  If
  $(\alphas,\allowbreak\{\betas^i\}_{i\in\IndI},\allowbreak\{\eta^{i<i'}\}_{i,i'\in\IndI},\allowbreak z)$
  is admissible (Definition~\ref{def:Admissible}) then
  $\CCFDa(\alphas,\{\betas^i\},\{\eta^{i<i'}\},z)$
  is bounded (Definition~\ref{def:BoundedTypeD}).
\end{proposition}

\begin{proof}
  This follows from Proposition~\ref{prop:AinftyRelationD}
  together with Equation~\eqref{eq:Compatibility}.
\end{proof}

\subsection{Bordered multi-diagrams with two boundary components and filtered bimodules}
\label{subsec:PolygonBimodule}
We now turn to the generalization to bordered Heegaard diagrams with
$2$ boundary components.

\begin{definition}
  \label{def:2AttachingCurves}
  Let $\Sigma$ be a compact, oriented surface with two boundary
  components, $\bdy_L\Sigma$ and $\bdy_R\Sigma$.  Fix pointed matched
  circles $\PMC_L$ and $\PMC_R$ consisting of $2k_L$ and $2k_R$ pairs
  of points $\CircPts_L\subset \bdy_L\Sigma$ and $\CircPts_R\subset
  \bdy_R\Sigma$, respectively. A {\em complete set of bordered attaching curves
    compatible with $\PMC_L$ and $\PMC_R$} is a collection
  $\alphas=\{\alpha_1,\dots,\alpha_{g+k_L+k_R}\}$ of curves in $\Sigma$ such that:
  \begin{itemize}
  \item The curves $\alpha_i\in\alphas$ are pairwise disjoint.
  \item $\alphas\cap \partial \Sigma =
    \bdy\alphas=\PMC_L\amalg\PMC_R$.  In particular, each $\alpha$-arc
    has both of its endpoints on the same boundary component of
    $\Sigma$.  We sometimes abbreviate this condition as $\partial
    (\Sigma,\alphas) = \PMC_L\amalg\PMC_R$.
  \item  The relative cycles $\{[\alpha_i]\}_{i=1}^{g+k_L+k_R}$,
    where $[\alpha_i]\in H_1(\Sigma,\partial \Sigma)$,
    are linearly independent.
  \end{itemize}
\end{definition}

\begin{definition}\label{def:2Admissible}
  Let $\alphas$ be a complete set of bordered attaching curves in $\Sigma$ in
  the sense of Definition~\ref{def:2AttachingCurves} (compatible with some
  $\PMC_L$ and $\PMC_R$),
  and let $\{\betas^i\}_{i=1}^n$ be an $n$-tuple of complete sets of 
  attaching circles (in the sense of
  Definition~\ref{def:AttachingCircles}). Fix also an arc
  $\arcz\subset \Sigma\setminus (\alphas\cup\bigcup_i\betas^i)$
  connecting $\bdy_L\Sigma$ and $\bdy_R\Sigma$. (Existence of such an
  arc is not guaranteed by the other hypotheses.)
  We call the data $(\Sigma,\alphas,\betas^1,\dots\betas^n,\arcz)$ a 
  {\em bordered multi-diagram with $2$ boundary components}.

  Define multi-periodic domains in the $2$ boundary component case exactly
  as in Definition~\ref{def:Admissible} (with the requirement that the 
  region containing $\arcz$ has coefficient $0$).
  A multi-periodic domain is
  called {\em provincial} if all of its multiplicities near $\partial
  \Sigma$ vanish. Admissibility and provincial admissibility are defined as before;
  these notions can be refined as follows. A domain is called {\em left-provincial} (respectively {\em
    right-provincial}) if all its local multiplicities around
  $\partial_L\Sigma$ (respectively $\partial_R \Sigma$) vanish. The diagram is
  called \emph{left-admissible} (respectively \emph{right-admissible})
  if all its non-zero
  right-provincial (respectively left-provincial) period domains have both
  positive and negative local multiplicities. 
\end{definition}

We can now define bimodule analogues of $\CCFAa$. This is a
straightforward synthesis of the definition of $\CFAAa$
from~\cite[Definition~\ref*{LOT2:def:AA-bimod}]{LOT2} and $\CCFAa$
from Definition~\ref{def:CCFA}. We sketch this
construction.

Fix a chain complex of attaching circles in $\Sigma$,
$(\IndI,\{\betas^i\}_{i\in\IndI},\{\eta^{i_1<i_2}\}_{i_1,i_2\in\IndI})$
and a complete set of bordered attaching curves $\alphas$ (compatible with some
$\PMC_L$ and $\PMC_R$). Assume that
$(\Sigma,\alphas,\{\betas^i\},\arcz)$ is provincially admissible.

As in~\cite[Definition~\ref*{LOT2:def:Drilling}]{LOT2}, we can form
the \emph{drilled Heegaard surface} $\Sigma_{\dr}$ by attaching a
one-handle to $\Sigma$ to connect the two endpoints of the arc
${\mathbf z}$. We fix a basepoint $z\in\partial\Sigma_{\dr}$ somewhere
on the boundary of the attached one-handle.  We can consider
$\CCFAa(\Sigma_{\dr},\alphas,\{\betas^i\},\{\eta^{i_1<i_2}\})$,
which is a filtered module over $\Alg(\PMC)$, where $\PMC=\PMC_L\conn\PMC_R=\partial
(\Sigma_\dr \cap \alphas)$. Then $\Alg(\PMC_L)$ and $\Alg(\PMC_R)$ are
commuting subalgebras of $\Alg(\PMC)$. 

Via restriction of scalars, we can view 
$\CCFAa(\Sigma_{\dr},\alphas,\{\betas^i\},\{\eta^{i_1<i_2}\})$
as a right filtered $\Ainf$-module over $\Alg(\PMC_L)\otimes\Alg(\PMC_R)$. The category
of right (filtered) $\Ainf$-modules over $\Alg(\PMC_L)\otimes\Alg(\PMC_R)$ is equivalent to
the category of right-right (filtered) $\Ainf$-bimodules over $\Alg(\PMC_L)\Hyph\Alg(\PMC_R)$ (see,
e.g.,~\cite[Section~\ref*{LOT2:sec:A-Bop-vs-bimods}]{LOT2}); 
let
\[
\CCFAAa(\Sigma,\alphas,\{\betas^i\},\{\eta^{i_1<i_2}\})
=
\CCFAAa(\Sigma,\alphas,\{\betas^i\}_{i\in\IndI},\{\eta^{i_1<i_2}\}_{i_1<i_2\in\IndI})
\]
denote
$\CCFAa(\Sigma_{\dr},\alphas,\{\betas^i\},\{\eta^{i_1<i_2}\})$
viewed as an $\Alg(\PMC_L)\Hyph\Alg(\PMC_R)$-bimodule.
Explicitly, the generators of $\CCFAAa$ are the generators of $\CCFAa$. 
The action of the sequences $(a_1,\dots,a_m)$ in $\Alg(\PMC_L)$
and $(b_1,\dots,b_n)$ in $\Alg(\PMC_R)$ on a generator $\x$
of $\CCFAAa$ is the sum
of the actions on $\x$
of all sequences $(c_1,\dots,c_{m+n})$ which interleave $(a_1,\dots,a_m)$
and $(b_1,\dots,b_n)$.

Now, define
\begin{align}
  \CCFDAa(\alphas,\{\betas^i\},\{\eta^{i_1<i_2}\})
  &\coloneqq 
  \CCFAAa(\alphas,\{\betas^i\},\{\eta^{i_1<i_2}\})
  \DT_{\Alg(\PMC_L)}\CFDDa(\Id_{\PMC_L})\label{eq:CCFDA-as-DT}\\
  \CCFDDa(\alphas,\{\betas^i\},\{\eta^{i_1<i_2}\})
  &\coloneqq 
  \CCFAAa(\alphas,\{\betas^i\},\{\eta^{i_1<i_2}\})
  \DT_{\Alg(\PMC_L)}\CFDDa(\Id_{\PMC_L})\\
  &
  \hphantom{\coloneqq\CCFAAa(\alphas,\{\betas^i\},\{\eta^{i_1<i_2}\})\ }
  \DT_{\Alg(\PMC_R)}\CFDDa(\Id_{\PMC_R}).\nonumber
\end{align}
(Compare Remark~\ref{remark:CFD-is-CFA-DT-Id}.)  We spell out this
definition more explicitly for the case of $\CCFDAa$.

As in the connected boundary case, when talking about holomorphic
curves we will abuse notation and let $\Sigma$ denote the result of
attaching cylindrical ends to the boundary components of
$\Sigma$. The definition of an admissible collection of almost-complex
structures (Definition~\ref{def:AdmissibleJ}) generalizes in an
obvious way to the two boundary component case; fix an admissible
collection of almost-complex structures.

Fix also the following data: 
\begin{itemize}
\item a sequence of sets of Reeb chords
  $\vec\rhos^R=(\rhos_1^R,\dots,\rhos_n^R)$ in $\PMC_R$ so that
  $(\x,\vec\rhos)$ is strongly boundary monotone,
\item a collection of generators
  $\eta^{i_m<i_{m+1}}\in\Gen(\betas^{i_m},\betas^{i_{m+1}})$ for
  $m=1,\dots,n-1$.
\end{itemize}
Define
\begin{align*}
\lsub{n}\delta^1_k&(\eta^{i_{n-1}<i_n},\dots,\eta^{i_1<i_2},\x,\rhos_1,\dots,\rhos_k)\\
&= \sum_{\substack{\y\in\Gen(\alphas,\betas^{i_n})\\ \vec\rho^L\\B\in\pi_2(\y,\eta^{n-1},\dots,\eta^1,\x;\vec{\rhos}^R,\vec{\rho}^L)\\\ind(B;\vec{\rhos}^R,\vec{\rho}^L)=2-n}}
a(-\vec\rho^L)\otimes
\#\Mod^B(\y,\eta^{i_{n-1}<i_n},\dots,\eta^{i_1<i_2},\x;\vec\rho^L;\vec\rhos^R)
\y,
\end{align*}
where here the sum is over all sequences of Reeb chords $\vec\rho^L$
in $\PMC_L$, and
\[\#\Mod^B(\y,\eta^{i_{n-1}<i_n},\dots,\eta^{i_1<i_2},\x;\vec\rho^L;\vec\rhos^R)\]
denotes the counts of holomorphic polygons
whose corners (in order) are mapped to
$\y$, $\eta^{i_{n_1}<i_n},\dots,\eta^{i_1<i_2}$, and $\x$;
the Reeb chords appearing along $\alpha$-arcs which land in 
$\partial_L\Sigma$ are, in order, $\vec\rho^L$; and the 
sets of Reeb chords appearing along the $\alpha$-arcs which land
in $\partial_R\Sigma$ are, in order, $\vec\rhos^R$.

Let $\CFDAa(\alphas,\betas^i)$ denote the vector space generated
by $\Gen(\alphas,\betas^i)$. We can extend
$\lsub{n}\delta^1_k$ linearly 
to give a map
\begin{align*}
  \lsub{n}\delta^1_k\co &\CFDAa(\betas^{i_{n-1}},\betas^{i_n})\otimes\dots\otimes
  \CFa(\betas^{i_1},\betas^{i_2},\arcz) \otimes \CFDAa(\alphas,\betas^{i_1})\otimes 
  \overbrace{\Alg(\PMC_R)\otimes\dots\otimes \Alg(\PMC_R)}^k \\
  & \to
  \Alg(\PMC_L)\otimes \CFDAa(\alphas,\betas^{i_n}).
\end{align*}

Fix a chain complex $(\{\betas^i\}_{i\in\IndI},\eta^{i_1<i_2},\arcz)$ of attaching circles.
Define
\[ 
\lsup{i<i'}\delta^1_k\co \CFDAa(\alphas,\betas^{i})
\otimes \overbrace{\Alg(\PMC_R)\otimes\dots\otimes\Alg(\PMC_R)}^k
\to \Alg(\PMC_L)\otimes \CFDAa(\alphas,\betas^{i'}) 
\]
by 
\begin{equation}\label{eq:CCFDA-direct}
\lsup{i<i'}\delta^1_k(\x,\rhos_1,\dots,\rhos_k) =
\sum_n\sum_{i=i_1<\dots<i_n=i'}
\lsup{n}\delta^1_{k}(\eta^{i_{n-1}<i_n},\dots,\eta^{i_1<i_2},\x,\rhos_1,\dots,\rhos_k).
\end{equation}

\begin{lemma}\label{lem:CCFDA-agree}
  The two definitions of
  $\CCFDAa(\alphas,\{\betas^i\},\{\eta^{i_1<i_2}\},\arcz)$, from
  Formula~\eqref{eq:CCFDA-as-DT} and Formula~\eqref{eq:CCFDA-direct}, agree.
\end{lemma}
\begin{proof}
  This is immediate from the definitions. (Compare
  Remark~\ref{remark:CFD-is-CFA-DT-Id}.)
\end{proof}

\begin{proposition}
  \label{prop:BimoduleMorphism}
  Let $(\Sigma,\alphas,\{\betas^i\}_{i\in\IndI},\arcz)$ be an admissible bordered
  multi-diagram with $2$ boundary components.
  The above maps
  $\lsup{i<i'}\delta^1_k$ give
  $\bigoplus_{i\in\IndI} \CFDAa(\alphas,\betas^i)$
  the structure of an $\IndI$-filtered type $DA$ bimodule.
  Moreover, if the diagram is left-admissible (respectively
  right-admissible) then the corresponding module is left-bounded
  (respectively right-bounded).
\end{proposition}

\begin{proof}
  The fact that these maps define a filtered type \DA\ bimodule is a straightforward combination of the 
  compactification of moduli spaces of polygons
  with the arguments from standard bordered Floer homology, as
  in~\cite[Proposition~\ref*{LOT2:prop:DA-module}]{LOT2}. 
  Alternately, this follows from Lemma~\ref{lem:CCFDA-agree}.
  The admissibility and boundedness statements follow exactly along
  the lines of~\cite[Lemma~\ref*{LOT2:lem:admis-bd-DA}]{LOT2}.
\end{proof}


%% file: pairing.tex
\section{The pairing theorem for chain complexes of attaching circles}
\label{sec:PairingTheorem}

The first goal of this section is to prove Theorem~\ref{intro:PairingForPolygons},
which we restate here:

\begin{theorem}
  \label{thm:PolygonPairing}
  Let $(\Sigma_1,\alphas^1,z)$ and $(\Sigma_2,\alphas^2,z)$
  be surfaces-with-boundary,
  each equipped with complete sets of bordered attaching curves $\alphas^1$ and
  $\alphas^2$ and basepoints $z\in\bdy\Sigma_i$ with 
  $\partial(\Sigma_1,\alphas^1)=\PMC$ and
  $\partial(\Sigma_2,\alphas^2)=-\PMC$.
  Let
  \[(\Sigma_1,\IndI,\{\betas^i\}_{i\in \IndI},\{\eta^{i<i'}\}_{i,i'\in \IndI})
  ~\text{and}~(\Sigma_2,\IndJ,\{\gammas^{j}\}_{j\in \IndJ},
  \{\zeta^{j<j'}\}_{j,j'\in \IndJ})\] be chain complexes
  of attaching circles, in $\Sigma_1$ and $\Sigma_2$ respectively.
  Suppose moreover that 
  both $(\Sigma_1,\alphas,\{\betas^i\}_{i\in\IndI},z)$ 
  $(\Sigma_2,\alphas,\{\gammas^j\}_{j\in\IndJ},z)$ 
  are provincially admissible, and at least one is admissible.
  Then, there is an $\IndI\times \IndJ$-filtered quasi-isomorphism
  \begin{equation}
    \label{eq:PairingQuasiIsomorphism}
    \begin{split}
      \CCFAa(&\Sigma_1,\alphas^1,\{\betas^i\}_{i\in \IndI},z)\DT
      \CCFDa(\Sigma_2,\alphas^2,\{\gammas^j\}_{j\in \IndJ},z)\\
      &\simeq \CCFa(\Sigma_1\cup\Sigma_2,\alphas^1\cup\alphas^2,
      (\{\betas^i\}_{i\in \IndI},\{\eta^{i<i'}\}_{i,i'\in \IndI})
      \#(\{\gammas^{j}\}_{j\in \IndJ},\{\zeta^{j<j'}\}_{j,j'\in \IndJ}),
      z).
    \end{split}
  \end{equation}
\end{theorem}

On the left side, we are taking chain complexes of type $A$
structures and type $D$ structures associated to chain complexes of
attaching circles as explained in
Subsection~\ref{subsec:ComplexAttachCircles} and then forming their tensor
product, which is an $\IndI\times \IndJ$-filtered complex (see
Lemma~\ref{lem:FilteredTensorProduct}).  On the right side, we
are taking the connected sum of the complexes of attaching circles, in the
sense of Definition~\ref{def:ConnectedSum}, and forming the associated
chain complex, as in Definition~\ref{def:AssociatedComplex}. In
particular, this involves using close
approximations of the attaching circles, as in
Proposition~\ref{prop:GlueChainComplexes}.

The proof of Theorem~\ref{thm:PolygonPairing} occupies
Sections~\ref{sec:pairing-prelim}--\ref{sec:time-dilate}.  In
Section~\ref{sec:on-boundedness} we formulate a version of
Theorem~\ref{thm:PolygonPairing} with weaker admissibility
hypotheses. We formulate and sketch a proof of the bimodule analogue
of Theorem~\ref{thm:PolygonPairing} in
Section~\ref{sec:pairing-bimodules}. There is a simpler pairing
theorem in the case of pairing holomorphic triangles with holomorphic
bigons, as described in Section~\ref{sec:pairing-triangles}, which we
use in Section~\ref{sec:exact-tri} to show that the surgery exact
triangle implied by bordered Floer theory~\cite{LOT1} agrees with the
original Heegaard Floer surgery exact
triangle~\cite{OS04:HolDiskProperties}.

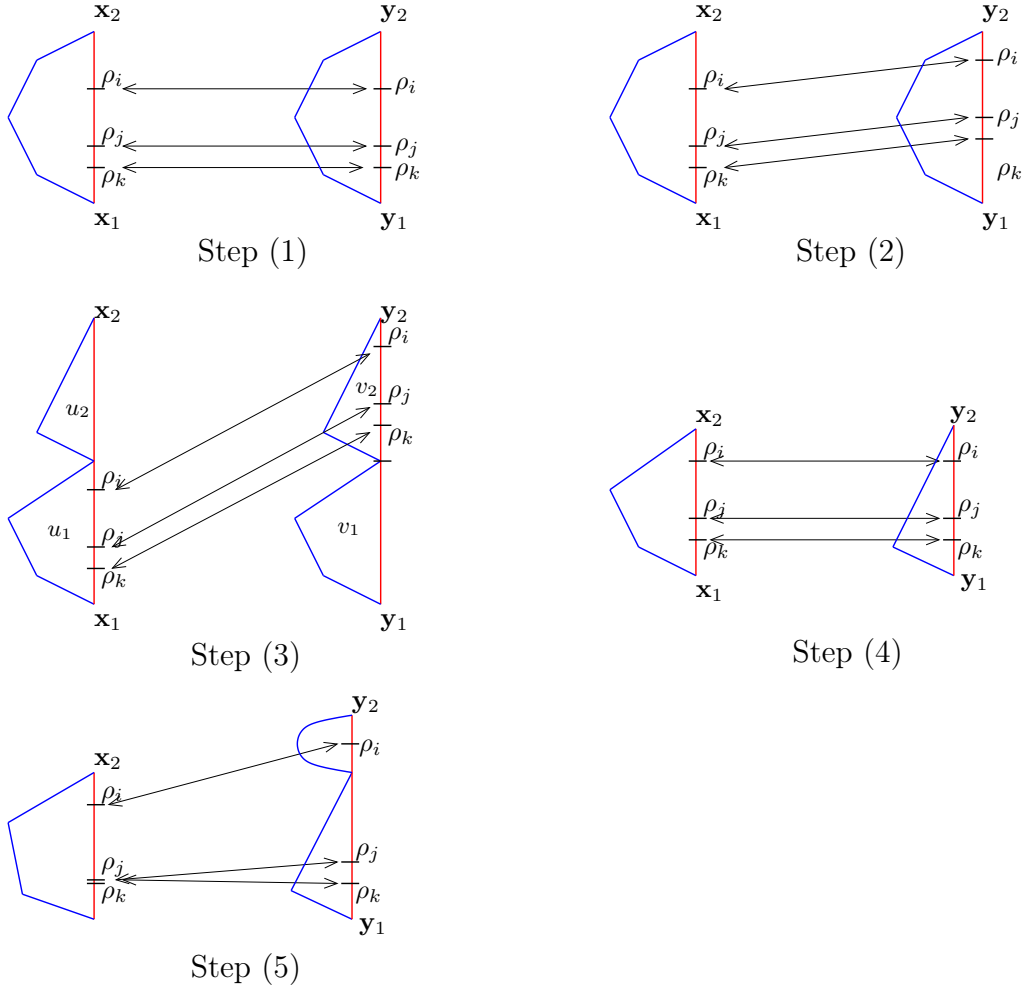
\begin{figure}
\centering
\input{PairingSketch}
\caption{\textbf{Sketch of the pairing theorem.}
  Schematic illustration of the curve configurations which contribute to the boundary operators of the chain complexes
  described in Steps~\ref{step:First}-\ref{step:LastComplex}. Note that in the first two steps, the precise positions of the chords
  are matched (in the second step, up to an overall, pre-specified translation factor); while in the other steps, only the relative 
  positions are matched. Moreover, in the first three steps, there is a matching of the conformal structures of the curves
  on the two sides, absent from the last two steps.}
\label{fig:PairingSketch}
\end{figure}
The proof of Theorem~\ref{thm:PolygonPairing} has the following steps,
as illustrated in Figure~\ref{fig:PairingSketch}:
\begin{enumerate}[label=(\arabic*),ref=(\arabic*)]
\item 
  \label{step:First}
  Form the connected sum of $\Sigma_1$ and $\Sigma_2$, and take a
  limit of almost-complex structures so that holomorphic polygons in $\Sigma_1\conn \Sigma_2$
  correspond to pairs of polygons in $\Sigma_1$ and $\Sigma_2$ which
  satisfy a matching condition for both the conformal structure of
  the polygon and the positions of the Reeb chords
  (Proposition~\ref{prop:TakeNeckToInfinity}). The moduli space of
  such matched polygons is denoted $\cMM_{[0]}$, and the complex
  which counts them is denoted $\MatchedComplex$.
\item Consider a new chain complex, $\tMatchedComplex$, whose
  differential counts points in the moduli space $\cMM_{[t]}$ of pairs
  of polygons where the matching condition
  for the Reeb chords is translated by a real parameter $t$.  The
  filtered quasi-isomorphism type of this complex is independent of
  the parameter $t$ (Lemma~\ref{lem:VaryTranslationParameter}). (This
  is similar to a step in the proof
  of~\cite[Theorem~\ref*{LOT2:thm:Hochschild}]{LOT2}.)
\item Send the $t$ parameter to $\infty$
  (Lemma~\ref{lem:TransToInfinity}). The chain complex
  $\tMatchedComplex$ for large enough $t$ is now identified with a
  different chain complex, $\ComplexCross$, whose differential counts
  points in the moduli space $\CrossMatched$ of
  \emph{cross-matched polygons}. Cross-matched polygons consists of
  two-story buildings $u_1*u_2$ and $v_1*v_2$ in $\Sigma_1$ and
  $\Sigma_2$ respectively, where the modulus of $u_i$ (in
  $\Conf(D_n)$) is matched with 
  the modulus of $v_i$, for $i=1,2$; $u_1$ and $v_2$ are provincial;
  and the (relative) 
  positions of the Reeb chords on $u_2$
  and $v_1$ are matched.
\item Now, move the approximation parameter $\epsilon$ in the
  construction of the attaching circles $\betas^{i\times j}$ and
  $\gammas^{i\times j}$. If $\epsilon$ is sufficiently small, the
  cross-matched polygons in the previous step can be simplified: the
  provincial curves $v_1$ and $u_2$ are determined uniquely, and hence
  they can be thrown out
  (Proposition~\ref{prop:SendEpsilonToZero}). This gives a new chain
  complex, denoted $\ChordMatchedComplex$, counting points in the
  moduli space $\ChordMatched$ of \emph{chord-matched
      polygon pairs}, consisting of a polygon $u_1$ in $\Sigma_1$ and
  $v_2$ in $\Sigma_2$, which satisfy a matching condition on the chord
  heights. At this point, the moduli of the polygons become
  unconstrained.
\item
  \label{step:LastComplex}
  Dilate time in the matching conditions on chord-matched
  polygon pairs, as in the proof of the usual pairing theorem in bordered
  Floer homology~\cite[Chapter~\ref*{LOT1:chap:tensor-prod}]{LOT1}. As in that
  proof, once the
  dilation parameter is sufficiently large
  (Proposition~\ref{prop:DilateToInfinity}), the chain complex is
  identified with a chain complex $\tsicComplex$, whose differential
  counts \emph{trimmed simple ideal-matched polygon pairs}
  (Definition~\ref{def:tsicPolygons}).
\item Again as in the proof of the original pairing theorem, counts of trimmed
  simple ideal-matched polygon pairs have an algebraic
  interpretation, establishing the pairing theorem for polygons.
\end{enumerate}

\subsection{Preliminaries}\label{sec:pairing-prelim}

Before launching into the above steps, we make an observation concerning the admissibility hypotheses:

\begin{proposition}
  \label{prop:GlueAdmissibility}
  If both $(\Sigma_1,\alphas^1,\{\betas^i\}_{i\in\IndI},z)$ and 
  $(\Sigma_2,\alphas^2,\{\gammas^j\}_{j\in\IndJ},z)$
  are provincially admissible, and one of the two is admissible, then
  their gluing
  \[
  (\Sigma_1\cup\Sigma_2,\alphas^1\cup\alphas^2, (\{\betas^i\}_{i\in
    \IndI},\{\eta^{i<i'}\}_{i,i'\in \IndI}) \conn (\{\gammas^{j}\}_{j\in
    \IndJ},\{\zeta^{j<j'}\}_{j,j'\in \IndJ}), z)
  \] 
  is admissible.
\end{proposition}

\begin{proof}
  In the usual bordered setting, the analogous result
  is~\cite[Lemma~\ref*{LOT1:lem:closed-admissible}]{LOT1}.
  The fact that we are dealing here with multi-diagrams causes no additional complications.
\end{proof}

Thus, the admissibility hypotheses ensure that the right-hand side of
Equation~\eqref{eq:PairingQuasiIsomorphism} is well-defined.
Moreover, provincial admissibility ensures that the tensor factors appearing on the left-hand side
are defined, and the admissibility hypotheses ensure that the tensor product is defined
(see Propositions~\ref{prop:DefCCFAa},~\ref{prop:DefCCFDa}, and Lemma~\ref{lem:FilteredTensorProduct}).

\subsection{Holomorphic curves in \texorpdfstring{$\Sigma_1\conn \Sigma_2$}{Sigma 1 \#  Sigma 2} and matched polygons}

In this subsection we discuss the limits of polygons in
$\Sigma_1\conn \Sigma_2$ when one stretches the neck. The results are
analogous
to~\cite[Section~\ref*{LOT1:sec:moduli-matched-pairs}]{LOT1}; we
assume the reader is familiar with the treatment there, and highlight
the parts where the case of polygons is somewhat more complicated.

\begin{convention}
  Throughout the rest of this section, we fix:
  \begin{itemize}
  \item Riemann surfaces $\Sigma_1$ and $\Sigma_2$.
  \item A pointed matched circle $\PMC$.
  \item Complete sets of bordered attaching curves $\alphas_1$
    (respectively $\alphas_2$) in $\Sigma_1$ (respectively $\Sigma_2)$
    compatible with $\PMC$ (respectively $-\PMC$).
  \item Complete sets of attaching circles
    $\betas^1,\dots,\betas^{n_1}$ (respectively
    $\gammas^1,\dots,\gammas^{n_2}$) in $\Sigma_1$ (respectively
    $\Sigma_2$).
  \item Chains $\eta^{i_1<i_2}\in
    \CFa(\Sigma_1,\betas^{i_1},\betas^{i_2},z)$ (respectively
    $\zeta^{j_1<j_2}\in \CFa(\Sigma_2,\gammas^{j_1},\gammas^{j_2},z)$)
    making $(\Sigma_1,\allowbreak\{\betas^i\}_{i\in
      \IndI},\allowbreak\{\eta^{i_1<i_2}\}_{i_1,i_2\in\IndI},z)$ (respectively
    $(\Sigma_2,\{\gammas^j\}_{j\in
      \IndJ},\{\zeta^{j_1<j_2}\}_{j_1,j_2\in\IndJ},z)$) into a chain
    complex of attaching circles.
  \item For each $\ell,m$, a close approximation $\betas^{i_\ell\times j_m}$ (respectively
    $\gammas^{i_\ell\times j_m}$) to $\betas^{i_\ell}$ (respectively
    $\gammas^{j_m}$).
  \end{itemize}
\end{convention}

Recall from Convention~\ref{conv:ij} that $\ij{k}$ denotes
$i_k\times j_k$. Also recall from Definition~\ref{def:ConnectedSum}
that the chain $\eta^{\ij{1}<\ij{2}}$ is the nearest-point map applied
to $\eta^{i_1<i_2}$ if $j_1=j_2$, or is the top generator if
$i_1=i_2$, or else is $0$, and similarly for $\zeta^{\ij{1}<\ij{2}}$
but with the roles of $i$ and $j$ exchanged,
and
\[
\omega_{\ij{1}<\ij{2}}=\eta^{\ij{1}<\ij{2}}\otimes \zeta^{\ij{1}<\ij{2}}.
\]

We find it convenient to make the following notational simplification.
Given generators $\x_1\in \Gen(\alphas^1,\betas^i)$ and
$\x_2\in\Gen(\alphas^1,\betas^{i'})$, and a sequence
$\ij{1}<\dots<\ij{n}$ with $i=i_1$ and $i'=i_n$, let
\[
\pi_2(\x_2, \eta^{\ij{n-1}<\ij{n}}, \dots,\eta^{\ij{1}<\ij{2}}, \x_1)
\] 
denote the union over all terms
${\mathbf w}_m$ in the chain $\eta^{\ij{m}<\ij{m+1}}$, $m=1,\dots,n-1$, of
\[
\pi_2(\x_2,{\mathbf w}_{n-1},\dots,{\mathbf w}_{1},\x_1).
\]
We make the corresponding notational simplification on $\Sigma_2$, as well.

\begin{definition}
  Let $\x\in\Gen(\alphas^1,\betas^{i\times j})$ 
  and $\y\in\Gen(\alphas^2,\gammas^{i\times j})$ 
  be a pair of generators with the property that the
  $\alphas^1$-arcs occupied by the generator $\x$ are complementary to
  the $\alphas^2$-arcs occupied by the generator $\y$, with respect to
  the identification induced by $\bdy
  (\Sigma_1,\alphas^1)=\PMC=-\bdy(\Sigma_2,\alphas^2)$. Then we say
  that $\x$ and
  $\y$ are a {\em complementary pair of generators}. Note that this 
  condition is equivalent to the condition that
  \[
  \x\cup \y\in \Gen(\alphas^1\cup\alphas^2,
  \betas^{i\times j}\cup \gammas^{i\times j}).
  \]
  We denote this induced generator by $\x\conn \y\coloneqq
  \x\cup \y$.

  Given two complementary pairs of generators $(\x_1,\y_1)$ and
  $(\x_2,\y_2)$,
  there is an obvious (injective) map 
  \begin{multline*}
  \pi_2(\x_2\conn \y_2,\eta^{\ij{n-1}< \ij{n}}\conn \zeta^{\ij{n-1}<\ij{n}},
  \dots, \eta^{\ij{1}<\ij{2}}\conn \zeta^{\ij{1}<\ij{2}},\x_1\conn \y_1)\\
  \to 
  \pi_2(\x_2,\eta^{\ij{n-1}<\ij{n}},\dots, \eta^{\ij{1}<\ij{2}},\x_1)
  \times
  \pi_2(\y_2,\zeta^{\ij{n-1}<\ij{n}},\dots,\zeta^{\ij{1}<\ij{2}},\y_1).
  \end{multline*}
  We call the image of this map the set of \emph{matched domains}
  (connecting these generators). If $(B_1,B_2)$ is in the image, then
  its preimage is denoted $B_1 \glue B_2$.
\end{definition}

Our next goal is to define matched holomorphic curves. Before doing
so, we need a little more notation. Fix generators
$\x_1$ and $\y_1$, a source $S_1$, Reeb chords $\rho_1,\dots,\rho_m$ and
a homology class $B_1$. We will define a map 
\[
\evbig{1}\co \Mod^{B_1}(\y_1,\eta^{\ij{n-1}<\ij{n}},\dots, \eta^{\ij{1}<\ij{2}},\x_1;\{\rho_1\},\dots,\{\rho_m\};S_1)\to 
(\RR^m\times \RR^{n-1})/\RR
\]
(where $\RR$ acts on $\RR^m\times\RR^{n-1}\cong \RR^{n+m-1}$ by
diagonal translation). To this end, consider an element $(j,u)\in
\Mod^{B_1}(\y_1,\eta^{\ij{n-1}<\ij{n}},\dots,
\eta^{\ij{1}<\ij{2}},\x_1;\{\rho_1\},\dots,\{\rho_m\};S_1)$. There is
a corresponding polygon $\kappa(j,u)\coloneqq j\in \Conf(D_{n+1})$. As
we have been suppressing $j$ from the notation, we will often abuse
notation and write $\kappa(u)$ instead of $\kappa(j,u)$. Now, consider
a biholomorphic map $\kappa(u)\to [0,1]\times\RR$ that sends the
corners of $\kappa(u)$ corresponding to $\x_1$ and $\y_1$ to $-\infty$
and $\infty$, respectively. (This map is well-defined up to
translation.) The remaining corners of $\kappa(u)$ give points
$(1,t'_1),\dots,(1,t'_{n-1})\in \{0\}\times\RR$ and the coordinates of
the Reeb chords of $u$ give points
$(0,t_1),\dots,(0,t_m)\in\{0\}\times\RR$. (Here, the orderings are
chosen so that $t'_1<\dots<t'_{n-1}$ and $t_1<\dots<t_m$.) Define
\[
\evbig{1}(u)=(t_1,\dots,t_m,t'_1,\dots,t'_{n-1})\in(\RR^m\times \RR^{n-1})/\RR.
\]
We define a map $\evbig{2}$ on the moduli space of curves in $\Sigma_2$ similarly.

\begin{definition}
  \label{def:MatchedPolygons}
  Fix $i<i'$ and $j<j'$, and let
  \begin{align*}
  \x_1&\in \Gen(\alphas^1,\betas^{i\times j}), &
  \x_2&\in \Gen(\alphas^1,\betas^{i'\times j'}),\\
  \y_1&\in \Gen(\alphas^2,\gammas^{i\times j}), &
  \y_2&\in \Gen(\alphas^2,\gammas^{i'\times j'})
  \end{align*}
  be generators with the property that $\x_1$ and $\y_1$ is a complementary
  pair of generators, and $\x_2$ and $\y_2$ is a complementary pair of 
  generators.
  Fix a sequence $i\times j=\ij{1}<\dots<\ij{n}=i'\times j'$ and
  matched homology classes 
  \begin{align*}
    B_1&\in \pi_2(\x_2, \eta^{\ij{n-1}<\ij{n}},\dots, \eta^{\ij{1}<\ij{2}},\x_1) \\
    B_2&\in \pi_2(\y_2, \zeta^{\ij{n-1}<\ij{n}},\dots, \zeta^{\ij{1}<\ij{2}}, \y_1).
  \end{align*}
  Fix also punctured Riemann surfaces with boundary $S_1$ and $S_2$
  and a sequence of Reeb chords $(\rho_1,\dots,\rho_m)$.

  Define the \emph{moduli space of matched polygons in the homology
    classes $B_1$ and $B_2$ with sources $S_1$ and $S_2$ and Reeb chords $(\rho_1,\dots,\rho_m)$} to be
  \begin{multline}\label{eq:MatchedPolygons}
    \cMM^{B_1\glue B_2}_{[0]}(\x_1\conn \y_1,\x_2\conn \y_2;\rho_1,\dots,\rho_m;S_1,S_2) \\
    \coloneqq 
    \Mod^{B_1}(\y_1,\eta^{\ij{n-1}<\ij{n}},\dots,\eta^{\ij{1}<\ij{2}},\x_1;\{\rho_1\},\dots,\{\rho_m\};S_1)\\
    \times_{\Id}
    \Mod^{B_2}(\y_2,\zeta^{\ij{n-1}<\ij{n}},\dots,\zeta^{\ij{1}<\ij{2}},\x_2;\{-\rho_1\},\dots,\{-\rho_m\};S_2),
  \end{multline}
  where by $\times_\Id$ we mean the pairs $(u_1,u_2)$ with $\evbig{1}(u_1)=\evbig{2}(u_2)$. 
 
  Note that there is a degenerate case of bigons with no Reeb chords
  (i.e., $n=1$ and $m=0$), and correspondingly there are no evaluation
  maps.
\end{definition}

(The notation $[0]$ comes from the fact that these moduli spaces will
fit into a one-parameter family indexed by a real number $t$, and
these correspond to $t=0$; see Definition~\ref{def:tSlid}.)

The expected dimension of
$\cMM^{B_1\glue B_2}_{[0]}(\x_1\conn \y_1,\x_2\conn \y_2;\rho_1,\dots,\rho_m;S_1,S_2)$
is given in terms of
\begin{equation}\label{eq:MM-ind-source}
  \begin{split}
    \ind(B_1,S_1;B_2,S_2;\rho_1,\dots,\rho_m)& \coloneqq
    \ind(B_1,S_1;\{\rho_1\},\dots,\{\rho_m\}) \\
    &\qquad\qquad +\ind(B_2,S_2;\{-\rho_1\},\dots,\{-\rho_m\})-m\\
    &=\left(\frac{3-n}{2}\right)(g_1+g_2)-\chi(S_1)-\chi(S_2)+2e(B_1)+2e(B_2)+m
  \end{split}
\end{equation}
according to the following:
\begin{lemma}\label{lem:matched-transversality}
  For a generic, admissible family of almost-complex structures for
  $\Sigma_1$ and $\Sigma_2$, the moduli space
  $\cMM^{B_1\glue B_2}_{[0]}(\x_1\conn \y_1,\x_2\conn \y_2;\rho_1,\dots,\rho_m;S_1,S_2)$
  is transversally cut out, and so is a (non-compact) manifold whose dimension 
  is
  $\ind(B_1,S_1;B_2,S_2;\rho_1,\dots,\rho_m)+n-2$
  (Formula~\eqref{eq:MM-ind-source}).
\end{lemma}
\begin{proof}
  This follows from standard transversality
  results. See~\cite[Lemma~\ref*{LOT1:lemma:Matched-Exp-Dim}]{LOT1}
  for the corresponding statement in the original bordered setting.
\end{proof}

Of course, we want to count embedded holomorphic curves; and as usual
this condition determines the Euler characteristic of $S$. Given 
\[
B\in
\pi_2(\x_2\conn \y_2,\eta^{\ij{n-1}<\ij{n}}\conn \zeta^{\ij{n-1}<\ij{n}}, \dots,
\eta^{\ij{1}<\ij{2}}\conn \zeta^{\ij{1}<\ij{2}},\x_1\conn \y_1),
\]
following~\cite{Sarkar11:IndexTriangles}, the
embedded Euler characteristic and index in the class $B$ are given by
\begin{align}
  \chi_\emb(B)&\coloneqq g+e(B)-n_{\x_0}(B)-n_{\x_n}(B)-\sum_{n\geq
      j>\ell\geq1}\bdy_j(B)\cdot\bdy_\ell(B)\\
  \ind_\emb(B)&\coloneqq e(B)+n_{\x_0}(B)+n_{\x_n}(B)-\left(\frac{n-1}{2}\right)g+\sum_{n\geq
      j>\ell\geq1}\bdy_j(B)\cdot\bdy_\ell(B).\label{eq:emb-ind-bord}
\end{align}
(Compare Proposition~\ref{prop:emb-ind}.)

\begin{definition}
  The \emph{moduli space $\cMM^{B_1\glue B_2}_{[0]}(\x_1\conn \y_1,\x_2\conn \y_2)$ of embedded matched polygons}
  is the union of the spaces
  $\cMM^{B_1\glue B_2}_{[0]}(\x_1\conn \y_1,\x_2\conn \y_2;\rho_1,\dots,\rho_m;S_1,S_2)$
  over all sequences of Reeb chords and all sources $S_1$, $S_2$ with 
  \[
  \chi(S_1\glue S_2)=\chi_{\emb}(B_1\glue B_2).
  \]
  (Here, $S_1\glue S_2$ denotes the gluing of $S_1$ and $S_2$ at the
  corresponding punctures. Note that the Euler characteristic of the
  glued surface
  $S_1\glue S_2$ is $\chi(S_1\glue S_2)=\chi(S_1)+\chi(S_2)-m$.)
\end{definition}

\begin{lemma}\label{lem:emb-matched-emb}
  A matched pair of polygons $(u_1,u_2)$ is in the corresponding
  embedded matched moduli space if and only if both $u_i$ are
  embedded. Moreover, the expected dimension of
  $\cMM^{B_1\glue B_2}_{[0]}(\x_1\conn \y_1,\x_2\conn \y_2)$ is given by
  $\ind_\emb(B_1\glue B_2)+n-2$.
\end{lemma}
(Compare~\cite[Lemma~\ref*{LOT1:lem:matched-embedded}]{LOT1}.)
\begin{proof}
  To see that the $u_i$ are embedded, it suffices to show that
  $\chi(S_i)$ is given by Formula~(\ref{eq:emb-chi}). We have
  \begin{align*}
    \chi(S_1\glue S_2)&=\chi(S_1)+\chi(S_2)-m\\
    &\geq \chi_\emb(u_1)+\chi_\emb(u_2)-m\\
    &=g_1+e(B_1)-n_{\x_0}(B_1)-n_{\x_n}(B_1)
    -\sum_{n\geq
      j>\ell\geq1}\bdy_j(B_1)\cdot\bdy_\ell(B_1)\\
    &\qquad-\sum_i\iota(\{\rho_i\})-\sum_{i<j}L(\rho_i,\rho_j)\\
    &\qquad + g_2+e(B_2)-n_{\x_0}(B_2)-n_{\x_n}(B_2)
    -\sum_{n\geq
      j>\ell\geq1}\bdy_j(B_2)\cdot\bdy_\ell(B_2)\\
    &\qquad-\sum_i\iota(\{-\rho_i\})-\sum_{i<j}L(-\rho_i,-\rho_j)-m\\
    &=\chi_\emb(B_1\glue B_2),
  \end{align*}
  where the last equality uses the facts that 
  \begin{align*}
  \sum_i\iota(\{\rho_i\})+\sum_i\iota(\{-\rho_i\})&=-m\\
  \sum_{i<j}L(\rho_i,\rho_j)+\sum_{i<j}L(-\rho_i,-\rho_j)&=0.
  \end{align*}
  Thus, if $\chi(S_1\glue S_2)=\chi_\emb(B_1\glue B_2)$ we must have
  $\chi(S_i)=\chi_\emb(u_i)$, as desired.

  The fact that the expected dimension is given by
  $\ind_\emb(B_1\glue B_2)+n-2$ follows from Formula~(\ref{eq:emb-ind-bord})
  and Lemma~\ref{lem:matched-transversality}.
\end{proof}

The moduli spaces defined above can be used to construct an $\IndI\times \IndJ$-filtered chain complex
$\MatchedComplex=\{\MatchedComplex^{i\times j},D^{i\times j<i'\times j'}\}$ defined by
$\MatchedComplex^{i\times
  j}=\CFa(\alphas^1\cup\alphas^2,\betas^{i\times
  j}\cup\gammas^{i\times j},z)$ and
\[
D^{i\times j<i'\times j'}(\x_1\conn  \y_1) = \hspace{-1em}
\sum_{\substack{\x_2\times
  \y_2\in\Gen(\alphas^1\cup\alphas^2,\betas^{i'\times
    j'}\cup\gammas^{i'\times j'})\\ 
  \ij{1} < \ij{2} < \dots <\ij{n} \\ 
(B_1,B_2)\text{ s.t.}\ind_\emb(B_1\glue B_2)=2-n}}\hspace{-1em}
\#\cMM_{[0]}^{B_1\glue B_2}(\x_1\conn \y_1,\x_2\conn \y_2) \x_2\conn  \y_2.
\]
 
The following is a generalization of Lemma~\ref{lem:FiberedProduct},
as well as~\cite[Theorem~\ref*{LOT1:thm:PrimitivePairing}]{LOT1}:
\begin{proposition}
  \label{prop:TakeNeckToInfinity}
  For sufficiently long connect sum parameters on $\Sigma_1\conn \Sigma_2$, 
  there is an identification of chain complexes
  \[
  \MatchedComplex\cong
  \CCFa(\Sigma_1\cup\Sigma_2,\alphas^1\cup\alphas^2,
  (\{\betas^i\}_{i\in \IndI},\{\eta^{i<i'}\}_{i,i'\in \IndI})
  \#(\{\gammas^{j}\}_{j\in \IndJ},\{\zeta^{j<j'}\}_{j,j'\in \IndJ}),
  z).\]
\end{proposition}

\begin{proof}
  This is a standard neck stretching argument;
  see~\cite[Theorem~\ref*{LOT1:thm:PrimitivePairing}]{LOT1} for the result
  for the case of bigons.
\end{proof}

\subsection{Time translation}

We now introduce a time translation parameter $t\in \RR$.
\[
\tau_t\co (\RR^{m}\times\RR^{n-1})/\RR \to (\RR^{m}\times\RR^{n-1})/\RR
\]
be the map
\[
\tau_t(t_1,\dots,t_m,t'_1,\dots,t'_{n-1})=(t+t_1,\dots,t+t_m,t'_1,\dots,t'_{n-1}).
\]

\begin{definition}\label{def:tSlid}
  The \emph{moduli space of $t$-slid matched polygons} is defined
  exactly like the moduli space of matched polygons
  (Definition~\ref{def:MatchedPolygons}) except with
  Equation~\ref{eq:MatchedPolygons} replaced by
  \begin{multline}
  \cMM^{B_1\glue B_2}_{[t]}(\x_1\conn \y_1,\x_2\conn \y_2;\rho_1,\dots,\rho_m;S_1,S_2) \\
  \coloneqq 
  \Mod^{B_1}(\x_2,\eta^{\ij{n-1}<\ij{n}},\dots,
  \eta^{\ij{1}<\ij{2}},\x_1;\{\rho_1\},\dots,\{\rho_m\};S_1)\\
\times_{\tau_t}
  \Mod^{B_2}(\y_2,\zeta^{\ij{n-1}<\ij{n}},\dots,
  \zeta^{\ij{1}<\ij{2}},\y_1;\{-\rho_1\},\dots,\{-\rho_m\};S_2),
  \end{multline}
  where by $\times_{\tau_t}$ we mean the pairs $(u_1,u_2)$ with $\tau_t(\evbig{1}(u_1))=\evbig{2}(u_2)$. 

  The \emph{moduli space of embedded $t$-slid matched polygons}
  $\cMM^{B_1\glue B_2}_{[t]}(\x_1\conn \y_1,\x_2\conn \y_2)$
  is the union of the spaces
  $\cMM^{B_1\glue B_2}_{[t]}(\x_1\conn \y_1;\x_2\conn \y_2;\rho_1,\dots,\rho_m;S_1,S_2)$
  over all sequences of Reeb chords and all sources $S_1$, $S_2$ with 
  $\chi(S_1\glue S_2)=\chi_{\emb}(B_1\glue B_2).$
\end{definition}

The notion of a $0$-slid matched polygon, of course, coincides
with the notion of a matched polygon as in
Definition~\ref{def:MatchedPolygons}. Moreover, the moduli space of
$t$-slid matched bigons ($n=1$) is independent of $t$, as is the degenerate
case of $t$-slid matched polygons with no Reeb chords ($m=0$).

The obvious analogue of Lemma~\ref{lem:matched-transversality} holds
for the $t$-slid matched moduli spaces: 
\begin{lemma}\label{lem:t-matched-transversality}
  For any fixed $t$ and generic $J$ (depending on $t$), the moduli
  spaces of $t$-slid matched polygons
  $\cMM^{B_1\glue B_2}_{[t]}(\x_1\conn \y_1;\x_2\conn \y_2;\rho_1,\dots,\rho_m;S_1,S_2)$
  are transversally cut out. Similarly, for fixed, generic $J$ and
  generic $t$ (depending on $J$), the moduli spaces
  \[\cMM^{B_1\glue B_2}_{[t]}(\x_1\conn \y_1;\x_2\conn \y_2;\rho_1,\dots,\rho_m;S_1,S_2)\]
  are transversally cut out. The dimension of
  $\cMM^{B_1\glue B_2}_{[t]}(\x_1\conn \y_1;\x_2\conn \y_2;\rho_1,\dots,\rho_m;S_1,S_2)$
  is given by $\ind(B_1,S_1;B_2,S_2;\rho_1,\dots,\rho_m)+n-2$.
\end{lemma}
\begin{proof}
  Like Lemma~\ref{lem:matched-transversality}, this follows from
  standard arguments (which we do not spell out here).
\end{proof}

Lemma~\ref{lem:emb-matched-emb} still applies to show that the
embedded moduli spaces consist of embedded curves.

The moduli spaces of $t$-slid matched polygons can be used to define a
chain complex $\tMatchedComplex$ generalizing
$\MatchedComplex$. Define $\tMatchedComplex=\{\tMatchedComplex^{i\times
  j},D_{[t]}^{i\times j<i'\times j'}\}$ to be $\tMatchedComplex^{i\times
  j}=\CFa(\alphas^1\cup\alphas^2,\betas^{i\times
  j}\cup\gammas^{i\times j},z)$ with differential
\[
D_{[t]}^{i\times j<i'\times j'}(\x_1\conn  \y_1) =\hspace{-1em}
\sum_{\substack{\x_2\times
  \y_2\in\Gen(\alphas^1\cup\alphas^2,\betas^{i'\times
    j'}\cup\gammas^{i'\times j'})\\[0.5pt]
  \ij{1} < \ij{2} < \dots \ij{n}\\[0.5pt]
  (B_1,B_2)\text{ s.t.}\ind_\emb(B_1\glue B_2)=2-n}}\hspace{-1em}
\#\cMM_{[t]}^{B_1\glue B_2}(\x_1\conn \y_1;\x_2\conn \y_2) \x_2\conn  \y_2.
\]
This is the same as $\MatchedComplex$, except that the differential
counts $t$-slid matched polygons rather than simply matched polygons.

\begin{remark}
  Moduli spaces with a shift in the time parameter appeared in the
  proof of~\cite[Theorem~\ref*{LOT2:thm:Hochschild}]{LOT2}.
\end{remark}

\begin{lemma}\label{lem:tMatched}
  For a generic $J$ and generic $t\in \RR$, $\tMatchedComplex$ is an
  $\IndI\times \IndJ$-filtered chain complex.
\end{lemma}

\begin{proof}
  As usual, this follows from looking at ends of one-dimensional
  moduli spaces;
  compare~\cite[Proposition~\ref*{LOT1:prop:DSquaredZero}]{LOT1}.  A
  one-dimensional moduli space of $t$-slid matched polygons can have
  the following kinds of ends:
  \begin{enumerate}[label=(e-\arabic*),ref=(e-\arabic*)]
  \item 
    \label{end:AlphaDegeneration}
    A two-story $t$-matched polygon (i.e., this looks like two
    copies of the picture from Figure~\ref{fig:attach-circ-cx} parts
    (d) or (e)).
  \item 
    \label{end:BetaDegeneration}
    A $t$-matched polygon, juxtaposed with a pair of polygons, with
    boundaries on the $\beta$-curves (i.e., this looks like two copies
    of the picture from Figure~\ref{fig:attach-circ-cx} parts (b) or
    (c)).
  \item 
    \label{end:Join}
    A $t$-slid matched comb, with a single curve at $e\infty$ which is
    a join component (in the sense
    of~\cite[Section~\ref*{LOT1:sec:curves_at_east_infinity}]{LOT1}).
  \item 
    \label{end:Split}
    A $t$-slid matched comb, with a single curve at $e\infty$ which is
    a split component (in the sense
    of~\cite[Section~\ref*{LOT1:sec:curves_at_east_infinity}]{LOT1}).
  \end{enumerate}
  Ends of Type~\ref{end:BetaDegeneration} cancel in pairs (just like
  in the proof of
  Proposition~\ref{prop:ChainComplexesOfAttachingCircles}), since the
  $\{\eta^{i\times j<i'\times j'}\conn \zeta^{i\times j<i'\times
    j'}\}_{i\times j<i'\times j'\in\IndI\times\IndJ}$ form a chain complex of attaching
  circles (according to Proposition~\ref{prop:GlueChainComplexes}).
    
  Ends of Types~\ref{end:Join} and~\ref{end:Split} are the only ends
  involving curves at $e\infty$, as in the proof
  of~\cite[Proposition~\ref{LOT1:prop:EndsOfTwoDimModuliSpacesT}]{LOT1}.
  Moreover, these ends cancel in pairs, as in the proof
  of~\cite[Proposition~\ref{LOT1:prop:CombsCancel}]{LOT1}.
  
  The remaining terms, corresponding to ends of
  Type~\ref{end:AlphaDegeneration}, are all counted in the $\x_2\times
  \y_2$ coefficient of
  \[
  D_{[t]}^{i_1\times j_1<i'\times j'}\circ D_{[t]}^{i\times j<i_1\times
    j_1}(\x_1\times \y_1) + \partial \circ D_{[t]}^{i\times j<i'\times
    j'}(\x_1\times \y_1) + D_{[t]}^{i\times j<i'\times
    j'}\circ \partial(\x_1\times \y_1).
  \] 
  (The last two terms correspond to when one of the degenerating
  polygons is a bigon.)
\end{proof}

We turn next to proving independence of the translation parameter
$t$. Again, we need more notation.

Fix a smooth function $\tau^1_{t_1,t_2}\co \RR\to \RR$ so that there is a constant $N$ with
\begin{equation}
  \label{eq:InterpolatingFunction}
  \tau^1_{t_1,t_2}(t)=\left\{\begin{array}{ll}
      t+t_1 & {\text{if $t<-N$}} \\
      t+t_2 & {\text{if $t>N$}}\end{array}\right.
\end{equation}
This induces a function 
\[
\tau_{t_1,t_2}\co \RR^{m}\times\RR^{n-1}\to\RR^{m}\times\RR^{n-1}
\]
by
\[
\tau_{t_1,t_2}(x_1,\dots,x_m,y_1,\dots,y_{n-1})=(\tau^1_{t_1,t_2}(x_1),\dots,\tau^1_{t_1,t_2}(x_m),y_1,\dots,y_{n-1}).
\]

Let $\pi\co \RR^{m+n-1}\to\RR^{m+n-1}/\RR$ denote projection. Let
\[
\tMod^{B_1}=\{(u,p)\in \Mod^{B_1}\times\RR^{m+n-1} \mid
\evbig{1}(u)=\pi(p)\in (\RR^m\times\RR^{n-1})/\RR\}. 
\]
So, $\tMod^{B_1}\cong \Mod^{B_1}\times\RR$. Define $\tMod^{B_2}$
similarly. There are well-defined maps $\evbig{i}\co \tMod^{B_i}\to
\RR^m\times\RR^{n-1}$.

\begin{definition}
  The
  \emph{moduli space of $t_1$-$t_2$-slid-matched polygons in $B_1$ and
    $B_2$} is defined exactly like the moduli space of matched polygons
  (Definition~\ref{def:MatchedPolygons}) except with
  Formula~(\ref{eq:MatchedPolygons}) replaced by 
  \begin{multline}
  \cMM^{B_1\glue B_2}_{[t_1;t_2]}(\x_1\conn \y_1,\x_2\conn \y_2;\rho_1,\dots,\rho_m;S_1,S_2) \\
  \coloneqq 
  \tMod^{B_1}(\x_2,\eta^{\ij{n-1}<\ij{n}},\dots,
  \eta^{\ij{1}<\ij{2}},\x_1;\{\rho_1\},\dots,\{\rho_m\};S_1)\\
  \times_{\tau_{t_1,t_2}}
  \tMod^{B_2}(\y_2,\zeta^{\ij{n-1}<\ij{n}},\dots,
  \zeta^{\ij{1}<\ij{2}},\y_1;\{-\rho_1\},\dots,\{-\rho_m\};S_2),
  \end{multline}
  where by $\times_{\tau_{t_1,t_2}}$ we mean the pairs $(u_1,u_2)$ with $\tau_{t_1,t_2}(\evbig{1}(u_1))=\evbig{2}(u_2)$. 

  The \emph{moduli space of embedded $t_1$-$t_2$-slid matched polygons}
  $\cMM^{B_1\glue B_2}_{[t_1;t_2]}(\x_1\conn \y_1,\x_2\conn \y_2)$
  is the union of the spaces
  $\cMM^{B_1\glue B_2}_{[t_1;t_2]}(\x_1\conn \y_1;\x_2\conn \y_2;\rho_1,\dots,\rho_m;S_1,S_2)$
  over all sequences of Reeb chords and all sources $S_1$, $S_2$ with 
  $\chi(S_1\glue S_2)=\chi_{\emb}(B_1\glue B_2).$
\end{definition}

Note that the space
$\cMM^{B_1\glue B_2}_{[t_1;t_2]}(\x_1\conn \y_1,\x_2\conn \y_2;\rho_1,\dots,\rho_m;S_1,S_2)$
is one dimension larger than the space
$\cMM^{B_1\glue B_2}_{[t]}(\x_1\conn \y_1,\x_2\conn \y_2;\rho_1,\dots,\rho_m;S_1,S_2)$.

In the case of no Reeb chords ($m=0$), if $n>1$ then the moduli space
\[
\cMM^{B_1\glue B_2}_{[t_1;t_2]} \cong
\RR\times \Mod^{B_1}\times_{ \kappa_{B_1}=\kappa_{B_2} } \Mod^{B_2}.
\]
is never rigid. In the case of bigons with no Reeb chords ($m=0$, $n=1$), the moduli space
$\cMM^{B_1\glue B_2}_{[t_1;t_2]}$ is rigid only in the special case that $B_1$ and $B_2$ are the trivial
(all $0$) domain of bigons. We do consider these $\RR$-invariant
bigons (unions of trivial strips) to be $t_1$-$t_2$-slid matched
polygons.

The obvious analogues of Lemmas~\ref{lem:emb-matched-emb}
and~\ref{lem:t-matched-transversality} hold, namely, the curves in the
embedded $t_1$-$t_2$-slid matched moduli spaces are embedded and for
generic $J$ and generic $t_1$, $t_2$ and $\tau_{t_1,t_2}$ these moduli
spaces are transversally cut out.

Counting $t_1$-$t_2$-slid matched polygons furnishes the chain
homotopy equivalence used in the following:

\begin{lemma}
  \label{lem:VaryTranslationParameter}
  Given generic $J$ and generic $t_1, t_2\in\RR$,
  there is a filtered chain homotopy equivalence ${\mathcal
    C}_{[t_1]}\simeq {\mathcal C}_{[t_2]}$.
\end{lemma}

\begin{proof}
  Define a map
  $f_{[t_1;t_2]}\co {\mathcal C}_{[t_1]}\to{\mathcal C}_{[t_2]}$
  by
  \[
  f(\x_1\conn \y_1)=
\sum_{\substack{\x_2\times
  \y_2\in\Gen(\alphas^1\cup\alphas^2,\betas^{i'\times
    j'}\cup\gammas^{i'\times j'})\\[0.5pt]
  \ij{1} < \ij{2} < \dots \ij{n}\\[0.5pt]
  (B_1,B_2)\text{ s.t.}\ind_\emb(B_1\glue B_2)=1-n}}\hspace{-1em}
  \#\cMM^{B_1\glue B_2}_{[t_1;t_2]}(\x_1\conn \y_1;\x_2\conn \y_1)
  \x_2\conn \y_2.
  \] 
  The verification that $f$ is a chain map goes by considering the
  ends of one-dimensional moduli spaces
  $\cMM^{B\glue C}_{[t_1;t_2]}(\x_1\conn \y_1,\x_2\conn \y_2)$, which are of the
  following kinds:
  \begin{enumerate}[label=(e-\arabic*),ref=(e-\arabic*)]
  \item 
    \label{end:FBoundary} Pairs of two-story polygons $u_1*u_2$ (on $\Sigma_1$) and
    $v_1*v_2$ (on $\Sigma_2$), where $(u_1,v_1)$ is a
    $t_1$-$t_2$-slid matched polygon and $(u_2,v_2)$ is a
    $t_2$-slid matched polygon; here, each of the four polygons $u_1$, $u_2$,
    $v_1$, and $v_2$ has an edge mapped into $\alphas$.
  \item 
    \label{end:BoundaryF} Pairs of two-story polygons $u_1*u_2$ (on $\Sigma_1$) and
    $v_1*v_2$ (on $\Sigma_2$), where $(u_1,v_1)$ is a
    $t_1$-slid matched polygon and $(u_2,v_2)$ is a
    $t_1$-$t_2$-slid matched polygon; here, each of the four polygons $u_1$, $u_2$,
    $v_1$, and $v_2$ has an edge mapped into $\alphas$.
  \item 
    \label{end:CancelInPairs}
    Pairs of two-story polygons $u_1*u_2$ (on $\Sigma_1$) and $v_1*v_2$
    (on $\Sigma_2$), where only $u_1$ and $v_1$ have edges mapped into
    $\alphas$; in this case, so $(u_1,v_1)$ is a
    $t_1$-$t_2$-slid matched polygon and $u_2$ and $v_2$ are ordinary,
    provincial polygons.
  \item 
    \label{end:JoinSplit}
    Triples $(u_1, w, v_2)$, where $w$ is a curve at $e\infty$.
  \end{enumerate}
  The count of ends of Type~\ref{end:FBoundary} contributes
  $\partial_{[t_2]}\circ f_{[t_1;t_2]}$, while the count of ends of
  Type~\ref{end:BoundaryF} contributes
  $f_{[t_1;t_2]}\circ \partial_{[t_1]}$.  Ends of
  Type~\ref{end:CancelInPairs} cancel in pairs, because the
  $\eta^{i\times j<i'\times j'}\conn \zeta^{i\times j<i'\times j'}$ form a
  chain complex of attaching circles.
  Ends of Type~\ref{end:JoinSplit} cancel in pairs, corresponding to
  viewing a split component for $u_1$ (respectively $u_2$) as a join
  component for $u_2$ (respectively $u_1$), or vice-versa (as in the
  proof of Lemma~\ref{lem:tMatched}
  or~\cite[Proposition~\ref{LOT1:prop:CombsCancel}]{LOT1}).

  The composition $f_{[t_1;t_2]}\circ f_{[t_2;t_1]}$ is chain
  homotopic to the identity map. The chain homotopy counts polygons in
  a one-parameter family of moduli spaces of the form
  $\cMM_{[t_1;t_1]}^{B_1\glue B_2}$ indexed by a real parameter $T$ which varies the choice of interpolating function
  (Equation~\eqref{eq:InterpolatingFunction}) implicit in the
  definition of the moduli space.
\end{proof}

\subsection{Translating time to \texorpdfstring{$\infty$}{infinity} and cross-matched polygons}
Consider the moduli space $\Mod^{B_1}(\y_1,\eta^{\ij{n-1}<\ij{n}},\dots, \eta^{\ij{1}<\ij{2}},\x_1;\{\rho_1\},\dots,\{\rho_m\};S_1)$. The
$\RR$-coordinates of the Reeb chords give an evaluation map
\[
\ev_{B_1}\co \Mod^{B_1}(\y_1,\eta^{\ij{n-1}<\ij{n}},\dots, \eta^{\ij{1}<\ij{2}},\x_1;\{\rho_1\},\dots,\{\rho_m\};S_1)\to \RR^{m}/\RR.
\]
Similarly, the conformal structure on the source gives a
forgetful map
\[
\kappa_{B_1}\co \Mod^{B_1}(\y_1,\eta^{\ij{n-1}<\ij{n}},\dots, \eta^{\ij{1}<\ij{2}},\x_1;\{\rho_1\},\dots,\{\rho_m\};S_1)\to \Conf(D_{n+1})
\]
(compare Equation~\eqref{eq:kappa}). In other words, these maps can be obtained from $\evbig{1}$ by projecting to $\RR^m/\RR$ or $\RR^{n-1}/\RR$, respectively:
\[
\begin{tikzpicture}
\node at (0,0) (MM) {$\Mod^{B_1}$};
\node at (4,0) (RR) {$(\RR^m\times\RR^{n-1})/\RR$};
\node at (8,2) (Rtop) {$\RR^{m}/\RR$};
\node at (8,-2) (Conf) {$\RR^{n-1}/\RR\cong \Conf(D_{n+1}).$};
\draw[->] (MM) to node[above]{\lab{\evbig{1}}} (RR);
\draw[->, bend left=12] (MM) to node[above]{\lab{\ev_{B_1}}} (Rtop);
\draw[->, bend right=12] (MM) to node[below]{\lab{\kappa_{B_1}}} (Conf);
\draw[->] (RR) to (Rtop);
\draw[->] (RR) to (Conf);
\end{tikzpicture}
\]
Similar remarks apply to curves in $\Sigma_2$. 

\begin{definition}
  \label{def:CrossMatchedPolygons}
  With notation as in Definition~\ref{def:MatchedPolygons}, define the
  \emph{moduli space of cross-matched polygons in the homology classes
    $B_1$ and $B_2$ with sources $S_1$, $S_2$, $T_1$ and $T_2$ and
    Reeb chords $(\rho_1,\dots,\rho_m)$} (for $m\geq 1$),
  \[
  \CrossMatched^{B_1\glue B_2}(\x_1\conn \y_1,\x_2\conn \y_2;\rho_1,\dots,\rho_m;S_1,T_1,S_2,T_2),
  \]
  to consist of quadruples of pseudo-holomorphic curves
  \begin{align*}
    u_1&\co S_1\to \Sigma_1\times[0,1]\times\RR & u_2&\co T_1\to\Sigma_1\times[0,1]\times\RR\\
    v_1&\co S_2\to \Sigma_2\times[0,1]\times\RR & v_2&\co T_2\to\Sigma_2\times[0,1]\times\RR
  \end{align*}
  satisfying:
  \begin{enumerate}[label=(X-\arabic*),ref=(X-\arabic*)]
  \item\label{item:XM-1} 
    $u_1$ represents a homology class
    \[
    B_1^1\in \pi_2(\x', \eta^{\ij{n_1-1}<\ij{n_1}}, \dots, \eta^{\ij{1}<\ij{2}}, \x_1);
    \]
    for some $1\leq n_1 \leq n$ and
    $\x'\in\Gen(\alphas_1,\betas^{\ij{n_1}})$ (so $u_1$ is
    an $n_1+1$-gon); and $u_2$ represents a homology class
    \[
    B_1^2\in \pi_2 (\x_2, \eta^{\ij{n-1}<\ij{n}},\dots, \eta^{\ij{n_2}<\ij{n_1+1}},\x')
    \] 
    (so $u_2$ is an $n_2+1$-gon with $n_1+n_2=n+1$).
  \item\label{item:XM-2}
    $v_1$ represents a homology class
    \[
    B_2^1\in \pi_2 (\y', \zeta^{\ij{n_1-1}<\ij{n_1}}, \dots, \zeta^{\ij{1}<\ij{2}}, \y_1),
    \]
    for some $\y'\in\Gen(\alphas_2,\betas^{\ij{n_1}})$ and
    $v_2$ represents a homology class
    \[
    B_2^2\in \pi_2(\y_2, \zeta^{\ij{n-1}<\ij{n}}, \dots, \zeta^{\ij{n_1}<\ij{n_1+1}}, \y').
    \]
  \item $B_1^1+B_1^2=B_1$, and $B_2^1+B_2^2=B_2$.
  \item $u_2$ and $v_1$ are provincial (i.e., they have no Reeb chords at  $e\infty$).
  \item $\ev_{B_1^1}(u_1)$ and $\ev_{B_2^2}(v_1^2)$ agree up to overall
    translation (i.e., as elements of $\RR^{m}/\RR$).
  \item 
    \label{crM:Moduli}
    The conformal structures of the bases of $u_1$ and $v_1$ coincide,
    as do the conformal structures of the bases of $u_2$ and $v_2$;
    i.e.,
    \begin{align*}
      \kappa_{B_1^1}(u_1)&=\kappa_{B_2^1}(v_1) \\
      \kappa_{B_1^2}(u_2)&=\kappa_{B_2^2}(v_2).
    \end{align*}
  \end{enumerate}

  The moduli space of \emph{embedded cross-matched polygons}
  $\CrossMatched^{B_1\glue B_2}(\x_1\conn \y_1,\x_2\conn \y_2)$ is the
  union of the spaces 
  \[
  \CrossMatched^{B_1\glue B_2}(\x_1\conn \y_1,\x_2\conn \y_2;\rho_1,\dots,\rho_m;S_1,T_1,S_2,T_2),
  \]
  over all sequences of Reeb chords and sources $S_1$, $T_1$, $S_2$,
  $T_2$ with 
  \[
  \chi(S_1\glue T_1\glue S_2\glue T_2)=\chi_{\emb}(B_1\glue B_2).
  \]
  (As usual, $\glue$ denotes gluing at the corresponding punctures.) In
  this union, we also include the matched polygons (from
  Definition~\ref{def:MatchedPolygons}) with no Reeb chords,
  corresponding to the degenerate case $m=0$.
\end{definition}

Fix a cross-matched holomorphic curve $(u_1*u_2,v_1*v_2)$.  Let
$\dim(u_1)$ denote the expected dimension of the moduli space of
holomorphic curves near $u_1$; i.e.,
$\dim(u_1)=\ind(u_1)+n_1-2$. (Note that in the degenerate case that
$u_1$ is an $\RR$-invariant bigon, $\dim(u_1)=-1$ while the moduli
space is, in fact, $0$-dimensional.)  Define $\dim(u_2)$, $\dim(v_1)$
and $\dim(v_2)$ similarly.

\begin{lemma}\label{lem:trans-cross-matched}
  With respect to a generic family of almost-complex structures, the
  moduli spaces of cross-matched holomorphic polygons are
  transversally cut out. Fix a cross-matched polygon
  $(u_1*u_2,v_1*v_2)$, and suppose that $c$ of $\{u_1,u_2,v_1,v_2\}$
  are $\RR$-invariant bigons (disjoint unions of trivial
  strips). Then, near $(u_1*u_2,v_1*v_2)$ the moduli space of
  cross-matched polygons has dimension
  \[
  \dimXM(u_1*u_2,v_1*v_2)=\dim(u_1)+\dim(u_2)+\dim(v_1)+\dim(v_2)+3-n+c-\min\{m-1,0\},
  \]
  where $m$ denotes the number of Reeb chords in $u_1$ (or
  equivalently $v_2$).
\end{lemma}
\begin{proof}
  Again, the transversality statement follows from standard
  techniques. For the statement about dimensions, observe that 
  we are matching the conformal structures on $u_1$ and $v_1$, which
  gives an $(n_1-2)$-dimensional constraint; then we are matching
  conformal structures on $u_2$ and $v_2$, which gives a further
  $(n_2-2)$-dimensional constraint; finally, we have an $(m-1)$-dimensional
  constraint coming from the matching condition on the
  chords (if $m\geq 1$). Together with the observation that $n=n_1+n_2-1$, this
  explains all of the terms in the formula except $c$, which comes
  from the fact that if $u_1$, say, is an $\RR$-invariant bigon then
  $\dim(u_1)$ differs by $1$ from the actual dimension of the moduli
  space near $u_1$ (because $u_1$ has $\RR$ as a stabilizer).
\end{proof}

Counting cross-matched polygons also gives a chain complex $\ComplexCross$.
In fact:

\begin{lemma}
  \label{lem:TransToInfinity}
  For all sufficiently large $t$, the differential in
  $\tMatchedComplex$ coincides with the differential counting
  cross-matched polygons.
\end{lemma}

\begin{proof}
  This follows from compactness and gluing arguments,
  similar to (but easier than) the proof
  of~\cite[Proposition~\ref*{LOT1:prop:LargeTLimit}]{LOT1}. Fix a
  rigid, non-provincial homology class, so the pair of curves are
  asymptotic to at least one Reeb chord at $e\infty$: the provincial case is
  trivial. The compactness theorem for pseudoholomorphic
  combs~\cite[Proposition 5.24]{LOT1}, or rather its obvious extension
  to polygons, implies that any sequence of $t$-slid matched polygons
  with $t\to\infty$ has a subsequence converging to some pair $U$, $V$
  of holomorphic combs. Further, the evaluation maps $\ev$ at the
  (far) east punctures of $U$ and $V$ must satisfy that the evaluation
  map on the $i\th$ story of $U$ matches with the evaluation map on
  the $(i-1)\st$ story of $V$, up to an overall translation on each
  story, while the conformal structures (i.e., the forgetful maps
  $\kappa$) on the $i\th$ story of $U$ and the $i\th$ story of $V$
  must agree.  In particular, each of $U$ and $V$ must have at least
  two stories.
  
  Next, a dimension count implies that each of $U$ and $V$ has no
  components at $e\infty$, and consists of exactly two stories. The
  fact that in the limit we consider the conformal structures and
  evaluation maps separately means that number of matching conditions
  has decreased by~$1$, i.e., the expected dimension has increased
  by~$1$. However, each story of $U$ or $V$ beyond the first reduces
  the
  expected dimension by $1$ on each side, so $2$ overall, but only
  decreases the number of matching conditions imposed by the
  evaluation maps $\evbig{}$ by $1$. 
  Thus, with two stories the expected dimension of the limit object is
  at most the same as for a 1-story $t$-slid matched curves; with
  three stories the expected dimension is at least one smaller than
  for a 1-story $t$-slid matched curve; and so on. So, for generic
  almost-complex structures, $U$ and $V$ must have exactly two
  stories.  Degenerating a curve at $e\infty$ or having two levels of
  Reeb chords collapse reduces the expected dimension of the moduli
  space on each side by at least $1$, but only reduces the number of
  matching conditions imposed
  (cf.~\cite[Theorem~\ref*{LOT1:thm:master_equation}]{LOT1}) by at
  most one, so is a codimension-1 degeneration overall. Thus, each of
  $U$ and $V$ must consist of two stories with no components at
  $e\infty$.
  
  So, we have shown that each pair $(U,V)$ occurring as a $t\to\infty$
  end of the moduli space of $t$-slid matched polygons is a
  cross-matched polygon.  Finally, the polygon analogue of the gluing
  result~\cite[Proposition~\ref*{LOT1:prop:gluing_two_story}]{LOT1}
  implies that near each cross-matched polygon the space
  $\bigcup_{t\in(T,\infty)}\cMM^{B_1\#B_2}_{[t]}$ is an open interval,
  so in particular the modulo $2$ count of $t$-slid matched polygons
  and cross-matched polygons agrees.
\end{proof}

\subsection{Small \texorpdfstring{$\epsilon$}{epsilon} approximation and chord-matched polygon pairs}

We next consider what happens when we take the approximation parameter
$\epsilon\to 0$. Our main goal is to show that, for $\epsilon$
sufficiently small, rigid cross-matched polygons correspond to rigid
chord-matched polygon pairs (Definition~\ref{def:ChordMatchedPolygons}); see
Proposition~\ref{prop:SendEpsilonToZero}. To prove this, we introduce
an auxiliary notion, simplified cross-matched polygons, which are the
Gromov limits of cross-matched polygons as $\epsilon\to 0$. (The
definition of simplified cross-matched polygons can be refined, but
Definition~\ref{def:SimplifiedCrossMatchedPolygons} will suffice
for our purposes: its main role is to restrict what kinds of
cross-matched polygons exist for small $\epsilon$.)

\begin{definition}
  \label{def:SimplifiedCrossMatchedPolygons}
  Fix $i<i'$ and $j<j'$, and let
  \begin{align*}
  \x_1&\in \Gen(\alphas^1,\betas^i), &
  \x_2&\in \Gen(\alphas^1,\gammas^{i'}),\\
  \y_1&\in \Gen(\alphas^2,\betas^{j}), &
  \y_2&\in \Gen(\alphas^2,\gammas^{j'})
  \end{align*}
  be generators with the property that $\x_1$ and $\y_1$ is a complementary
  pair of generators, and $\x_2$ and $\y_2$ is a complementary pair of 
  generators.
  Fix sequences $i=i_0<\dots<i_{n_1}=i'$ and $j=j_0<\dots<j_{n_2}=j'$ and
  homology classes 
  \begin{align*}
    B_1&\in \pi_2(\x_2, \eta^{i_{n_1-1}< i_{n_1}}, \dots,
    \eta^{i_1<i_2}, \x_1) \\
    B_2&\in \pi_2(\y_2, \zeta^{j_{n_2-1}<j_{n_2}}, \dots,
    \zeta^{j_1<j_2}, \y_1).
  \end{align*}
  Fix also punctured Riemann surfaces with boundary $S_1$, $S_2$,
  $T_1$ and $T_2$ and a sequence of Reeb chords
  $(\rho_1,\dots,\rho_m)$.

  Define the \emph{moduli space of simplified cross-matched polygons
    in the homology classes $B_1$ and $B_2$ with sources $S_1$, $S_2$,
  $T_1$ and $T_2$},
  \[ \SimpCrossMatched^{B_1\glue B_2}(\x_1\conn \y_1;\x_2\conn \y_2;\rho_1,\dots,\rho_m;S_1,T_1,S_2,T_2),\]
  to consist of quadruples of pseudo-holomorphic curves 
  \begin{align*}
    u_1&\co S_1\to \Sigma_1\times[0,1]\times\RR & u_2&\co T_1\to\Sigma_1\times[0,1]\times\RR\\
    v_1&\co S_2\to \Sigma_2\times[0,1]\times\RR & v_2&\co T_2\to\Sigma_2\times[0,1]\times\RR
  \end{align*}
  such that
  \begin{enumerate}[label=(SX-\arabic*),ref=(SX-\arabic*)]
    \item 
    $u_1$ represents some
    \[
    B_1^1\in \pi_2 (\x', \eta^{i_{\ell_1-1}<i_{\ell_1}}, \dots, \eta^{i_1<i_2}, \x_1)
    \]
    for some $1\leq \ell_1 \leq n_1$ and
    $\x'\in\Gen(\alphas_1,\betas^{i_{\ell_1}})$ (so $B_1^1$
    is an $\ell_1+1$-gon); and $u_2$ represents some
    \[
    B_1^2\in \pi_2 (\x_2, \eta^{i_{n_1-1}<i_{n_1}}, \dots,
    \eta^{i_{\ell_1}<i_{\ell_1+1}}, \x')
    \]
    (so $B_1^2$ is an $\ell_1'+1$-gon for $\ell_1+\ell_1'-1=n_1$);
  \item $v_1$ represents some
    \[
    B_2^1\in \pi_2 (\y', \zeta^{j_{\ell_2-1}< j_{\ell_2}}, \dots,
    \zeta^{j_1<j_2}, \y_1)
    \]
    for some $1\leq \ell_2 \leq n_2$ and
    $\y'\in\Gen(\alphas_2,\betas^{j_{\ell_2}})$ (so $B_2^1$
    is an $\ell_2+1$-gon); and $v_2$
    represents some
    \[
    B_2^2\in \pi_2 (\y_2, \zeta^{j_{n_2-1}<j_{n_2}}, \dots,
    \zeta^{j_{\ell_2}<j_{{\ell_2}+1}}, \y');
    \]
    (so $B_2^2$ is an $\ell_2'+1$-gon for $\ell_2+\ell_2'-1=n_2$);
  \item $B_1^1+B_1^2=B_1$ and $B_2^1+ B_2^2=B_2$;
  \item $u_2$ and $v_1$ are provincial (i.e., they have no Reeb chords at $e\infty$); and
  \item $\ev_{B_1^1}(u_1)$ and $\ev_{B_2^2}(v_2)$ agree up to overall
    translation (i.e., as elements of $\RR^{m}/\RR$).
  \end{enumerate}

  The moduli space of \emph{embedded simplified cross-matched polygons}
  \[\SimpCrossMatched^{B_1\glue B_2}(\x_1\conn \y_1,\x_2\conn \y_2)\] is the
  union of the spaces 
  \[
  \SimpCrossMatched^{B_1\glue B_2}(\x_1\conn \y_1,\x_2\conn \y_2;\rho_1,\dots,\rho_m;S_1,T_1,S_2,T_2),
  \]
  over all sequences of Reeb chords and sources $S_1$, $T_1$, $S_2$,
  $T_2$ with 
  \[
  \chi(S_1\glue T_1\glue S_2\glue T_2)=\chi_{\emb}(B_1\glue B_2).
  \]
  Again, in this union, we also include the matched polygons (from
  Definition~\ref{def:MatchedPolygons}) with no Reeb chords,
  corresponding to the degenerate case $m=0$.
\end{definition}

When comparing Definition~\ref{def:SimplifiedCrossMatchedPolygons}
with Definition~\ref{def:CrossMatchedPolygons}, the reader should be
aware of two key differences: the Heegaard diagrams appearing in
Definition~\ref{def:SimplifiedCrossMatchedPolygons} do not involve
approximations to the $\betas^i$ or $\gammas^j$, whereas those in
Definition~\ref{def:CrossMatchedPolygons} do; and, in
Definition~\ref{def:SimplifiedCrossMatchedPolygons}, the conformal
moduli of the polygons are unconstrained whereas in
Definition~\ref{def:CrossMatchedPolygons}, there is a restriction
coming from Property~\ref{crM:Moduli}.

\begin{lemma}\label{lem:trans-simp-cross-matched}
  With respect to a generic family of almost-complex structures, the
  moduli spaces of simplified cross-matched holomorphic polygons are
  transversally cut out. Fix a generic $\{J_j\}$ and a simplified
  cross-matched polygon $(u_1*u_2,v_1*v_2)$, and suppose that $c$ of
  $\{u_1,u_2,v_1,v_2\}$ are $\RR$-invariant bigons (disjoint unions of
  trivial strips). Then, near $(u_1*u_2,v_1*v_2)$ the moduli space of
  cross-matched polygons has dimension
  \[
  \dimSXM(u_1*u_2,v_1*v_2)=\dim(u_1)+\dim(u_2)+\dim(v_1)+\dim(v_2)+c-\min\{m-1,0\},
  \]
  where $m$ denotes the number of Reeb chords in $u_1$ (or
  equivalently $v_2$).
\end{lemma}
\begin{proof}
  The proof of transversality is similar to the (omitted) proof of
  transversality in
  Lemma~\ref{lem:trans-cross-matched} (but easier, since we do not
  have to worry about matching the conformal structures).
  
  For the dimension counting, note that we have $(m-1)$ constraints
  coming from the Reeb chords (if $m$ is at least one), and recall
  that $\dim(u_1)$, say, differs from the actual dimension of the
  moduli space near $u_1$ by $1$ if $u_1$ is an $\RR$-invariant bigon.
\end{proof}

As in Definition~\ref{def:ApproximatingHomotopyClass}, looking at the
local multiplicities away from the isotopy region gives a map
\[
\phi_\epsilon\co \pi_2(\x_2, \eta^{\ij{n-1}<\ij{n}}, \dots, \eta^{\ij{1}<\ij{2}},\x_1)\to
\pi_2(\x_2,\eta^{i_{s_{n-k-1}}<i_{s_{n-k}}},\dots,\eta^{i_{s_1}<i_{s_2}},\x_1),
\]
where $k = R(i_1,\dots,i_{n})$ is the number of repeated entries in the sequence $i_1,\dots,i_n$.
As before, we will call
$B_\epsilon\in \phi_\epsilon^{-1}(B')$ an \emph{approximation to
  $B'$}.

The following is an analogue of Lemma~\ref{lem:PseudoHolomorphicRepresentative}.
\begin{proposition}\label{prop:XM-to-SXM}
  With notation as in
  Definition~\ref{def:SimplifiedCrossMatchedPolygons}, suppose that
  \begin{align*}
    B_{1,\epsilon}&\in \pi_2(\x_2, \eta^{\ij{n-1}<\ij{n}}, \dots, \eta^{\ij{1}<\ij{2}},\x_1)\\
    B_{2,\epsilon}&\in \pi_2(\y_2, \zeta^{\ij{n-1}<\ij{n}}, \dots, \zeta^{\ij{1}<\ij{2}}, \y_1)
  \end{align*}
  are approximations to $B_1$ and $B_2$. If
  $\CrossMatched^{B_{1,\epsilon}\glue B_{2,\epsilon}}(\x_1\conn \y_1,\x_2\conn \y_2)$ is non-empty for
  $\epsilon>0$ arbitrarily small then
  $\SimpCrossMatched^{B_1\glue B_2}(\x_1\conn \y_1;\x_2\conn \y_2)$
  is non-empty, as well.
\end{proposition}
\begin{proof}
  This follows from the same argument as
  Lemma~\ref{lem:PseudoHolomorphicRepresentative}, treating east
  $\infty$ as in the proof
  of~\cite[Proposition~\ref*{LOT1:prop:compactness}]{LOT1}.
\end{proof}

As mentioned earlier, we will never actually count the moduli space of simplified
cross-matched polygons. The moduli spaces we will count (and continue
to study in Section~\ref{sec:time-dilate}) are the moduli spaces of
chord-matched polygon pairs:
\begin{definition}
  \label{def:ChordMatchedPolygons}  
  Fix generators 
  \begin{align*}
    \x_1&\in\Gen(\alphas^1,\betas^{i})
  &\x_2&\in\Gen(\alphas^1,\betas^{i'}) &
  \y_1&\in\Gen(\alphas^2,\gammas^j)
  &\y_2&\in\Gen(\alphas^2,\gammas^{j'})
  \end{align*}
  so that $\x_1$ and
  $\y_1$ are a complementary pair and $\x_2$ and $\y_2$ are a
  complementary pair. Fix homology classes $B_1\in\pi_2(\x_2,
  \eta^{i_{{n_1}-1}<i_{n_1}}, \dots, \eta^{i_1<i_2}, \x_1)$ and
  $B_2\in\pi_2(\y_2, \zeta^{j_{{n_2}-1}<j_{n_2}}, \dots,\allowbreak \zeta^{j_1<j_2},
  \y_1)$, surfaces $S_1$ and $S_2$, and a sequence of Reeb chords
  $\rho_1,\dots,\rho_m$.  The {\em{moduli space of
      chord-matched polygon pairs in the homology classes $B_1$ and
      $B_2$ with sources $S_1$ and $S_2$}} is
  the fibered product
  \begin{align*}
  \ChordMatched^{B_1\glue B_2}&(\x_1\conn \y_1,\x_2\conn \y_2;\rho_1,\dots,\rho_m;S_1,S_2)\\
  &=\Mod^{B_1}(\x_2,\eta^{i_{{n_1}-1}<i_{n_1}},\dots,\eta^{i_2<i_1},\x_1;\rho_1,\dots,\rho_m;S_1)\\
  &\qquad\times_{\ev_{B_1}=\ev_{B_2}} \Mod^{B_2}(\y_2,\zeta^{j_{{n_2}-1}<j_{n_2}},\dots,\zeta^{j_2<j_1},\y_1;\rho_1,\dots,\rho_m;S_2).
  \end{align*}

  The moduli space of \emph{embedded chord-matched polygon pairs}
  $\ChordMatched^{B_1\glue B_2}(\x_1\conn \y_1;\x_2\conn \y_2)$ is the union of
  the spaces
  \[ 
  \ChordMatched^{B_1\glue B_2}(\x_1\conn \y_1,\x_2\conn \y_2;\rho_1,\dots,\rho_m;S_1,S_2)
  \]
  over all sequences of Reeb chords and sources $S_1$ and $S_2$ with 
  \begin{equation}\label{eq:chord-matched-emb-source}
  \chi(S_1)=\chi_{\emb}(B_1)\qquad\text{and}\qquad\chi(S_2)=\chi_\emb(B_2).
  \end{equation}
  In this union we also include the degenerate case where there are no
  Reeb chords (i.e., $m=0$) and one of $B_1$ or $B_2$ is the trivial
  homology class of bigons.
\end{definition}

Note that the two polygons in
Definition~\ref{def:ChordMatchedPolygons} typically have a different
number of sides, and their conformal structures are unconstrained. The
moduli space of embedded chord-matched polygon pairs has expected
dimension
\[
\ind_\emb(B_1\glue B_2)+n_1+n_2-3.
\]

\begin{proposition}
  \label{prop:SendEpsilonToZero}
  Fix a pair of domains $B_{1,\epsilon}$ and $B_{2,\epsilon}$ which are approximations to a pair of domains $B_1$ and $B_2$, so that the expected dimension of polygons in the glued homology class $B_1\glue B_2$ is $0$.  
  Then for $\epsilon$ sufficiently small, the moduli space of
  cross-matched polygons
  $\CrossMatched^{B_{1,\epsilon}\glue B_{2,\epsilon}}(\x_1\conn \y_1,\x_2\conn \y_2)$
  in the homology class $B_{1,\epsilon}\glue B_{2,\epsilon}$ is empty except in the following
  cases:
  \begin{itemize}
  \item (The non-degenerate case.) If there are $m>0$ Reeb chords,
    $i_{n_1+1}=\dots=i_{n}$, $j_{1}=\dots=j_{n_1}$, and
    $\phi_\epsilon(B_1^2)$ and $\phi_\epsilon(B_2^1)$ are trivial
    domains (all of whose multiplicities are $0$). (Here, $B_1^2$ and
    $B_2^1$ are from items~\ref{item:XM-1} and~\ref{item:XM-2},
    respectively.)
  \item (The degenerate case.) If there are no Reeb chords ($m=0$),
    and one of $B_1$ or $B_2$ is the
    trivial domain (all of whose multiplicities are $0$).
  \end{itemize}
  Further, 
  \[
  \sum_{\substack{\phi_\epsilon(B_{1,\epsilon})=B_1\\\phi_\epsilon(B_{2,\epsilon})=B_2}}
  \#\CrossMatched^{B_{1,\epsilon}\glue B_{2,\epsilon}}(\x_1\conn \y_1,\x_2\conn \y_2)=
  \#\ChordMatched^{B_1\glue B_2}(\x_1\conn \y_1,\x_2\conn \y_2).
  \]
\end{proposition}

\begin{proof}
  This proof can be thought of as a variant of the proof of
  Proposition~\ref{prop:GlueChainComplexes}. In the degenerate case of
  curves with no Reeb chords ($m=0$), the result in fact follows from
  the proof of Proposition~\ref{prop:GlueChainComplexes}, so we will
  now restrict attention to the case of $m>0$ Reeb chords.

  Let $\{(u_1^k*u_2^k,v_1^k*v_2^k)\}_{k=1}^{\infty}$ be a sequence of
  cross-matched polygons in the homology classes $B_{1,\epsilon_k}^1$, $B_{1,\epsilon_k}^2$,
  $B_{2,\epsilon_k}^1$ and $B_{2,\epsilon_k}^2$ with
  \[
  \dimXM(u_1^k*u_2^k,v_1^k*v_2^k)=0
  \]
  and perturbation parameters $\epsilon_k$ converging to~$0$. Here,
  $B_{1,\epsilon_k}^1$ is an approximation (in the sense of
  Definition~\ref{def:ApproximatingHomotopyClass}) to some homology
  class $B_1^1$, and similarly for $B_{1,\epsilon_k}^2$,
  $B_{2,\epsilon_k}^1$ and $B_{2,\epsilon_k}^2$.  Restricting attention to
  a subsequence, we can assume that $u_1^k$ and $v_1^k$ are
  $(n_1+1)$-gons and $u_2^k$ and $v_2^k$ are $(n_2+1)$-gons, for $n_1$
  and $n_2$ independent of~$k$. Proposition~\ref{prop:XM-to-SXM} then
  guarantees that
  $\SimpCrossMatched^{B_1\glue B_2}(\x_1\conn \y_1,\x_2\conn \y_2)$
  is non-empty. Let $(u_1'*u_2',v_1'*v_2')\in
  \SimpCrossMatched^{B_1\glue B_2}(\x_1\conn \y_1,\x_2\conn \y_2)$. 

  Observe that
  Lemma~\ref{lem:ApproximateDimension} holds even in the present case
  where one of the sets of attaching curves---the
  $\alphas$'s---contains some arcs, giving:
  \begin{align*}
    \dim(u_1^k)-\dim(u_1')&= R(i_1,\dots,i_{n_1})\\
    \dim(u_2^k)-\dim(u_2')&= R(i_{n_1},\dots,i_{n})\\
    \dim(v_1^k)-\dim(v_1')&= R(j_1,\dots,j_{n_1})\\
    \dim(v_2^k)-\dim(v_2')&= R(j_{n_1},\dots,j_n).
  \end{align*}
  (The notation $\dim$ has a slightly different meaning in
  this section from in Lemma~\ref{lem:ApproximateDimension}: here, an
  $\RR$-invariant $u$ has $\dim(u)=-1$. Also, in the present setting,
  Case~(\ref{case:AD-1gons}) of Lemma~\ref{lem:ApproximateDimension}
  does not occur.)
  It is easy to see that
  \begin{align*}
    R(i_1,\dots,i_{n_1})+R(j_1,\dots,j_{n_1})
    &= n_1-1 \\
    R(i_{n_1},\dots,i_{n})+R(j_{n_1},\dots,j_{n})
    &= n_2-1.
  \end{align*}
  We conclude that
  \begin{align*}
    0\leq \dimSXM(u_1'*u_2',v_1'*v_2') &=
    \dim(u_1')+\dim(u_2')+\dim(v_1')+\dim(v_2')+c'-m+1 \\
    &= \dim(u_1^k)+\dim(u_2^k)+\dim(v_1^k)+\dim(v_2^k)+2\\
    &\qquad -n_1-n_2+c'-m+1\\
    &=\dimXM(u_1^k*u_2^k,v_1^k*v_2^k)-2+(c'-c)\\
    &= (c'-c)-2,
  \end{align*}
  where $c'$ is the number of $\{u_1',u_2',v_1',v_2'\}$ which are
  $\RR$-invariant bigons, and $c$ is the number of
  $\{u_1^k,u_2^k,v_1^k,v_2^k\}$ which are $\RR$-invariant
  bigons. Since there is at least one Reeb chord, $u_1^k$, $u_1'$,
  $v_2^k $ and $v_2'$ are not $\RR$-invariant bigons, so $c'\leq
  2$. Moreover, if $c'=2$ then $i_{n_1+1}=\dots=i_{n}$ and
  $j_{1}=\dots=j_{n_1}$. This proves the first half of the statement.

  For the second half, Lemma~\ref{lem:FiberedProductDegreeOne} says
  that the forgetful maps $\kappa_{B_1^2}$ and $\kappa_{B_2^1}$ are
  both degree $1$. So, dropping Condition~\ref{crM:Moduli} of
  Definition~\ref{def:CrossMatchedPolygons} and forgetting the
  components $u_{2,\epsilon}$ and $v_{1,\epsilon}$ does not change the
  holomorphic curve counts. Since $i_{n_1+1}=\dots=i_{n}$ and
  $j_{1}=\dots=j_{n_1}$, we have $i_{1}<\dots<i_{n_1}$ and
  $j_{n_1}<\dots<j_n$, and so the remaining curves $(u_1^k,v_2^k)$ can
  be viewed as elements of
  $\ChordMatched^{B_1\glue B_2}(\x_1\conn \y_1,\x_2\conn \y_2)$. This
  completes the proof.
\end{proof}

The moduli space of chord-matched polygon pairs can be used to
construct a chain complex. Specifically, define
$\ChordMatchedComplex=\{\ChordMatchedComplex^{i\times
  j}\}_{i\times j\in \IndI\times \IndJ}$ to be $\ChordMatchedComplex^{i\times
  j}=\CFa(\alphas^1\cup\alphas^2,\betas^i\cup\gammas^{j},z)$ with differential
\[
D^{i\times j<i'\times j'}(\x_1\conn  \y_1) =\sum_{\substack{\x_2\conn 
  \y_2\in\Gen(\alphas^1\cup\alphas^2,\betas^{i'\times
    j'}\cup\gammas^{i'\times j'})\\[0.5pt]
  i = i_1 < i_2 < \dots< i_{n_1} = i'\\[0.5pt]
   j = j_1 < j_2 < \dots < j_{n_2} = j'\\[1pt]
(B_1,B_2)\text{ s.t.}\ind_\emb(B_1\glue B_2)=3-n_1-n_2}}
\#\ChordMatched^{B_1\glue B_2}(\x_1\conn \y_1,\x_2\conn \y_2) \x_2\conn  \y_2.
\]

One could verify directly that $\partial^2=0$ on $\ChordMatchedComplex$, but this also follows
from Proposition~\ref{prop:SendEpsilonToZero}.

\subsection{Time dilation for chord-matched polygon pairs}\label{sec:time-dilate}

The ``time dilation''
argument from~\cite[Chapter~\ref*{LOT1:chap:tensor-prod}]{LOT1} identifies
the chain complex $\ChordMatchedComplex$ with the chain complex for
the tensor product. Before giving it, we introduce one more piece of
terminology. We have avoided describing the compactifications of the
moduli spaces of polygons, via \emph{holomorphic polygonal combs},
because of the cumbersome notation. But we will need one special case
of these objects:
\begin{definition}
  Fix a sequence of sets of attaching circles
  $\betas^0,\dots,\betas^n$, a subsequence $i_0=0<i_1<\dots<i_m=n$,
  generators $\x_j\in\Gen(\alphas,\betas^{i_j})$ ($j=0,\dots,m$) and
  $\eta^j\in\Gen(\betas^j,\betas^{j+1})$ ($j=0,\dots,n-1$), and
  homology classes $B_j\in
  \pi_2(\x_{j+1},\eta^{i_{j+1}},\dots,\eta^{i_j},\x_j)$.  A
  \emph{spinal holomorphic polygonal comb} is a sequence of
  holomorphic polygons (as in Definition~\ref{def:emb-moduli-sp})
  $(u_1,u_2,\dots,u_m)$ where $u_j\in\cM^{B_j}$. We say that this
  polygonal comb has $m$ stories and represents the homology class
  $B_1+\cdots+B_m\in\pi_2(\x_m,\eta^{n-1},\dots,\eta^0,\x_0)$; and
  $u_i$ is the $i\th$ story of the comb.
\end{definition}
There is a trivial spinal holomorphic polygonal comb, with $m=0$
stories. Holomorphic polygons can be viewed as $1$-story spinal
holomorphic polygonal combs.

The following definition is a generalization
of~\cite[Definition~\ref*{LOT1:def:TrimmedSimpleIdealMatchedCurve}]{LOT1}
to polygons:

\begin{definition}
  \label{def:tsicPolygons}
  A {\em trimmed simple ideal-matched polygon pair} connecting
  complementary pairs of generators $\x_1\conn \y_1$ and $\x_2\conn \y_2$ in the
  homology classes
  \[ B_1\in\pi_2(\x_2, \eta^{i_{n-1}<i_{n}}, \dots, \eta^{i_1<i_2}, \x_1)
  ~\text{and}~
  B_2\in\pi_2(\y_2, \eta^{j_{m-1}<j_{m}}, \dots, \zeta^{j_1<j_2},
  \y_1)\] is a pair of spinal holomorphic polygonal combs $w_1$ and $w_2$ where
  \begin{enumerate}[label=(T-\Alph*),ref=T-\Alph*,leftmargin=*]
  \item\label{TSIC:provincial} One of $w_1$ or $w_2$ is trivial and the other is a
    rigid (i.e., index $3-c$ where $c$ is the number of corners)
    holomorphic polygon with no $e$ punctures or
  \item\label{TSIC:cosmopolitan} $(w_1,w_2)$ has the following properties:
  \begin{enumerate}[label=\hspace{1em}(T-B\arabic*),ref=T-B\arabic*,leftmargin=*]
  \item The comb $w_1$ is a (one story) holomorphic curve representing $B_1$ which
    is asymptotic to the non-trivial sequence of non-empty sets of
    Reeb chords ${\vec\rhos}=(\rhos_1,\dots,\rhos_q)$.
    \item The curve $w_1$ is rigid (with respect to $\vec{\rhos}$).
    \item The comb $w_2$ is a $q$-story spinal holomorphic polygonal building 
      representing the homology class $B_2$.
    \item Each story of $w_2$ is rigid.
    \item Each of $w_1$ and $w_2$ is strongly boundary monotone.
    \item For each $i=1,\dots,q$,
      the east punctures of the $i\th$ story of $w_2$ are labeled,
      in order, by a non-empty sequence of Reeb chords
      $(-\rho^i_1,\dots,-\rho^i_{\ell})$ which have the property that
      the sequence of singleton sets of chords
      $\vec{\rho}^i=(\{\rho^i_1\},\dots,\{\rho^i_\ell\})$ is composable.
    \item The composition of the sequence of singleton sets of Reeb
      chords on the $i\th$ story of $w_2$ (with reversed
      orientation) coincides with the $i\th$ set of Reeb chords
      $\rhos_i$ appearing on the boundary of $w_1$.
    \end{enumerate}
    \end{enumerate}
\end{definition}

We can define a chain complex $\tsicComplex$,
which counts points in zero-dimensional moduli spaces
of trimmed simple ideal-matched polygon pairs. Rather than proving
directly that this does in fact define a chain complex, we identify it (up to
homotopy equivalence) with the chain complex $\ChordMatchedComplex$:

\begin{proposition}
  \label{prop:DilateToInfinity}
  The chain complex whose differential counts chord-matched polygon pairs
  $\ChordMatchedComplex$ is homotopy equivalent to the complex whose
  differential counts trimmed simple ideal-matched polygon pairs $\tsicComplex$.
\end{proposition}

\begin{proof}
  The proof of this proposition follows most of the time dilation
  proof of the pairing theorem for
  $\HFa$~\cite[Chapter~\ref*{LOT1:chap:tensor-prod}]{LOT1}.  In words,
  we consider yet another chain complex which counts chord-matched
  polygon pairs where now the matching condition on the chords is
  further perturbed by scaling out by a parameter $T$;
  i.e.,
  $T\ev_{B_1}(u)=\ev_{B_2}(v)$.  There are chain homotopy equivalences between
  these complexes as we vary the parameter
  $T$. (See~\cite[Proposition~\ref*{LOT1:prop:VaryTChain}]{LOT1}.)
  The novelty in the present application of this argument is that now,
  there can be canceling ends in the moduli spaces of polygons which
  correspond to polygons connecting the various $\betas^i$; but this
  came up already in the proof of
  Lemma~\ref{lem:VaryTranslationParameter} above.

  Next, we make the parameter $T$ very large. For large $T$, counts in
  the moduli spaces of $T$-matched polygons stabilize to counts of trimmed
  simple ideal-matched polygon pairs, according to the argument
  from~\cite[Proposition~\ref*{LOT1:prop:LargeTLimit}]{LOT1}.
\end{proof}

Now we put together the above steps to prove
Theorem~\ref{thm:PolygonPairing}:

\begin{proof}[Proof of Theorem~\ref{thm:PolygonPairing}]
  We identify 
  \[\CCFa(\Sigma_1\cup\Sigma_2,\alphas^1\cup\alphas^2,
  (\{\betas^i\}_{i\in \IndI},\{\eta^{i<i'}\}_{i,i'\in \IndI})
  \conn (\{\gammas^{j}\}_{j\in \IndJ},\{\zeta^{j<j'}\}_{j,j'\in \IndJ}),
  z)\] up to homotopy equivalence with the chain complex counting
  trimmed simple ideal-matched polygon pairs $\tsicComplex$, by
  applying, in turn, Proposition~\ref{prop:TakeNeckToInfinity},
  Lemma~\ref{lem:VaryTranslationParameter},
  Lemma~\ref{lem:TransToInfinity}, and
  Proposition~\ref{prop:SendEpsilonToZero}.  The differential in
  $\tsicComplex$ is identified with the differential in the tensor
  product complex
  \[\CCFAa(\Sigma_1,\alphas_1,\{\betas^i\},z)
  \DT \CCFDa(\Sigma_2,\alphas_2,\{\gammas^j\},z).\] This
  follows from the expression for $D^{i\times j<i'\times j'}$ from
  Equation~\eqref{eq:DefineDifferentialOnDT}: the string of operations
  on the type $D$ side counts $k$-story spinal holomorphic polygonal combs
  in $\Sigma_2$ and the node $F^{i\leq i'}$ pairs these with
  corresponding holomorphic polygons in $\Sigma_1$ (with the
  understanding that the terms in $\delta_{j}$ and $F^{i\leq i'}$ when
  $i=i'$ count bigons).
\end{proof}

\subsection{On boundedness}\label{sec:on-boundedness}

Suppose that both $(\Sigma,\alphas^1,\{\betas^i\}_{i\in\IndI},z)$ and
$(\Sigma,\alphas^2,\{\gammas^i\}_{i\in\IndI},z)$ are provincially
admissible, but neither is admissible.  Then, a
variant of Theorem~\ref{thm:PolygonPairing} remains true:

\begin{theorem}
  \label{thm:PolygonPairingDT}
  Let $(\Sigma_1,\alphas^1,z)$ and $(\Sigma_2,\alphas^2,z)$ be
  surfaces-with-boundary, each equipped with complete sets of
  bordered attaching curves $\alphas^1$ and $\alphas^2$ and basepoints
  $z\in\bdy\Sigma_i$ with $\partial(\Sigma_1,\alphas^1)=\PMC$ and
  $\partial(\Sigma_2,\alphas^2)=-\PMC$.  Let
  \[(\Sigma_1,\IndI,\{\betas^i\}_{i\in \IndI},\{\eta^{i<i'}\}_{i,i'\in \IndI})
  ~\text{and}~(\Sigma_2,\IndJ,\{\gammas^{j}\}_{j\in \IndJ},
  \{\zeta^{j<j'}\}_{j,j'\in \IndJ})\] be chain complexes
  of attaching circles in $\Sigma_1$ and $\Sigma_2$. respectively.
  Suppose that both of the multi-diagrams 
  $(\Sigma_1,\alphas^1,\{\betas^i\}_{i\in\IndI},z)$ and
  $(\Sigma_2,\alphas^2,\{\gammas^j\}_{j\in\IndJ},z)$ 
  are provincially admissible.
  Let $M$ be a module with
  \[ M \simeq
  \CCFAa(\Sigma_1,\alphas^1,\{\betas^i\},\{\eta^{i<i'}\},z) \]
  and such that $M$ is bounded.
  Then, there is an $\IndI\times \IndJ$-filtered quasi-isomorphism
  \begin{multline*}
    M\DT
   \CCFDa(\Sigma_2,\alphas^2,\{\gammas^j\},\{\zeta^{j<j'}\},z) \\
  \simeq \CCFa\bigl(\Sigma_1\cup\Sigma_2,\alphas,
  (\{\betas^i\},\{\eta^{i<i'}\})
  \conn (\{\gammas^{j}\},\{\zeta^{j<j'}\}),
  z\bigr),
  \end{multline*}
  where $\alphas$ is an isotopic copy of $\alphas^1\cup\alphas^2$,
  chosen so that the second multi-diagram is admissible.
\end{theorem}

\begin{proof}
  There is an isotopic copy $\xis^2$ of $\alphas^2$ so that
  $(\Sigma,\xis^2,\{\gammas^j\}_{j\in\IndJ},z)$ is admissible
  (see~\cite[Proposition~\ref*{LOT1:prop:admis-achieve-maintain}]{LOT1}).
  The isotopy induces a filtered chain homotopy equivalence
  \[
  \CCFDa(\Sigma_2,\alphas^2,\{\gammas^j\},\{\zeta^{j<j'}\},z)\simeq 
  \CCFDa(\Sigma_2,\xis^2,\{\gammas^j\},\{\zeta^{j<j'}\},z),
  \]
  with $\CCFDa(\Sigma_2,\xis^2,\{\gammas^j\},\{\zeta^{j<j'}\},z)$
  bounded.
  This in turn gives the first of the following homotopy equivalences:
  \begin{align*}
    M\DT 
    \CCFDa(&\Sigma_2,\alphas^2,\{\gammas^j\},\{\zeta^{j<j'}\},z) \\
    & \simeq 
    M\DT\CCFDa(\Sigma_2,\xis^2,\{\gammas^j\},\{\zeta^{j<j'}\},z)\\
    & \simeq 
    \CCFAa(\Sigma_1,\alphas^1,\{\betas^i\},\{\eta^{i<i'}\},z)
    \DT\CCFDa(\Sigma_2,\xis^2,\{\gammas^j\},\{\zeta^{j<j'}\},z)\\
    & \simeq \CCFa\bigl(\Sigma_1\cup\Sigma_2,\alphas^1\cup\xis^2,
    (\{\betas^i\},\{\eta^{i<i'}\})
    \conn (\{\gammas^{j}\},\{\zeta^{j<j'}\}),
    z\bigr) \\
    & \simeq \CCFa\bigl(\Sigma_1\cup\Sigma_2,\alphas,
    (\{\betas^i\},\{\eta^{i<i'}\})
    \conn (\{\gammas^{j}\},\{\zeta^{j<j'}\}),
    z\bigr).
  \end{align*}
  The third homotopy equivalence above is
  Theorem~\ref{thm:PolygonPairing}, and the fourth is induced by the
  isotopies from $\alphas^1\cup\xis^2$ to $\alphas$
  (Proposition~\ref{prop:IsotopiesInduceEquivalences}).
\end{proof}

The bounded models $M$ needed in Theorem~\ref{thm:PolygonPairingDT}
can be constructed either geometrically, by isotoping the $\alphas$
curves (as in the above proof) or by more algebraic
considerations. One more algebraic approach is to form
$\CCFAa(\Sigma_1,\alphas,\{\betas^i\},\{\eta^{i<i'}\},z)\DT
\lsup{\Alg}\Barop^{\Alg}\DT \lsub{\Alg}\Alg_{\Alg}$
(see~\cite[Section~\ref*{LOT2:sec:bar-cobar-modules}]{LOT2}). Another
more finite construction is to use a combinatorially describable,
bounded model for $\CFDAa(\Id)$, for example
using~\cite[Proposition~\ref*{LOT2:prop:DDAA-duality}]{LOT2}.

\subsection{The pairing theorem for bimodules}\label{sec:pairing-bimodules}
We turn next to the bimodule analogue of Theorem~\ref{thm:PolygonPairing}:

\begin{theorem}
  \label{thm:PolygonPairingDA}
  Let $(\Sigma_1,\alphas^1,\arcz)$ (respectively
  $(\Sigma_2,\alphas^2,\arcz)$) be a surface with two boundary
  components, equipped with a complete sets of bordered attaching curves
  $\alphas^1$ (respectively $\alphas^2$), compatible with $\PMC_0$ and
  $\PMC$ (respectively $-\PMC$ and $\PMC_2$), for some $\PMC_0$,
  $\PMC$ and $\PMC_2$.  Let 
  \[ (\Sigma_1,\IndI,\{\betas^i\}_{i\in
    \IndI},\{\eta^{i_1<i_2}\}_{i_1,i_2\in \IndI})~\text{and}~
  (\Sigma_2,\IndJ,\{\gammas^{j}\}_{j\in \IndJ},
  \{\zeta^{j_1<j_2}\}_{j_1,j_2\in \IndJ}) \]
  be chain complexes of
  attaching circles in $\Sigma_1$ and $\Sigma_2$ respectively.
  Suppose that both
  $\HD_1=(\Sigma_1,\alphas,\{\betas^i\}_{i\in\IndI},\arcz)$ and
  $\HD_2=(\Sigma_2,\alphas,\{\gammas^j\}_{j\in\IndJ},\arcz)$ are
  provincially admissible, and either $\HD_1$ is right-admissible or
  $\HD_2$ is left-admissible.  Then, there is an $\IndI\times
  \IndJ$-filtered quasi-isomorphism
  \begin{multline*}
   \lsup{\Alg_L}\CCFDAa(\Sigma_1,\alphas^1,\{\betas^i\},\{\eta^{i_1<i_2}\},\arcz)\DT_{\Alg(\PMC)}
   \CCFDAa(\Sigma_2,\alphas^2,\{\gammas^j\},\{\zeta^{j_1<j_2}\},\arcz)_{\Alg_R}\\
  \simeq \lsup{\Alg_L}\CCFDAa(\Sigma_1\cup\Sigma_2,\alphas^1\cup\alphas^2,
  (\{\betas^i\},\{\eta^{i_1<i_2}\})
  \conn (\{\gammas^{j}\},\{\zeta^{j_1<j_2}\}),
  \arcz)_{\Alg_R},
  \end{multline*}
  where, $\Alg_L=\Alg(-\partial_L\Sigma_1)$ and 
  $\Alg_R=\Alg(\partial_R\Sigma_2)$.
\end{theorem}

\begin{proof}
  The proof is similar to the proof of
  Theorem~\ref{thm:PolygonPairing}. We review this proof in outline,
  indicating the changes necessary for the bimodule case:
  \begin{enumerate}
  \item In the same vein as Proposition~\ref{prop:GlueAdmissibility}, 
    the admissibility conditions guarantee that the glued diagram is admissible.
    (See~\cite[Lemma~\ref*{LOT2:lem:admiss-glues}]{LOT2}.)
  \item Glue $\Sigma_1$ and $\Sigma_2$, identifying $\partial_R
    \Sigma_1$ and $\partial_L\Sigma_2$.  The straightforward
    generalization of Proposition~\ref{prop:TakeNeckToInfinity}
    identifies holomorphic curves in $\Sigma_1\cup\Sigma_2$ with pairs
    of matched polygons $u$ and $v$ in $\Sigma_1$ and $\Sigma_2$,
    respectively. In the bimodule case, the chords along $\partial_R
    u$ are matched with chords in $\partial_L v$, and some chords in
    $\bdy_R v$ are constrained to lie at the same heights. The
    analogue of $\MatchedComplex$ is a filtered type $\DA$ bimodule.
  \item Consider the analogue of $t$-matched curves, where now the heights
    of the chords in $u$ which map to $\partial_R \Sigma_1$ are
    translated as compared with the heights of chords in $v$
    which map to $\partial_L \Sigma_2$. Counting these curves gives the $\DA$ bimodule
    operations on the generalization of $\tMatchedComplex$.
    The filtered $\DA$ quasi-isomorphism type is independent of $t$.
  \item Send $t$ to infinity as before.  The appropriate
    generalizations of cross-matched polygons which appear in the $t\to
    \infty$ limit are pairs of polygons $u_1*u_2$ and $v_1*v_2$ with
    the following properties:
    \begin{enumerate}
    \item The conformal moduli of the polygons underlying $u_1$ and
      $v_1$ are matched, as are the conformal moduli of $u_2$ and $v_2$.
    \item Relative heights of the chords in $u_1$ along
      $\partial_R\Sigma_1$ are matched with relative heights of the
      chords in $v_2$ along $\partial_L \Sigma_2$.  Rather than
      being required to be provincial, the polygons $u_2$ and $v_1$
      are required to be disjoint from $\partial_R \Sigma_1$ and
      $\partial_L \Sigma_2$. 
    \item Heights of chords in $v_1*v_2$ appearing along $\partial_R\Sigma_2$
      satisfy the constraints dictated by the action of $\Alg(\PMC_R)$.
    \end{enumerate}
  \item Make the approximation parameter $\epsilon$ sufficiently
    small. According to the analogue of
    Proposition~\ref{prop:SendEpsilonToZero}, cross-matched polygons now
    correspond to chord-matched polygon pairs, analogous to
    Definition~\ref{def:ChordMatchedPolygons}, where once again the
    chord matching occurs along the interface between
    $\Sigma_1$ and $\Sigma_2$.
  \item Dilate time along the interface between $\Sigma_1$ and $\Sigma_2$, so chord-matched polygon
    pairs converge to trimmed simple ideal-matched polygonal pairs.
  \item The resulting curve counts are identified now
    with the tensor product of the two filtered $\DA$ bimodules.
  \end{enumerate}
  With these modifications, the proof is now complete.
\end{proof}    

We have the following variant of Theorem~\ref{thm:PolygonPairingDA}
with weaker admissibility hypotheses:

\begin{theorem}
  \label{thm:PolygonPairingDA-DTP}
  Let $(\Sigma_1,\alphas^1,\arcz)$ (respectively
  $(\Sigma_2,\alphas^2,\arcz)$) be a surface with two boundary
  components, equipped with a complete set of bordered attaching curves
  $\alphas^1$ (respectively $\alphas^2$) compatible with $\PMC_0$ and
  $\PMC$ (respectively $-\PMC$ and $\PMC_2$). Let
  \[
  (\Sigma_1,\IndI,\{\betas^i\}_{i\in \IndI},\{\eta^{i<i'}\}_{i,i'\in
    \IndI}) ~\text{and}~(\Sigma_2,\IndJ,\{\gammas^{j}\}_{j\in \IndJ},
  \{\zeta^{j<j'}\}_{j,j'\in \IndJ})
  \] 
  be chain complexes of attaching circles in $\Sigma_1$ and
  $\Sigma_2$, respectively.  Suppose that both of the multi-diagrams
  $(\Sigma_1,\alphas,\{\betas^i\}_{i\in\IndI},\arcz)$
  $(\Sigma_2,\alphas,\{\gammas^j\}_{j\in\IndJ},\arcz)$ are provincially
  admissible.  Let $M$ be a module with
  \[ 
  M \simeq \CCFDAa(\Sigma_1,\alphas^1,\{\betas^i\},\{\eta^{i<i'}\},\arcz)
  \]
  and such that $M$ is bounded.
  Then, there is an $\IndI\times \IndJ$-filtered quasi-isomorphism
  \begin{multline*}
    M\DT
   \CCFDAa(\Sigma_2,\alphas^2,\{\gammas^j\},\{\zeta^{j<j'}\},\arcz) \\
  \simeq \CCFDAa\bigl(\Sigma_1\cup\Sigma_2,\alphas,
  (\{\betas^i\},\{\eta^{i<i'}\})
  \conn (\{\gammas^{j}\},\{\zeta^{j<j'}\}),
  \arcz\bigr),
  \end{multline*}
  where $\alphas$ is an isotopic copy of $\alphas^1\cup\alphas^2$,
  chosen so that the second multi-diagram is admissible.
\end{theorem}

\begin{proof}
  This follows from Theorem~\ref{thm:PolygonPairingDA} exactly as
  Theorem~\ref{thm:PolygonPairingDT} follows from
  Theorem~\ref{thm:PolygonPairing}.
\end{proof}

\subsection{The pairing theorem for triangles}\label{sec:pairing-triangles}

There is a somewhat simpler statement of the pairing theorem for triangles.

Let $\HD_L=(\Sigma_L,\alphas_L,\betas_L,\arcz_L)$ be a bordered Heegaard
diagram with two boundary components, and let
$\HD_R=(\Sigma_R,\alphas_R,\betas_R^0,\betas_R^1,\arcz_R) $ be a
bordered Heegaard triple with two boundary components, with $\partial_R\HD_L=\partial_L\HD_R$.
Choose a cycle $\eta^{0<1}\in\CFa(\betas_R^0,\betas_R^1,\arcz_R)$.
Let $\betas_L'$ be an approximation to $\betas_L$, and let
$\Theta\in\CFa(\betas_L,\betas_L',\arcz_L)$ be a cycle generating
the top-dimensional homology group. 
Abbreviate
\begin{align*}
  \HD_L'&=(\Sigma_L,\alphas_L,\betas_L',\arcz_L) &
  \HD_R^0&=(\Sigma_R,\alphas_R,\betas_R^0,\arcz_R) &
  \HD_R^1&=(\Sigma_R,\alphas_R,\betas_R^1,\arcz_R).
\end{align*}
Assume that both of $\HD_L$ and $\HD_R$ are provincially admissible
and that either $\HD_L$ is right-admissible or $\HD_R$ is
left-admissible.

\begin{proposition}
  \label{prop:IdentifyTheMapsDA}
  There is a homotopy-commutative square
  \begin{equation}\label{eq:bimod-tri}
    \mathcenter{
      \begin{tikzpicture}
        \node at (0,0) (T1) {$\CFDAa(\HD_L\cup\HD^0_{R})$};
        \node at (8,0) (T2) {$\CFDAa(\HD_L'\cup\HD^1_{R})$};
        \node at (0,-2) (B1) {$\CFDAa(\HD_L)\DT\CFDAa(\HD^0_R)$};
        \node at (8,-2) (B2) {$\CFDAa(\HD_L')\DT\CFDAa(\HD^{1}_R)$,};
        \draw[->] (T1) to node[above]{\lab{m_2(\Theta\otimes \eta^{0<1},\cdot)}} (T2);
        \draw[->] (T1) to (B1);
        \draw[->] (B1) to node[above]{\lab{m_2(\Theta,\cdot)\DT m_2(\eta^{0<1},\cdot)}} (B2);
        \draw[->] (T2) to (B2);
      \end{tikzpicture}}
  \end{equation}
  where the vertical maps are induced by the pairing theorem for bigons 
  (\cite[Theorem~\ref*{LOT1:thm:TensorPairing}]{LOT1}). Analogous
  statements hold for pairing \AAm\ bimodules with \DA\ or \DD\ bimodules,
  \DA\ bimodules with \DD\ bimodules, and \DD\ bimodules with \AAm\
  bimodules; or for pairing bimodules with modules.
\end{proposition}

\begin{proof}
  This follows from the pairing theorem,
  Theorem~\ref{thm:PolygonPairingDA}.
  Consider the $1$-step chain complex $(\{\betas_L\},)$ in $\Sigma_L$
  and the $2$-step chain complex
  $(\{\betas_R^0,\betas_R^1\},\eta^{0<1})$ in
  $\Sigma_R$. Theorem~\ref{thm:PolygonPairingDA} gives a $\{0,1\}$-filtered homotopy equivalence
  \begin{multline*}
    \CFDAa(\alphas_L,\betas_L,\arcz_L)\DT
    \CCFDAa(\alphas_R,\{\betas_R^0,\betas_R^1\},\{\eta^{0<1}\},\arcz_R)\\
    \simeq
    \CCFDAa(\alphas_L\cup\alphas_R,\{\betas_L\cup\betas_R^0,\betas_L\cup\betas_R^1\},\{\Theta\otimes
    \epsilon_{0<1}\},\arcz_L\cup\arcz_R).
  \end{multline*}
  Unpacking the definitions, the right-hand side is exactly the top
  row of Diagram~\eqref{eq:bimod-tri}, and the left-hand side is the
  bottom row of Diagram~\eqref{eq:bimod-tri}. The homotopy equivalence
  furnishes the vertical arrows, as well as a diagonal arrow
  $\CFDAa(\HD_L\cup \HD_R^0)\to \CFDAa(\HD'_L)\DT\CFDAa(\HD^1_R)$,
  which is the homotopy in ``homotopy-commutative.''
\end{proof}

\subsection{The exact triangles agree}\label{sec:exact-tri}

Let $K$ be a framed knot in a $3$-manifold $Y$, and let $Y_r(K)$
denote $r$-surgery on $K$. In~\cite{OS04:HolDiskProperties}, an exact
triangle 
\begin{equation}\label{eq:surg-tri}
\mathcenter{\begin{tikzpicture}
  \node at (0,0) (infty) {$\HFa(Y_\infty(K))$};
  \node at (4,0) (minus) {$\HFa(Y_{-1}(K))$};
  \node at (2,-1.5) (zero) {$\HFa(Y_{0}(K))$};
  \draw[->] (infty) to (minus);
  \draw[->] (minus) to (zero);
  \draw[->] (zero) to (infty);
\end{tikzpicture}}
\end{equation}
was constructed. The maps in~\eqref{eq:surg-tri} were defined by
counting holomorphic triangles.

In~\cite[Chapter~\ref*{LOT1:chap:TorusBoundary}]{LOT1} we gave another
construction of an exact triangle of the form~\eqref{eq:surg-tri} by
explicitly writing down maps between three different framed solid tori
and invoking the pairing theorem. This was generalized slightly
in~\cite[Theorem~\ref*{DCov1:thm:Dehn-is-MC}]{LOT:DCov1} to the case
that $Y$ is a bordered $3$-manifold with two boundary components (and
$K$ lies in the interior of $Y$).

In fact, both constructions prove slightly more: they give
quasi-isomorphisms 
\begin{align}
  \CFa(Y_{-1}(K))&\simeq \Cone\bigl(\theta\co \CFa(Y_0(K))\to
  \CFa(Y_\infty(K))\bigr)\label{eq:surg-cone-closed}\\
  \shortintertext{or}
  \CFDDa(Y_{-1}(K))&\simeq \Cone\bigl(\theta\co \CFDDa(Y_0(K))\to
  \CFDDa(Y_\infty(K))\bigr)\label{eq:surg-cone-bord}
\end{align}
in the closed or bordered cases, respectively. Again, the map $\theta$
is defined in~\cite{OS04:HolDiskProperties} by counting holomorphic
triangles and in~\cite{LOT1,LOT:DCov1} by an explicit formula together
with the pairing theorem.

The goal of this section is to identify the two surgery exact
triangles. We will prove:

\begin{proposition}\label{prop:surg-tri-counts}
  The map $\theta$ in Formula~(\ref{eq:surg-cone-bord}), as defined
  in~\cite[Theorem~\ref*{DCov1:thm:Dehn-is-MC}]{LOT:DCov1}, is given
  up to homotopy by counting holomorphic triangles in a suitable
  bordered Heegaard triple-diagram. More precisely, with notation as
  in Sections~\ref{sec:const-bord-tri} and~\ref{sec:const-orig-tri},
  there is a homotopy-commutative square
  \[
  \begin{tikzpicture}
    \node at (0,0) (T1) {$\CFDDa(\HD\cup\HD_0)$};
    \node at (8,0) (T2) {$\CFDDa(\HD\cup \HD_\infty)$};
    \node at (0,-2) (B1) {$\CFDDAa(\HD)\DT\CFDa(\HD_0)$};
    \node at (8,-2) (B2) {$\CFDDAa(\HD)\DT\CFDa(\HD_\infty)$.};
    \draw[->] (T1) to node[above]{\lab{\theta_{OS}}} (T2);
    \draw[->] (T1) to node[left]{\lab{\simeq}} (B1);
    \draw[->] (B1) to node[below]{\lab{\theta_{LOT}}} (B2);
    \draw[->] (T2) to node[right]{\lab{\simeq}} (B2);
  \end{tikzpicture}
  \]
  where the vertical arrows are induced by the pairing theorem for
  bigons, the map $\theta_{OS}$ is defined by counting holomorphic
  triangles in a bordered Heegaard triple diagram, and the map
  $\theta_{LOT}=\Id\DT\theta$ with $\theta$ as described in
  Section~\ref{sec:const-bord-tri}.
  
  Analogous statements hold for type \DA\ and \AAm\ bimodules; for
  $D$ and $A$ modules in the one-boundary-component case; and for
  $\CFa$ in the closed case.
\end{proposition}

\begin{corollary}\label{cor:surg-tri-agree}
  The surgery exact triangle constructed
  in~\cite[Chapter~\ref*{LOT1:chap:TorusBoundary}]{LOT1} agrees with
  the original surgery exact triangle constructed
  in~\cite{OS04:HolDiskProperties}. More precisely,
  there is a homotopy-commutative square
  \[
  \begin{tikzpicture}
    \node at (0,0) (T1) {$\CFa(\HD\cup\HD_{0})$};
    \node at (8,0) (T2) {$\CFa(\HD\cup\HD_{\infty})$};
    \node at (0,-2) (B1) {$\CFAa(\HD)\DT\CFDa(\HD_0)$};
    \node at (8,-2) (B2) {$\CFAa(\HD)\DT\CFDa(\HD_{\infty})$,};
    \draw[->] (T1) to node[above]{\lab{m_2(\Theta_{\gamma,\beta},\cdot)}} (T2);
    \draw[->] (T1) to node[left]{\lab{\simeq}} (B1);
    \draw[->] (B1) to node[above]{\lab{\Id_{\CFAa(\HD)}\DT \theta}} (B2);
    \draw[->] (T2) to node[right]{\lab{\simeq}} (B2);
  \end{tikzpicture}
  \]  
  where the vertical maps are quasi-isomorphisms induced by the
  pairing theorem for bigons
  (\cite[Theorem~\ref*{LOT1:thm:TensorPairing}]{LOT1}),
  and $\theta$ is defined as in Equation~\eqref{eq:ComputeTheta}.
\end{corollary}
(This is also~\cite[Proposition~\ref*{DCov1:prop:triangles-agree}]{LOT:DCov1},
the proof of which was deferred to here.)

To prove Proposition~\ref{prop:surg-tri-counts} and
Corollary~\ref{cor:surg-tri-agree}, we start by recalling the two
constructions of the surgery exact triangle, and reformulating them in
the language of chain complexes of attaching circles. The result then
follows quickly from the pairing theorem.

\subsubsection{The bordered construction of the surgery triangle}\label{sec:const-bord-tri}
Let ${\mathcal T}$ be a punctured genus one surface. Fix three linear curves
$\beta^{\infty}$, $\beta^{-1}$, and $\beta^0$ in $T$ of slopes $\infty$,
$-1$, and $0$ respectively. Let $\xi=\beta^{\infty}\cap \beta^{-1}$
and $\eta=\beta^{-1}\cap\beta^0$.  Letting $\IndI=\{\infty,-1,0\}$, we
can view these curves and intersection points as determining an
$\IndI$-filtered chain complex of bordered attaching curves, with
$\eta^{\infty<-1}=\xi$, $\eta^{-1<0}=\eta$, and $\eta^{\infty<0}=0$. (This is in fact
Example~\ref{ex:AttachingCircles} with a little reparameterization. The reader
bothered by the ordering $\infty<-1$ should think of $\infty$ as $-\infty$.)

Choose $\alphas$ to consist of two curves going out to the
puncture in ${\mathcal T}$, one parallel to $\beta^{\infty}$ and
the other parallel to $\beta^{0}$. Place the basepoint $z$ and order the chords at the punctures as illustrated
in Figure~\ref{fig:ExactTriangle}.

\begin{figure}
  \centering
  \input{ExactTriangle}
  \caption{{\bf A bordered Heegaard multi-diagram.}
  \label{fig:ExactTriangle}}
\end{figure}
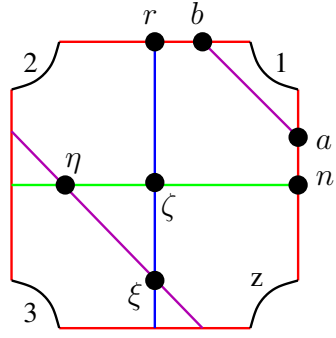

The associated $\IndI$-filtered type $D$ module
$\CCFDa(\alphas,\{\betas^i\}_{i\in\IndI},\{\eta^{i<j}\}_{i,j\in\IndI},z)$ has three summands:
\begin{align*}
\CFDa(\HD_\infty)&=\CFDa(\alphas,\betas^{\infty},z) &
\CFDa(\HD_{-1})&=\CFDa(\alphas,\betas^{-1},z) &
\CFDa(\HD_{0})&=\CFDa(\alphas,\betas^0,z).
\end{align*}
These have generating sets $\{r\}$, $\{a,b\}$, and $\{n\}$ in filtration levels $\infty$,
$-1$, and $0$ respectively, and differentials
\begin{align*}
 \delta^1 r &= \rho_{23}\otimes r  &
 \delta^1 a &= (\rho_{1}+\rho_{3})\otimes b \\
 \delta^1 b &= 0  &
 \delta^1 n &= \rho_{12}\otimes n.
\end{align*}
The maps changing filtration are:
\begin{align*}
  F^{\infty<-1}(r)&= (\rho_2\otimes a) + b &
  F^{-1<0}(a)&= n \\
  F^{-1<0}(b) &= \rho_2\otimes n &
  F^{\infty<0} &= 0.
\end{align*}
All these maps can be computed by counting holomorphic disks in the torus.
Note that $F^{-\infty<-1}$ and $F^{-1<0}$ are the maps denoted $\varphi$ and $\psi$
in~\cite[Section~\ref*{LOT1:sec:surg-exact-triangle}]{LOT1}.

We can also consider a third map $\theta=F^{0<\infty}$,  the map gotten by reordering $\{-1,0,\infty\}$,
i.e., counting holomorphic triangles based at the intersection point $\zeta=\beta^{\infty}\cap\beta^0$.
This map can be readily computed as
\begin{equation}
  \label{eq:ComputeTheta}
  \theta(n)=(\rho_1+\rho_3)\otimes r.
\end{equation}
There is an explicit isomorphism
\begin{equation}
  \label{eq:HminusOneIsMappingCone}
  \CFDa(\HD_{-1}) 
  \stackrel{\simeq}{\longrightarrow} 
  {\Cone}(\theta\co \CFDa(\HD_0)\to\CFDa(\HD_{\infty}))
\end{equation}
given by
\begin{align*}
  b & \mapsto r+\rho_{2}\otimes n \\
  a & \mapsto n.
\end{align*}
(Isomorphisms
\begin{align*}
\CFDa(\HD_0)&\simeq  \Cone(\phi\co \CFDa(\HD_{\infty})\to\CFDa(\HD_{-1}))  \\
  \CFDa(\HD_{\infty})&\simeq \Cone(\psi\co \CFDa(\HD_{-1})\to\CFDa(\HD_{0}))
\end{align*}
 are even easier to write down.)

 Now, suppose that $Y$ is a $3$-manifold with a framed knot
 $(K,\lambda)$ in it and
let $\HD$
be a bordered diagram for $Y\setminus\nbd(K)$. The identification
\[ 
{\Cone}(\theta\co \CFDa(\HD_0)\to\CFDa(\HD_{\infty}))\simeq
\CFDa(\HD_{-1})
\]
can be tensored with $\CFAa(\HD)$ to give a quasi-isomorphism
\begin{equation}\label{eq:CF-is-cone}
\CFAa(\HD)\DT\CFDa(\HD_{-1})
\simeq
{\Cone}\bigl(\Id_{\CFAa(\HD)}\DT \theta\co \CFAa(\HD)\DT
\CFDa(\HD_0)\to \CFAa(\HD)\DT \CFDa(\HD_{\infty})\bigr). 
\end{equation}
The pairing theorem for
bigons~\cite[Theorem~\ref*{LOT1:thm:TensorPairing}]{LOT1} gives
quasi-isomorphisms
\begin{equation}\label{eq:handlebodies}
  \begin{split}
  \CFAa(\HD)\DT\CFDa(\HD_\infty)&\simeq \CFa(Y) \\
  \CFAa(\HD)\DT\CFDa(\HD_{-1})&\simeq \CFa(Y_{-1}(K)) \\
  \CFAa(\HD)\DT\CFDa(\HD_0)&\simeq \CFa(Y_0(K)).
  \end{split}
\end{equation}
Together, Equations~\eqref{eq:CF-is-cone} and~\eqref{eq:handlebodies}
rise to a long exact sequence relating $\HFa(Y_0(K))$, $\HFa(Y)$, and
$\HFa(Y_{-1}(K))$ as in Equation~(\ref{eq:surg-tri}).

As noted in~\cite{LOT:DCov1}, this approach extends easily to give
surgery triangles for bordered Floer homology as well. Suppose $Y$ is
a bordered $3$-manifold with two boundary components and $(K,\lambda)$
is a framed knot in the interior of $Y$. Choose an arced bordered
Heegaard diagram $\HD=(\Sigma,\alphas,\betas,\arcz)$ with three
boundary components for $Y\setminus\nbd(K)$. (Here, ``arced'' means,
for instance, that there is a component of
$\Sigma\setminus(\alphas\cup\betas)$ adjacent to all three components
of $\bdy\Sigma$, and we place the basepoint $z$ in this region.) There
are bordered trimodules $\CFDDDa(\HD)$, $\CFDDAa(\HD)$ and so on
associated to $\HD$, one for each labeling of the boundary components
by elements of $\{D,A\}$; the definitions of these trimodules are
trivial adaptations of the definitions of the bimodules
in~\cite{LOT2}. Consider in particular the trimodules $\CFDDAa(\HD)$,
$\CFDAAa(\HD)$ and $\CFAAAa(\HD)$ where the boundary component
corresponding to $K$ is labeled by $A$. The pairing theorem gives
quasi-isomorphisms
\begin{equation}\label{eq:bord-handlebodies}
  \begin{split}
  \CFDDAa(\HD)\DT\CFDa(\HD_\infty)&\simeq \CFDDa(Y) \\
  \CFDDAa(\HD)\DT\CFDa(\HD_{-1})&\simeq \CFDDa(Y_{-1}(K)) \\
  \CFDDAa(\HD)\DT\CFDa(\HD_0)&\simeq \CFDDa(Y_0(K)),
  \end{split}
\end{equation}
and similarly with \DD\ replaced by \DA\ or \AAm.
Combining Equations~(\ref{eq:CF-is-cone}) (and its \DA\ and \AAm\ versions)
and Equation~(\ref{eq:bord-handlebodies}) gives
\begin{equation}\label{eq:bord-exact-tri}
  \begin{split}
    \CFDDa(Y_{-1}(K))&\simeq \Cone\bigl(\theta_{LOT}\co \CFDDa(Y_0(K))\to\CFDDa(Y_{\infty}(K))\bigr)\\
    \CFDAa(Y_{-1}(K))&\simeq \Cone\bigl(\theta_{LOT}\co \CFDAa(Y_0(K))\to\CFDAa(Y_{\infty}(K))\bigr)\\
    \CFAAa(Y_{-1}(K))&\simeq \Cone\bigl(\theta_{LOT}\co \CFAAa(Y_0(K))\to\CFAAa(Y_{\infty}(K))\bigr),
  \end{split}
\end{equation}
the most natural analogues of the exact triangle~(\ref{eq:surg-tri})
in this more complicated algebraic setting. In particular, the map $\theta=\theta_{LOT}$ from
Formula~(\ref{eq:surg-cone-bord}) is defined as:
\begin{multline*}
\theta_{LOT}=\Id\DT\theta\co
\CFDDa(Y_0(K))\simeq\CFDDAa(\HD)\DT\CFDa(\HD_0)\\
\to\CFDDAa(\HD)\DT\CFDa(\HD_{\infty})\simeq \CFDDa(Y_{\infty}(K)).
\end{multline*}

Of course, there are also surgery triangles for bordered
$3$-manifolds with connected boundary, via the same argument but with
trimodules replaced by bimodules.

\subsubsection{The original construction of the surgery triangle}\label{sec:const-orig-tri}
The original construction of the surgery exact triangle is somewhat
different.  Again, 
suppose that $Y$ is a three-manifold with a framed knot
$(K,\lambda)$ in it. There is a corresponding Heegaard triple as
in~\cite{OS04:HolDiskProperties} $(\Sigma,\alphas,\gammas,\betas,z)$
representing the two-handle cobordism from the surgery $Y_\lambda(K)$
to $Y$. In particular, $(\Sigma,\gammas,\betas,z)$ represents a
connected sum of $(S^2\times S^1)$'s.  Counting triangles with one input a cycle $\Theta$
representing the top-graded homology class in
$\HFa(\Sigma,\gammas,\betas,z)$ gives a map
\[ 
m_2(\cdot\otimes \Theta_{\gamma,\beta})\co
\CFa(\alphas,\gammas,z)\to \CFa(\alphas,\betas,z).
\]

In the language of chain complexes of attaching circles, 
take a bordered diagram
$\HD=(\Sigma_1,\alphas^1,\betas^1)$ for the complement of $K$ in $Y$.
Consider the Heegaard triple $({\mathcal
  T},\alphas,\{\beta^0,\beta^{\infty}\},z)$ described above, where
we think of $({\mathcal T},\{\beta^0,\beta^{\infty}\},z)$ as a
two-step chain complex of attaching circles, using the cycle
$\eta^{0<\infty}=\beta^0\cap\beta^{\infty}$.  We can glue this to
$\betas^1$ (in the sense of Definition~\ref{def:ConnectedSum}), which
we now think of as a one-step complex of attaching circles in $\Sigma_1$,
to get a two-step chain complex
\[(\Sigma^1,\{\betas^1\},z)\conn({\mathcal
  T},\{\beta^0,\beta^\infty\},z)
=(\Sigma,\{\gammas,\betas\},z),\]  equipped with 
distinguished chain
$\Theta_{\gamma,\beta}\in\CFa(\Sigma_1\conn \Sigma,\gammas,\betas,z)$.
Pairing this two-step complex with $\alphas^1\cup\alphas$ gives the
two-step complex which is the mapping cone of
\[
m_2(\cdot\otimes \Theta_{\gamma,\beta})\co \CFa(\Sigma_1\cup\Sigma,
\alphas^1\cup\alphas,\gammas,z)\to
\CFa(\Sigma_1\cup\Sigma,\alphas^1\cup\alphas,\betas,z).
\]

The proof of the surgery exact triangle
from~\cite{OS04:HolDiskProperties} gives a quasi-isomorphism
\[ 
\CFa(\HD\cup\HD_{-1})\simeq
{\Cone}(m_2(\cdot,\Theta_{\gamma,\beta})\co
\CFa(\Sigma,\alphas,\gammas,z)\to \CFa(\Sigma,\alphas,\betas,z)).
\]
The long exact sequence~(\ref{eq:surg-tri}) relating $\HFa(Y_0)$,
$\HFa(Y)$, and $\HFa(Y_{-1})$ follows at once.

\subsubsection{The two constructions agree}

\begin{proof}[Proof of Proposition~\ref{prop:surg-tri-counts}]
  It is immediate from the definitions that the Heegaard triple
  $(\Sigma,\alphas^1,\betas^1)\conn ({\mathcal
    T},\alphas,\{\beta^0,\beta^{\infty}\},z)$ agrees with the Heegaard
  triple $(\Sigma,\alphas,\gammas,\betas,z)$ associated to the
  two-handle cobordism from $Y_0$ to $Y$.  Further, the chain
  $\Theta_{\gamma,\beta}$ constructed in the gluing of chain complexes
  represents the top-graded generator of
  $\HFa(\gammas,\betas,z)$.
  So, the commutative square is a direct consequence of the pairing
  theorem for triangles, Proposition~\ref{prop:IdentifyTheMapsDA}.
\end{proof}

\begin{proof}[Proof of Corollary~\ref{cor:surg-tri-agree}]
  This is the special case of Proposition~\ref{prop:surg-tri-counts}
  in which both of the boundary components of $Y$ are empty, together
  with the observation that $\theta_{OS}$ is given by $m_2(\Theta_{\gamma,\beta},\cdot)$.
\end{proof}


%% file: draws/PairingSketch.tex
\begin{picture}(0,0)%
\includegraphics{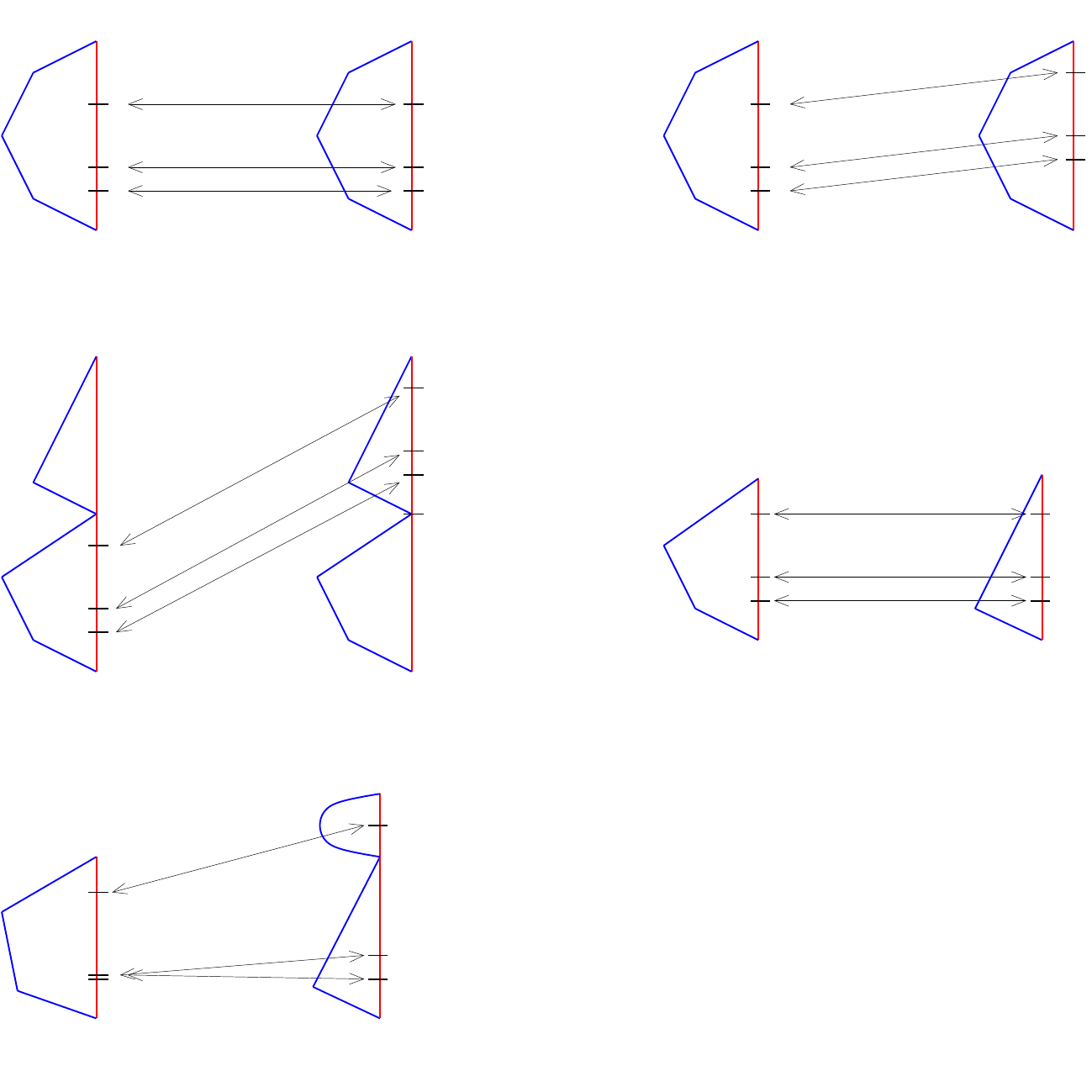}%
\end{picture}%
\setlength{\unitlength}{1184sp}%
\begingroup\makeatletter\ifx\SetFigFont\undefined%
\gdef\SetFigFont#1#2#3#4#5{%
  \reset@font\fontsize{#1}{#2pt}%
  \fontfamily{#3}\fontseries{#4}\fontshape{#5}%
  \selectfont}%
\fi\endgroup%
\begin{picture}(20748,20742)(2368,-20956)
\put(4351,-12511){\makebox(0,0)[lb]{\smash{{\SetFigFont{10}{12.0}{\rmdefault}{\mddefault}{\updefault}{\color[rgb]{0,0,0}$\rho_k$}%
}}}}
\put(4351,-10411){\makebox(0,0)[lb]{\smash{{\SetFigFont{10}{12.0}{\rmdefault}{\mddefault}{\updefault}{\color[rgb]{0,0,0}$\rho_i$}%
}}}}
\put(4351,-11611){\makebox(0,0)[lb]{\smash{{\SetFigFont{10}{12.0}{\rmdefault}{\mddefault}{\updefault}{\color[rgb]{0,0,0}$\rho_j$}%
}}}}
\put(4201,-13411){\makebox(0,0)[lb]{\smash{{\SetFigFont{10}{12.0}{\rmdefault}{\mddefault}{\updefault}{\color[rgb]{0,0,0}$\x_1$}%
}}}}
\put(10351,-9511){\makebox(0,0)[lb]{\smash{{\SetFigFont{10}{12.0}{\rmdefault}{\mddefault}{\updefault}{\color[rgb]{0,0,0}$\rho_k$}%
}}}}
\put(10351,-7411){\makebox(0,0)[lb]{\smash{{\SetFigFont{10}{12.0}{\rmdefault}{\mddefault}{\updefault}{\color[rgb]{0,0,0}$\rho_i$}%
}}}}
\put(10351,-8611){\makebox(0,0)[lb]{\smash{{\SetFigFont{10}{12.0}{\rmdefault}{\mddefault}{\updefault}{\color[rgb]{0,0,0}$\rho_j$}%
}}}}
\put(10201,-6961){\makebox(0,0)[lb]{\smash{{\SetFigFont{10}{12.0}{\rmdefault}{\mddefault}{\updefault}{\color[rgb]{0,0,0}$\y_2$}%
}}}}
\put(16801,-12811){\makebox(0,0)[lb]{\smash{{\SetFigFont{10}{12.0}{\rmdefault}{\mddefault}{\updefault}{\color[rgb]{0,0,0}$\x_1$}%
}}}}
\put(22276,-11911){\makebox(0,0)[lb]{\smash{{\SetFigFont{10}{12.0}{\rmdefault}{\mddefault}{\updefault}{\color[rgb]{0,0,0}$\rho_k$}%
}}}}
\put(22276,-9811){\makebox(0,0)[lb]{\smash{{\SetFigFont{10}{12.0}{\rmdefault}{\mddefault}{\updefault}{\color[rgb]{0,0,0}$\rho_i$}%
}}}}
\put(22276,-11011){\makebox(0,0)[lb]{\smash{{\SetFigFont{10}{12.0}{\rmdefault}{\mddefault}{\updefault}{\color[rgb]{0,0,0}$\rho_j$}%
}}}}
\put(22126,-9061){\makebox(0,0)[lb]{\smash{{\SetFigFont{10}{12.0}{\rmdefault}{\mddefault}{\updefault}{\color[rgb]{0,0,0}$\y_2$}%
}}}}
\put(16951,-11911){\makebox(0,0)[lb]{\smash{{\SetFigFont{10}{12.0}{\rmdefault}{\mddefault}{\updefault}{\color[rgb]{0,0,0}$\rho_k$}%
}}}}
\put(16951,-9811){\makebox(0,0)[lb]{\smash{{\SetFigFont{10}{12.0}{\rmdefault}{\mddefault}{\updefault}{\color[rgb]{0,0,0}$\rho_i$}%
}}}}
\put(16951,-11011){\makebox(0,0)[lb]{\smash{{\SetFigFont{10}{12.0}{\rmdefault}{\mddefault}{\updefault}{\color[rgb]{0,0,0}$\rho_j$}%
}}}}
\put(16801,-9136){\makebox(0,0)[lb]{\smash{{\SetFigFont{10}{12.0}{\rmdefault}{\mddefault}{\updefault}{\color[rgb]{0,0,0}$\x_2$}%
}}}}
\put(22351,-12586){\makebox(0,0)[lb]{\smash{{\SetFigFont{10}{12.0}{\rmdefault}{\mddefault}{\updefault}{\color[rgb]{0,0,0}$\y_1$}%
}}}}
\put(10501,-2161){\makebox(0,0)[lb]{\smash{{\SetFigFont{10}{12.0}{\rmdefault}{\mddefault}{\updefault}{\color[rgb]{0,0,0}$\rho_i$}%
}}}}
\put(10501,-3436){\makebox(0,0)[lb]{\smash{{\SetFigFont{10}{12.0}{\rmdefault}{\mddefault}{\updefault}{\color[rgb]{0,0,0}$\rho_j$}%
}}}}
\put(10501,-3961){\makebox(0,0)[lb]{\smash{{\SetFigFont{10}{12.0}{\rmdefault}{\mddefault}{\updefault}{\color[rgb]{0,0,0}$\rho_k$}%
}}}}
\put(4351,-2011){\makebox(0,0)[lb]{\smash{{\SetFigFont{10}{12.0}{\rmdefault}{\mddefault}{\updefault}{\color[rgb]{0,0,0}$\rho_i$}%
}}}}
\put(4351,-3211){\makebox(0,0)[lb]{\smash{{\SetFigFont{10}{12.0}{\rmdefault}{\mddefault}{\updefault}{\color[rgb]{0,0,0}$\rho_j$}%
}}}}
\put(4351,-4111){\makebox(0,0)[lb]{\smash{{\SetFigFont{10}{12.0}{\rmdefault}{\mddefault}{\updefault}{\color[rgb]{0,0,0}$\rho_k$}%
}}}}
\put(4201,-661){\makebox(0,0)[lb]{\smash{{\SetFigFont{10}{12.0}{\rmdefault}{\mddefault}{\updefault}{\color[rgb]{0,0,0}$\x_2$}%
}}}}
\put(4201,-5011){\makebox(0,0)[lb]{\smash{{\SetFigFont{10}{12.0}{\rmdefault}{\mddefault}{\updefault}{\color[rgb]{0,0,0}$\x_1$}%
}}}}
\put(10201,-661){\makebox(0,0)[lb]{\smash{{\SetFigFont{10}{12.0}{\rmdefault}{\mddefault}{\updefault}{\color[rgb]{0,0,0}$\y_2$}%
}}}}
\put(10201,-5011){\makebox(0,0)[lb]{\smash{{\SetFigFont{10}{12.0}{\rmdefault}{\mddefault}{\updefault}{\color[rgb]{0,0,0}$\y_1$}%
}}}}
\put(10201,-13411){\makebox(0,0)[lb]{\smash{{\SetFigFont{10}{12.0}{\rmdefault}{\mddefault}{\updefault}{\color[rgb]{0,0,0}$\y_1$}%
}}}}
\put(23101,-3961){\makebox(0,0)[lb]{\smash{{\SetFigFont{10}{12.0}{\rmdefault}{\mddefault}{\updefault}{\color[rgb]{0,0,0}$\rho_k$}%
}}}}
\put(16951,-2011){\makebox(0,0)[lb]{\smash{{\SetFigFont{10}{12.0}{\rmdefault}{\mddefault}{\updefault}{\color[rgb]{0,0,0}$\rho_i$}%
}}}}
\put(16951,-3211){\makebox(0,0)[lb]{\smash{{\SetFigFont{10}{12.0}{\rmdefault}{\mddefault}{\updefault}{\color[rgb]{0,0,0}$\rho_j$}%
}}}}
\put(16951,-4111){\makebox(0,0)[lb]{\smash{{\SetFigFont{10}{12.0}{\rmdefault}{\mddefault}{\updefault}{\color[rgb]{0,0,0}$\rho_k$}%
}}}}
\put(16801,-661){\makebox(0,0)[lb]{\smash{{\SetFigFont{10}{12.0}{\rmdefault}{\mddefault}{\updefault}{\color[rgb]{0,0,0}$\x_2$}%
}}}}
\put(16801,-5011){\makebox(0,0)[lb]{\smash{{\SetFigFont{10}{12.0}{\rmdefault}{\mddefault}{\updefault}{\color[rgb]{0,0,0}$\x_1$}%
}}}}
\put(22801,-661){\makebox(0,0)[lb]{\smash{{\SetFigFont{10}{12.0}{\rmdefault}{\mddefault}{\updefault}{\color[rgb]{0,0,0}$\y_2$}%
}}}}
\put(22801,-5011){\makebox(0,0)[lb]{\smash{{\SetFigFont{10}{12.0}{\rmdefault}{\mddefault}{\updefault}{\color[rgb]{0,0,0}$\y_1$}%
}}}}
\put(23101,-1561){\makebox(0,0)[lb]{\smash{{\SetFigFont{10}{12.0}{\rmdefault}{\mddefault}{\updefault}{\color[rgb]{0,0,0}$\rho_i$}%
}}}}
\put(23101,-2836){\makebox(0,0)[lb]{\smash{{\SetFigFont{10}{12.0}{\rmdefault}{\mddefault}{\updefault}{\color[rgb]{0,0,0}$\rho_j$}%
}}}}
\put(9676,-19111){\makebox(0,0)[lb]{\smash{{\SetFigFont{10}{12.0}{\rmdefault}{\mddefault}{\updefault}{\color[rgb]{0,0,0}$\rho_k$}%
}}}}
\put(9676,-18211){\makebox(0,0)[lb]{\smash{{\SetFigFont{10}{12.0}{\rmdefault}{\mddefault}{\updefault}{\color[rgb]{0,0,0}$\rho_j$}%
}}}}
\put(4351,-19111){\makebox(0,0)[lb]{\smash{{\SetFigFont{10}{12.0}{\rmdefault}{\mddefault}{\updefault}{\color[rgb]{0,0,0}$\rho_k$}%
}}}}
\put(4351,-17011){\makebox(0,0)[lb]{\smash{{\SetFigFont{10}{12.0}{\rmdefault}{\mddefault}{\updefault}{\color[rgb]{0,0,0}$\rho_i$}%
}}}}
\put(4201,-16336){\makebox(0,0)[lb]{\smash{{\SetFigFont{10}{12.0}{\rmdefault}{\mddefault}{\updefault}{\color[rgb]{0,0,0}$\x_2$}%
}}}}
\put(9751,-19786){\makebox(0,0)[lb]{\smash{{\SetFigFont{10}{12.0}{\rmdefault}{\mddefault}{\updefault}{\color[rgb]{0,0,0}$\y_1$}%
}}}}
\put(9601,-15136){\makebox(0,0)[lb]{\smash{{\SetFigFont{10}{12.0}{\rmdefault}{\mddefault}{\updefault}{\color[rgb]{0,0,0}$\y_2$}%
}}}}
\put(9751,-16036){\makebox(0,0)[lb]{\smash{{\SetFigFont{10}{12.0}{\rmdefault}{\mddefault}{\updefault}{\color[rgb]{0,0,0}$\rho_i$}%
}}}}
\put(4351,-18511){\makebox(0,0)[lb]{\smash{{\SetFigFont{10}{12.0}{\rmdefault}{\mddefault}{\updefault}{\color[rgb]{0,0,0}$\rho_j$}%
}}}}
\put(6376,-5761){\makebox(0,0)[lb]{\smash{{\SetFigFont{12}{14.4}{\rmdefault}{\mddefault}{\updefault}{\color[rgb]{0,0,0}Step (1)}%
}}}}
\put(18901,-5761){\makebox(0,0)[lb]{\smash{{\SetFigFont{12}{14.4}{\rmdefault}{\mddefault}{\updefault}{\color[rgb]{0,0,0}Step (2)}%
}}}}
\put(18826,-14161){\makebox(0,0)[lb]{\smash{{\SetFigFont{12}{14.4}{\rmdefault}{\mddefault}{\updefault}{\color[rgb]{0,0,0}Step (4)}%
}}}}
\put(6226,-20761){\makebox(0,0)[lb]{\smash{{\SetFigFont{12}{14.4}{\rmdefault}{\mddefault}{\updefault}{\color[rgb]{0,0,0}Step (5)}%
}}}}
\put(3226,-11536){\makebox(0,0)[lb]{\smash{{\SetFigFont{9}{10.8}{\rmdefault}{\mddefault}{\updefault}{\color[rgb]{0,0,0}$u_1$}%
}}}}
\put(9676,-8536){\makebox(0,0)[lb]{\smash{{\SetFigFont{9}{10.8}{\rmdefault}{\mddefault}{\updefault}{\color[rgb]{0,0,0}$v_2$}%
}}}}
\put(9301,-11461){\makebox(0,0)[lb]{\smash{{\SetFigFont{9}{10.8}{\rmdefault}{\mddefault}{\updefault}{\color[rgb]{0,0,0}$v_1$}%
}}}}
\put(6226,-14236){\makebox(0,0)[lb]{\smash{{\SetFigFont{12}{14.4}{\rmdefault}{\mddefault}{\updefault}{\color[rgb]{0,0,0}Step (3)}%
}}}}
\put(4201,-6961){\makebox(0,0)[lb]{\smash{{\SetFigFont{10}{12.0}{\rmdefault}{\mddefault}{\updefault}{\color[rgb]{0,0,0}$\x_2$}%
}}}}
\put(3601,-8986){\makebox(0,0)[lb]{\smash{{\SetFigFont{9}{10.8}{\rmdefault}{\mddefault}{\updefault}{\color[rgb]{0,0,0}$u_2$}%
}}}}
\end{picture}%

%% file: draws/ExactTriangle.tex
\begin{picture}(0,0)%
\includegraphics{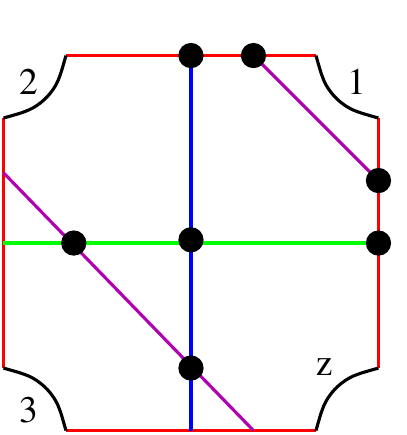}%
\end{picture}%
\setlength{\unitlength}{1973sp}%
\begingroup\makeatletter\ifx\SetFigFont\undefined%
\gdef\SetFigFont#1#2#3#4#5{%
  \reset@font\fontsize{#1}{#2pt}%
  \fontfamily{#3}\fontseries{#4}\fontshape{#5}%
  \selectfont}%
\fi\endgroup%
\begin{picture}(3873,4161)(2368,-3994)
\put(4051,-3661){\makebox(0,0)[rb]{\smash{{\SetFigFont{12}{14.4}{\rmdefault}{\mddefault}{\updefault}{\color[rgb]{0,0,0}$\xi$}%
}}}}
\put(6226,-2161){\makebox(0,0)[lb]{\smash{{\SetFigFont{12}{14.4}{\rmdefault}{\mddefault}{\updefault}{\color[rgb]{0,0,0}$n$}%
}}}}
\put(4276,-2536){\makebox(0,0)[lb]{\smash{{\SetFigFont{12}{14.4}{\rmdefault}{\mddefault}{\updefault}{\color[rgb]{0,0,0}$\zeta$}%
}}}}
\put(6226,-1711){\makebox(0,0)[lb]{\smash{{\SetFigFont{12}{14.4}{\rmdefault}{\mddefault}{\updefault}{\color[rgb]{0,0,0}$a$}%
}}}}
\put(4651,-136){\makebox(0,0)[lb]{\smash{{\SetFigFont{12}{14.4}{\rmdefault}{\mddefault}{\updefault}{\color[rgb]{0,0,0}$b$}%
}}}}
\put(4051,-136){\makebox(0,0)[lb]{\smash{{\SetFigFont{12}{14.4}{\rmdefault}{\mddefault}{\updefault}{\color[rgb]{0,0,0}$r$}%
}}}}
\put(3076,-1936){\makebox(0,0)[lb]{\smash{{\SetFigFont{12}{14.4}{\rmdefault}{\mddefault}{\updefault}{\color[rgb]{0,0,0}$\eta$}%
}}}}
\end{picture}%

%% file: diagram.tex
\section{Identifying the spectral sequences}
\label{subsec:IdentifySS}

\subsection{The multi-diagram for the branched double cover}\label{sec:diag-for-dcov}
Let $L$ be a link in $S^3$. In this section we describe a particular
Heegaard multi-diagram for the branched double cover of the cube of
resolutions of $L$.  Our main reason for interest in this diagram is
that it decomposes as a concatenation of particularly simple bordered
Heegaard diagrams, but it has other nice properties, as well.  In
fact, this diagram is essentially the one studied by
J.~Greene~\cite{Greene:spanning-tree}. (See
Proposition~\ref{prop:IdentifyWithGreene} for a precise statement.)

Draw the plat closure of a $2n$-braid. Rotate it 90$^\circ$ clockwise
(so that the maxima are on the right and the minima are on the
left, rather than top and bottom), as this convention is
best adapted to the bordered setting.  The knot projection is thought of
as gotten from a diagram for the unlink with $n$ maxima (on the
right) and $n$ minima
(on the left) by modifying the picture so as to introduce crossings
between various consecutive strands, as specified by the braid. 
Decompose the knot projection into
three regions: the {\em cap region}, consisting of the $n$ maxima,
the {\em cup region}, consisting of the $n$ minima, and the
{\em braid region}, which contains the braid. 
We will assume the following properties of the knot projection:
\begin{itemize}
\item It is given as a standard plat closure of a braid (i.e., where
  all the cups (respectively caps) happen at the same time).
\item In the braid, the first two strands never cross each other -- i.e.,
  the first strand is stationary throughout the braid.
  (We will think of the ``first two strands'' as the top two strands in the picture.)
\end{itemize}
Both of these properties can be at the cost of possibly
introducing more crossings.

We describe a corresponding Heegaard diagram for the branched
double cover of the $n$-component unlink, and then describe local
changes at the crossings needed to obtain the desired multi-diagram
associated to the projection of $L$. (For an example of the resulting
multi-diagram, see Figure~\ref{fig:DCovDiag}.)

Label the strands in the (trivial) braid from bottom to top $s_1,\dots,s_{2n}$.
We think of these as $2n$ horizontal segments in the unlink
projection.
We first build the part of a Heegaard multi-diagram
associated to the braid region. This part of the Heegaard diagram is
built from an annulus, thought as a rectangle in the plane whose top and bottom edges are identified.
The annulus is equipped with
horizontal arcs which will eventually be
used to build closed curves ($\alpha$-circles) in the Heegaard
diagram, as follows. Each strand $s_i$ for $1<i<2n-1$ is replaced
by a pair of horizontal arcs, one just above and one just below the
horizontal strand $s_i$, while the  strands 
$s_1$ and $s_{2n-1}$ both induce single horizontal arcs in  the Heegaard
diagram (and $s_{2n}$ induces none).  For $1<i<2n-1$, label the arc
in the Heegaard diagram 
just above $s_i$ by $(s_{i-1},s_i)_+$, and the
one just below $s_i$ by $(s_{i},s_{i+1})_-$, replacing $s_1$ by
$(s_1,s_2)_-$ and $s_{2n-1}$ by $(s_{2n-2},s_{2n-1})_+$. 
Thus, coming up
from the bottom, the horizontal arcs in the Heegaard diagram are
labeled as follows:
\begin{align*} 
&(s_1,s_2)_-, (s_2,s_3)_-, (s_1,s_2)_+, \\
&(s_3,s_4)_-, (s_2,s_3)_+,
(s_4,s_5)_-,(s_3,s_4)_+,  \\
& \dots \\
& (s_i,s_{i+1})_-, (s_{i-1},s_i)_+,
(s_{i+1},s_{i+2})_-,(s_i,s_{i+1})_+ \\
&\dots \\
&(s_{2n-2},s_{2n-1})_+.
\end{align*}

So far, we have specified the Heegaard diagram in the braid region,
provided that there are no crossings. We will next describe the cap
and cup regions of the diagram, and finally we turn to the
modifications to the braid region needed in the case where there are
crossings.

At the cap region, we add $(n-1)$ one-handles, after which we close
off the 
pairs of arcs $(s_{i},s_{i+1})_-$ and $(s_{i},s_{i+1})_+$ for $i=1,\dots,2n-1$, 
and draw
$(n-1)$ $\beta$-circles. In more detail, at the cap region
(the right of the diagram) we close off the rightmost endpoints of the
arcs $(s_{2i-1},s_{2i})_-$ and $(s_{2i-1},s_{2i})_+$ for $i=1,\dots,n$,
by arcs denoted $(s_{2i-1},s_{2i})_r$.  Next, we draw $n$
circles labeled $\beta_i^r$ for $i=1,\dots,n-1$, where the $i\th$ $\beta$-circle
encircles the rightmost endpoints of $(s_{2i-2},s_{2i-1})_+$ and
$(s_{2i},s_{2i+1})_-$, except when $i=1$, in which case it encircles
only $(s_1,s_2)_-$. We draw these circles small enough that they are
disjoint from the arcs $(s_{2i-1},s_{2i})_r$.  Next, we stabilize
the picture by attaching one-handles joining up the endpoints of
$(s_{2i},s_{2i+1})_-$ and $(s_{2i},s_{2i+1})_+$; call the resulting arc $(s_{2i},s_{2i+1})_r$. 
Now the arcs $(s_{i},s_{i+1})_-$ and $(s_{i},s_{i+1})_+$ are connected
in the cap region for $i=1,\dots,2n-2$.

The diagram in the cup region is the mirror image of the diagram in
the cap region.
We label the $n-1$ new $\beta$-circles
here $\{\beta^\ell_i\}_{i=1}^{n-1}$, and the arcs joining
$(s_{i},s_{i+1})_-$ and $(s_{i},s_{i+1})_+$ in the left region by
$(s_{i},s_{i+1})_\ell$.

So far, we have a surface of genus $2n-2$, equipped with $2n-2$
$\beta$-circles. For
$i=1,\dots,2n-2$, the closed curves $(s_i,s_{i+1})_\ell\cup
(s_i,s_{i+1})_-\cup (s_i,s_{i+1})_+\cup (s_i,s_{i+1})_r$ are our
$\alpha$-circles. This gives a Heegaard
diagram for $\#_{i=1}^{2n-2}(S^2\times S^1)$, which is the branched
double cover of the plat closure of the trivial braid on $2n$ strands.

We describe now how to modify the Heegaard diagram (in the braid
region), in the presence of crossings.  If the $k\th$ crossing occurs
between the strands $s_i$ and $s_{i+1}$, then choose small
disks $D_k^-$ intersecting $(s_i,s_{i+1})_-$ and $D_k^+$ intersecting
$(s_i,s_{i+1})_+$.  Remove the interiors of these disks and identify
their boundaries via reflection across a horizontal axis.  This has the effect of
attaching a one-handle with two arcs
running through it to the Heegaard multi-diagram. With this one-handle
attachment, we have increased
the number of $\alpha$-circles by one.  There will be four choices for
how to add a corresponding $\beta$-circle. Either we take a 
meridian for the newly attached one-handle, $\mu_k$; or we take a $\beta$-circle
$\lambda_k$ which runs through the handle, meeting only the two arcs
$(s_{i-1},s_i)_+$ and $(s_i,s_{i+1})_-$; or we take a curve which
is one of the two resolutions of $\mu_k\cup \lambda_k$.
For the negative braid generator, we let $\beta_k^0$ be the meridian
$\mu_k$, $\beta_k^1$ be the longitude $\lambda_k$, and
$\beta_k^{\infty}$ be their resolution pictured on the top of
Figure~\ref{fig:LocalChanges}. For the positive braid
generator, we let $\beta_k^0$ be the longitude $\lambda_k$,
$\beta_k^1$ be the meridian $\mu_k$, and $\beta_k^{\infty}$ be the
other resolution, pictured on the bottom of Figure~\ref{fig:LocalChanges}.

\begin{figure}
  \centering
  \input{LocalChanges}
  \caption{\textbf{Modifications for the Heegaard multi-diagram at each crossing.}
    The two types of braid generator are illustrated in the left column.
    This corresponds to a modification of the Heegaard multi-diagram as
    illustrated on the right.}
    \label{fig:LocalChanges}
\end{figure}
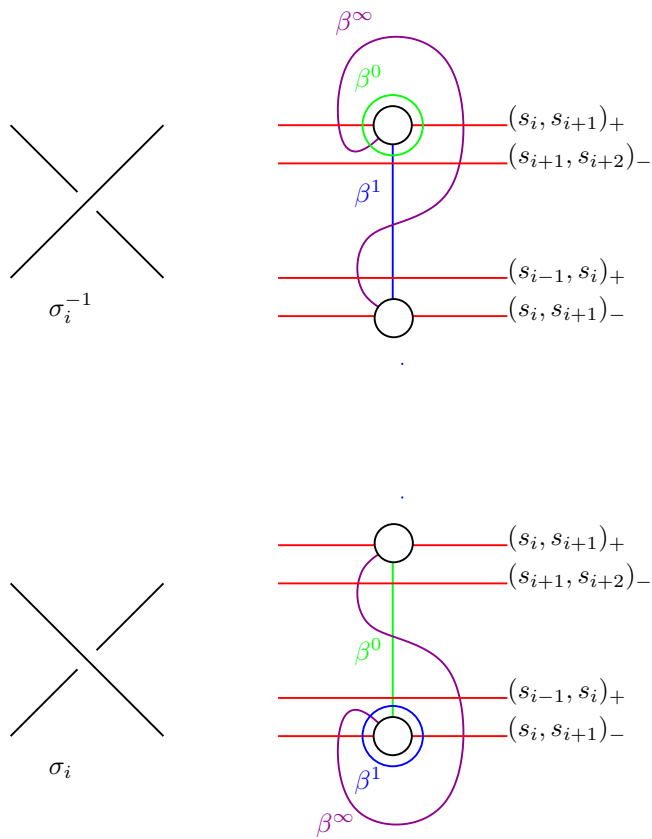

We will see in Lemma~\ref{lem:IdentifyTheDiagrams} that the Heegaard
diagram $(\Sigma,\alphas,\betas^\infty,z)$ represents the branched
double cover $\Sigma(L)$ of $L$. 

We will think of the diagram
$(\Sigma,\alphas,\{\betas^j\}_{j\in\{0,1,\infty\}},z)$ as sliced up
into $c+2$ slices: reading right to left, the first of these corresponds to the (rightmost)
cap region, $c$ of these correspond to the crossings
and the last corresponds to the cup regions.
In more detail, we  cut up the Heegaard multi-diagram along $c+1$ concentric circles
$S_0,\dots,S_c$ based at some central point off to the right of the
diagram. The {\em{$0\th$ slice}} of the Heegaard multi-diagram is the
region encircled by $S_0$ (which is drawn large enough to contain all
of the cap region); for $k=1,\dots,c$, the {\em{$k\th$ slice of
    the Heegaard multi-diagram}} is the region in the annulus between the
circle $S_{k-1}$ and $S_{k}$. The $(c+1)\st$ slice is the region outside
$S_c$ (including a point at infinity). The crossing modifications are
arranged so that both $D_k^-$ and
$D_k^+$ occur in the $k\th$ slice (in the same order as the
crossings appear in the knot projection).

The resulting surface has genus $g=2n-2+c$,
and it is equipped with a
$(2n-2+c)$-tuple of attaching circles $\alphas$.
Moreover, for each $j\in\{0,1,\infty\}^c$, there is also
a $(2n-2+c)$-tuple of curves $\betas^j$, consisting of all the above
$\beta$-circles on the right and the left of the diagram, and 
choosing $\beta^0_k$, $\beta^1_k$, or $\beta^{\infty}_k$ at the
$k\th$ crossing as prescribed by the component $j_k$.
An example is illustrated in Figure~\ref{fig:DCovDiag}.
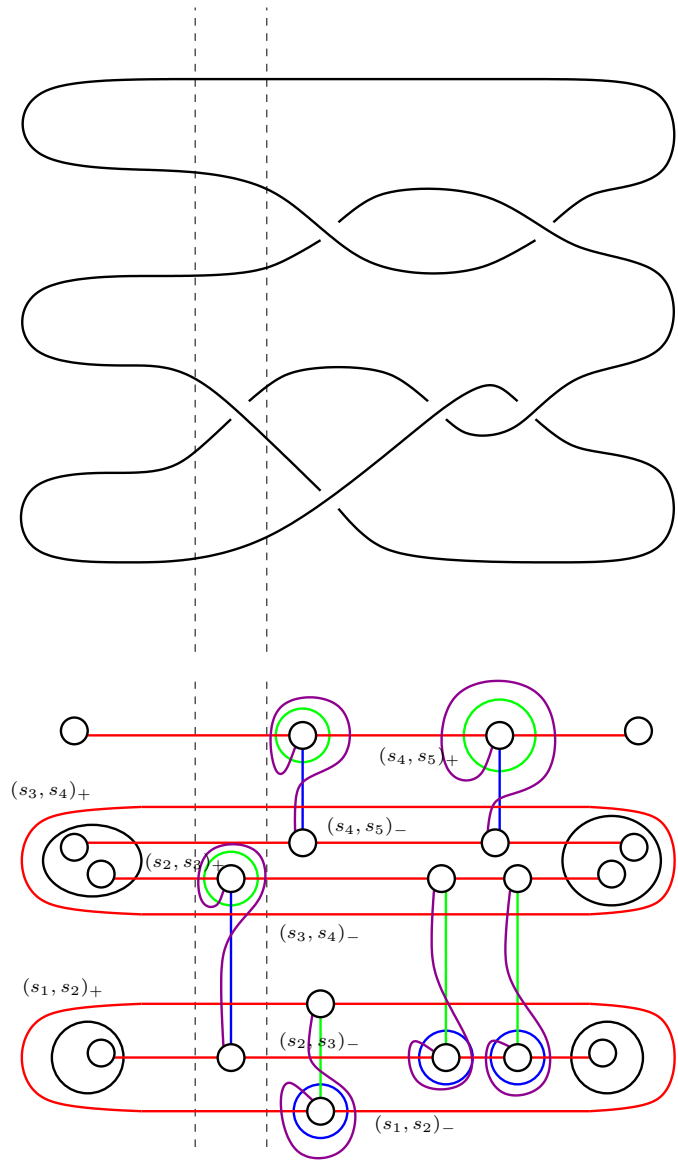
\begin{figure}
  \centering
  \input{DCovDiag}
  \caption{\textbf{Heegaard multi-diagram for a branched
      double cover.}
    Start from the knot projection above, and
    build the diagram below. We have distinguished a vertical slice
    of the knot diagram containing a single crossing. The corresponding
    vertical slice of the Heegaard multi-diagram, after suitable
    stabilizations, is a bordered multi-diagram for a Dehn twist.}
    \label{fig:DCovDiag}
\end{figure}

This diagram is not ideal for our purposes: the slices of this
diagram are not yet bordered multi-diagrams in the sense of
Definition~\ref{def:2Admissible}, as most of the $\alpha$-arcs run
from one boundary component to the other.  Specifically, the $k\th$
slice, involving a crossing between strands $s_i$ and $s_{i+1}$, has
pairs of arcs corresponding to $(s_\ell,s_{\ell+1})_-$ and
$(s_\ell,s_{\ell+1})_+$ for all $\ell\neq i$ which connect the two
boundary components.  We can, however, attach one-handles to reconnect
these pairs of arcs so that each of the new arcs has both endpoints on
one boundary component, as in Figure~\ref{fig:reconnect-arc}.  This
modification has the effect of increasing the genus by another
$c(2n-3)$, and introducing the same number of new $\alpha$-circles and
$\beta$-circles $\betas^j$ (which are the same for all choices of
$j\in\{0,1,\infty\}^{c}$). This additional stabilization does not
increase the total number of Heegaard Floer generators for any of the
diagrams -- it serves only to make sensible the holomorphic curve
counts in the various bordered slices. We refer to the destabilized
version as the {\em small diagram for the branched double cover};
but we will typically consider its stabilized version.

\begin{figure}
  \centering
  \includegraphics{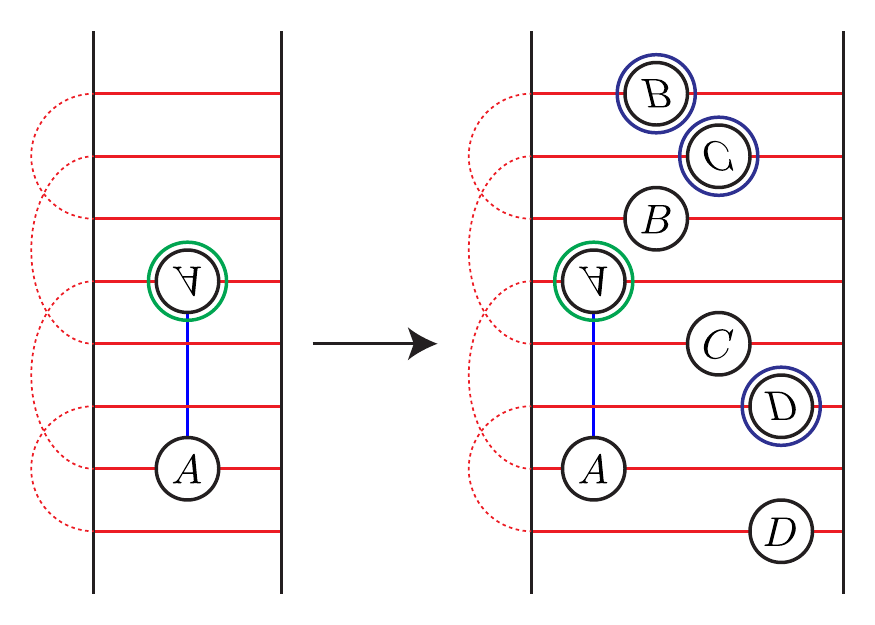}
  \caption{\textbf{Attaching handles and reconnecting arcs.} Each
    letter corresponds to a handle.}
  \label{fig:reconnect-arc}
\end{figure}

Finally, the Heegaard multi-diagram is equipped with a basepoint
corresponding to the point at infinity. This can be extended to an arc
of basepoints crossing all the slices, but disjoint from all attaching
circles, in a straightforward way.

For $k=0,\dots,c+1$, write the part of the Heegaard multi-diagram in the $k\th$ slice as
$(\Sigma_k,\alphas_k,\{\betas^j_k\}_{j\in\{0,1,\infty\}},z)$, so that $k=0$ and $k=c+1$ correspond to the cup and cap regions.

If $Y=\Sigma(L)$
is the branched double cover of a link $L$ with $c$ crossings, then
$Y$ is equipped with a framed link $L'$ whose components are in
one-to-one correspondence with the crossings of $L$.  
(Each crossing $c$ in the diagram for $L$ specifies an arc whose boundary lies in $L$. The branched double cover of this arc gives the link $L'$.
The framing of this link is specified so that $0$-framed surgery gives
the branched double cover of the braid-like or the anti-braid-like resolution,
depending on whether the crossing is of type $\sigma_i$ or $\sigma_i^{-1}$.

\begin{lemma}
  \label{lem:IdentifyTheDiagrams}
  The $\IndI=\{0,1,\infty\}^c$-filtered chain complex of
  attaching circles
  \[
  (\Sigma,\{\betas^j\}_{j\in \{0,1,\infty\}^c},z)
  =
  \#_{k=0}^{c+1} (\Sigma_k,\{\betas^j_k\}_{j\in\{0,1,\infty\}},z) 
  \]
  is a chain complex of attaching circles associated to the framing changes 
  as in Definition~\ref{def:ChangeFramingsComplex}
  on the link $L'\subset Y$ specified by 
  $\alphas$ and $\{\betas^j_k\}$.
\end{lemma}

It is natural to see this from the bordered perspective; so we postpone the proof a moment.

\subsection{The bordered decomposition of the multi-diagram}

The slices of the diagram considered above were studied
in~\cite{LOT:DCov1}. As noted above, each slice is a bordered Heegaard
(multi-)diagram with two boundary components. Each boundary component
is parameterized by the {\em linear pointed matched circle}, as
in~\cite[Section~\ref*{DCov1:sec:diagrams}]{LOT:DCov1} or
Figure~\ref{fig:linear-pmc}, which we think of as the $2$-sphere branched at
$2n$ collinear points. Each of the matched pairs in the
linear pointed matched circle $\PMC$ specifies a circle in the surface
$F(\PMC)$ containing the core of the corresponding handle. We call these the
\emph{generating curves} in $F(\PMC)$. 
These generating curves correspond to branched double covers of straight arcs connecting 
two consecutive of the $2n$ collinear points in the $2$-sphere.

\begin{figure}
  \centering
  \input{LinearPMC}
  \caption{\textbf{The linear pointed matched circle.} The genus $3$
    (i.e., $n=4$) case is shown.}
  \label{fig:linear-pmc}
\end{figure}
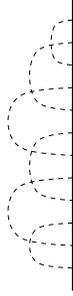

\begin{proof}[Proof of Lemma~\ref{lem:IdentifyTheDiagrams}]
  We find this easiest to verify one slice at a time.

  Consider the $k\th$ slice and suppose for definiteness that the
  crossing there is a positive braid generator. By inspection, the
  intersection of $\alphas$ and $\betas^1$ with this slice specifies the
  identity cobordism, which we think of as the branched double cover
  of the sphere times an interval branched at $2n$ non-crossing arcs
  connecting the two boundary components (i.e., the braid-like
  resolution of the positive crossing).  We also claim
  that the intersection of $\alphas$ and $\betas^0$ with this slice
  specifies the branched double cover of a cup followed by a cap
  (i.e., the anti-braid-like resolution of the positive crossing). This
  can be seen by noting that $\beta^0_k$ is a copy of one of the
  generating curves on $F$, supported at the mid-level of the product
  cobordism. It follows that $\alphas$ and $\betas^0$ represents some
  (Morse) surgery on the knot in the branched double cover of the
  trivial braid on $2n$ strands, which is the branched double cover on
  an arc connecting two of the consecutive strands. To see that it is,
  in fact, a cup followed by a cap (which is surgery with respect to
  the surface framing) follows from straightforward homological
  considerations. The positive Dehn twist (branched double cover of
  the positive braid generator) is now obtained as surgery on this
  same curve with a new framing ($-1$ with respect to the surface
  framing): it is specified as a suitable resolution of the sum
  of $\beta^0_k$ and $\beta^1_k$.  Permuting the roles of the ambient
  and the surgered three-manifold, we have that $\beta^0_k$ and
  $\beta^1_k$ denote two different framings on the knot in the
  three-manifold $Y$ corresponding to the $k\th$ crossing.
  
  The bottom and top pieces correspond to standard handlebodies,
  thought of as branched double covers of the $3$-ball branched along $n$ arcs.
  
  Gluing the pieces together, the result follows.
\end{proof}

\begin{lemma}
  \label{lem:EachPiece}
  The bordered diagram $(\Sigma_k,\alphas_k,\betas_k^{\infty},\arcz)$
  represents a Dehn twist along one of the preferred
  generating curves.  Moreover, counting holomorphic triangles in the
  two-step chain complex
  $(\Sigma_k,\alphas_k,\{\betas^j_k\}_{j\in\{0,1\}},\theta^{0<1})$
  gives a map
  \[
  F^-\co \CFDAa(\Id) \to \CFDAa({\check \sigma}_i)~\text{or}~F^+\co
  \CFDAa({\check \sigma}_i) \to \CFDAa(\Id),
  \]
  according to whether the $k\th$ crossing is of the form $\sigma_i$ or $\sigma_i^{-1}$ respectively, whose mapping cone
  is identified with 
  $\CFDAa(\Sigma_i,\alphas_i,\betas_i^{\infty},z)$.
\end{lemma}

\begin{proof}
  The first sentence follows from an inspection of the diagram. The
  identification between the mapping cone of the triangle map and the
  Dehn twist follows from the bordered proof of the surgery exact
  triangle~\cite[Theorem~\ref*{DCov1:thm:Dehn-is-MC}]{LOT:DCov1}
  together with Proposition~\ref{prop:surg-tri-counts}, which
  shows the maps in the surgery exact triangle come from counting
  holomorphic triangles.
\end{proof}

\begin{lemma}
  \label{lem:IdentifyTheMapsInSlices}
  The map from Lemma~\ref{lem:EachPiece} induced by counting
  holomorphic triangles in the bordered
  diagram $(\Sigma_k,\alphas_k,\{\betas_k^0,\betas_k^1\},z)$,
  \[
  F^-\co \CFDAa(\Id) \to  \CFDAa({\check \sigma}_i)~\text{or}~F^+\co \CFDAa({\check \sigma}_i) \to \CFDAa(\Id),
  \]
  agrees up to homotopy with the maps (with the same notation)
  from~\cite{LOT:DCov1}.
\end{lemma}

\begin{proof}
  To keep the notation simple, we will talk about $F^-$; the proof
  for $F^+$ is the same. 

  By~\cite[Propositions~\ref*{DCov1:prop:CalculateNegMorphism} and~\ref*{DCov1:prop:CalculateNegMorphismDeg}]{LOT:DCov1}, for the type \DD\ maps we have
  \[
  F^-_{\DD,\text{hol}}=F^-_{\DD,\text{comb}}\co \CFDDa(\Id) \to \CFDDa({\check \sigma}_i).
  \]
  The map $F^-_{\text{comb}}\co \CFDAa(\Id) \to  \CFDAa({\check
    \sigma}_i)$ from~\cite{LOT:DCov1} is (by definition) gotten by
  tensoring $F^-_{\DD,\text{comb}}$ with the identity map of
  $\CFAAa(\Id)$:
  \[
  \xymatrix{
    \CFDAa(\Id)\ar[rr]^{F^-_{\text{comb}}}\ar[d]_\simeq & &  \CFDAa({\check
      \sigma}_i)\\
    \CFDDa(\Id)\DT\CFAAa(\Id)\ar[rr]^{F^-_{\DD,\text{comb}}\DT\Id} & &  \CFDDa({\check
      \sigma}_i)\DT\CFAAa(\Id).\ar[u]_\simeq
  }
  \]
  To see that $F^-_{\text{comb}}\sim F^-_{\text{hol}}$, tensor both
  sides with the identity map of $\CFDDa(\Id)$.  It follows from the
  pairing theorem (Theorem~\ref{thm:PolygonPairingDA}) that
  $F^-_{\text{hol}}\DT\Id_{\DD}\sim F^-_{\DD,\text{hol}}$, and it
  follows from the fact that $\CFDDa(\Id)\DT\CFAAa(\Id)\simeq \CFDAa(\Id)=[\Id]$ that 
  $F^-_{\text{comb}}\DT\Id_{\DD}\sim F^-_{\DD,\text{comb}}$
  But tensoring with $\CFDDa(\Id)$ is a quasi-equivalence of \dg
  categories, so this implies that $F^-_{\text{hol}}\sim
  F^-_{\text{comb}}$, as desired.
\end{proof}

\subsection{Putting together the pieces}

\begin{proof}[Proof of Theorem~\ref{thm:IdentifySpectralSequences}]
  Using the diagram explained above,
  Lemma~\ref{lem:IdentifyTheDiagrams} identifies the chain complex of
  attaching circles used to construct the spectral sequence
  from~\cite{BrDCov} with the chain complex of attaching circles which
  decomposes as a
  concatenation of bordered diagrams for Dehn twists in the linear
  pointed matched circle. 

  The pairing theorem for polygons (Theorem~\ref{thm:PolygonPairingDA-DTP}) then
  identifies the filtered complex with
  the one gotten by an iterated tensor product of $DA$ bimodule
  morphisms, where the morphisms are defined by counting
  pseudo-holomorphic triangles. 

  Lemma~\ref{lem:IdentifyTheMapsInSlices} then identifies these $DA$
  bimodule morphisms with the combinatorially defined morphisms constructed
  in~\cite{LOT:DCov1}, and the filtered complex gotten from the iterated tensor
  product of these combinatorial models induces the spectral sequence
  from~\cite{LOT1}.
\end{proof}

\subsection{Kauffman states and Greene's diagram}\label{sec:Kauffman}

As mentioned earlier, the Heegaard multi-diagram considered here is a
stabilization of the diagram considered by
Greene~\cite{Greene:spanning-tree}.  As such, if we forget about the
filtration on our chain complex, we end up with a chain complex
for the branched double cover whose generators correspond to Kauffman
states. This correspondence can be seen locally. The $\beta$-circle at
each crossing meets four of the $\alpha$-curves (except if the strand
is one of the two extremal strands). Thus, each generator picks out
one of these four intersection points.  Which of those four is chosen
corresponds to a local choice of Kauffman states, as indicated in
Figure~\ref{fig:LocalKauffman}.  It is easy to see that these local
correspondences piece together to give a correspondence between
Heegaard Floer generators and Kauffman states
(compare~\cite{Greene:spanning-tree};
Proposition~\ref{prop:IdentifyWithGreene}; and
also~\cite{AltKnots}). See Figure~\ref{fig:KauffmanStates} for a
global example.

\begin{figure}
  \centering
  \input{Kauffman}
  \caption{\textbf{Local correspondence between Kauffman states and
      Heegaard Floer generators.}}
    \label{fig:LocalKauffman}
\end{figure}
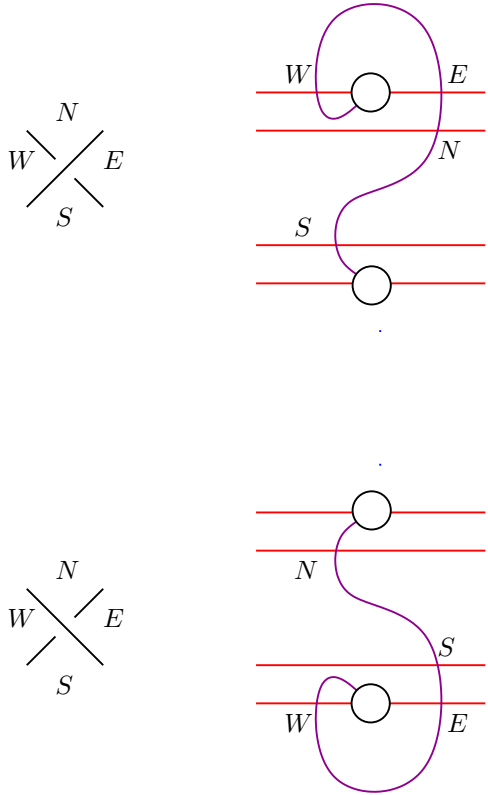

\begin{figure}
  \centering
  \input{KauffmanStates}
  \caption{\textbf{A knot, a Kauffman state, the branched double cover
      Heegaard multi-diagram, and corresponding Heegaard Floer generator.}
    For the Kauffman state indicated above, we have drawn components
    of the corresponding Heegaard Floer generator. The Heegaard Floer
    generator has more components which are not indicated; but those
    are all uniquely determined.}
    \label{fig:KauffmanStates}
\end{figure}
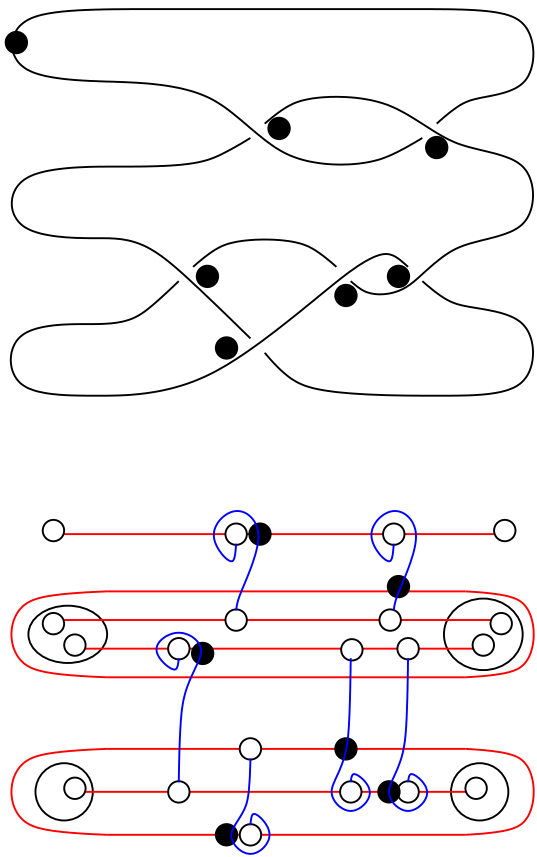

This correspondence extends to the resolutions: at each resolution,
there is a correspondence between Kauffman states in the resolved
diagram and Heegaard-Floer generators. 
It is interesting to note that for the complete resolutions, the
Heegaard multi-diagram is typically not admissible: a disconnected,
complete resolution has no Kauffman states, but the Floer homology is
non-trivial. 


Our aim now is to relate our Heegaard diagrams with Greene's. As a first step,
we paraphrase Greene's description.

Let $K$ be a projection of a knot, with one marked edge. For
consistency with Section~\ref{sec:diag-for-dcov} we assume (unlike
Greene) that it is the plat closure of a braid, lying on its
side, with no crossings involving the top strand.  We think of the top strand as the marked
edge.

Take a regular neighborhood $T$ of the knot diagram in the plane of
the knot projection, with the induced orientation from the plane. This
region $T$ will eventually form half the Heegaard surface (the ``top
half'').  The boundary of $T$ is a collection of circles. We label all
of those circles except the top edge as $\alpha$-circles.

We now construct portions of $\beta$-circles corresponding to crossings.
Think of crossings as involving two strands, one of which connects
$SW$ and $NE$, and the other of which connects $SE$ and
$NW$. Correspondingly, there are four corners in $\bdy T$ at each crossing,
which we label $N$, $S$, $E$, and $W$.  We draw pairs of arcs in
$T$, which will eventually close up to form $\beta$-circles
corresponding to the crossings. If the $SW/NE$ strand is an
overcrossing, one arc connects $W$ to $N$ and the other $S$ to $E$;
otherwise, one arc connects $E$ to $N$ and the other connects $S$ to $W$.

There is one additional $\beta$ arc which cuts across the  marked
edge in $T$.

The Heegaard surface now is gotten by doubling $T$ along its
boundary. Let $B$ denote the other half of the double (the ``bottom
half''). In the
bottom half, there is a $\beta$ arc which closes up the special
$\beta$ arc at the marked edge. Also, at each crossing, we draw
pairs of $\beta$ arcs in $B$, consisting of an arc connecting $N$ to
$S$ and another connecting $E$ to $W$.
These arcs are all drawn to be pairwise
disjoint, but are not contained in a neighborhood of the crossing. (This can be done uniquely, up to diffeomorphism.)

This is, essentially, Greene's description. Note that the
distinguished $\beta$-circle corresponding to the marked edge meets a
single $\alpha$-circle; thus, this pair of curves can be destabilized
to construct what we shall call the {\em small Greene diagram}.

A knot can be rotated $180^\circ$ around the $x$-axis, to obtain a
new planar projection. If the knot is the plat closure of a braid,
this has the effect of replacing each $\sigma_i$  with $\sigma_{n-i}$.
We will call the resulting knot diagram the {\em rotated diagram}.

\begin{proposition}
  \label{prop:IdentifyWithGreene}
  For any connected plat braid diagram for~$K$, the standard small
  Heegaard diagram for the branched
  double cover of $K$ (as in this paper) can be destabilized $2n-2$
  times (supported in the cup and cap regions) so that it
  becomes homeomorphic to the small Greene diagram
  associated to the rotated diagram of $K$.
  In particular, there
  is a canonical identification between generators for the standard diagram
  studied here and Greene's diagram (for the rotate of $K$), which identifies absolute
  gradings of generators.
\end{proposition}

\begin{proof}
  Consider the standard small diagram for the branched double cover of
  a knot.  As a first step, erase all the $\beta$-circles supported
  inside cup and cap regions and the $\alpha$-circles which meet those
  $\beta$-circles.  Destabilize all the handles supported in the cup
  and cap regions. This is the destabilization described at the
  beginning of the proposition; call this diagram $\mathcal D$.

  It is perhaps easiest to see the identification after finding
  regions corresponding to $T$ and $B$ in ${\mathcal D}$.  To this
  end, we will describe a circle $\gamma$
  in ${\mathcal D}$. 

  Draw a portion of $\gamma$
  which is mostly parallel, but below, the topmost horizontal arcs, and then
  enters the handles at the leftmost and rightmost ends of the topmost arcs.
  Next, continue $\gamma$ so that it consists of arcs which connect
  the two handles which are the leftmost ends of the arcs 
  $(s_i,s_{i+1})_\ell$, and similarly so that it consists of arcs which 
  connect the two handles which are the rightmost ends of the arcs
  $(s_i,s_{i+1})_r$ (all this for $i=1,\dots,2n-2$). Finally, run
  an arc nearly parallel and just below $(s_2,s_{3})^-$, running
  through the leftmost and rightmost handles at the two ends of this curve.
  
  All the $\alpha$ curves together with $\gamma$ divide our Heegaard surface
  into two regions, one of which we denote $T'$ and the other $B'$. 
  (We label these so that the point at infinity is contained in $B'$.)
  See Figure~\ref{fig:GreeneDiag}.

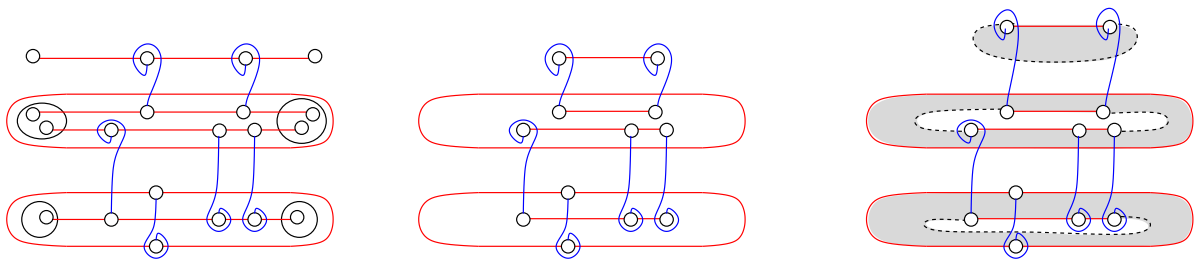
\begin{figure}
  \centering
  \input{GreeneDiag}
  \caption{\textbf{ Modifications to the small standard diagram to get
      Greene's diagram.}  Passing from the first to the second picture
    is a sequence of destabilizations at the cup and cap regions. The
    third picture is equipped with the curve $\gamma$ (drawn dashed),
    and the region $T'$ is shaded.}
    \label{fig:GreeneDiag}
\end{figure}

It is now straightforward to see that $T'$ corresponds to $T$ in
Greene's diagram: this follows from local considerations at each
crossing as in Figure~\ref{fig:LocalGreene}.  Note that in the
homeomorphism between $T$ and $T'$, the directions $N$ and $S$ are
reversed, while $E$ and $W$ are not. Locally in $T$ (and in $T'$)
there are four directions which can be connected to the other local pieces associated to the
crossings. Label these directions $NE$, $NW$, $SE$, and $SW$ in the obvious way. The homeomorphism between $T$ and $T'$ also
switches $NE$ and $SE$, and $NW$ and $SW$.  This effectively switches
$\sigma_i$ to $\sigma_{n-i}$. This homeomorphism clearly extends over $B$ and $B'$.
\end{proof}

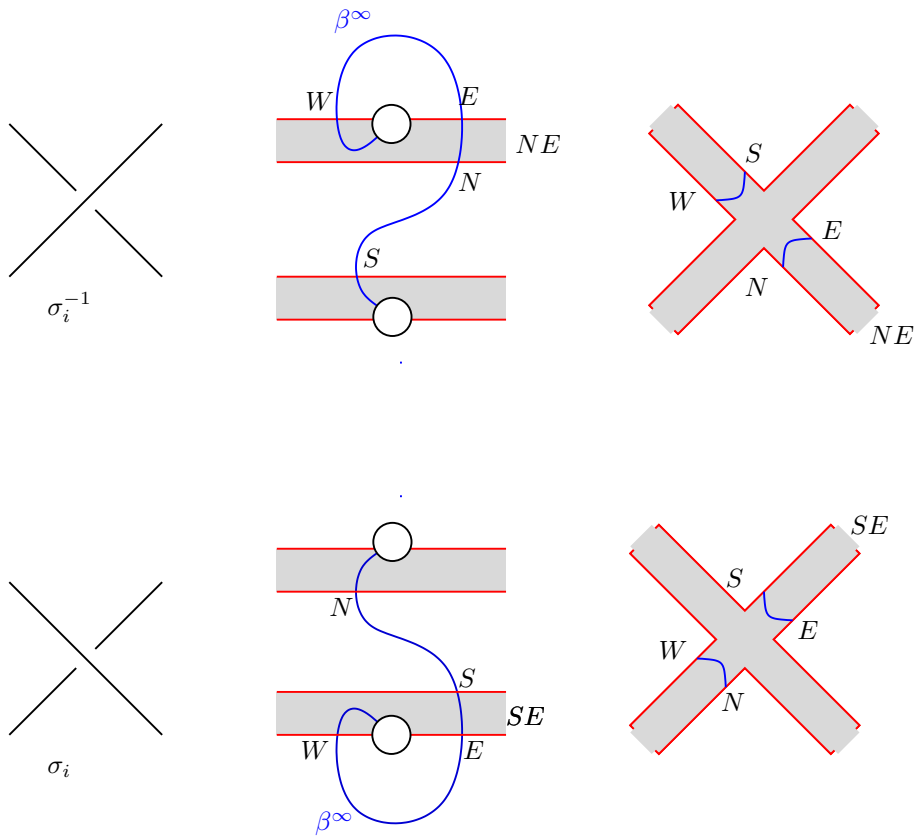
\begin{figure}
  \centering
  \input{LocalGreene}
  \caption{\textbf{Local identification between the small standard
      diagram and Greene's diagram.}  At the left is a crossing; in
    the middle the corresponding portion of the standard diagram (with
    shaded $T'$), and at right the corresponding portion of Greene's
    diagram (with $T$ shaded).}
    \label{fig:LocalGreene}
\end{figure}

Proposition~\ref{prop:IdentifyWithGreene} is particularly useful
because of the thoroughness of Greene's work: for example, he
explicitly computes the gradings of the generators on the branched
double cover, and these computations extend quickly to compute the
gradings of the branched double cover described here.  Moreover, the
proof of Proposition~\ref{prop:IdentifyWithGreene} above can be
adapted easily to give a comparison between our multi-diagram and
Greene's multi-diagram.


%% file: draws/LocalChanges.tex
\begin{picture}(0,0)%
\includegraphics{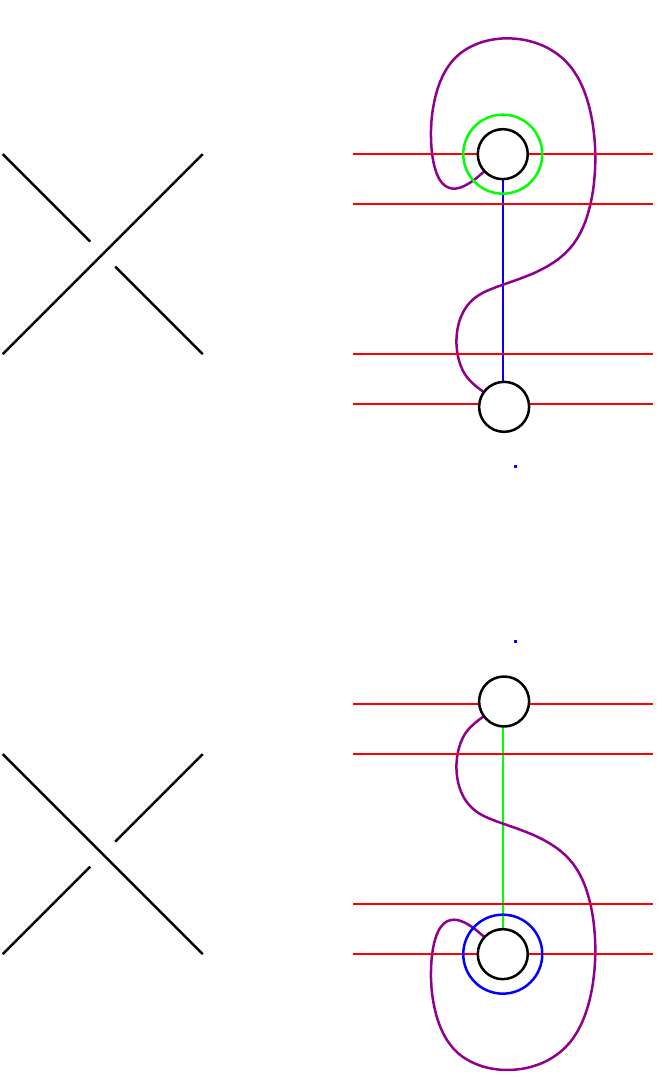}%
\end{picture}%
\setlength{\unitlength}{1579sp}%
\begingroup\makeatletter\ifx\SetFigFont\undefined%
\gdef\SetFigFont#1#2#3#4#5{%
  \reset@font\fontsize{#1}{#2pt}%
  \fontfamily{#3}\fontseries{#4}\fontshape{#5}%
  \selectfont}%
\fi\endgroup%
\begin{picture}(7866,13059)(-632,-9793)
\put(7201,-1561){\makebox(0,0)[lb]{\smash{{\SetFigFont{10}{12.0}{\rmdefault}{\mddefault}{\updefault}{\color[rgb]{0,0,0}$(s_i,s_{i+1})_-$}%
}}}}
\put(7201,-961){\makebox(0,0)[lb]{\smash{{\SetFigFont{10}{12.0}{\rmdefault}{\mddefault}{\updefault}{\color[rgb]{0,0,0}$(s_{i-1},s_i)_+$}%
}}}}
\put(7201,839){\makebox(0,0)[lb]{\smash{{\SetFigFont{10}{12.0}{\rmdefault}{\mddefault}{\updefault}{\color[rgb]{0,0,0}$(s_{i+1},s_{i+2})_-$}%
}}}}
\put(7201,1439){\makebox(0,0)[lb]{\smash{{\SetFigFont{10}{12.0}{\rmdefault}{\mddefault}{\updefault}{\color[rgb]{0,0,0}$(s_i,s_{i+1})_+$}%
}}}}
\put(7201,-8161){\makebox(0,0)[lb]{\smash{{\SetFigFont{10}{12.0}{\rmdefault}{\mddefault}{\updefault}{\color[rgb]{0,0,0}$(s_i,s_{i+1})_-$}%
}}}}
\put(7201,-7561){\makebox(0,0)[lb]{\smash{{\SetFigFont{10}{12.0}{\rmdefault}{\mddefault}{\updefault}{\color[rgb]{0,0,0}$(s_{i-1},s_i)_+$}%
}}}}
\put(7201,-5761){\makebox(0,0)[lb]{\smash{{\SetFigFont{10}{12.0}{\rmdefault}{\mddefault}{\updefault}{\color[rgb]{0,0,0}$(s_{i+1},s_{i+2})_-$}%
}}}}
\put(7201,-5161){\makebox(0,0)[lb]{\smash{{\SetFigFont{10}{12.0}{\rmdefault}{\mddefault}{\updefault}{\color[rgb]{0,0,0}$(s_i,s_{i+1})_+$}%
}}}}
\put(4801,-8986){\makebox(0,0)[lb]{\smash{{\SetFigFont{10}{12.0}{\rmdefault}{\mddefault}{\updefault}{\color[rgb]{0,0,1}$\beta^{1}$}%
}}}}
\put(4801,2039){\makebox(0,0)[lb]{\smash{{\SetFigFont{10}{12.0}{\rmdefault}{\mddefault}{\updefault}{\color[rgb]{0,1,0}$\beta^{0}$}%
}}}}
\put(4801,239){\makebox(0,0)[lb]{\smash{{\SetFigFont{10}{12.0}{\rmdefault}{\mddefault}{\updefault}{\color[rgb]{0,0,1}$\beta^{1}$}%
}}}}
\put(4801,-6961){\makebox(0,0)[lb]{\smash{{\SetFigFont{10}{12.0}{\rmdefault}{\mddefault}{\updefault}{\color[rgb]{0,1,0}$\beta^{0}$}%
}}}}
\put(4201,-9661){\makebox(0,0)[lb]{\smash{{\SetFigFont{10}{12.0}{\rmdefault}{\mddefault}{\updefault}{\color[rgb]{.56,0,.56}$\beta^{\infty}$}%
}}}}
\put(4501,2939){\makebox(0,0)[lb]{\smash{{\SetFigFont{10}{12.0}{\rmdefault}{\mddefault}{\updefault}{\color[rgb]{.56,0,.56}$\beta^{\infty}$}%
}}}}
\put(  1,-1561){\makebox(0,0)[lb]{\smash{{\SetFigFont{10}{12.0}{\rmdefault}{\mddefault}{\updefault}{\color[rgb]{0,0,0}$\sigma_i^{-1}$}%
}}}}
\put(  1,-8761){\makebox(0,0)[lb]{\smash{{\SetFigFont{10}{12.0}{\rmdefault}{\mddefault}{\updefault}{\color[rgb]{0,0,0}$\sigma_i$}%
}}}}
\end{picture}%

%% file: draws/DCovDiag.tex
\begin{picture}(0,0)%
\includegraphics{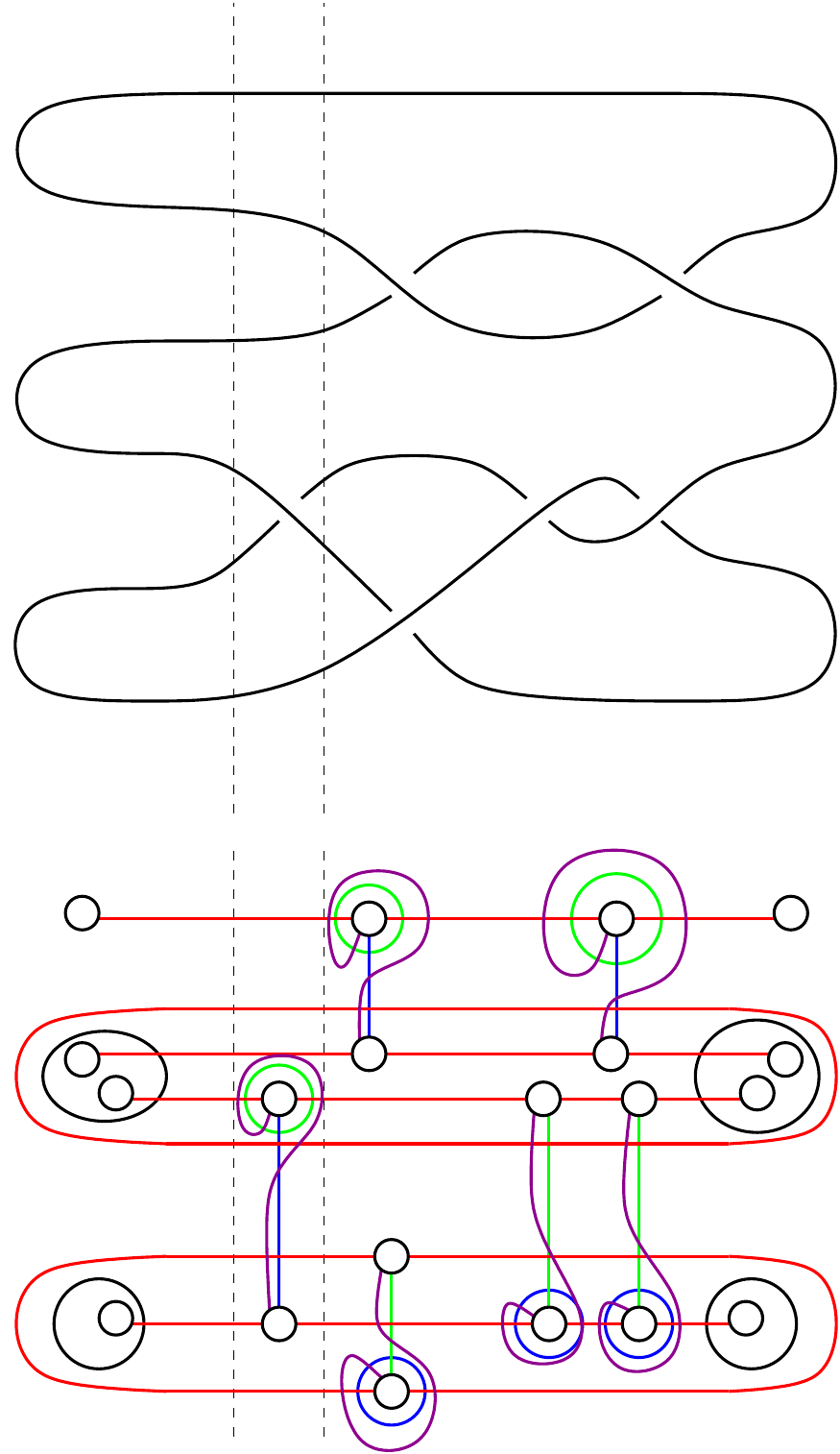}%
\end{picture}%
\setlength{\unitlength}{1973sp}%
\begingroup\makeatletter\ifx\SetFigFont\undefined%
\gdef\SetFigFont#1#2#3#4#5{%
  \reset@font\fontsize{#1}{#2pt}%
  \fontfamily{#3}\fontseries{#4}\fontshape{#5}%
  \selectfont}%
\fi\endgroup%
\begin{picture}(8396,14511)(-389,-14260)
\put(4257,-9234){\makebox(0,0)[lb]{\smash{{\SetFigFont{7}{8.4}{\rmdefault}{\mddefault}{\updefault}{\color[rgb]{0,0,0}$(s_{4}, s_{5})_{+}$}%
}}}}
\put(3001,-11461){\makebox(0,0)[lb]{\smash{{\SetFigFont{7}{8.4}{\rmdefault}{\mddefault}{\updefault}{\color[rgb]{0,0,0}$(s_{3},s_{4})_{-}$}%
}}}}
\put(4201,-13861){\makebox(0,0)[lb]{\smash{{\SetFigFont{7}{8.4}{\rmdefault}{\mddefault}{\updefault}{\color[rgb]{0,0,0}$(s_{1},s_{2})_{-}$}%
}}}}
\put(3601,-10111){\makebox(0,0)[lb]{\smash{{\SetFigFont{7}{8.4}{\rmdefault}{\mddefault}{\updefault}{\color[rgb]{0,0,0}$(s_{4}, s_{5})_{-}$}%
}}}}
\put(3001,-12811){\makebox(0,0)[lb]{\smash{{\SetFigFont{7}{8.4}{\rmdefault}{\mddefault}{\updefault}{\color[rgb]{0,0,0}$(s_{2},s_{3})_{-}$}%
}}}}
\put(-224,-12136){\makebox(0,0)[lb]{\smash{{\SetFigFont{7}{8.4}{\rmdefault}{\mddefault}{\updefault}{\color[rgb]{0,0,0}$(s_{1},s_{2})_{+}$}%
}}}}
\put(-374,-9661){\makebox(0,0)[lb]{\smash{{\SetFigFont{7}{8.4}{\rmdefault}{\mddefault}{\updefault}{\color[rgb]{0,0,0}$(s_{3},s_{4})_{+}$}%
}}}}
\put(1313,-10561){\makebox(0,0)[lb]{\smash{{\SetFigFont{7}{8.4}{\rmdefault}{\mddefault}{\updefault}{\color[rgb]{0,0,0}$(s_{2},s_{3})_{+}$}%
}}}}
\end{picture}%

%% file: draws/LinearPMC.tex
\begin{picture}(0,0)%
\includegraphics{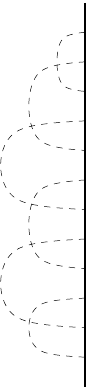}%
\end{picture}%
\setlength{\unitlength}{1243sp}%
\begingroup\makeatletter\ifx\SetFigFont\undefined%
\gdef\SetFigFont#1#2#3#4#5{%
  \reset@font\fontsize{#1}{#2pt}%
  \fontfamily{#3}\fontseries{#4}\fontshape{#5}%
  \selectfont}%
\fi\endgroup%
\begin{picture}(1333,5916)(1401,-5494)
\end{picture}%

%% file: draws/Kauffman.tex
\begin{picture}(0,0)%
\includegraphics{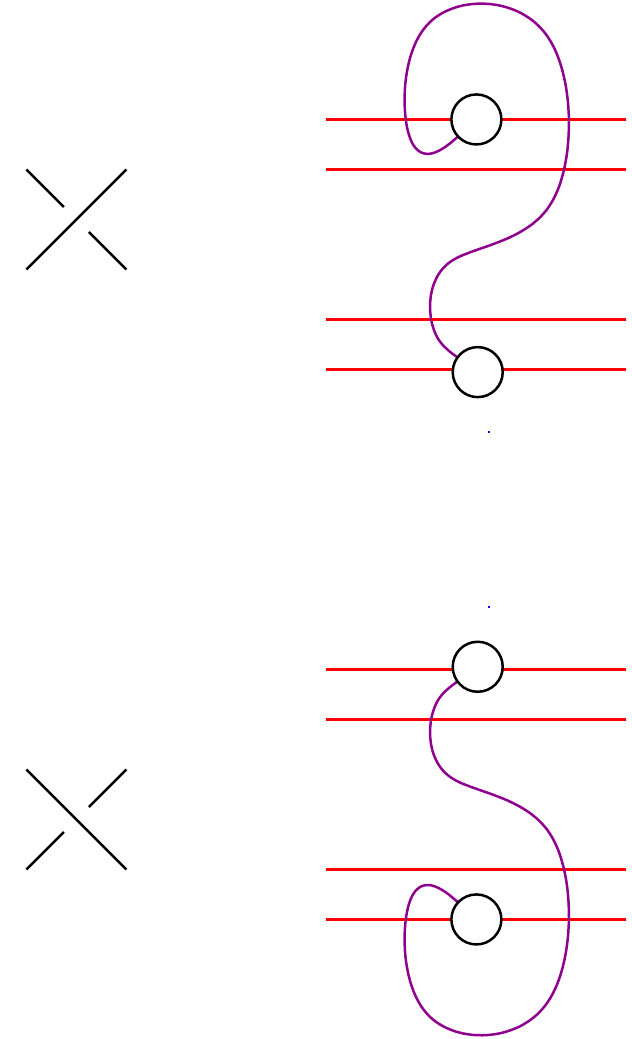}%
\end{picture}%
\setlength{\unitlength}{1579sp}%
\begingroup\makeatletter\ifx\SetFigFont\undefined%
\gdef\SetFigFont#1#2#3#4#5{%
  \reset@font\fontsize{#1}{#2pt}%
  \fontfamily{#3}\fontseries{#4}\fontshape{#5}%
  \selectfont}%
\fi\endgroup%
\begin{picture}(7548,12446)(-314,-9584)
\put(1201,-6961){\makebox(0,0)[lb]{\smash{{\SetFigFont{10}{12.0}{\rmdefault}{\mddefault}{\updefault}{\color[rgb]{0,0,0}$E$}%
}}}}
\put(451,-8011){\makebox(0,0)[lb]{\smash{{\SetFigFont{10}{12.0}{\rmdefault}{\mddefault}{\updefault}{\color[rgb]{0,0,0}$S$}%
}}}}
\put(451,-6211){\makebox(0,0)[lb]{\smash{{\SetFigFont{10}{12.0}{\rmdefault}{\mddefault}{\updefault}{\color[rgb]{0,0,0}$N$}%
}}}}
\put(-299,-6961){\makebox(0,0)[lb]{\smash{{\SetFigFont{10}{12.0}{\rmdefault}{\mddefault}{\updefault}{\color[rgb]{0,0,0}$W$}%
}}}}
\put(4051,-8611){\makebox(0,0)[lb]{\smash{{\SetFigFont{10}{12.0}{\rmdefault}{\mddefault}{\updefault}{\color[rgb]{0,0,0}$W$}%
}}}}
\put(6601,-8611){\makebox(0,0)[lb]{\smash{{\SetFigFont{10}{12.0}{\rmdefault}{\mddefault}{\updefault}{\color[rgb]{0,0,0}$E$}%
}}}}
\put(6451,-7411){\makebox(0,0)[lb]{\smash{{\SetFigFont{10}{12.0}{\rmdefault}{\mddefault}{\updefault}{\color[rgb]{0,0,0}$S$}%
}}}}
\put(4201,-6211){\makebox(0,0)[lb]{\smash{{\SetFigFont{10}{12.0}{\rmdefault}{\mddefault}{\updefault}{\color[rgb]{0,0,0}$N$}%
}}}}
\put(4201,-811){\makebox(0,0)[lb]{\smash{{\SetFigFont{10}{12.0}{\rmdefault}{\mddefault}{\updefault}{\color[rgb]{0,0,0}$S$}%
}}}}
\put(6451,389){\makebox(0,0)[lb]{\smash{{\SetFigFont{10}{12.0}{\rmdefault}{\mddefault}{\updefault}{\color[rgb]{0,0,0}$N$}%
}}}}
\put(6601,1589){\makebox(0,0)[lb]{\smash{{\SetFigFont{10}{12.0}{\rmdefault}{\mddefault}{\updefault}{\color[rgb]{0,0,0}$E$}%
}}}}
\put(4051,1589){\makebox(0,0)[lb]{\smash{{\SetFigFont{10}{12.0}{\rmdefault}{\mddefault}{\updefault}{\color[rgb]{0,0,0}$W$}%
}}}}
\put(1201,239){\makebox(0,0)[lb]{\smash{{\SetFigFont{10}{12.0}{\rmdefault}{\mddefault}{\updefault}{\color[rgb]{0,0,0}$E$}%
}}}}
\put(-299,239){\makebox(0,0)[lb]{\smash{{\SetFigFont{10}{12.0}{\rmdefault}{\mddefault}{\updefault}{\color[rgb]{0,0,0}$W$}%
}}}}
\put(451,-661){\makebox(0,0)[lb]{\smash{{\SetFigFont{10}{12.0}{\rmdefault}{\mddefault}{\updefault}{\color[rgb]{0,0,0}$S$}%
}}}}
\put(451,989){\makebox(0,0)[lb]{\smash{{\SetFigFont{10}{12.0}{\rmdefault}{\mddefault}{\updefault}{\color[rgb]{0,0,0}$N$}%
}}}}
\end{picture}%

%% file: draws/KauffmanStates.tex
\begin{picture}(0,0)%
\includegraphics{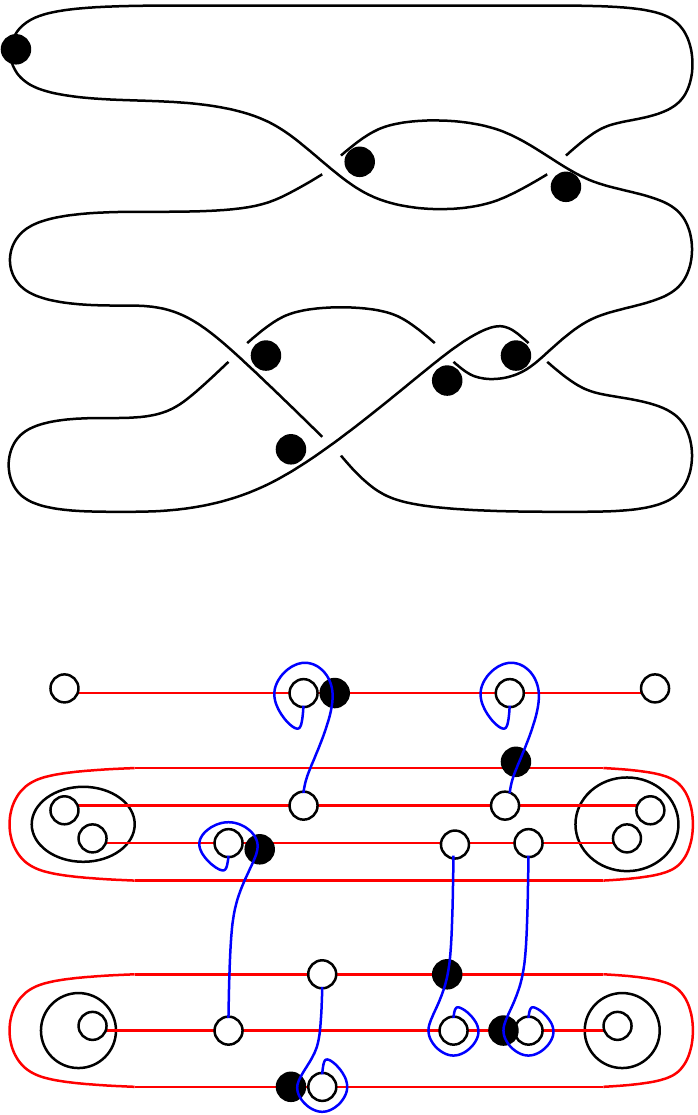}%
\end{picture}%
\setlength{\unitlength}{1579sp}%
\begingroup\makeatletter\ifx\SetFigFont\undefined%
\gdef\SetFigFont#1#2#3#4#5{%
  \reset@font\fontsize{#1}{#2pt}%
  \fontfamily{#3}\fontseries{#4}\fontshape{#5}%
  \selectfont}%
\fi\endgroup%
\begin{picture}(8347,13341)(-340,-13969)
\end{picture}%

%% file: draws/GreeneDiag.tex
\begin{picture}(0,0)%
\includegraphics{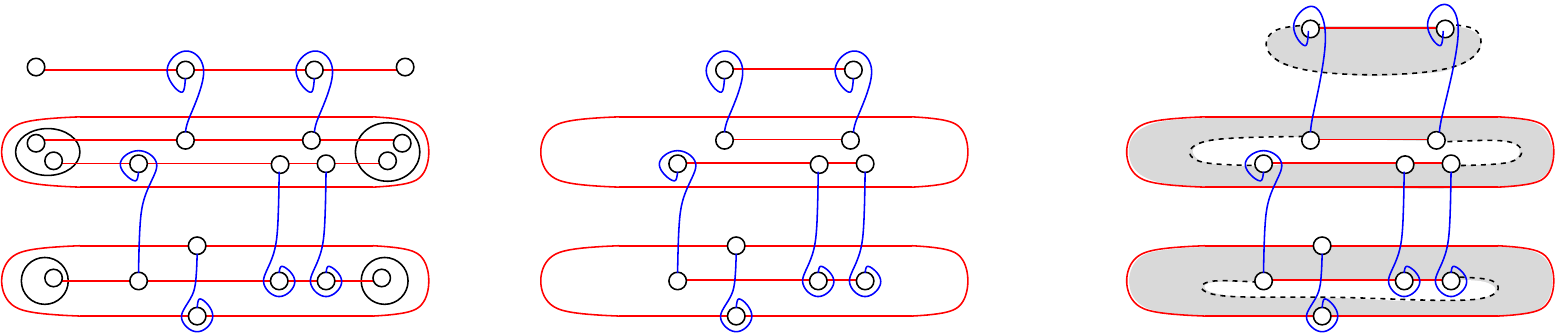}%
\end{picture}%
\setlength{\unitlength}{987sp}%
\begingroup\makeatletter\ifx\SetFigFont\undefined%
\gdef\SetFigFont#1#2#3#4#5{%
  \reset@font\fontsize{#1}{#2pt}%
  \fontfamily{#3}\fontseries{#4}\fontshape{#5}%
  \selectfont}%
\fi\endgroup%
\begin{picture}(29866,6349)(-259,-13969)
\end{picture}%

%% file: draws/LocalGreene.tex
\begin{picture}(0,0)%
\includegraphics{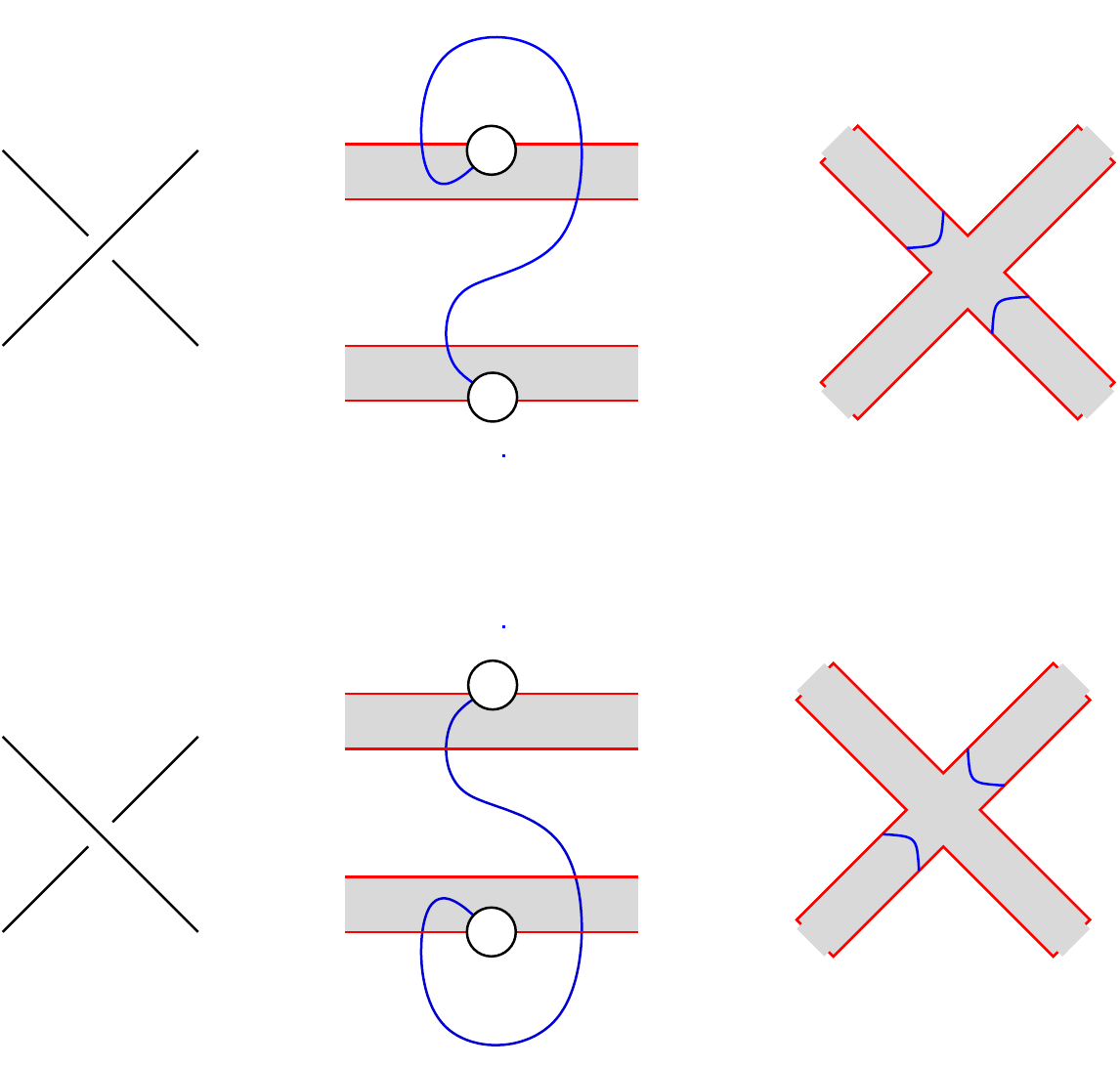}%
\end{picture}%
\setlength{\unitlength}{1579sp}%
\begingroup\makeatletter\ifx\SetFigFont\undefined%
\gdef\SetFigFont#1#2#3#4#5{%
  \reset@font\fontsize{#1}{#2pt}%
  \fontfamily{#3}\fontseries{#4}\fontshape{#5}%
  \selectfont}%
\fi\endgroup%
\begin{picture}(13716,13075)(-632,-9797)
\put(4201,-9661){\makebox(0,0)[lb]{\smash{{\SetFigFont{10}{12.0}{\rmdefault}{\mddefault}{\updefault}{\color[rgb]{0,0,1}$\beta^{\infty}$}%
}}}}
\put(4501,2939){\makebox(0,0)[lb]{\smash{{\SetFigFont{10}{12.0}{\rmdefault}{\mddefault}{\updefault}{\color[rgb]{0,0,1}$\beta^{\infty}$}%
}}}}
\put(  1,-8761){\makebox(0,0)[lb]{\smash{{\SetFigFont{10}{12.0}{\rmdefault}{\mddefault}{\updefault}{\color[rgb]{0,0,0}$\sigma_i$}%
}}}}
\put(  1,-1561){\makebox(0,0)[lb]{\smash{{\SetFigFont{10}{12.0}{\rmdefault}{\mddefault}{\updefault}{\color[rgb]{0,0,0}$\sigma_i^{-1}$}%
}}}}
\put(4951,-811){\makebox(0,0)[lb]{\smash{{\SetFigFont{10}{12.0}{\rmdefault}{\mddefault}{\updefault}{\color[rgb]{0,0,0}$S$}%
}}}}
\put(6451,389){\makebox(0,0)[lb]{\smash{{\SetFigFont{10}{12.0}{\rmdefault}{\mddefault}{\updefault}{\color[rgb]{0,0,0}$N$}%
}}}}
\put(6451,1739){\makebox(0,0)[lb]{\smash{{\SetFigFont{10}{12.0}{\rmdefault}{\mddefault}{\updefault}{\color[rgb]{0,0,0}$E$}%
}}}}
\put(10951,839){\makebox(0,0)[lb]{\smash{{\SetFigFont{10}{12.0}{\rmdefault}{\mddefault}{\updefault}{\color[rgb]{0,0,0}$S$}%
}}}}
\put(12151,-361){\makebox(0,0)[lb]{\smash{{\SetFigFont{10}{12.0}{\rmdefault}{\mddefault}{\updefault}{\color[rgb]{0,0,0}$E$}%
}}}}
\put(10951,-1261){\makebox(0,0)[lb]{\smash{{\SetFigFont{10}{12.0}{\rmdefault}{\mddefault}{\updefault}{\color[rgb]{0,0,0}$N$}%
}}}}
\put(7351,989){\makebox(0,0)[lb]{\smash{{\SetFigFont{10}{12.0}{\rmdefault}{\mddefault}{\updefault}{\color[rgb]{0,0,0}$NE$}%
}}}}
\put(12901,-2011){\makebox(0,0)[lb]{\smash{{\SetFigFont{10}{12.0}{\rmdefault}{\mddefault}{\updefault}{\color[rgb]{0,0,0}$NE$}%
}}}}
\put(6451,-7411){\makebox(0,0)[lb]{\smash{{\SetFigFont{10}{12.0}{\rmdefault}{\mddefault}{\updefault}{\color[rgb]{0,0,0}$S$}%
}}}}
\put(6526,-8536){\makebox(0,0)[lb]{\smash{{\SetFigFont{10}{12.0}{\rmdefault}{\mddefault}{\updefault}{\color[rgb]{0,0,0}$E$}%
}}}}
\put(11776,-6661){\makebox(0,0)[lb]{\smash{{\SetFigFont{10}{12.0}{\rmdefault}{\mddefault}{\updefault}{\color[rgb]{0,0,0}$E$}%
}}}}
\put(10651,-5836){\makebox(0,0)[lb]{\smash{{\SetFigFont{10}{12.0}{\rmdefault}{\mddefault}{\updefault}{\color[rgb]{0,0,0}$S$}%
}}}}
\put(10576,-7786){\makebox(0,0)[lb]{\smash{{\SetFigFont{10}{12.0}{\rmdefault}{\mddefault}{\updefault}{\color[rgb]{0,0,0}$N$}%
}}}}
\put(7201,-8011){\makebox(0,0)[lb]{\smash{{\SetFigFont{10}{12.0}{\rmdefault}{\mddefault}{\updefault}{\color[rgb]{0,0,0}$SE$}%
}}}}
\put(7201,-8011){\makebox(0,0)[lb]{\smash{{\SetFigFont{10}{12.0}{\rmdefault}{\mddefault}{\updefault}{\color[rgb]{0,0,0}$SE$}%
}}}}
\put(12601,-5011){\makebox(0,0)[lb]{\smash{{\SetFigFont{10}{12.0}{\rmdefault}{\mddefault}{\updefault}{\color[rgb]{0,0,0}$SE$}%
}}}}
\put(4051,1664){\makebox(0,0)[lb]{\smash{{\SetFigFont{10}{12.0}{\rmdefault}{\mddefault}{\updefault}{\color[rgb]{0,0,0}$W$}%
}}}}
\put(4426,-6286){\makebox(0,0)[lb]{\smash{{\SetFigFont{10}{12.0}{\rmdefault}{\mddefault}{\updefault}{\color[rgb]{0,0,0}$N$}%
}}}}
\put(3976,-8536){\makebox(0,0)[lb]{\smash{{\SetFigFont{10}{12.0}{\rmdefault}{\mddefault}{\updefault}{\color[rgb]{0,0,0}$W$}%
}}}}
\put(9601,-6961){\makebox(0,0)[lb]{\smash{{\SetFigFont{10}{12.0}{\rmdefault}{\mddefault}{\updefault}{\color[rgb]{0,0,0}$W$}%
}}}}
\put(9751, 89){\makebox(0,0)[lb]{\smash{{\SetFigFont{10}{12.0}{\rmdefault}{\mddefault}{\updefault}{\color[rgb]{0,0,0}$W$}%
}}}}
\end{picture}%